\documentclass[a4paper,11pt]{amsart}

\usepackage[utf8]{inputenc}

\usepackage{graphicx}
\usepackage{tikz}
\usetikzlibrary{decorations.pathreplacing}
\usepackage{subcaption}
\usepackage{color}

\usepackage[colorlinks=true,urlcolor=blue,
citecolor=red,linkcolor=blue,linktocpage,pdfpagelabels,
bookmarksnumbered,bookmarksopen]{hyperref}

\usepackage{ae}
\usepackage[utf8]{inputenc}
\numberwithin{figure}{section}

\usepackage{mathrsfs,a4wide}
\usepackage{esint}
\usepackage{color}
\usepackage{import}

\usepackage[leqno]{amsmath}

\numberwithin{figure}{section}

\newtheorem{theorem}{Theorem}[section]
\newtheorem*{mtheorem}{Main Theorem}
\newtheorem{lemma}[theorem]{Lemma}
\newtheorem{proposition}[theorem]{Proposition}
\newtheorem{corollary}[theorem]{Corollary}
\theoremstyle{definition}
\newtheorem{definition}[theorem]{Definition}

\newtheorem{remark}[theorem]{Remark}
\numberwithin{equation}{section}

\newcommand{\curl}{\operatorname{curl}}

\newcommand{\beq}{\begin{equation}}
\newcommand{\eeq}{\end{equation}}
\newcommand{\Div}{\operatorname{div}}

\def\R{\mathbb R}

\def\N{\mathbb N}
\def\H{\mathcal H}

\def\D{\mathcal D}
\def\e{\varepsilon}
\def\pa{\partial}
\def\St{\Sigma_t}
\def\En{\mathcal{E}}
\newcommand{\la}{\langle}
\newcommand{\ra}{\rangle}
\def \Qt{Q(t)}
\def\Cap{\text{Cap}}

\def\restrict#1{\raise-.5ex\hbox{\ensuremath|}_{#1}}

\title[A priori estimates for the motion of charged liquid drop]{A priori estimates for the motion of charged liquid drop:  A dynamic approach via free boundary Euler equations}

\author[v. Julin, D. A. La Manna]{Vesa Julin, Domenico Angelo La Manna}

\address[V. Julin]{University of Jyv\"{a}skyl\"{a},
Department of Mathematics and Statistics,
P.O.Box 35 (MaD) FI-40014, Finland}
\email{vesa.julin@jyu.fi}

\address[D. A. La Manna]{Universit\`{a} di Napoli ``Federico II'',
Dipartimento di Matematica e Applicazioni,
Via Cintia, Monte Sant'Angelo, I-80126 Napoli, Italy}
\email{domenicolamanna@hotmail.it}

\date{\today}

\begin{document}

\begin{abstract} 
We study the motion of charged liquid drop in three dimensions where the equations of motions are given by the Euler equations with free boundary with an  electric field. This is a well-known problem in physics going back to the famous work by Rayleigh. Due to experiments and numerical simulations one may expect the charged drop to form conical singularities called Taylor cones, which we interpret as singularities of the flow. In this paper, we study the well-posedness of the problem and regularity of the solution. Our main theorem is a criterion which roughly states that if the flow remains $C^{1,\alpha}$-regular in shape and the velocity remains  Lipschitz-continuous, then the flow remains smooth, i.e., $C^\infty$ in time and space, assuming that the initial data is smooth. Our main focus is on the regularity of the shape of the drop. Indeed, due to the appearance of Taylor cones, which are singularities  with Lipschitz-regularity, we expect  the $C^{1,\alpha}$-regularity assumption to be optimal. We also quantify the $C^\infty$-regularity via high order energy estimates which, in particular, implies the well-posedness of the problem.    

\medskip

\noindent \textbf{  MSC classes}: 35Q35, 76B03, 76B07.

\end{abstract}


\maketitle
\tableofcontents

\section{Introduction and the main result}

\subsection{State-of-the-art}

In this paper we study the problem of charged liquid drop from rigorous mathematical point of view. In the model the two effecting forces are the surface tension, which prefers to keep the drop spherical, and the repulsive electrostatic force, which both act on the boundary of the drop. The problem is well-known and goes back to Rayleigh \cite{Ray} who studied the linear stability of the sphere and showed that the sphere becomes unstable when the total electric charge is above a given threshold. When the total electric charge is above this Rayleigh threshold, the drop begins to elongate and may eventually form a conical singularity at the tip with a certain opening angle. Such singularities are called \emph{Taylor cones} due to the work by Taylor \cite{Tay} and the numerical and experimental evidence suggest that the charged drop typically forms such a singularity \cite{GGS, Mik, Tay, Zel}. In this paper our goal is to study the well-posedness of the problem and the regularity of the solution. We refer to \cite{MN} and \cite{GR} for an introduction to the topic.

The static problem of charged liquid drop can be seen as a nonlocal isoperimetric problem and it
has been studied from the point of view of Calculus of Variations in recent years \cite{GNR, GNR2, La, MN}. The main issue is that the
associated minimization problem, formulated in the framework of Calculus of Variations, is not
well-posed, in the sense that the problem does not have a minimizer \cite{GNR, MN}. Even more surprising is that the results in \cite{GNR, MN} show that even if the total electric charge is below the  Rayleigh  threshold, the sphere is not a local minimizer of the associated energy. This means that the electrostatic term is not lower order with respect to the surface tension, which makes the problem mathematically challenging. In order to make the variational problem well-posed one may restrict the problem to convex sets \cite{GNR2} or regularize the functional by adding a curvature term \cite{GR} which could lead to the existence of minimizer as the result in \cite{GNRo} suggests.   

Here we  study this problem from the point of view of fluid-dynamics, which is the framework studied e.g. in  \cite{GGS}, where the authors derive the PDE system in the irrotational case (see also \cite{Bas}). Indeed, as it is observed in \cite{GGS} the problem is by nature evolutionary, where the drop deforms as a function of time given by the Euler equations with the
surface of the drop being the free boundary, which law of motions is coupled with the system which we give in \eqref{system} below. The problem can thus be seen as the Euler equations for incompressible fluids with free boundary with an additional term given by the electric field. The Euler equations with free boundary without the electric field has been studied rather extensively in recent years.  We give only a brief overview on this challenging problem below and refer to   \cite{CS3, MR} for more detailed introduction to the topic. Regarding the problem with electric field we mention the resent works by Yang \cite{Yang1, Yang2} and Wang-Yang \cite{WY}, where the authors study the case of the water-wave problem. We also mention the work \cite{FF}, where the authors study the Stokes flow associated with the charged liquid drop near the sphere and show that under smallness assumption the flow is well defined.  

We stress that in our case it is crucial to include the surface tension in the model since otherwise the problem might be ill-posed. For the problem without the electric field one may study the Euler equations also without the surface tension, when one assumes the so called Rayleigh-Taylor sign condition \cite{ebin}, which one should not confuse with the Rayleigh threshold mentioned above. For the water-wave problem the well-posedness is proven by Wu \cite{Wu1, Wu2} and the general case is due to  Lindblad \cite{Lin}, see also \cite{AM, Lan}. Concerning the problem with surface tension, which is closer to ours, the short time existence of solution in the irrotational case for starshaped sets is due to Beyer-G\"unther \cite{BG1, BG2} and the general case is proven by Coutand-Shkoller \cite{CS1}. We also mention the earlier works concerning the well-posedness of the problem in the planar case \cite{Amb, Igu, Yos}.  The works that are closest to ours are Shatah-Zheng \cite{CPAM} and Schweizer \cite{Shw}, where the authors prove regularity estimates for the free boundary Euler equations with surface tension. Our work is also inspired by  Masmoudi-Rousset \cite{MR}, where the authors prove similar estimates for the Euler equations without the surface tension. 

As usual with geometric evolution equations, the Euler equations with free boundary may develop singularities in finite time. In \cite{CS2} the authors construct an example where the equations develop singularities where the drop changes its topology. We stress that in the absence of the electric field, we do not expect the flow to develop conical singularities predicted by Taylor  \cite{Tay}, where both the curvature and the velocity become singular. Indeed, Taylor cones are special type of singularities as the evolving sets $\Omega_t$ do not change their topology, but  only lose their regularity. We also point out that the analysis in \cite{Tay} does not give a rigorous mathematical  proof for the fact  that the Taylor cone is a singularity of the associated flow. Indeed, since there is no monotonicity formula for the Euler equations, similar to the one by Huisken \cite{Hui, MantegazzaBook} for the mean curvature flow, there is little hope to have general classification of the singularities at the moment. We refer to \cite{FJ}  and  \cite{GVW} which both study critical points of energy functionals, which are very much similar to ours, with  conical or cusp-like singularities. Then again, as we interpret the Taylor cone as a singularity of the flow given by the Euler equations \eqref{system},  it is not clear why the singularity is a critical point of the potential energy.

\subsection{Statement of the Main Theorem}
We study the motion of an incompressible charged drop in vacuum in $\R^3$ and  denote the fluid domain by $\Omega_t$. We assume that we have an initial smooth and compact set $\Omega_0 \subset \R^3$ and a smooth initial velocity field $v_0: \Omega_0 \to \R^3$ which evolve to  a smooth family of sets and vector fields  $(\Omega_t, v(\cdot, t))_{t \in [0,T)}$.  The total energy is given by 
\beq
\label{energy}
J_t(\Omega_t, v) = \frac12  \int_{\Omega_t} |v(x,t)|^2 \, dx  + \H^2(\pa \Omega_t) + \frac{Q}{\text{Cap}(\Omega_t)},
\eeq 
where $Q>0$ and $\text{Cap}(\Omega)$ is the electrostatic capacity given by
\[
\text{Cap}(\Omega) : = \inf \big{\{} \int_{\R^3}\frac12|\nabla u|^2 \, dx : \,  u(x) \geq 1 \,\, \text{for all } \, x \in  \Omega \, , u \in \dot{H}^1(\R^3) \big{\}},
\]
and by $\H^2(\pa \Omega)$ we denote the two dimensional Hausdorff measure of the set $\pa \Omega$. Define the norm $\|u\|_{\dot{H}^1(\R^3)} = \|\nabla u\|_{L^2(\R^3)} + \|u\|_{L^6(\R^3)}$. We denote the capacitary potential as $U_\Omega  \in \dot{H}^1(\R^3)$ which is the function for which $\text{Cap}(\Omega)  = \int_{\R^3} \frac12|\nabla U_\Omega|^2 \, dx$. This is equivalent to say that $U_\Omega \in \dot{H}^1(\R^3)$ satisfies
\beq \label{eq:def-capapot}
\begin{cases}
-\Delta U_\Omega = 0 \qquad \text{in }  \,  \R^3 \setminus \overline{\Omega} \\
U_\Omega = 1 \qquad \text{on }  \, \overline{\Omega}.
\end{cases}
\eeq
We denote the mean curvature of $\St = \pa \Omega_t$ by $H_{\St}$, which for us is the sum of the principal curvatures given by orientation via the outer normal $\nu_{\St}$. With this convention convex sets have positive mean curvature.   As usual we denote the material derivative of a vector field $F$ by  
\[
\D_t F = \pa_t F + (v\cdot \nabla)F. 
\]
The equations of motion are given by  the Euler equations with free boundary (for the derivation see \cite{GGS})
\begin{eqnarray}\label{system}
\begin{cases}
\D_t v  + \nabla p  =0 \qquad  \text{in } \,  \Omega_t
\\
 \Div  v=0 \qquad \text{in } \, \Omega_t
\\
v_n = V_t \qquad \text{on } \, \St = \pa  \Omega_t
\\
p= H_{\St} - \frac{Q}{2 C_t^2} |\nabla U_{\Omega_t}|^2 \qquad \text{on } \St,
\end{cases}
\end{eqnarray}
where $C_t = \text{Cap}(\Omega_t) $, $V_t$ is the normal velocity, $v_n = v \cdot \nu$ and $p$ is the pressure. We say that the system \eqref{system} has a smooth solution in time-interval $(0,T)$ with initial data  $(\Omega_0, v_0)$, if there is a family of $C^\infty$-diffeomorphisms $(\Phi_t)_{t \in [0,T)}$, which depend smoothly on $t$, such that $\Phi_0 = id$ and $\Phi_t(\Omega_0) = \Omega_t$, the functions $v(t, \Phi(t,x))$ and $p(t, \Phi(t,x))$ are smooth and the  equations hold in the classical sense. Moreover we require that $v(t, \Phi(t,\cdot)) \to v_0$ as $t \to 0$, where $v_0: \Omega_0 \to \R^3$ is the initial velocity field.    
When the total electric charge is zero, i.e.   $Q= 0$, the system reduces to the more familiar Euler equations with free boundary with surface tension. We stress that formally it may seem that the term given by the electric field in the pressure is of lower order than the curvature. However, even for Lipschitz domains this naive intuition fails as we will observe in the beginning of Section 3.

The characteristic property of the solution of  \eqref{system}  is the conservation of the energy \eqref{energy}, i.e., 
\[
\frac{d}{dt} J(\Omega_t, v) = 0, 
\]
which follows from straightforward calculation. Therefore one could guess, and we will prove this in our main theorem, that  assuming that the flow given by the system  \eqref{system} does not develop singularities, then it  preserves the regularity of the initial data $(\Omega_0, v_0)$ at least  in sense of certain Sobolev norm. In particular, we point out that, unlike the mean curvature flow \cite{MantegazzaBook}, the flow given by the system  \eqref{system} is not smoothing.   

We parametrize the moving boundary $\Sigma_t = \pa \Omega_t$ by using  a fixed reference surface $\Gamma$ which we assume to be smooth and compact. We use   the  height function  parametrization which means that for every $t$ we associate  the function $h(\cdot, t) : \Gamma \to \R$  with  the moving boundary $\St$ as
\[
\Sigma_t = \{x+  h(x,t)\nu_{\Gamma}(x)  : x \in \Gamma\}.
\]
We assume that $\Gamma$ satisfies the interior and exterior ball condition with radius $\eta >0$ and note that $\eta$ is not necessarily small. For example, from application point of view a relevant case is when the initial set is star-shaped in which case it is natural to choose the  reference manifold to be a sphere  in which case $\eta$ is its radius. It is clear that the height-function parametrization is well defined as long as
\[
\sup_{t \in[0,T)} \|h(\cdot ,t)\|_{L^\infty(\Gamma) }< \eta.
\]
Therefore we define the quantity 
\begin{equation}
\label{eq:height_well}
\sigma_T := \eta - \sup_{t \in[0,T)} \|h(\cdot ,t)\|_{L^\infty(\Gamma) }
\end{equation}
and the  above condition reads as $\sigma_T >0$.

As in \cite{MR,Shw, CPAM} we note  that we do not consider the existence in this paper.  Instead, as in \cite{Shw} we assume that the following qualitative short time existence result holds. 

\emph{Throughout the paper we assume that for every smooth initial set and smooth initial velocity field the system \eqref{system} has a smooth solution which exists a short interval of time. }

Since we will prove a priori estimates, we expect the existence to follow from an argument in the spirit of  \cite{SZ2}.

In this paper we are interested in finding a priori estimates which guarantee that the system \eqref{system} does not develop singularities. To this aim we fix a small $\alpha>0$ and  define 
\begin{equation}
\label{def:apriori_est}
\Lambda_T:= \sup_{t \in[0,T)} \left( \|h(\cdot ,t)\|_{C^{1,\alpha}(\Gamma)} + \|\nabla v(\cdot,t) \|_{L^\infty(\Omega_t)} + \|v_n(\cdot,t) \|_{H^2(\St) }     \right).
\end{equation}
We note that $\alpha$ can be any positive number. We could also replace the $C^{1,\alpha}$-norm by the $C^{1,\text{Dini}}$-norm, but we choose to work with H\"older norms as the problem is already technically involved.   Our goal is to show that if the quantity $\Lambda_T$ is bounded then the flow can be extended beyond time $T$ and is smooth if the initial data is smooth. We will prove this in a quantitative way and define an energy quantity   of order $l \geq 1$ as
\begin{equation}
\label{def:Energy-tilde}
\hat E_l(t) := \sum_{k = 0}^l\|\D_t^{l+1-k} v\|_{H^{\frac32k}(\Omega_t)}^2 + \|v(\cdot, t)\|_{H^{\lfloor \frac32(l+1) \rfloor}(\Omega_t)}^2, 
\end{equation}
where $\lfloor \frac32(l+1) \rfloor$ denotes the integer part of  $\frac32(l+1)$. We define the Hilbert space for half-integers $H^{\frac32k}(\Omega_t)$ via extension in Section 2. In the last term we use a Hilbert space of integer order since it simplifies the calculations. The fact that the boundedness of $\tilde E_l(t)$  for every $l$ implies the smoothness of the flow will be clear from the results in Section 8. Indeed, we first show that the bound on $\hat E_l(t) $ implies a bound for  the pressure $p$. By the a priori estimate we know that the fluid domain remains $C^{1,\alpha}$-regular. We use this and estimates for harmonic functions to conclude that the bound on the pressure implies bound on the curvature  (see Lemma \ref{lem:press-curv}), which then  gives the regularity of the fluid domain $\Omega_t$.    

Our main result reads as follows. Recall that we assume that the reference surface $\Gamma \subset \R^3$ satisfies the interior and exterior ball condition with radius $\eta$.
\begin{mtheorem}
Assume that $\Omega_0$ is a smooth initial set which boundary satisfies $\pa \Omega_0 = \{x+  h_0(x)\nu_{\Gamma}(x)  : x \in \Gamma\}$ with 
$ \|h_0\|_{L^\infty(\Gamma)}< \eta$ and let $v_0  \in C^{\infty}(\Omega_0; \R^3)$ be the initial velocity  field. Assume that the system \eqref{system} has a smooth solution in  time-interval $[0,T)$  and the parametrization satisfies 
\beq \label{eq:apriori_est}
 \Lambda_T \leq M  \quad  \text{and} \quad  \sigma_T \geq \frac{1}{M}
\eeq
for some $M>0$, where $\sigma_T$ is defined in \eqref{eq:height_well} and $\Lambda_T$ in \eqref{def:apriori_est}. Then for every $l \in \N$ there is a constant $C_l$, which depends on $M, l, \hat E_l(0) $, and on $T$ if $T>1$, such that  the flow satisfies 
\[
\sup_{0 < t < T} \hat E_l(t) \leq C_l,
\]
where $\hat E_l(t)$ is defined in \eqref{def:Energy-tilde}. In particular, the system \eqref{system} does not develop singularity at time $T$, but remains quantitatively smooth.

Moreover, there are $T_0 >0 $ and $M$, which depend on $\sigma_0$, i.e., $\sigma_t$ at $t=0$,  $\|H_{\Sigma_0}\|_{H^2(\Sigma_0)}$, $\|v\|_{H^3(\Omega_0)}$ and the $C^{1,\alpha}$-norm of $h_0$, such that the a priori estimates \eqref{eq:apriori_est} hold for $M$  up to time $\hat T = \min\{ T, T_0\}$. 
\end{mtheorem}

Let us make a few comments on the Main Theorem. First,  from the point of view of the shape of the drop, the result says that if the parametrization of the flow remains $C^{1,\alpha}$-regular then the flow does not develop singularities. We expect this to be optimal in the sense that, we cannot  relax the $C^{1,\alpha}$-regularity to Lipschitz regularity as the flow may create conical singularities as discussed before. 

From the point of view of the velocity, the assumption on Lipschitz regularity of $v$, which is  stronger than the boundedness of the $\curl v$,   is in the spirit of  the Beale-Kato-Majda criterion and thus natural in the theory of the Euler equations \cite{MB}.  Indeed, in the case when the drop does not chance its shape, i.e., $\Omega_t = \Omega_0$ the condition \eqref{eq:apriori_est} reduces to 
\[
\sup_{t \in [0,T)} \|\nabla v (\cdot ,t)\|_{L^{\infty}(\Omega_0)}  < \infty,
\]
which guarantees that the equations do not develop singularities by standard results for the Euler-equations \cite{MB}. Whether one may remove this condition is beyond our reach at the moment as the gradient level estimates are a fundamental  problem in the theory of the Euler equations without the free boundary. The condition on the $H^2$-integrability of the normal component of the velocity $v$ on the other hand is related to the fact that the boundary $\St$ is moving. We do not expect this to be  optimal but again  this problem is too involved  for us to solve at the moment. Our main contribution to the problem is to find  the optimal sufficient condition for the shape of the drop which guarantee that the flow is well-defined and provide the regularity estimates of all order $l$. For a drop without surface tension similar type  of estimate is proven by Ginsberg \cite{Gin} with an a priori assumption on the uniform curvature bound.

Finally the last statement of the Main Theorem says that the first statement is not empty, i.e., that the a priori estimates stay bounded up to time $T_0$, which depends on the initial data by requiring that $\|H_{\Sigma_0}\|_{H^2(\Sigma_0)}$ and $\| v_0\|_{H^3(\Omega_0)}$ are bounded. We also note that since the regularity estimates in Main Theorem are quantitative, the result can be applied for non-smooth initial data by standard approximation.  We note that all quantities in the paper depend of course on the chosen reference surface $\Gamma$ even if it is not  explicitly mentioned.

\subsection{Overview of the proof and the structure of the paper}

As the paper is long we give a brief overview of the proof of the Main Theorem and of the structure of the paper. The proof is based on energy estimates and to that aim we define the energy functional of order $l \geq 1$ as 
\[
\begin{split}
\En_l(t) =\frac12 \int_{\Omega_t} &|\D_t^{l+1} v|^2 \, dx +  \frac12\int_{\Sigma_t} |\nabla_\tau (\D_t^l v \cdot \nu)|^2 \, d \H^2 \\
&-  \frac{Q}{2C_t^2} \int_{\Omega_t^c} |\nabla (\pa_t^{l+1} U_{\Omega_t}) |^2 \, dx  +  \int_{\Omega_t} |\nabla^{\lfloor \frac12(3l+1) \rfloor}( \curl v)|^2 \, dx,
\end{split}
\]
which is similar to the quantity in \cite{Shw} defined on graphs. Here $v$ is the velocity, $\D_t^l v$ is the material derivative of order $l$ and $\lfloor \frac12(3l+1) \rfloor$ denotes the largest integer smaller than $\frac12(3l+1)$. Note that we need an additional term involving the time derivative of the capacitary potential $U_{\Omega_t}$, as it  appears as a high order term in the linearization of the pressure (see Lemma \ref{formula:Dtp}). This additional term causes problems as it is not immediately clear why the energy is positive or even bounded from below. We also define  the associated energy quantity, where we include  the spatial regularity
\[
E_l(t) =  \sum_{k=0}^{l}  \|\D_t^{l+1-k}v\|_{H^{\frac32 k}(\Omega_t)}^2 + \|v\|_{H^{\lfloor \frac32(1+1)\rfloor}(\Omega_t)}^2+ 
 \|\D_t^l v \cdot \nu \|_{H^1(\Sigma_t)}^2+1.
\]
Note that this quantity takes into account the  natural scaling of the system \eqref{system}, where time scales of order $\frac32$ with respect to space as observed e.g. in \cite{CPAM}. 

We prove high order energy estimates by first showing that if $\Lambda_T, \sigma_T$ satisfy \eqref{eq:apriori_est} then for all $t<T$ it holds 
\beq \label{eq:energy1-intro}
\frac{d}{dt} \En_l(t) \leq C_l E_l(t)
\eeq
for $l \geq 2$. The novelty of \eqref{eq:energy1-intro} is that the RHS has linear and not polynomial dependence on $E_l(t)$, which is crucial in order to show that the flow remains smooth as long as \eqref{eq:apriori_est} holds. This makes the proof technically challenging as we need to estimate all nonlinear error terms in $\frac{d}{dt} \En_l(t)$ in an optimal way. We complete the argument by proving 
\beq \label{eq:energy2-intro}
 E_l(t) \leq C_l(\tilde C_l + \En_l(t))
\eeq
which holds for $l \geq 1$. The inequalities \eqref{eq:energy1-intro} and \eqref{eq:energy2-intro} then imply the  energy estimates for $l\geq 2$ and the quantitative $C^\infty$-regularity of the flow. 

We prove  \eqref{eq:energy1-intro} and \eqref{eq:energy2-intro} by an induction argument over $l$, where the constants depend on $\sup_{t<T} E_{l-1}(t)$ which is bounded by the previous step. Therefore the  first challenge is to start the argument and to bound $E_1(t)$. The issue is that \eqref{eq:energy1-intro} does not hold for $l=1$. Instead, we show  a weaker estimate  
\beq \label{eq:energy3-intro}
\frac{d}{dt} \En_1(t) \leq C_1 (1+ \| p\|_{H^2(\Omega_t)}^2) E_1(t),
\eeq
which we expect to be sharp. Therefore in order to start the induction argument we use an ad-hoc argument to show  
\beq \label{eq:energy4-intro}
\int_0^T \| p\|_{H^2(\Omega_t)}^2 \, dt \leq  C.
\eeq
The inequalities \eqref{eq:energy2-intro}, \eqref{eq:energy3-intro} and \eqref{eq:energy4-intro} then imply the first order energy estimate. 

We show \eqref{eq:energy4-intro} by studying the function
\[
\Phi(t) = -\int_{\St} p \, \Delta_{\St} v_n \, d \H^2,
\]
where $p$ is the pressure and $\Delta_{\St}$ the Laplace-Beltrami operator, and prove that it holds
\[
\frac{d}{dt}\Phi(t)  \leq - \frac13 \| p\|_{H^2(\Omega_t)}^2 + \text{lower order terms}.
\]
We show that the a priori estimates \eqref{eq:apriori_est} imply that $\Phi$ is bounded and thus we obtain \eqref{eq:energy4-intro} by integrating the above inequality over $(0,T)$. We point out that the low order energy estimate is the most challenging part of the proof as we have to work with domains with low regularity. We need rather deep results from differential geometry, boundary regularity for harmonic functions and elliptic regularity in order to overcome this problem. Let us finally outline the structure of the paper.

In \textbf{Section 2} we introduce our notation. Due to the presence of the surface tension, the problem is geometrically involved and we need notation and tools from differential geometry to overcome these issues. We also define the function spaces that we need which include the Hilbert spaces in the domain  $H^{k}(\Omega)$ and on the boundary $H^{k}(\Sigma)$ for half-integers $k =0, \frac12, 1, \dots$. We also recall functional inequalities such as interpolation inequality and Kato-Ponce inequality. 

 In \textbf{Section 3} we prove div-curl type estimates in order to transform the  high order energy estimates into regularity for the shape and the velocity.  We  first recall the result from \cite{CS} and prove its lower order version in Theorem \ref{prop:CS-k=1}.  In Theorem \ref{teo:reg-capa} we prove sharp boundary regularity estimates for harmonic functions with Dirichlet boundary data by using methods from \cite{FJM3D}. We believe  that these two results are of independent interest.

In \textbf{Section 4} we derive commutation formulas as in \cite{CPAM} and formula for the material derivatives of the pressure on the moving boundary. These formulas include four different error terms  which we bound  in \textbf{Section 5}. All the error terms have different structure and therefore we need to treat them one by one, which makes the Section 5 long. The further difficulty is due to the fact that the time and space derivatives have different scaling. 

The core of the proof of the Main Theorem is in the next three sections. In \textbf{Section 6} we prove \eqref{eq:energy4-intro}, in \textbf{Section 7} we prove \eqref{eq:energy1-intro} and \eqref{eq:energy3-intro}, and   in  \textbf{Section 8} we prove \eqref{eq:energy2-intro}. The short final section then contains the proof of the Main Theorem. 

\section{Notation and Preliminary results}

In this section we introduce our notation and recall some basic results on function spaces and geometric inequalities. Many of these results are well-known for experts but we include them since they might be difficult to find, while some results we did not find at all in the existing literature.    Throughout the paper $C$ denotes a large constant, which value may change from line to line.

We first introduce notation related to Riemannian geometry.  As an introduction to the topic we refer to \cite{Lee}. We will always deal with compact hypersurfaces $\Sigma \subset \R^3$, which then can be seen as boundaries of sets $\Omega$, i.e., $\pa \Omega = \Sigma$. We denote its outer unit normal by $\nu_{\Omega}$ and denote it sometimes by $ \nu_{\Sigma}$ or merely $\nu$ when its meaning is obvious from the context. We use the outward orientation and denote the second fundamental form by $B_\Sigma$ and the mean curvature by $H_\Sigma$, which is defined as the sum of the principal curvatures. Again we write simply $B$ and $H$ when the meaning is clear from the context. We note that we use the convention in our notation that $\Sigma = \pa \Omega$ denotes a generic surface, $\St = \pa \Omega_t$ denotes the evolving surface given by the equations \eqref{system} and $\Gamma = \pa G$ is our reference surface which we introduce later.  We note that the constants in the paper will depend on the chosen reference surface. We take this for granted and do not mention it in the statements.

Since $\Sigma$ is  embedded in $\R^{3}$  it has natural metric $g$ induced by the Euclidian metric. Then $(\Sigma, g)$ is a Riemannian manifold and we denote the inner product on each tangent space $X, Y \in T_x \Sigma$ by $\la X, Y \ra$, which we may write in local coordinates as 
\[
\la X,Y \ra = g(X,Y) = g_{ij} X^iY^j.
\]
We extend the inner product in a natural way for tensors. We denote smooth vector fields on $\Sigma$ by $\mathscr{T}(\Sigma)$ and by  a slight abuse of notation we denote smooth $k$th order tensor fields on $\Sigma$  by $\mathscr{T}^k(\Sigma)$. We write $X^i$ for vectors and $Z_i$ for covectors in local coordinates. 

We denote the Riemannian connection on  $\Sigma$ by $\bar \nabla$ and recall that  for a function $u \in C^\infty(\Sigma)$ the covariant derivative  $\bar \nabla u $ is a $1$-tensor field defined for  $X  \in \mathscr{T}( \Sigma)$  as
\[
\bar \nabla u(X)  = \bar \nabla_X u = X u,
\]
i.e., the derivative  of $u$ in the direction of $X$.  The  covariant derivative  of  a smooth $k$-tensor field $F \in \mathscr{T}^k( \Sigma)$, denoted  by  $\bar \nabla F$, is a $(k+1)$-tensor field    and we have the following recursive formula for $ Y_1, \dots, Y_k, X \in \mathscr{T}( \Sigma)$  
\[
\bar \nabla F(Y_1, \dots, Y_k, X) = (\bar \nabla_X F)(Y_1, \dots, Y_k) ,
\]
where
\[
(\bar \nabla_X F)(Y_1, \dots, Y_k) = X F(Y_1, \dots, Y_k) - \sum_{i=1}^k F(Y_1, \dots,  \bar \nabla_X Y_i ,\dots, Y_k).
\]
Here $\bar \nabla_X Y$ is  the covariant derivative of $Y$ in the direction of $X$ (see \cite{Lee}) and since $\bar \nabla$ is the Riemannian connection it holds  $\bar \nabla_X Y = \bar \nabla_Y X  + [X,Y]$ for every $X, Y \in \mathscr{T}( \Sigma)$. We denote the $k$th order covariant derivative of a function $u$ on $\Sigma$ by $\bar \nabla^k u \in \mathscr{T}^k( \Sigma)$. The notation $\bar \nabla_{i_1} \cdots \bar \nabla_{i_k} u$ means a  coefficient of $\bar \nabla^k u$ in local coordinates. We may raise the index of $\bar \nabla_i u$ by using the inverse of the metric tensor  $g^{ij}$  as $\bar \nabla^i u = g^{ij}\bar \nabla_j u$. We denote the divergence of a vector field $X \in \mathscr{T}(\Sigma)$ by  $\Div_{\Sigma} X$ and the Laplace-Beltrami operator for a function $u : \Sigma \to \R$ by $\Delta_\Sigma u$. We recall that by the divergence theorem 
\[
\int_{\Sigma} \Div_{\Sigma} X \, d \H^2 = 0.
\]

We will first fix our reference surface which we denote by $\Gamma$ which is a boundary of a smooth, compact set $G$, i.e., $\Gamma = \pa G$. Since $G$ is smooth it satisfies the interior and exterior ball condition with radius $\eta$,  and we denote the tubular neighborhood of $\Gamma$ by $\mathcal{N}_\eta(\Gamma)$ which is defined as
\[
\mathcal{N}_\eta(\Gamma) = \{ x \in \R^3: \text{dist}(x,\Gamma) < \eta\}. 
\]
Then the map $\Psi : \Gamma \times (-\eta, \eta)  \to \mathcal{N}_\eta(\Gamma)$ defined as $\Psi(x,s) = x+ s \nu_{\Gamma}(x)$ is a diffeomorphism. We say that  a hypersurface  $\Sigma$,  or a domain $\Omega$ with $\pa \Omega = \Sigma$, is $C^{1,\alpha}(\Gamma)$-regular for some small $\alpha >0$, when it can be written as 
\[
\Sigma = \{ x + h(x) \nu_{\Gamma}(x) : x \in \Gamma \},
\]
for a $C^{1,\alpha}(\Gamma)$-regular function $h : \Gamma \to \R$ with $\|h\|_{L^\infty} < \eta$. In particular, all $C^{1,\alpha}(\Gamma)$-regular sets are diffeomorphic. We say that a set $\Sigma$ is uniformly $C^{1,\alpha}(\Gamma)$-regular  if the height-function satisfies $\|h\|_{C^{1,\alpha}(\Gamma)} \leq C$ and $\|h\|_{L^\infty} \leq c \eta$ for constants $C$ and $c <1$. Finally we say that $\Sigma$ is uniformly $C^1$-regular if $\|h\|_{C^{1}(\Gamma)} \leq C$.

Let us next fix our notation in the ambient space $\R^3$. We denote the $k$th order  differential of a vector field $F : \R^3 \to \R^m$ by $\nabla^k F$,  the divergence of $F : \R^3 \to \R^3$ by $\Div F$ and the Laplace operator in $\R^3$ by $\Delta$. The notation $(\nabla F)^T$ stands for the transpose of $\nabla F$.  When we restrict $F: \R^3 \to \R^3$ on $\Sigma$, we define its normal and tangential part as 
\[
F_n := F \cdot \nu_\Sigma \qquad \text{and} \qquad F_\tau = F - F_n \, \nu_\Sigma. 
\]
We use the notation $x\cdot y$ for the inner product of two vectors in $\R^n$. 

Since  $\Sigma$ is a smooth hypersurface we may extend every function and vector field defined on $\Sigma$ to $\R^3$. We may thus define a tangential differential of a vector field $F : \Sigma \to \R^m$ by
\[
\nabla_\tau F = \nabla F - (\nabla F   \nu_\Sigma) \otimes  \nu_\Sigma 
\]
where we have extended $F$ to $\R^3$. We may then extend the definition of $\Div_\Sigma$ to fields $F : \Sigma \to \R^3$ by $\Div_\Sigma F = \text{Tr}(\nabla_\tau F)$ and the divergence theorem generalizes to
\[
\int_{\Sigma} \Div_{\Sigma} F \, d \H^2 = \int_{\Sigma} H_\Sigma (F \cdot \nu_\Sigma) \, d \H^2 .
\]
We note that the tangential gradient of $u \in C^\infty(\Sigma)$ is equivalent to its covariant derivative in the sense that for every vector field $X \in \mathscr{T}(\Sigma)$ we find a vector field  $\tilde{X} : \Sigma \to \R^{3}$ which satisfies $\tilde{X}\cdot \nu_\Sigma = 0$  and 
\[
\bar \nabla_X u = \nabla_\tau u \cdot \tilde{X}.
\]

Let us comment briefly on the notation  related to the equations \eqref{system}. We denote the derivative with respect to time by $\pa_t F$ and the material derivative as
\[
\D_t F := \pa_t F +(v \cdot \nabla) F.
\]
The material derivative does not commute with the spatial  derivative and we denote the commutation  
\[
[\D_t,\nabla] u = \D_t\nabla u - \nabla \D_t u.
\]
We denote by $U_\Omega$ the capacitary potential defined in \eqref{eq:def-capapot} and denote $U_t = U_{\Omega_t}$, $H_t = H_{\Sigma_t}$ etc... when the meaning is clear from the context. To shorten further the notation we denote
\beq \label{def:constant-q}
\Qt: = \frac{Q}{(\text{Cap}(\Omega_t))^2}.
\eeq
We may thus write the pressure in \eqref{system} as
\[
p = H_t - \frac{\Qt}{2} |\nabla U_t|^2.
\]

Let us next fix the notation for the function spaces.  We define  the Sobolev space $W^{l,p}(\Sigma)$ in  a standard way for $p \in [1,\infty]$, see e.g.  \cite{AubinBook2}, denote the Hilbert space $H^l(\Sigma) = W^{l,2}(\Sigma)$ and  define the associated norm for $u \in W^{l,p}(\Sigma)$ as
\[
\| u\|_{W^{l,p}(\Sigma)}^p = \sum_{k = 0}^l \int_\Sigma |\bar \nabla^k u|^p\, d \H^2
\]
and for $p = \infty$
\[
\| u\|_{W^{l,\infty}(\Sigma)} = \sum_{k = 0}^l \sup_{x \in \Sigma} |\bar \nabla^k u|.
\]
We often denote $\|u\|_{C^0(\Sigma)} = \|u\|_{L^\infty(\Sigma)} = \sup_{x \in \Sigma} |u(x)|$ 
for continuous function $u : \Sigma \to \R$ and $\| u\|_{C^{m}(\Sigma)} = \| u\|_{W^{m,\infty}(\Sigma)}$ . We define the H\"older norm of a continuous function $u : \Sigma \to \R$ by 
\[
\| u\|_{C^\alpha(\Sigma)} = \|u\|_{L^\infty(\Sigma)} + \sup_{\substack{x \neq y \\ x,y \in \Sigma}} \frac{|u(y) - u(x)|}{|y-x|^\alpha}.
\]
We define the H\"older norm for a  tensor field $F \in \mathscr{T}^k(\Sigma)$ as in \cite{JL}
\[
\| F\|_{C^\alpha(\Sigma)} = \sup \{ \| F(X_1, \dots, X_k) \|_{C^\alpha(\Sigma)} : X_i \in \mathscr{T}(\Sigma) \, \text{ with } \, \|X_i \|_{C^1(\Sigma)} \leq 1\}.
\]
Finally we define the $H^{-1}(\Sigma)$-norm by duality, i.e., 
\[
\|u \|_{H^{-1}(\Sigma)} := \sup \Big{\{} \int_{\Sigma} u \, g \, d\H^2 : \|g\|_{H^{1}(\Sigma)} \leq 1 \Big{\}}.  
\]

For functions defined in the domain $u:\Omega \to \R$ we define the Sobolev space  $W^{l,p}(\Omega)$ as functions which have $k$th order weak derivative in $\Omega$ and the corresponding norm is bounded
\[
\|u\|_{W^{l,p}(\Omega)}^p := \sum_{k=0}^l \int_{\Omega} |\nabla^k u|^p \, dx <\infty. 
\]
As before we denote the Hilbert space as  $H^l(\Omega) = W^{l,2}(\Omega)$ and define $H^{-1}(\Omega)$ by duality. 
Finally given an index vector $\alpha = (\alpha)_{i=1}^k \in \N^k$ we define its norm by 
\[
|\alpha| = \sum_{i =1}^k \alpha_i.
\]

Throughout the paper we use the notation $S \star T$ from \cite{Ham, Mantegazza2002} to denote a tensor formed by contraction on some indexes of  tensors $S$ and $T$, using the coefficients of the metric tensor $g_{ij}$ if $S$ and $T$ are defined on the boundary $\Sigma$. We also use the convention that $\bar \nabla^k u \star \bar \nabla^l v$ denotes contraction of some indexes of  tensors $\bar \nabla^{i} u$ and $ \bar \nabla^{j} v$  for any $i \leq k$ and $j \leq l$. In other words, we include also the lower order covariant derivatives.  

Following the notation from \cite{Triebel}, we first introduce the real interpolation method and then the interpolation spaces. Let $X$ and $Y$ be Banach spaces endowed respectively with the norm $\|\cdot\|_X$ and $\|\cdot\|_Y$.
The couple $(X,Y)$ is said to be an interpolation couple if both $X$ and $Y$ are
embedded in a Hausdorff topological vector space $V$. In this case we have that 
$X\cap Y$ endowed with the norm $\|v\|_{X\cap Y}=\max \{\|v\|_X, \|v\|_Y\}$ is a Banach space.
Moreover, 
we also have that $X+Y=\{z=x+y,\, x\in X,\,y\in Y \}$
endowed with the norm
\[
\|z\|_{X+Y}= \inf_{x\in X,\, y \in Y}\{ \|x\|_X+\|y\|_Y, \, z=x+y\}
\]
is a Banach space
and it is immediate to check that
$$
X\cap Y\subset X,Y\subset X+ Y.
$$
For $z \in X+Y$ and $t>0$ we introduce the $K$ functional
$$
K(t,z,X,Y)=\inf_{x\in X,\, y\in Y} \{\|x\|_X+t\|y\|_Y,\, x+y=z\}.
$$
For $\theta \in (0,1)$, $p\in [1,\infty)$ and $z\in X+Y$ we let
$$
\|z\|^p_{\theta, p}= \int_{0}^\infty \left(\frac{K(t,z,X,Y)}{t^{\theta }}\right)^p
\frac{dt}{t}
$$
and  define
$$
(X,Y)_{\theta, p}=\{z\in X+Y: \|z\|_{\theta,p}<\infty\}.
$$
We note that $(X,Y)_{\theta,p}$ is a Banach space.

Finally, we recall that if $(X_1,X_2)$ and $(Y_1,Y_2)$ are interpolation couples and $\mathcal{F}: X_1+X_2 \to Y_1+Y_2$ is a linear operator which is bounded $X_i \to Y_i$ by $M_i$. Then  for $\theta \in (0,1)$ and $p\in [1,\infty)$ the operator
\beq \label{ineq:inter-functional}
\mathcal{F}: (X_1,X_2)_{\theta, p} \to (Y_1,Y_2)_{\theta, p}
\eeq
is bounded  and we may estimate its norm by $M_1^{1-\theta}M_2^\theta$.

\subsection{Half-integer Sobolev spaces}

Before giving the definition of half-integer Sobolev space in a domain, we exploit the extension properties of Sobolev functions. Throughout the paper we assume that the boundary $\Sigma = \partial \Omega$ is uniformly $C^{1,\alpha}(\Gamma)$-regular and thus it is $H^1$ extension domain, i.e., there is a linear operator  $T:H^{1}(\Omega)\to H^{1} (\R^3)$ such that 
\[
 \|T(u)\|_{H^{1} (\R^3)} \leq C\|u\|_{H^{1}(\Omega)} .
\]
We refer to \cite{CWHM} for the study of Sobolev spaces under Lipschitz-regularity and the references therein. 

We need more regularity for the boundary for higher order Sobolev extension $m \geq 2$, although we do not need the  optimal condition. Instead, we assume the following for the second fundamental form     
\beq\label{eq:notationhp}
\tag{H$_{m}$}
 \|B_\Sigma\|_{L^4(\Sigma)}\leq C_m \quad \text{ if }\,  m=2,\qquad
\, \|B_\Sigma\|_{L^\infty(\Sigma)}+ \|B_\Sigma\|_{H^{m-2}(\Sigma)}\leq C_m \quad \text{ if } \,  m>2,
\eeq
which guarantees that we may extend a given function $u \in H^{m}(\Omega)$ to the whole space. Note that for $m \geq 4$ the condition \eqref{eq:notationhp} is implied by $\|B_\Sigma\|_{H^{m-2}(\Sigma)}\leq C_m $ by the Sobolev-embedding, which agrees with the assumption e.g. in \cite{CS}.  In the following we do not specify that a given quantity depends on the constant $C_m$, but take it for granted when we refer to the condition \eqref{eq:notationhp}. 
  
Even if there are many results for extensions of Sobolev functions in the literature, the condition \eqref{eq:notationhp} is too weak to apply them. To this aim we need the following result.   
\begin{proposition}
\label{prop:extension}
Let $m \in \N$, with $m\geq 2$,  and let $\Omega$ be a smooth domain which is  uniformly $C^{1,\alpha}(\Gamma)$-regular and satisfies  \eqref{eq:notationhp}. Then there is an extension operator $T:H^{m}(\Omega)\to H_0^{m} (\R^3)$ such that
\[
\|T(u)\|_{H^m(\R^3)}\leq 
C\|u\|_{H^m(\Omega)}.
\]
\end{proposition}
\begin{proof}
Let  $x_0\in \pa \Omega$. There exist $\delta>0$ and a diffeomorphism $\Psi: \R^3 \to \R^3 $
such that $\Psi^{-1}( \Omega\cap B_\delta (x_0))= B_1^+ = B_1 \cap \R_+^3$. We note that the $C^{1,\alpha}$ regularity of $\pa \Omega$ and \eqref{eq:notationhp} imply that we may choose the diffeomorphism such that it satisfies $\|\Psi\|_{C^{1,\alpha}(\R^3)} \leq C$ and
\begin{equation}
    \label{eq:extension1}
     \|\nabla^2 \Psi \|_{L^4(\R^3)}\leq C \quad \text{ if }\,  m=2,\qquad
 \|\Psi\|_{H^{m}(\R^3)}\leq C \quad \text{ if } \,  m>2.
\end{equation}

For $u : \Omega\to \R$ smooth we let  $u'=u\circ \Psi$.
We may  extend $u'$ to a function $T(u') \in H^m (B_1)$ such that 
\[
\|T(u')\|_{H^m(B_1)}\leq C\|u'\|_{H^m(B_1^+)}.
\]
The construction of $T(u') \in H^m (B_1)$ is classical but we recall it for the reader's convenience. 
We define 
\begin{equation*}
    T(u')(x',x_n)= 
    \begin{cases}
u'(x',x_n) \,\,\,\,\ &x_n\geq 0
\\
\sum_{j=1}^{m+1} \lambda_j u'(x', -j^{-1}x_n )\,\,\,\,\, &x_n<0
\end{cases}
\end{equation*}\
where $\lambda=(\lambda_1,\dots,\lambda_{m+1})$
solve the system
\begin{equation*}
    \begin{cases}
    \sum_{j}\lambda_j=1
    \\
    \sum _{j}(-j)^{-1}\lambda_j=1
    \\
    \vdots
    \\
    \sum_{j} (-j)^{-m}\lambda_j=1.
    \end{cases}
\end{equation*}
This system, known as Vandermonde system, has a unique solution, hence $Tu'$ is well defined. Finally we define the extension operator as 
\[
T(u) := T(u') \circ \Psi^{-1}. 
\]
Let us show that $T$ is a bounded operator.  

It is straightforward to check that
\[
\|Tu'\|_{H^m(B_1)}\leq C\|u'\|_{H^m(B_1^+)}.
\]
Let us then show that 
\begin{equation}
    \label{eq:extension2}
    \|u'\|_{H^m(B_1^+)}\leq C  \|u\|_{H^m(\Omega)}.
\end{equation}
We first note that
\[
\nabla u'= \nabla u \star \nabla \Psi 
\] 
and for $m\geq 2$
\[
\nabla^m u'= \sum_{|\alpha|\leq m-1 } \nabla^{1+\alpha_1}\Psi\star \cdots \star  \nabla^{1+\alpha_m}
\Psi \star \nabla^{1+\alpha_{m+1}} u.
\]
For $m=2$ we have then by H\"older's inequality, by \eqref{eq:extension1} and by the Sobolev embedding
\[
\|\nabla^2 u'\|_{L^2(B_1)} \leq 
C \|\Psi\|_{C^1(\R^3)}\|\nabla^2 u\|_{L^2(\Omega)} + \|\nabla^2 \Psi\|_{L^4(\R^3)}\|\nabla u\|_{L^4(\Omega)} \leq C\|u\|_{H^2(\Omega)}.
\]
To treat the case $m\geq 3$ we first observe  that by Sobolev embedding it holds 
\[
\|\nabla^{m-2}u\|_{L^\infty(B_1^+)}\leq \|u\|_{H^m(\Omega)} \quad \text{and} \quad \|\nabla^{m-2}\Psi\|_{L^\infty(B_1^+)}\leq \|\Psi\|_{H^m(\R^3)} . 
\]
Hence for $m \geq  3$ we have  by H\"older's inequality, by \eqref{eq:extension1} and by the Sobolev embedding
\[
\begin{split}
    \|\nabla^m u'\|_{L^2(B_1)} &\leq 
C\big((1+\|\Psi\|_{H^m(\R^3)}^m)\|u\|_{H^m(\Omega)}   + \|\nabla^2 \Psi\|_{L^4(\R^3)}\|\nabla^{m-1} u\|_{L^4(\Omega)}\\
&\,\,\,\,\,\,\,+\|\nabla^{m-1} \Psi\|_{L^4(\R^3)}\|\nabla^{2} u\|_{L^4(\Omega)}  \big)\\
&\leq C (1+\|\Psi\|_{H^m(\R^3)}^m)\|u\|_{H^m(\Omega)}\\
&\leq C \|u\|_{H^m(\Omega)}.
\end{split}
\]
Thus we have \eqref{eq:extension2}. 

Similarly we show that 
\[
\|T(u') \circ \Psi^{-1}\|_{H^m(\Omega\cap B_\delta (x_0))} \leq C\|Tu'\|_{H^m(B_1)}.
\]
The claim then follows from standard covering argument. 
\end{proof}

Throughout the paper we will refer to the operator $T$ as the canonical extension operator or simply as the extension operator. We define the half-integer Sobolev space in the domain $\Omega$ using the canonical extension operator.
\begin{definition} \label{dfn:fracsobobulk}
We say that a function $u\in L^2(\Omega)$ is in $H^\frac12(\Omega)$ if
\[
\|u\|^2_{H^\frac12(\Omega)}:= \|u\|^2_{L^2(\Omega)}+\int_{\R^3}\int_{\R^3}\frac{|Tu(x)-Tu(y)|^2}{|x-y|^3}\,dxdy<\infty .
\]
For $k\geq 1$ we say that $u\in H^{k+\frac12}(\Omega)$ if $u\in H^k(\Omega)$ and
\[
\|u\|_{H^{k+\frac12}(\Omega)}
:=\|u\|_{H^k(\Omega)}+ \|T(\nabla^k u)\|_{H^\frac12(\Omega)} < \infty .
\]
Finally we define  the $H^{-\frac12}(\Omega)$-norm by duality. 
\end{definition}
We define $H^\frac12_\star(\Omega)$ as the space of functions via interpolation such that $u \in H^\frac12_\star(\Omega)$ if $T(u) \in (L^2(\R^3),H^1(\R^3))_{\frac12,2}$ and endow it with the norm
\begin{equation} \label{def:fracnormbulk}
\|u\|_{H_\star^\frac12(\Omega)}^2:= \|u\|_{L^2(\Omega)}^2+ \int_0^\infty\left(\frac{ K(t,T(u),L^2(\R^3),H^1(\R^3))}{t^{1/2}}\right)^2\, \frac{dt}{t}.
\end{equation}
This gives an equivalent definition for the half-integer Sobolev space. 
\begin{proposition}
    \label{prop:sobolev2}
    Let $m \in \N$, with $m\geq 2$,  and let $\Omega$ be a smooth domain which is  uniformly $C^{1,\alpha}(\Gamma)$-regular and  satisfies  \eqref{eq:notationhp}. The  norms in Definition \ref{dfn:fracsobobulk} and \eqref{def:fracnormbulk} are equivalent, i.e.,
\[
\|u\|_{H^\frac12(\Omega)}\simeq \|u\|_{H_\star^\frac12(\Omega)}.
\]    
\end{proposition}
We do not give the details of the proof, but only refer to \cite{CWHM} and mention that it follows from the fact that $H^\frac12(\R^3)=(L^2(\R^3),H^1(\R^3))_{\frac12,2}$, see \cite{Triebel}.

\subsection{Half-integer Sobolev spaces on a surfaces}
We begin by defining the space $H^\frac12(\Sigma)$. Again there are many ways to do this. We choose the definition via harmonic extension.  
\begin{definition} \label{def:H-1/2-boundary}
Let $\Sigma = \pa \Omega$ be uniformly $C^{1,\alpha}(\Gamma)$-regular. We say  that $u\in H^\frac12(\Sigma)$ if $u\in L^2(\Sigma)$ and
\[
\|u\|_{H^\frac12(\Sigma)}=\|u\|_{L^2(\Sigma)}+\inf\{\|\nabla v\|_{L^2(\Omega)}: \,v-u \in H^{1}_0 (\Omega)\} <\infty.
\]
We define the space $H^{-\frac12}(\Sigma)$ and its norm  by duality.
\end{definition}
By standard theory the $C^{1,\alpha}(\Gamma)$-regularity of $\Sigma$ ensures that Definition \ref{def:H-1/2-boundary} is equivalent to the definition via  Gagliardo seminorm 
\[
\|u\|_{H^\frac12(\Sigma)}^2 \simeq \|u\|_{L^2(\Sigma)}^2
+\int_\Sigma\int_\Sigma \frac{|u(x)-u(y)|^2}{|x-y|^3}\, d\H^2_xd\H^2_y.
\]
Moreover,  this norm is also equivalent to the norm obtained via interpolation. Indeed, let us define the interpolation space (see beginning of Section 2)
\[
H^\frac12_\star(\Sigma)=(L^2(\Sigma),H^1(\Sigma))_{\frac12,2}. 
\]
Let us show that 
\beq 
\label{eq:equiv-frac-1/2}
\|u\|_{H^\frac12(\Sigma)} \simeq \|u\|_{H_\star^\frac12(\Sigma)}.
\eeq
Due to the non-local nature of the problem we give the proof of \eqref{eq:equiv-frac-1/2} in detail. 

We fix a small $\delta>0$ and cover 
$\Sigma$ with finitely many balls of radius $\delta$ centered at $x_i \in \Sigma$, i.e., 
\[
\Sigma \subset \bigcup_{i=1}^N B_\delta (x_i).
\]
Since $\Sigma$ is $C^{1,\alpha}(\Gamma)$, there are $C^{1,\alpha}$-regular functions $\phi_i$ such that $\Sigma \cap B_{2\delta} (x_i)$ is contained in the graph of $\phi_i$ for every $i =1, \dots, N$, when $\delta$ is small enough. Let  $\{\eta_i\}_{i=1,\dots, N}$ be a partition of unity subordinated to the open covering $B_{\delta}(x_i)$. Then it holds 
 \[
 \|u\|_{H^\frac12(\Sigma)}\leq \sum_{i=1}^N\|\eta_i u\|_{H^\frac12(\Sigma)}.
 \]
 
 Let us fix $i =1, \dots, N$ and by rotating and translating the coordinates we may assume that $x_i = 0$ and $\Sigma \cap B_{2\delta} \subset \{ (x',\phi_i(x')): x' \in \R^2\}$ with $\phi_i(0) = 0$ and $\nabla \phi_i(0) = 0$. Denote  $u_i= \eta_i u$ and $v_i(x') = u_i(x', \phi_i(x'))$. Note that then $v_i : \R^2 \to \R$, $\text{supp}\, v_i \subset B_{\delta'} \subset \R^2$ and $\text{supp}\, u_i \subset B_{\delta} \subset \R^3$ with $\delta/2 \leq \delta' \leq \delta$. Therefore we deduce by the $C^{1,\alpha}$-regularity of $\phi_i$ that 
\[
\begin{split}
    &\int_{\Sigma}\int_{\Sigma}\frac{|u_i(x)-u_i(y)|^2}{|x-y|^3}\, d\H^2_x\,d\H^2_y\\
    &= \int_{\Sigma \cap B_{2\delta} }\int_{\Sigma\cap B_{2\delta}}\frac{|u_i(x)-u_i(y)|^2}{|x-y|^3}\, d\H_x\,d\H_y + 2\int_{\Sigma \setminus B_{2\delta} }\int_{\Sigma\cap B_{2\delta}}\frac{|u_i(x)-u_i(y)|^2}{|x-y|^3}\, d\H_x\,d\H_y
    \\
    &\leq C\int_{B_{2\delta'}}\int_{B_{2\delta'}} \frac{|u_i(x',\phi_i(x'))-u_i(y',\phi_i(y'))|^2}{(|x'-y'|^2 +(\phi_i(x')-\phi_i(y')^2)^{3/2}}\, dx'\,dy' +C \|u\|_{L^2(\Sigma)}^2
    \\
    &\leq C\|v_i\|_{H^\frac12(\R^2)}^2 + C\|u\|_{L^2(\Sigma)}^2\\
    &\leq C\|v_i\|_{H_{\star}^\frac12(\R^2)}^2+ C\|u\|_{L^2(\Sigma)}^2. 
    \end{split}
\]
This implies $\|u_i\|_{H^{\frac12}(\Sigma)} \leq C\|v_i\|_{H_{\star}^\frac12(\R^2))}+ C\|u\|_{L^2(\Sigma)}$. 
Let us denote by $L_0^2(B_{2\delta'})$ and $H_0^1(B_{2\delta'})$ for functions $f \in L^2(\R^2)$, and respectively  $f \in H^1(\R^2)$, with $\text{supp} f \subset B_{2\delta'}$. Denote also $\Psi : B_{2\delta'} \to \Psi(B_{2\delta'}) \subset  \R^3$, $\Psi(x') = (x', \phi_i(x'))$. We may estimate 
\[
\begin{split}
&K(t,v_i, L^2(\R^2),H^1(\R^2))\\
&=\inf_{f+g=v_i}\{ \|f\|_{L^2(\R^2)}+ t \|g\|_{H^1(\R^2)},\, f \in L^2(\R^2),\, g\in H^1(\R^2) \}\\
&\leq \inf_{f+g=v_i}\{ \|f\|_{L^2(\R^2)}+ t \|g\|_{H^1(\R^2)},\, f \in L_0^2(B_{2\delta'}),\, g\in H_0^1(B_{2\delta'}) \}
\\
&\leq C\inf_{f+ g = v_i}\{ \|f \circ \Psi^{-1}\|_{L^2(\Sigma \cap B_{2\delta})}+ t \|g\circ \Psi^{-1}\|_{H^1(\Sigma \cap B_{2\delta})},\, f \in L_0^2(B_{2\delta'}),\, g\in H_0^1(B_{2\delta'}) \} \\
&\leq C\inf_{\tilde f+ \tilde g = u_i}\{ \|\tilde f \|_{L^2(\Sigma \cap B_{2\delta})}+ t \|\tilde g\|_{H^1(\Sigma \cap B_{2\delta})},\,\tilde f \in L_0^2(\Sigma \cap B_{2\delta}),\, \tilde g\in H_0^1(\Sigma \cap B_{2\delta}) \}
\\
&=C \, K(t,u_i, L_0^2(\Sigma \cap B_{2\delta}),H_0^1(\Sigma \cap B_{2\delta})).
\end{split}
\]
Since $\text{supp} \, u_i \subset \Sigma \cap B_{\delta}$, it is easy to see that 
\[
 K(t,u_i, L_0^2(\Sigma \cap B_{2\delta}),H_0^1(\Sigma \cap B_{2\delta})) \leq  C\, K(t,u_i, L^2(\Sigma),H^1(\Sigma)).
\]
Therefore, we deduce 
\[
\|v_i\|_{H_{\star}^\frac12(\R^2)}\leq C\|u_i\|_{H_{\star}^\frac12(\Sigma)} \leq C\|u\|_{H_{\star}^\frac12(\Sigma)} .
\]
Repeating the argument for every $i =1, \dots, N$ yields
\[
\|u\|_{H^\frac12(\Sigma)}\leq C\|u\|_{H_{\star}^\frac12(\Sigma)}.
\]
The opposite inequality can be proved in a similar way.

We also note that we may interpolate between $H^{\frac12}(\Sigma)$ and its dual and obtain by using \cite[Theorem 4.1]{CWHM}  and by a localization argument that 
\[
(H^{-\frac12}(\Sigma), H^{\frac12}(\Sigma))_{\frac12, 2} = L^2(\Sigma).
\]

In order to define higher order half-integer Sobolev spaces on the boundary we use the fact that in our setting the boundary $\pa \Omega = \Sigma$ is given by the parametrization $\Psi_{\Sigma}: \Gamma \to \Sigma$,  $\Psi_{\Sigma}(x) = x + h(x)\nu_\Gamma(x)$, where $\Gamma$ is the reference surface. In our case $\Gamma$ is the boundary of a smooth  set $G$ and the map $\Psi : \Gamma \times (-\eta, \eta) \to \mathcal{N}_\eta(\Gamma)$, $\Psi(x,s) = x + s \nu_{\Gamma}(x)$, is a diffeomorphism. Here $\mathcal{N}_\eta(\Gamma)$ is the tubular neighborhood of $\Gamma$. Therefore the projection map $\pi_\Gamma: \mathcal{N}_\eta(\Gamma) \to \Gamma$ is well defined  as 
\beq \label{def:projection}
\pi_\Gamma(y) = x  \qquad \text{where } \,   y = x + s \nu_{\Gamma}(x) \,\,\, \text{for some  } \, \, s \in (- \eta, \eta). 
\eeq
 We extend $\pi_\Gamma$ to whole $\R^3$ and thus we may extend a given function $u : \Gamma \to \R$ to  $\R^3$ as $(u \circ \pi_\Gamma) : \R^3 \to \R$.  In particular, the $k$th order derivative $\nabla^k (u \circ \pi_\Gamma)(x)$ is well defined for  all $x \in \Gamma$, and for $x \in \Gamma$ the function $s \mapsto  (u \circ \pi_\Gamma)(x+ s\nu_{\Gamma}(x))$ is constant for $|s|$ small. We use this extension to define the half-integer Sobolev norm on the reference surface.

\begin{definition} \label{defn:fractionalnorm}
For $m \geq 2$ we say that $u \in H^{m-\frac12}(\Gamma)$ if 
$u \in H^{m-1}(\Gamma)$ and the norm 
\[
\| u \|_{H^{m-\frac12}(\Gamma)}  := \|\nabla^{m-1} (u \circ \pi_\Gamma)\|_{H^{\frac12}(\Gamma)}+ \|u\|_{H^{m-1}(\Gamma)} 
\]
is bounded.
\end{definition}
We define the half-integer Sobolev spaces on $\Sigma$  by mapping a function $u \in C^\infty(\Sigma)$ back to $\Gamma$ by using the parametrization $\Psi_{\Sigma}: \Gamma \to \Sigma$, $\Psi_{\Sigma}(x) = x + h(x)\nu_{\Gamma}(x)$.  Let us fix $m \geq 2$ and recall that   
\[
\bar \nabla^m (u\circ\Psi_\Sigma)= \sum_{|\alpha|\leq m-1}\bar \nabla^{1+\alpha_1}\Psi_\Sigma \star\cdots \star \bar \nabla^{1+\alpha_{k}}\Psi_\Sigma \star\bar \nabla^{1+\alpha_{m+1}} u.
\]
If $\Sigma$ satisfies  \eqref{eq:notationhp}, then arguing as in the proof of Proposition \ref{prop:extension}  we deduce
\[
\|u\|_{H^m(\Sigma)}
\simeq \|u\circ \Psi_\Sigma\|_{H^{m}(\Gamma)}.
\]
Based on this we define the half-integer Sobolev space of order $m - 1/2$ on $\Sigma$ in the following way.

\begin{definition} \label{defn:fractionalnorm2}
Let  $m\geq 2$ be an integer and assume $\Sigma$ is $C^{1,\alpha}(\Gamma)$-regular.  We say that  $u$ is in the space    $H^{m-\frac12}(\Sigma)$  if $(u\circ \Psi_\Sigma) \in H^{m-\frac12}(\Gamma)$
and define the norm as 
\[
\|u\|_{H^{m-\frac12}(\Sigma)}:= \|u\circ \Psi_\Sigma \|_{H^{m-\frac12}(\Gamma)}, 
\]
where $\Psi_{\Sigma}: \Gamma \to \Sigma$ is the parametrization $\Psi_{\Sigma}(x) = x + h(x)\nu_{\Gamma}(x)$.
\end{definition}

We define the space $H^{m-\frac12}_\star(\Sigma)$ via interpolation as the functions $u\in H^{m-1}(\Sigma)$ such that
\begin{equation}\label{def:H_star}
\|u\|_{H_\star^{m-\frac12}(\Sigma)}^2:= \|u\|_{H^{m-1}(\Sigma)}^2+
\int_{0}^\infty \left(\frac{K(t,u,H^{m-1}(\Sigma),H^m(\Sigma))}{t^{1/2}} \right)^2\, \frac{dt}{t} <\infty.
\end{equation}

We note that if $\Sigma$  satisfies the assumption  \eqref{eq:notationhp} for $m\geq 2$  the norm $H^{m-\frac12}(\Sigma)$  in Definition \ref{defn:fractionalnorm2} is equivalent with the interpolation norm $\|u\|_{H_\star^{m-\frac12}(\Sigma)}$ in \eqref{def:H_star}. We state this in the next proposition. The proof is similar to the argument for  \eqref{eq:equiv-frac-1/2}  and we omit it. 

\begin{proposition} 
\label{prop:sobolev} Let  $m\geq 2$ and  assume that $\Sigma$ is uniformly $C^{1,\alpha}(\Gamma)$-regular and  satisfies the assumption  \eqref{eq:notationhp}. Then it holds 
\[
H^{m-\frac12}(\Sigma)=H_\star^{m-\frac12}(\Sigma) \qquad \text{and} \qquad \|u\|_{H^{m-\frac12}(\Sigma)} \simeq \|u\|_{H_\star^{m-\frac12}(\Sigma)}.
\]
\end{proposition}

\subsection{Geometric Preliminaries}

We begin by recalling basic results from differential geometry. We define the Riemann curvature tensor $R \in \mathscr{T}^4(\Sigma)$ \cite{Lee, MantegazzaBook}  via interchange of covariant  derivatives of a vector field  $Y^i$ and a covector field  $Z_i$ as   
\begin{equation}
\label{eq:curv-tensor}
\begin{split}
&\bar \nabla_i \bar \nabla_j Y^s - \bar \nabla_j \bar \nabla_i Y^s = R_{ijkl} g^{ks} Y^l,\\
&\bar \nabla_i \bar \nabla_j Z_k - \bar \nabla_j \bar \nabla_i Z_k =  R_{ijkl} g^{ls} Z_s,
\end{split}
\end{equation}
where we have used the Einstein summation convention. We may write the Riemann tensor in local coordinates by using the second fundamental form $B$ as
\begin{equation}
\label{eq:curv-tensor2}
R_{ijkl} = B_{ik}B_{jl} - B_{il}B_{jk}.
\end{equation}
We will also need the Simon's identity which reads as 
\begin{equation}
    \label{eq:Simon}
    \Delta_{\Sigma} B_{ij} =  \bar \nabla_i \bar \nabla_j  H + H B_{il} g^{ls} B_{sj} - |B|^2 B_{ij}.
\end{equation}

Let us recall  that  the interpolation inequality  holds for smooth compact  $n$-dimensional hypersurface $\Sigma \subset \R^{n+1}$, see e.g. \cite{AubinBook2},
\beq \label{eq:class-inter}
 \| \bar \nabla^k u \|_{L^p(\Sigma)} \leq C_{\Sigma}  \|u \|_{W^{l,r}(\Sigma)}^\theta  \| u \|_{L^q(\Sigma)}^{(1-\theta)}, 
\eeq
where
\[
\frac{1}{p} = \frac{k}{n} + \theta \left(\frac{1}{r} - \frac{l}{n} \right) + \frac{1}{q}(1- \theta).
\]
In particular, \eqref{eq:class-inter} holds on  the reference surface $\Gamma \subset \R^3$ and   in $\R^{n}$ for functions with compact support $\text{supp} \, u \subset B_R$. 
 
 In order to have the interpolation inequality for a general surface $\Sigma \subset \R^{n+1}$ with control on the constant $C_{\Sigma}$,  we use the result in    \cite{Mantegazza2002}, which states that once the mean curvature $H_\Sigma$ satisfies the bound  $\|H_{\Sigma}\|_{L^{n+\delta}(\Sigma)} \leq C$, then the above interpolation inequality holds on $\Sigma$ with uniform bound on the constant. We state this for our purpose, where $\Sigma$  is $2$-dimensional surfaces  that  is  uniformly $C^{1,\alpha}(\Gamma)$-regular and satisfies the bound  $\|B_\Sigma\|_{L^4(\Sigma)} \leq C$. The reason for the $L^4$-curvature bound will be clear from the results in Section 6.   The following interpolation inequality follows from  \cite[Proposition 6.5]{Mantegazza2002}. 

\begin{proposition}
\label{prop:interpolation}
Assume $\Sigma \subset \R^3$  is a compact $2$-dimensional hypersurface which is  uniformly $C^{1,\alpha}(\Gamma)$-regular and  satisfies the bound  $\|B_\Sigma\|_{L^4(\Sigma)} \leq M$. Then  for integers $k,l$,  $0 \leq k <l$ and numbers $p, r \in [1,\infty)$ and $q \in [1,\infty]$ we have for all tensor fields $T$ that 
\[
 \| \bar \nabla^k T \|_{L^p(\Sigma)} \leq C  \|T \|_{W^{l,r}(\Sigma)}^\theta  \| T \|_{L^q(\Sigma)}^{(1-\theta)}, 
\]
where $p$ and $\theta \in [0,1]$ are given by 
\[
\frac{1}{p} = \frac{k}{2} + \theta \left(\frac{1}{r} - \frac{l}{2} \right) + \frac{1}{q}(1- \theta).
\]
The constant $C$ depends on $M, k, p, l, r, q$. 
\end{proposition}

In particular, we have the Sobolev embedding, i.e., for $p \in [1,n)$ it holds  $\| u\|_{L^{p^*}(\Sigma)} \leq C \|u\|_{W^{1,p}(\Sigma)}$ with  $p^* = \frac{np}{n-p}$, for $p = n$ it holds $\| u\|_{L^{q}(\Sigma)} \leq C \|u\|_{W^{1,p}(\Sigma)}$ for all $q <\infty$ and for $p >n$  it holds $\|u\|_{C^{\alpha}(\Sigma)}\leq C \|u\|_{W^{1,p}(\Sigma)}$ for $\alpha = 1- \frac{n}{p}$. 

There is a danger for confusion  in terminology when we use interpolation of function spaces and interpolation inequality. We use the term 'interpolation argument', when we interpolate between two function spaces, and  'interpolation inequality' or merely 'interpolation' when we refer to Proposition \ref{prop:interpolation}.

Let $\Sigma = \pa \Omega\subset \R^3$ be a compact hypersurface  in $\R^3$ such that $\Sigma = \pa \Omega$ which is  $C^{1,\alpha}(\Gamma)$-regular. Then the Sobolev embedding extends to half-integers, i.e., it holds 
\[
\|u\|_{L^p(\Sigma)}\leq C \|u\|_{H^{\frac12}(\Sigma)} , \qquad \text{for } \, p \leq 4
\]
and
\[
\|u\|_{L^p(\Omega)}\leq C \|u\|_{H^{\frac12}(\Omega)} , \qquad \text{for } \, p \leq 3.
\]
We need the above interpolation inequality also for half-integers and for functions defined in $\Omega$. To this aim we need to assume that $\Sigma$ satisfies the condition \eqref{eq:notationhp}.

\begin{corollary}
\label{coro:interpolation}
Let $m \in \N$  and  $\Sigma \subset \R^3$  is compact $2$-dimensional hypersurface which is uniformly $C^{1,\alpha}(\Gamma)$-regular such that $\Sigma = \pa \Omega$ and satisfies the condition \eqref{eq:notationhp}. Then for all half-integers $k$ and $l$ with $k< l \leq m$ and for $q \in [1,\infty]$ it holds
\[
\| u\|_{H^{k}(\Sigma)}\leq C \|u\|_{H^{l}(\Sigma)}^\theta
\|u\|_{L^q(\Sigma)}^{1-\theta},
\]
  where $\theta \in [0,1]$ is given by  
  \[
  1 = k - \theta (l - 1) + \frac{2}{q}(1- \theta).
  \]

In addition, it holds 
  \[
\| u\|_{H^{k}(\Omega)}\leq C \|u\|_{H^{l}(\Omega)}^\theta
\|u\|_{L^q(\Omega)}^{1-\theta},
\]
  where $\theta \in [0,1]$ is given by
  \[
  \frac{1}{2} = \frac{k}{3} + \theta \left(\frac{1}{2} - \frac{l}{3} \right) + \frac{1}{q}(1- \theta).
  \]
  Moreover, the inequality \eqref{eq:class-inter} holds on $\Omega \subset \R^3$ with $r=2$ and integers $k <l \leq m$. The constants depends on $m, q$ and on the $C^{1,\alpha}$-norm of the heightfunction. 
\end{corollary}
\begin{proof}
We sketch the proof  only for the first claim when $k=\tilde k -\frac12$ for $\tilde k \in \N$ and $l$ is an integer. By Proposition \ref{prop:sobolev} and by the classical interpolation theory stated in \eqref{ineq:inter-functional}  we have 
\[
\|u\|_{H^{k}(\Sigma)} \leq C\|u\|_{H_\star^{k}(\Sigma)} \leq 
C\|u\|_{H^{\tilde k}(\Sigma)}^\frac12\|u\|_{H^{\tilde k-1}(\Sigma)}^\frac12.
\]
Proposition \ref{prop:interpolation} yields
\[
\|u\|_{H^{\tilde k}(\Sigma)} \leq C
\|u\|_{H^{l}(E)}^{\theta_1}
\|u\|_{L^q(E)}^{1-\theta_1},
\]
where $\theta_1$ is given by $1 = \tilde k - \theta_1 (l - 1) + \frac{2}{q}(1- \theta_1)$, and 
  \[
\|u\|_{H^{\tilde k-1}(\Sigma)} \leq C
\|u\|_{H^{l}(E)}^{\theta_2}
\|u\|_{L^q(E)}^{1-\theta_2},
\]
where $\theta_2$ is given by $1 = (\tilde k -1) - \theta_2 (l - 1) + \frac{2}{q}(1- \theta_2)$. This implies the claim. The  case when $l$ is half-integer follows from the  same argument. Finally the second interpolation inequality follows by extending $u$ to whole $\R^3$, where the inequality is well-known,  and using  Proposition \ref{prop:extension}. 
\end{proof}

\subsection{Functional and geometric inequalities}

We begin by recalling the extension of the interpolation inequality \eqref{eq:class-inter}, or the Gagliardo-Nirenberg inequality,  in $\R^n$ for fractional Sobolev spaces \cite{BM}. We state the result in the setting that we need, where for all $f \in C_0^\infty(B_R)$ it holds   
\beq \label{eq:BM}
\|f\|_{W^{s,p}(B_R)} \leq C \|f\|_{W^{s_1,p_1}(B_R)}^\theta \|f\|_{L^{p_2}(B_R)}^{1-\theta}, 
\eeq
for  $0 \leq s \leq s_1$, $p_2 \in (1,\infty)$ and $\theta \in (0,1)$ which satisfy
\[
s = \theta s_1 \qquad \text{and}\qquad \frac{1}{p} = \frac{\theta}{p_1} + \frac{1-\theta}{p_2}.
\]

Next we recall the Kato-Ponce inequality, or the fractional Leibniz rule, in $\R^n$ which is proven e.g. in \cite{GOh}. We may define the norm $\|f\|_{W^{k, p}(\R^n)}$ for half-integer $k \geq 0$ and $p \in (1,\infty)$ by using Bessel potentials  $\langle D \rangle^k$ as
\[
\|f\|_{W^{k, p}(\R^n)} = \| \langle D \rangle^k f\|_{L^p(\R^n)}.
\]
The Kato-Ponce inequality, in the form we are interested in, states that for $f,g \in C_0^\infty(\R^n)$ and for numbers $2 \leq  p_1, q_2 < \infty$ and $2\leq  p_2, q_1 \leq \infty$ with  
\begin{equation}
    \label{eq:kato-ponce-con}
\frac{1}{p_1} + \frac{1}{q_1} = \frac{1}{p_2} + \frac{1}{q_2} = \frac12
\end{equation}
it holds 
\begin{equation}
    \label{eq:kato-ponce}
\|fg\|_{H^k(\R^n)}\leq C\|f\|_{W^{k,p_1}(\R^n)}\|g\|_{L^{q_1}(\R^n)}+C \|f\|_{L^{p_2}(\R^n)}\|g\|_{W^{k,q_2}(\R^n)}.
\end{equation}

We need the following generalization of the Kato-Ponce inequality both on the boundary $\Sigma$ and in the domain $\Omega$.
\begin{proposition}
    \label{prop:kato-ponce}
Let $m\geq 1$ be an integer and assume $\Sigma$ is uniformly $C^{1,\alpha}(\Gamma)$-regular and satisfies the condition \eqref{eq:notationhp}.  Then for all half-integers  $k \leq m$ it holds     
\[
\|fg\|_{H^k(\Sigma)}\leq C\|f\|_{H^{k}(\Sigma)}\|g\|_{L^{\infty}(\Sigma)}+C\|f\|_{L^{\infty}(\Sigma)}\|g\|_{H^{k}(\Sigma)}
\]
and
\[
\|fg\|_{H^k(\Omega)}\leq C\|f\|_{H^{k}(\Omega)}\|g\|_{L^{\infty}(\Omega)}+C\|f\|_{L^{\infty}(\Omega)}\|g\|_{H^{k}(\Omega)}.
\]

Moreover, assume that $\|B\|_{L^4} \leq M$ and let $k \in \N$. Then for $p_1,p_2, q_1,q_2 \in [2,\infty]$ with $p_1, q_2 <\infty$ which satisfies \eqref{eq:kato-ponce-con} it holds 
\[
\|fg\|_{H^k(\Sigma)}\leq C\|f\|_{W^{k,p_1}(\Sigma)}\|g\|_{L^{q_1}(\Sigma)}+C\|f\|_{L^{p_2}(\Sigma)}\|g\|_{W^{k,q_2}(\Sigma)}.
\]
The constants depend on $M, m, k, p_1,p_2, q_1,q_2$ and on the  $C^{1,\alpha}$-norm of the heightfunction.
\end{proposition}

\begin{proof}
The second inequality follows immediately from the property of the extension operator given by Proposition \ref{prop:extension} and by the classical Kato-Ponce inequality \eqref{eq:kato-ponce}, see e.g. \cite{CS}.  Also the first inequality follows from a similar localization argument as we used in \eqref{eq:equiv-frac-1/2}.

We  prove the third inequality, since we  will use the argument also later.   First by Leibniz formula we may write  
\[
\bar \nabla^k (fg) = \sum_{i+j=k} \bar \nabla^i f \star \bar  \nabla^j g .
\]
The claim thus follows once we prove
\begin{equation}
    \label{eq:est-product}
    \sum_{i+j=k} \|\bar \nabla^i f \star \bar \nabla^j g \|_{L^2(\Sigma)} \leq C\|f\|_{W^{k,p_1}(\Sigma)}\|g\|_{L^{q_1}(\Sigma)}+\|f\|_{L^{p_2}(\Sigma)}\|g\|_{W^{k,q_2}(\Sigma)}.  
\end{equation}

To this aim we use H\"older's inequality as 
\[
 \sum_{i+j=k} \|\bar  \nabla^i f \star \bar  \nabla^j g \|_{L^2(\Sigma)} \leq  \sum_{i+j=k} \|\bar  \nabla^i f\|_{L^4(\Sigma)} \|\bar  \nabla^j g \|_{L^4(\Sigma)} .
\]
By interpolation inequality in Proposition \ref{prop:interpolation} we have 
\[
\|\bar \nabla^{i} f\|_{L^4(\Sigma)} \leq C\|f\|_{W^{k,p_1}(\Sigma)}^{\theta_i} \| f\|_{L^{p_2}(\Sigma)}^{1-\theta_i}
\qquad\text{ with }\qquad
\theta_i= \frac{\frac{i}{2}-\frac{1}{4} + \frac{1}{p_2}}{\frac{k}{2} + \frac{1}{p_2} - \frac{1}{p_1}}
\]
and recalling that $\frac{1}{p_1} + \frac{1}{q_1} = \frac{1}{p_2} + \frac{1}{q_2} = \frac12$ we have 
\[
\|\bar \nabla^{j} g\|_{L^4(\Sigma)} \leq C\|g\|_{W^{k,q_2}(\Sigma)}^{\theta_j} \| g\|_{L^{q_1}(\Sigma)}^{1-\theta_j}
\qquad\text{ with }\qquad
\theta_j= \frac{\frac{j}{2}-\frac{1}{4} + \frac{1}{q_1}}{\frac{k}{2} + \frac{1}{p_2} - \frac{1}{p_1}}.
\]
In particular, $i+j = k$ implies $ \theta_i + \theta_j =1$. Therefore we have by Young's inequality $ a^{\theta_i} b^{\theta_j} \leq \theta_i a+ \theta_j b \leq a+b $ that 
\[
\begin{split}
 \sum_{i+j=k} \|\bar\nabla^i f \star \bar \nabla^j g \|_{L^2} &\leq C \| f\|_{L^{p_2}}\| g\|_{L^{q_1}}\sum_{i+j=k}  \|f\|_{W^{k,p_1}}^{\theta_i} \| f\|_{L^{p_2}}^{-\theta_i}\|g\|_{W^{k,q_2}}^{\theta_j} \| g\|_{L^{q_1}}^{-\theta_j}\\
 &\leq C \| f\|_{L^{p_2}}\| g\|_{L^{q_1}} \left(\frac{\|f\|_{W^{k,p_1}}}{\| f\|_{L^{p_2}}} + \frac{\|g\|_{W^{k,q_2}}}{\| g\|_{L^{q_1}}} \right)
 \end{split}
\]
and the claim follows. 
\end{proof}

We remark that we do not generalize the last inequality in Proposition \ref{prop:kato-ponce} for half-integers $k$ since we do not define the space $W^{k,p}(\Sigma)$ for $p \neq 2$, when $k$ is not an integer. However, under the assumption of Proposition \ref{prop:kato-ponce}, we obtain a weaker version which reads as follows
\beq \label{eq:kato-weak}
\|fg\|_{H^{\frac12}(\Sigma)}\leq C \|f\|_{H^{\frac12}(\Sigma)}\|g\|_{L^{\infty}(\Sigma)}+\|f\|_{L^{p}(\Sigma)}\|g\|_{W^{1,q}(\Sigma)},
\eeq
for $\frac{1}{p} + \frac{1}{q} = \frac12$. Again, since the proof is similar to the argument we used in \eqref{eq:equiv-frac-1/2} we leave it for the reader, but refer to  \cite[Lemma 4.3]{DDM} for the proof of the case $p=q=4$. In particular, when $g$ is   Lipschitz, we may estimate the product simply by
\[
\|fg\|_{H^{\frac12}(\Sigma)} \leq C\|f\|_{H^{\frac12}(\Sigma)}\|g\|_{C^{1}(\Sigma)}.
\]

Next we recall (see e.g. \cite{FJM3D})    that it holds $\|u\|_{H^{k+2}(\Sigma)} \leq C_{\Sigma}(\|\Delta_\Sigma u\|_{H^{k}(\Sigma)} + \|u \|_{L^2(\Sigma)})$. However,  the constant depends on the curvature of $\Sigma$ and we need to quantify this dependence. 
\begin{proposition}
\label{prop:laplace-bound}
Assume that $\Sigma$ is uniformly $C^{1,\alpha}(\Gamma)$-regular and satisfies  $\|B_\Sigma\|_{L^4} \leq M$. For all $\e>0$ there exists a constant $C_\e$ such that for $k = 0, \frac12, 1$ it holds 
\[
\|u\|_{H^{k+2}(\Sigma)} \leq (1+\e)\|\Delta_\Sigma u\|_{H^{k}(\Sigma)} +   C_{\e} \|u \|_{L^2(\Sigma)}.
\]
Let $m$ be an integer with $m \geq 3$ and assume that $\Sigma$ satisfies in addition the condition \eqref{eq:notationhp}. Then for every half-integer $2 \leq k \leq m$  it holds
\[
\|u\|_{H^{k}(\Sigma)} \leq (1+\e)\|\Delta_\Sigma u\|_{H^{k-2}(\Sigma)} +   C_{\e} \|u \|_{L^2(\Sigma)}.
\]
The constant $C_\e$ depends on $\e, M, m$ and on the  $C^{1,\alpha}$-norm of the heightfunction. 
\end{proposition}

\begin{proof}
The case $k = 0$ follows from  \cite[Lemma 4.11]{DDM} but we give the proof for the reader's convenience. We  recall that the Riemann tensor $R$ satisfies by \eqref{eq:curv-tensor2} $|R|\leq C|B|^2$ and deduce by     \cite[Remark 2.4]{FJM3D} (see also \cite{AubinBook2}) that 
\[
\|\bar \nabla^2 u\|_{L^{2}(\Sigma)}^2 \leq \|\Delta_\Sigma u\|_{L^{2}(\Sigma)}^2 + C \int_\Sigma |B|^2 |\bar \nabla u|^2\, d \H^2 .
\]
By Proposition  \ref{prop:interpolation} there is $\theta \in (0,1)$ such that 
\[
 \int_\Sigma |B|^2 |\bar \nabla u|^2\, d \H^2 \leq \|B\|_{L^4}^2 \|\bar \nabla u\|_{L^{4}}^2 \leq C  \|  u\|_{H^{2}}^{2\theta} \|u\|_{L^{2}}^{2(1-\theta)} \leq \e  \|  u\|_{H^{2}} + C_{\e}\|u\|_{L^{2}}^2.
\]
This implies the claim for $k =0$.

For the case $k=1$ we use \eqref{eq:curv-tensor} and integration by parts
\[
\begin{split}
\|\bar \nabla \Delta_\Sigma u\|_{L^{2}(\Sigma)}^2 &=  \int_{\Sigma}   \bar \nabla^k \bar \nabla_i\bar \nabla^i u  \, \bar \nabla_k \bar \nabla^j\bar \nabla_j u \, d \H^2 \\
&= \int_{\Sigma} \bar \nabla_i  \bar \nabla^k \bar \nabla^i u  \, \bar \nabla_k \bar \nabla^j\bar \nabla_j u \, d \H^2 + \int_{\Sigma} (R  \star \bar \nabla u \star \bar \nabla^3 u) \, d \H^2\\
&\geq -\int_{\Sigma}   \bar \nabla^k \bar \nabla^i u  \, \bar \nabla_i \bar \nabla_k \bar \nabla^j\bar \nabla_j u \, d \H^2 - \|R \star \bar \nabla u\|_{L^2}\, \|\bar \nabla^3 u\|_{L^2}.
\end{split}
\]
As before we have by Proposition \ref{prop:interpolation} and by $|R|\leq C|B|^2$ that 
\begin{equation}
    \label{eq:laplace-bound1}
\|R \star \bar \nabla u\|_{L^2} \leq  C\|B\|_{L^4}^2\|\bar \nabla u\|_{L^\infty} \leq  \e \|u\|_{H^3} + C_{\e} \| u\|_{L^2}.
\end{equation}
We proceed by using \eqref{eq:curv-tensor} and by integrating by parts
\[
\begin{split}
 -\int_{\Sigma}   &\bar \nabla^k \bar \nabla^i u  \, \bar \nabla_i \bar \nabla_k \bar \nabla^j\bar \nabla_j u \, d \H^2 \\
 &\geq -\int_{\Sigma}   \bar \nabla^k \bar \nabla^i u  \,  \bar \nabla_k \bar \nabla_i \bar \nabla^j\bar \nabla_j u \, d \H^2  + \int_{\Sigma}   \bar \nabla^2 u \star R \star \bar \nabla^2 u\, d \H^2 \\
 &\geq -\int_{\Sigma}   \bar \nabla^k \bar \nabla^i u  \,  \bar \nabla_k \bar \nabla^j \bar \nabla_i  \bar \nabla_j u \, d \H^2 - \int_{\Sigma}   \bar \nabla^k \bar \nabla^i u  \,  \bar \nabla_k \bar [ \bar \nabla_i\nabla^j -  \bar \nabla_j\nabla^i] \bar \nabla_j u \, d \H^2\\
 &\,\,\,\,\,\,\,\,- C\|B\|_{L^4}^2 \|\bar \nabla^2 u\|_{L^4}^2\\
 &\geq \int_{\Sigma}   \bar \nabla^j \bar \nabla^k \bar \nabla^i u  \,  \bar \nabla_k  \bar \nabla_i  \bar \nabla_j u \, d \H^2 + \int_{\Sigma}   \bar \nabla_k   \bar \nabla^k \bar \nabla^i u  \, \bar [ \bar \nabla_i\nabla^j -  \bar \nabla_j\nabla^i] \bar \nabla_j u \, d \H^2\\
 &\,\,\,\,\,\,\,\,- C\|B\|_{L^4}^2 \|\bar \nabla^2 u\|_{L^4}^2\\
 &\geq \|\bar \nabla^3 u\|_{L^2}^2 - C\|B\|_{L^4}^2 \|\bar \nabla^2 u\|_{L^4}^2 - \|R  \star \bar \nabla u\|_{L^2}\, \|\bar \nabla^3 u\|_{L^2}.
 \end{split}
\]
The inequality for $k =1$ then follows from \eqref{eq:laplace-bound1} and from Proposition \ref{prop:interpolation} which yields 
\[
\|\bar \nabla^2 u\|_{L^4} \leq \e \|u\|_{H^3} + C_\e \|u\|_{L^2}. 
\]

The case $k=1/2$ follows from the previous two estimates and   Proposition \ref{prop:sobolev} with standard interpolation argument which we briefly sketch here for the reader's convenience. 
We define a linear operator $\mathcal{F}: \tilde H^{k}(\Sigma) \to \tilde H^{k+2}(\Sigma)$ such that $\mathcal{F}(g) = u$, where $u$ is the solution of 
\[
\Delta_\Sigma u = g \qquad \text{on } \, \Sigma, 
\]
and $\tilde H^k(\Sigma) = \{ f \in H^k(\Sigma) : \int_{\Sigma} f \, d \H^2 = 0\}$. The operator $\mathcal{F}$ is well-defined and by the previous estimates it satisfies 
\[
\|\mathcal{F}\|_{\mathcal L(L^2,H^2)} \leq C \qquad \text{and} \qquad \|\mathcal{F}\|_{\mathcal L(H^1,H^3)} \leq C 
\]
 By the  interpolation theory discussed in \eqref{ineq:inter-functional}  it holds 
\[
\|\mathcal{F}(g)\|_{H_\star^{\frac52}(\Sigma)} \leq C \|g\|_{H_\star^{\frac12}(\Sigma)}.
\]
Proposition \ref{prop:sobolev} then yields
\[
\|\mathcal{F}(g)\|_{H^{\frac52}(\Sigma)} \leq C \|g\|_{H^{\frac12}(\Sigma)}.
\]
We apply this to $\tilde u = u - \bar u$, where $\bar u = \fint_{\Sigma} u \, d \H^2$ and the claim follows. 

 The argument for  higher $m$ and $k$ is similar and we merely sketch it. Let $k$ be an integer with $2 \leq k \leq m$. Using  \eqref{eq:curv-tensor} and arguing as above we obtain after long but straightforward calculations that
\[
\|\bar \nabla^{k} u\|_{L^{2}(\Sigma)}^2 \leq  \|\bar \nabla^{k-2} \Delta_\Sigma u\|_{L^{2}(\Sigma)}^2  + C \sum_{\alpha+ \beta\leq k-2}\|\bar \nabla^\alpha R \star \bar \nabla^{1+\beta} u \|_{L^2(\Sigma)}^2.
\]
Then by \eqref{eq:curv-tensor2}, \eqref{eq:est-product},  Proposition \ref{prop:interpolation} and by the assumption  $\|B\|_{L^\infty}, \|B\|_{H^{k-2}}\leq C$ we have  
\[
\begin{split}
\sum_{\alpha+ \beta\leq k-2}\|\bar \nabla^\alpha (R \star \bar \nabla^{1+\beta} u) \|_{L^2(\Sigma)} &\leq C \|B\|_{L^\infty}^2 \|u\|_{H^{k-1}} + C \|B\|_{L^\infty}  \|B\|_{H^{k-2}} \|\bar \nabla u\|_{L^\infty}\\
&\leq \e \|u\|_{H^k(\Sigma)} + C_{\e} \|u\|_{L^2(\Sigma)}.
\end{split}
\]
This yields the claim for integers $2 \leq k \leq l$. 

If $k \leq m-\frac12$ is an half-integer but not an integer, then we may use the previous argument for integer $l = k + \frac12 \leq m$  and deduce 
\[
\| u\|_{H^{l}} \leq  (1+\e)\|\Delta_\Sigma u\|_{H^{l-2}(\Sigma)} +   C_{\e} \|u \|_{L^2(\Sigma)}.
\]
The same holds for $l-1$. Hence, the claim follows by Proposition \ref{prop:sobolev} and by the same interpolation argument we used above. 
\end{proof}

By using the previous proposition and the Simon's identity \eqref{eq:Simon} we deduce that we may bound the second fundamental form by the mean curvature.

\begin{proposition}
\label{prop:meancrv-bound}
Assume that $\Sigma$ is  uniformly $C^{1, \alpha}(\Gamma)$-regular. Then for every $p \in (1,\infty)$ it holds
\[
\| B_\Sigma\|_{L^p(\Sigma)} \leq C(1+ \|H_\Sigma\|_{L^p(\Sigma)}). 
\]
If in addition  $\|B_\Sigma\|_{L^4(\Sigma)} \leq M$, then  for  $k =\frac12, 1, 2$ it holds
\[
\|B_\Sigma\|_{H^{k}(\Sigma)} \leq C(1+ \|H_\Sigma\|_{H^{k}(\Sigma)}).
\]
Finally let $m\geq 3$ be an integer  and assume that $\Sigma$ satisfies in addition the condition \eqref{eq:notationhp} for $m$. Then the  above estimate holds for all half-integers  $k\leq m$. The constants depend on $M, p, m$ and on the  $C^{1,\alpha}$-norm of the heightfunction. 
\end{proposition}

\begin{proof}
The first claim follows from standard Calderon-Zygmund estimate \cite{GT} and we omit it. Let us proof the second claim for $k=\frac12$. We recall the geometric fact 
\[
\Delta_\Sigma x_i = - H_\Sigma \nu_i,
\]
where $x_i = x \cdot e_i$ and $\nu_i = \nu_\Sigma \cdot e_i$.  Then we have by Proposition \ref{prop:laplace-bound} and  \eqref{eq:kato-weak}
\[
\begin{split}
\|B_\Sigma \|_{H^{\frac12}(\Sigma)} &\leq C\sum_{i=1}^3(1+ \|\nabla^2_\Sigma x_i  \|_{H^{\frac12}(\Sigma)})  \leq \sum_{i=1}^3 C(1+ \|\Delta_\Sigma x_i  \|_{H^{\frac12}(\Sigma)}) \\
&= \sum_{i=1}^3 C(1+ \| H_\Sigma \nu_i\|_{H^{\frac12}(\Sigma)}) \\
&\leq C(1+  \| H_\Sigma\|_{H^{\frac12}(\Sigma)} + \| H_\Sigma\|_{L^{4}(\Sigma)} \|\nu_\Sigma\|_{W^{1,4}(\Sigma)}) \leq C(1+  \| H_\Sigma\|_{H^{\frac12}(\Sigma)}).
\end{split}
\]
The argument for $k=1$ is similar. 

In the case $k =2$ we use the Simon's identity \eqref{eq:Simon}  to deduce 
\[
\|\Delta_\Sigma B\|_{L^2(\Sigma)}^2 \leq \|\bar \nabla^{2} H\|_{L^2(\Sigma)}^2 + C \| B\|_{L^6(\Sigma)}^6.
\]
 Proposition \ref{prop:laplace-bound}   yields  $\| B\|_{H^2(\Sigma)} \leq 2\|\Delta_\Sigma B\|_{L^{2}(\Sigma)} +C$. The claim then follows from interpolation inequality (Proposition \ref{prop:interpolation})
 \[
\| B\|_{L^6(\Sigma)}^6 \leq \|B\|_{H^2(\Sigma)}^{\frac23} \|B\|_{L^4(\Sigma)}^{\frac{16}{3}} \leq \e \|B\|_{H^2(\Sigma)} + C_\e.
 \]

Let us then fix $m\geq 3$, assume that $\Sigma$ satisfies the condition \eqref{eq:notationhp} for $m$ and let $k \leq m$.
We use the Simon's identity \eqref{eq:Simon} and Proposition  \ref{prop:kato-ponce} to deduce 
\[
\begin{split}
\|\Delta_\Sigma B\|_{H^{k-2}(\Sigma)} &\leq \| H\|_{H^k(\Sigma)} + C \|B\star B \star B\|_{H^{k-2}(\Sigma)}\\
&\leq \| H\|_{H^k(\Sigma)} + C \|B\|_{L^\infty(\Sigma)}^2 \|B\|_{H^{k-2}(\Sigma)}\\
&\leq \| H\|_{H^k(\Sigma)} + \e\|B\|_{H^{k}(\Sigma)} + C_\e,
\end{split}
\]
where the last inequality follows from $\|B\|_{L^\infty}\leq C$ and from interpolation. The claim then follows from  Proposition \ref{prop:laplace-bound}.  
\end{proof}

Note that by the definition of the space $\|\cdot \|_{H^{k}(\Sigma)}$ in Definition \ref{defn:fractionalnorm2} it is not yet clear if it holds 
\[
\|\nabla_\tau u\|_{H^{k-1}(\Sigma)} \leq C \| u\|_{H^{k}(\Sigma)}
\]
when $k$ is not an integer. We conclude this by section by proving this in the following technical lemma.  

\begin{lemma}
\label{lem:frac-gradient}
Let $m$ be an integer with $m \geq 3$ and assume that $\Sigma$ is   uniformly $C^{1, \alpha}(\Gamma)$-regular and satisfies the condition \eqref{eq:notationhp}. Then it holds
\[
\|\nabla_\tau u \|_{H^{m-\frac32}(\Sigma)} \leq C \| u \|_{H^{m-\frac12}(\Sigma)}.
\]
\end{lemma}
\begin{proof}
Let us denote $\tilde u = u \circ \Psi : \Gamma \to \R$ and in order to simplify the notation denote the extension given by the projection  in \eqref{def:projection} $(\tilde u \circ \pi_\Gamma)$ simply  by $\tilde u$. We observe that there is  a matrix field $A(x) = A(x,h, \bar \nabla h)$ such that 
\[
(\nabla_\tau u \circ \Psi)(x)  = A(x) \nabla \tilde u(x) \qquad \text{for  } \, x \in \Gamma.
\]
Therefore we have by Definition \ref{defn:fractionalnorm2} and by Proposition \ref{prop:kato-ponce}
\[
\begin{split}
\|\nabla_\tau u  \|_{H^{m-\frac32}(\Sigma)} &= \|\nabla_\tau u \circ \Psi \|_{H^{m-\frac32}(\Gamma)} = \|A \, \nabla \tilde u \|_{H^{m-\frac32}(\Gamma)} \\
&\leq C\|A\|_{L^\infty} \, \|\nabla \tilde u \|_{H^{m-\frac32}(\Gamma)} + C \|A\|_{H^{m-\frac32}(\Gamma)} \, \|\nabla \tilde u \|_{L^\infty}.
\end{split}
\]

The assumption $\|B\|_{L^\infty(\Sigma)},\|B\|_{H^{m-2}(\Sigma)}\leq C$ implies  for the height function  $\|h\|_{C^2(\Gamma)}\leq C$ and  $ \|h\|_{H^m(\Gamma)} \leq C$ and therefore $\|A\|_{L^{\infty}(\Gamma)},  \|A\|_{H^{m-1}(\Gamma)} \leq C$. Moreover,  since $m\geq 3$  the Sobolev embedding  yields $\|\nabla  \tilde u\|_{L^\infty(\Gamma)} \leq C\|\nabla \tilde u\|_{H^{m-\frac32}(\Gamma)}$. Therefore we have 
\[
\|\nabla_\tau u  \|_{H^{m-\frac32}(\Sigma)} \leq C \|\nabla  \tilde u \|_{H^{m-\frac32}(\Gamma)} \leq C\| \tilde  u \|_{H^{m-\frac12}(\Gamma)} = C \| u  \|_{H^{m-\frac12}(\Sigma)}.
\]
\end{proof}

\section{Elliptic estimates for vector fields and functions}

In this section we recall some known and  provide some new  div-curl type estimates for vector fields in the domain, i.e.,  $F : \Omega \to \R^3$. We will need estimates where we control the norm $\|F\|_{H^k(\Omega)}$ by the $\Div F$, $\curl F$ in $\Omega$ and with $F_n$ on the boundary $\Sigma$. The main result of the section is Theorem \ref{prop:CS-k=1} where we prove this estimate for $k=1$ and require the boundary merely to satisfy $\|B_\Sigma\|_{L^4} \leq C$. We do not expect the  $L^4$-integrability  to be the optimal condition.  However, related to this we note that we may   construct a cone $\Omega \subset \R^3$ and a harmonic function $u : \Omega \to \R$ with zero Neumann boundary data $\pa_\nu u = 0$  arguing as in \cite[Section 3]{FJ}, such that $u$ can be written in spherical coordinates as $u(\rho,\theta) = \sqrt{\rho} f(\theta)$ for a smooth functions $f$. In particular,  $ u \notin H^2(\Omega \cap B_R)$  and therefore we  may deduce that a necessary condition for the curvature is at least  $\|B_\Sigma\|_{L^2} \leq C$ for Lemma \ref{lem:poisson1} and Theorem \ref{prop:CS-k=1} to hold.

We will also prove boundary regularity estimates for harmonic functions in Theorem \ref{teo:reg-capa}, which quantify the boundary regularity of the harmonic functions with respect to the regularity of the boundary. We note that in Theorem \ref{teo:reg-capa} it is crucial to assume that the boundary is uniformly $C^{1,\alpha}(\Gamma)$-regular. Indeed, the statement does not hold for Lipschitz domains.  

\subsection{Regularity estimates for vector fields}

We begin this section by recalling the following result which is essentially from \cite{CS} (see also \cite{Taylor}). Recall that we define 
\[
\curl F = \nabla F - (\nabla F)^T.
\]
Throughout the section we assume that $\Omega$ is connected, but its boundary $\Sigma = \pa \Omega$ may have many components.   
\begin{theorem}
\label{thm:chen-shkoller}
Let $l \geq 2$ be an integer and let  $\Omega$ be a  domain such that  $\Sigma= \pa \Omega$ is uniformly $C^{1,\alpha}(\Gamma)$-regular and $\|B_{\Sigma}\|_{H^{\frac32l-1}(\Sigma)}\leq M$. Then there exists a constant $C$, which depends on $M, l$ and on the $C^{1,\alpha}$-norm of the heightfunction, such that for all smooth vector fields $F:\Omega \to \R^3$   and every half-integers $1\leq k \leq \frac32 l$ it holds
\[
\|F\|_{H^{k}(\Omega)} \leq  C( \| F_n\|_{H^{k-\frac12}(\Sigma)}+ \| F\|_{L^2( \Omega)}  +\|\Div F\|_{H^{k-1}( \Omega)}+\|\curl F\|_{H^{k-1}} ).
\]
Moreover, for $k = \lfloor \frac32(l+1)\rfloor$ it holds 
\[
\|F\|_{H^{k}( \Omega)} \leq C( \| \nabla_\tau F_n\|_{H^{k-\frac32}(\Sigma)}+  (1+\|B_\Sigma\|_{H^{\frac32 l}}) \| F\|_{L^\infty}  +\|\Div F\|_{H^{k-1}( \Omega)}+\|\curl F\|_{H^{k-1}} ).
\]
\end{theorem}
\begin{proof}
We first note that the assumption $\|B\|_{H^{\frac32l-1}(\Sigma)}\leq M$ implies that $\Sigma$ satisfies the condition \eqref{eq:notationhp} for $m =\lfloor \frac32 l + 1\rfloor \geq 4$. 
We use \cite[Theorem 1.3]{CS} to  deduce 
\[
\|F\|_{H^{k}(\Omega)} \leq  C( \| \nabla_\tau F \cdot \nu \|_{H^{k-\frac32}(\Sigma)}+ \| F\|_{L^2(\Omega)}  +\|\Div F\|_{H^{k-1}(\Omega)}+\|\curl F\|_{H^{k-1}(\Omega)} )
\]
for all $k \leq \lfloor \frac32(l+1)\rfloor$. We write $\nabla_\tau F \cdot \nu = \nabla_\tau F_n + F \star B$ and  use Proposition \ref{prop:kato-ponce} to obtain
\[
\| \nabla_\tau F \cdot \nu \|_{H^{k-\frac32}(\Sigma)} \leq \| \nabla_\tau F_n \|_{H^{k-\frac32}(\Sigma)} + C\|  F\|_{L^\infty} \| B\|_{H^{k-\frac32}(\Sigma)} + C\|  B\|_{L^\infty} \| F\|_{H^{k-\frac32}(\Sigma)}.
\]
Interpolation inequality yields
\[
\| F\|_{H^{k-\frac32}(\Sigma)} \leq \| F\|_{H^{k-1}(\Omega)} \leq \e \| F\|_{H^{k}(\Omega)} + C_{\e}\| F\|_{L^{\infty}(\Omega)}.
\]
Thus we have the second inequality. The first one follows from the fact that for $k \leq \frac32 l$ it holds $\| B\|_{H^{k-\frac32}(\Sigma)} \leq M$ by the assumption. 
\end{proof}

We combine Proposition \ref{prop:laplace-bound} and Theorem \ref{thm:chen-shkoller} and obtain  the following  inequality which is suitable to our purpose. 

\begin{proposition}
\label{prop:vector-high}
Let $l$ and $\Omega$ be as in Theorem \ref{thm:chen-shkoller}. Then for all smooth vector fields $F:\Omega \to \R^3$   and every half-integer $\frac52 \leq k \leq \frac32 l$ it holds
\[
\|F\|_{H^{k}(\Omega)} \leq  C( \|\Delta_\Sigma F_n\|_{H^{k-\frac52}(\Sigma)}+ \| F\|_{L^2(\Omega)}  +\|\Div F\|_{H^{k-1}(\Omega)}+\|\curl F\|_{H^{k-1}(\Omega)} ).
\]
Moreover, for $k = \lfloor \frac32(l+1)\rfloor$ it holds 
\[
\|F\|_{H^{k}(\Omega)} \leq  C( \| \Delta_\Sigma F_n\|_{H^{k-\frac52}(\Sigma)}+  (1+\|B\|_{H^{\frac32 l}}) \| F\|_{L^\infty}  +\|\Div F\|_{H^{k-1}(\Omega)}+\|\curl F\|_{H^{k-1}(\Omega)} ).
\]
\end{proposition}
\begin{proof}
Recall that  $\|B\|_{H^{\frac32 l-1}(\Sigma)}\leq M$ implies that $\Sigma $  satisfies the condition \eqref{eq:notationhp} for $m =\lfloor \frac32 l + 1\rfloor \geq 4$. The first  inequality then follows from Theorem \ref{thm:chen-shkoller} and  Proposition \ref{prop:laplace-bound}. 

Let us then prove the last inequality. We have by Proposition  \ref{prop:laplace-bound} that 
\[
\| \nabla_\tau F_n\|_{H^{k-\frac32}(\Sigma)} \leq C(\| \Delta_{\Sigma} \nabla_\tau F_n\|_{H^{k-\frac72}(\Sigma)} + \|F_n\|_{L^2(\Sigma)}).
\]
We use the commutation formula  \eqref{eq:curv-tensor} for the tangential gradient of  $u : \Sigma \to \R$  and obtain
\[
\Delta_{\Sigma} (\nabla_\tau u) =  \nabla_\tau (\Delta_\Sigma u) + (B \star B) \star \nabla_\tau u. 
\]
Therefore we have by Lemma \ref{lem:frac-gradient}
\[
\begin{split}
\| \Delta_{\Sigma} \nabla_\tau F_n\|_{H^{k-\frac72}(\Sigma)} &\leq C\left( \|\nabla_\tau (\Delta_\Sigma F_n) \|_{H^{k-\frac72}(\Sigma)} +  \| (B \star B) \star \nabla_\tau F_n \|_{H^{k-\frac72}(\Sigma)} \right)\\
&\leq C ( \|\Delta_\Sigma F_n \|_{H^{k-\frac52}(\Sigma)} +  \| (B \star B) \star \nabla_\tau F_n \|_{H^{k-\frac72}(\Sigma)} ).
\end{split}
\]
Recall that $k = \lfloor \frac32(l+1)\rfloor$. We have by Proposition \ref{prop:kato-ponce},  by Lemma \ref{lem:frac-gradient} and by the assumption $\|B\|_{H^{\frac32 l -1}(\Sigma)} \leq C$ that 
\[
\begin{split}
\| (B \star B) \star \nabla_\tau F_n \|_{H^{k-\frac72}(\Sigma)} &\leq C\left( \|B\|_{L^\infty}^2 \|\nabla_\tau F_n \|_{H^{k-\frac72}} + \|B\|_{L^\infty} \|B\|_{H^{k-\frac72}} \|\nabla_\tau F_n\|_{L^\infty}  \right)\\
&\leq C (\| F_n \|_{H^{k-\frac52}(\Sigma)} + \|\nabla_\tau F_n\|_{L^\infty(\Sigma)}  ).
\end{split}
\]
Finally we have by the Sobolev embedding and  by Corollary \ref{coro:interpolation} 
\[
\|\nabla_\tau F_n\|_{L^\infty(\Sigma)} +\| F_n \|_{H^{k-\frac52}(\Sigma)}    \leq \e \|\nabla_\tau  F_n\|_{H^{k-\frac32}(\Sigma)} + \e \|F\|_{H^k(\Omega )}+ C_{\e} \|F\|_{L^\infty}.
\]
The second inequality then follows from  Theorem \ref{thm:chen-shkoller} and  combining the above inequalities.  
\end{proof}

Proposition  \ref{prop:vector-high} provides the inequality we need when we have the bound $\|B_\Sigma\|_{H^{\frac32l -1}(\Sigma)} \leq C$ for  $l \geq 2$. When $l=1$ the above bound reduces to $\|B_\Sigma\|_{H^{\frac12}(\Sigma)}$, which is the bound that we are able to prove in Section 6, but is not enough to apply the  results from \cite{CS, Taylor}. Note that by the Sobolev embedding this implies $\|B_{\Sigma}\|_{L^{4}(\Sigma)} \leq C$. We need to work more in order to prove the first inequality in Theorem \ref{thm:chen-shkoller} under the assumption   $\|B_{\Sigma}\|_{L^{4}(\Sigma)} \leq C$.

We begin by recalling  the following Reilly's type identity for vector fields. First, if $\psi:\Omega \to \R^3$ is a smooth divergence free vector field such that $\psi \cdot \nu = 0$ on $\Sigma$ then it holds 
 \begin{equation}
 \label{eq:divcurl1}
 \|\nabla \psi\|_{L^2(\Omega)}^2 =  \frac12\|\curl \psi\|_{L^2(\Omega)}^2 - \int_{\Sigma} \langle B_\Sigma \, \psi, \psi \rangle \, d \H^2.
  \end{equation}
 Second, if $u: \Omega \to \R$ is a smooth function then it holds
 \begin{equation}
 \label{eq:divcurl2}
 \begin{split}
 \|\nabla^2 u\|_{L^2(\Omega)}^2 =  &\|\Delta u\|_{L^2(\Omega)}^2 -2\int_{\Sigma} \Delta_\Sigma u \, \pa_\nu u \, d \H^2 \\
 &- \int_{\Sigma} \langle B_\Sigma \bar \nabla  u, \bar \nabla u \rangle \, d \H^2 - \int_{\Sigma} H_\Sigma (\pa_\nu u)^2 \, d \H^2.
 \end{split}
  \end{equation}

 We give the calculations for \eqref{eq:divcurl1} and \eqref{eq:divcurl2} for the reader's convenience. First, for a generic smooth vector field $F: \Omega\to \R^3$  it holds 
\[
\begin{split}
\int_{\Omega} |\Div F|^2 + \frac12 |\text{curl}\, F|^2\,dx &=
\sum_{i,j=1}^3\int_{\Omega} (\pa_i F_j)^2\,dx
+ \sum_{i,j=1}^3\int_{\Omega} (\pa_i F_i\pa_j F_j - \pa_i F_j \pa_j F_i) \,dx\\
&= \|\nabla F\|_{L^2(\Omega)}^2 + \sum_{i,j=1}^3\int_{\Omega} (\pa_i F_i\pa_j F_j - \pa_i F_j \pa_j F_i) \,dx. 
\end{split}
\]
By using divergence theorem twice we obtain
\[
\begin{split}
\int_{\Omega} \pa_i F_i\pa_j F_j\, dx &= - \int_{\Omega} F_i \pa_i \pa_j F_j\, dx + \int_{\Sigma}  \pa_j F_j  \,  F_i \nu_i,\ d \H^2\\
&= \int_{\Omega} \pa_j F_i \pa_i  F_j\, dx  + \int_{\Sigma}  \pa_j F_j  \,  F_i \nu_i,\ d \H^2 - \int_{\Sigma} F_i \pa_i F_j \nu_j \ d \H^2.
\end{split}
\]
Combining the two above equalities yield 
\begin{equation}
 \label{eq:divcurl3}
\begin{split}
\|\nabla F\|_{L^2(\Omega)}^2  &= \|\Div F\|_{L^2(\Omega)}^2 + \frac12\|\text{curl}\, F\|_{L^2(\Omega)}^2\\
&+ \int_{\Sigma} (\nabla F \, F) \cdot  \nu \ d \H^2 - \int_{\Sigma} \Div F \, F_n \, d \H^2 .
\end{split}
  \end{equation}

Assume now that  $\psi $ is a divergence free vector field such that $\psi\cdot \nu = 0$ on $\Sigma$. Since $\psi$ is a tangent field on $\Sigma$  we have  $\nabla_\psi (\psi\cdot \nu)=0$ and thus $ \nabla \psi \, \psi \cdot \nu= -\langle \bar \nabla_\psi\nu,\psi \rangle= -\langle B_\Sigma \psi,\psi\rangle$.
Therefore the equality \eqref{eq:divcurl1} follows from \eqref{eq:divcurl3} and from $F_n = \psi \cdot \nu = 0$.

To obtain  \eqref{eq:divcurl2} we apply \eqref{eq:divcurl3} for $F = \nabla u$ and deduce 
\[
\|\nabla^2 u\|_{L^2(\Omega)}^2 = \|\Delta u\|_{L^2(\Omega)}^2 +  \int_{\Sigma} ((\nabla^2  u \,  \nabla u) \cdot  \nu  -\Delta u \, \pa_\nu u) \, d \H^2.
\]
We write $\nabla u = \nabla_\tau u + \pa_\nu u \, \nu$ and observe
\[
\begin{split}
(\nabla^2  u \,  \nabla u )\cdot  \nu &= (\nabla^2  u \,  \nabla_\tau u) \cdot  \nu  + (\nabla^2  u \,  \nu )\cdot  \nu \,\pa_\nu u \\
&=\nabla_\tau (\pa_\nu u) \cdot \nabla_\tau u - \langle B_\Sigma \, \bar \nabla u, \bar \nabla u \rangle  + (\nabla^2  u \,  \nu) \cdot  \nu\,\pa_\nu u.
\end{split}
\]
Again by divergence theorem 
\[
\int_{\Sigma} \nabla_\tau (\pa_\nu u) \cdot \nabla_\tau u \, d \H^2 = - \int_{\Sigma} \Delta_\Sigma u \, \pa_\nu u\, d \H^2.
\]
The equality \eqref{eq:divcurl2}  then  follows from 
\begin{equation}
    \label{eq:tang-laplace}
\Delta_\Sigma u = \Delta u - (\nabla^2  u \,  \nu )\cdot  \nu - H_\Sigma \, \pa_\nu u.
\end{equation}

We remark that it is crucial in Theorem \ref{thm:chen-shkoller} that the boundary term on the RHS has only the normal component of the vector field. The next lemma is a generalization of \cite{JK} and it essentially states that we may control the vector field on the boundary by its normal or its tangential component. 

\begin{lemma} 
\label{lem:divcurl}
Let $\Omega\subset \R^3$ with  $\Sigma =\pa \Omega$ be uniformly  $C^{1,\alpha}(\Gamma)$-regular. Then for all  vector fields $F \in \dot{H}^1(\Omega; \R^3)$ 
it holds
\[
\|F\|_{L^2(\Sigma)}^2 \leq C \left(\|F_n \|_{L^2(\Sigma)}^2  + \|F\|_{L^2(\Omega)}^2 +\|\Div F\|_{L^2(\Omega)}^2 +\|\curl F\|_{L^2(\Omega)}^2 
\right)
\]
and
\[
\|F\|_{L^2(\Sigma)}^2 \leq C \left(\|F_\tau \|_{L^2(\Sigma)}^2  + \|F\|_{L^2(\Omega)}^2 +\|\Div F\|_{L^2(\Omega)}^2 +\|\curl F\|_{L^2(\Omega)}^2
\right),
\]
where $F_n = F\cdot \nu$ and $F_\tau = F - F_n \, \nu$. Here $F \in \dot{H}^1(\Omega; \R^3)$ means that $\|F\|_{\dot{H}^1(\Omega)} = \|\nabla F \|_{L^2(\Omega)} + \| F \|_{L^6(\Omega)} < \infty$. Note that $\Omega$ may be unbounded, but 
its boundary is compact. 
\end{lemma}
\begin{proof}
We only consider the case when $\Omega$ is bounded. Let us first  assume that $\Omega$ is uniformly star shaped with respect to the origin, i.e.,  we have that there exists a constant $c_0>0$ such that $x \cdot \nu_\Omega \geq c_0$ for all $x \in \Sigma$.  We claim that the following identity holds
\beq \label{eq:dividentity}
\Div \left( |F|^2x- 2 (F\cdot x)  F   \right)
= |F|^2 -2 \curl F (F  \cdot x)
-2 \Div F \,  (F \cdot x).
\eeq
Indeed, this follows from the following straightforward computation, where we denote  the Dirac delta  by $\delta_i^j$,  
\[
\begin{split}
    \Div ( |F|^2x- &2 (F\cdot x)  F  ) = \sum_{i,j =1}^3
    \pa_i \left( F_j^2 x_i - 2x_jF_j F_i\right)\\
    &= 3|F|^2 + \sum_{i,j =1}^3  2x_i\pa_iF_jF_j  - 2 \delta_i^j F_jF_i - 2 x_j \pa_iF_j F_i-2 x_j F_j\pa_iF_i \\
    &=  |F|^2 -2 ( F\cdot x) \Div F -2\sum_{i,j =1}^3
    (\pa_jF_i-\pa_i F_j )F_j x_i.
\end{split}
\]
Thus we integrate \eqref{eq:dividentity} 
to find
\[
\int_{\Sigma} \big((x \cdot \nu)|F|^2 -2 F_n (F \cdot x)\big)\, d\H^2
=
\int_{\Omega} |F|^2 - 2\curl F (F \cdot  x)
-2 \Div F (F \cdot x) \,dx .
\]
Note that   $|F|^2 = |F_\tau|^2 + F_n^2$ and $(F \cdot x) = ( x \cdot \nu) F_n + (F_\tau \cdot x)$. Therefore we have the equality
\[
\begin{split}
\int_{\Sigma} \big(-( x \cdot \nu) F_n^2 &+ ( x \cdot \nu)  |F_\tau|^2  -2 F_n (F_\tau \cdot x) \big)\, d\H^2\\
&=\int_{\Omega} |F|^2 + 2\curl F (F \cdot  x)
-2 \Div F (F \cdot x) \,dx. 
\end{split}
\]
We use the fact that  $x \cdot \nu \geq c_0$ on $\Sigma$ and obtain the first claim by re-organizing the terms in  above and estimating $| F_n (F_\tau \cdot x)| \leq \e |F_\tau|^2 + C_e F_n^2 $
\[
\begin{split}
c_0\int_{\Sigma} |F_\tau|^2\, d\H^2
\leq &\int_{\Sigma} (\e |F_\tau|^2 + C_\e|F_n|^2) \, d\H^2 \\
&+ C(\|F\|_{L^2(\Omega)}^2 + \|\curl F\|_{L^2(\Omega)}^2 + \|\Div F\|_{L^2(\Omega)}^2).
\end{split}
\]
This yields the first inequality. The second follows from  similar argument. 

To prove the general case, i.e. when $\Omega$ is not starshaped, we use a localization argument which is similar to  \cite{AFJM}. 
\end{proof}

\begin{remark}
\label{rem:annoying}
We observe that the proof gives us a slightly stronger estimate. Indeed, we may improve the second inequality in Lemma \ref{lem:divcurl} as 
\[
\|F\|_{L^2(\Sigma)}^2 \leq C \left(\|F_\tau \|_{L^2(\Sigma)}^2  + \|F\|_{L^2(\Omega)}^2 +\| F\Div F\|_{L^1(\Omega)} +\|F\curl F\|_{L^1(\Omega)})\right).
\]
\end{remark}

In order to estimate $\|\nabla F\|_{H^1(\Omega)}$ we first consider the case when $F$ is curl-free, i.e., $F = \nabla u$. Since we define the norm $\|\pa_\nu u\|_{H^{\frac12}(\Sigma)}$ via the harmonic extension, we prove the next lemma using standard results from harmonic analysis instead of localizing and flattening the boundary.

\begin{lemma}
\label{lem:poisson1}
Assume that $\Omega$, with $\Sigma= \pa \Omega$, is uniformly  $C^{1,\alpha}(\Gamma)$-regular and  $\|B_\Sigma\|_{L^4} \leq M$ and $u :\Omega \to \R$ is a smooth function. There exists a constant $C$, depending on $M$ and the $C^{1,\alpha}$-norm of the heightfunction,  such that it holds
\[
 \| u\|_{H^2(\Omega)} \leq  C\left(\|\pa_\nu u\|_{H^{\frac12}(\Sigma)} + \|u\|_{L^2(\Omega)}  + \|\Delta u\|_{L^2(\Omega)} \right).
\]
The reverse also holds $\|\pa_\nu u\|_{H^{\frac12}(\Sigma)} \leq C \| u\|_{H^2(\Omega)}$.
\end{lemma}

\begin{proof}
We have by \eqref{eq:divcurl2} and \eqref{eq:tang-laplace} that 
\[
 \|\nabla^2 u\|_{L^2(\Omega)}^2 \leq \|\Delta u\|_{L^2(\Omega)}^2 + 2 \int_{\Sigma}  ((\nabla^2 u\, \nu )\cdot  \nu  - \Delta u) \, \pa_\nu u \, d \H^2 + C \int_{\Sigma} |B_\Sigma| \, |\nabla u|^2\, d \H^2 .
\]
To estimate the last terms, we use the interpolation inequality (Proposition \ref{prop:interpolation}), Lemma \ref{lem:divcurl} and the assumption $\|B\|_{L^4} \leq C$ and have for some $\theta \in (0,1)$
\begin{equation}
    \label{eq:poisson1.2}
\begin{split}
\int_{\Sigma} |B_\Sigma| \, |\nabla u|^2\, d \H^2 &\leq C \|B\|_{L^4(\Sigma)} \|\nabla u \|_{L^{\frac83}(\Sigma)}^2 \leq   \|\nabla u \|_{H^{\frac12}(\Sigma)}^{2\theta} \|\nabla u \|_{L^{2}(\Sigma)}^{2(1 -\theta)}\\
&\leq \e \|u \|_{H^{2}(\Omega)}^{2} + C_\e\|\nabla u \|_{L^{2}(\Sigma)}^{2}\\
&\leq \e \|u \|_{H^{2}(\Omega)}^{2} + C_\e(\|\pa_\nu u\|_{L^2(\Sigma)}^2 +  \|\nabla u \|_{L^{2}(\Omega)}^{2} + \|\Delta u\|_{L^2(\Omega)}^2 )\\
&\leq \e \|u \|_{H^{2}(\Omega)}^{2} + C_\e(\|\pa_\nu u\|_{L^2(\Sigma)}^2 +  \|u \|_{L^{2}(\Omega)}^{2} + \|\Delta u\|_{L^2(\Omega)}^2 ).  
\end{split}
\end{equation}

Let us then  show that 
\begin{equation}
    \label{eq:poisson1-1}
    \int_{\Sigma}  ((\nabla^2 u\, \nu )\cdot  \nu  - \Delta u) \, \pa_\nu u \, d \H^2 \leq  \e \|u \|_{H^{2}(\Omega)}^{2} + C_\e( \|\pa_\nu u\|_{H^{\frac12}(\Sigma)}^2 +  \|u \|_{L^{2}(\Omega)}^{2} + \|\Delta u\|_{L^2(\Omega)}^2 ). 
\end{equation}
To this aim we denote the harmonic extension of $\nu$ by $\tilde \nu$ and denote $f = \Delta u$. Then we have 
\[
\begin{split}
\int_{\Sigma} (\nabla^2 u\, \nu )\cdot  \nu \, \pa_\nu u  \, d \H^2 &=  \int_{\Sigma} \big(\pa_\nu (\nabla u \cdot \tilde \nu)  -(\pa_\nu \tilde \nu \cdot \nabla u) \big)\,(\nabla u \cdot \tilde \nu)\,  d \H^2  \\
&\leq \int_{\Sigma} \pa_\nu (\nabla u \cdot \tilde \nu)\,(\nabla u \cdot \tilde \nu)\,  d \H^2 + C\|\pa_\nu \tilde \nu\|_{L^4(\Sigma)} \| \nabla u \|_{L^{\frac83}(\Sigma)}^2 .
\end{split}
\]
We argue as in \eqref{eq:poisson1.2}  and obtain 
\[
 \| \nabla u \|_{L^{\frac83}(\Sigma)}^2 \leq  \e \|u \|_{H^{2}(\Omega)}^{2} + C_\e(\|\pa_\nu u\|_{L^2(\Sigma)}^2 +  \|u \|_{L^{2}(\Omega)}^{2} + \|f\|_{L^2(\Omega)}^2 ).
\]
Next we use the result from \cite{FJR}  for harmonic functions $\varphi:\Omega \to \R$ in $C^{1,\alpha}$-domains which states that 
\[
\|\pa_\nu \varphi\|_{L^p(\Sigma)} \leq C_p \|\nabla_\tau \varphi\|_{L^p(\Sigma)} \qquad \text{for }\, p \in (1,\infty).
\]
We use  this for $\tilde \nu$ component-wise, use the fact that on $\Sigma$ it holds $\tilde \nu = \nu$ and obtain
\begin{equation}
\label{eq:JK-for-nu}    
\|\nabla  \tilde \nu\|_{L^4(\Sigma)} \leq C\|\nabla_\tau \tilde \nu\|_{L^4(\Sigma)} \leq \|B \|_{L^4(\Sigma)}\leq C. 
\end{equation}
Therefore we have 
\begin{equation}
    \label{eq:poisson1-3}
    \begin{split}
\int_{\Sigma} (\nabla^2 u\, \nu \cdot ) \nu \, \pa_\nu u \, d \H^2 \leq &\int_{\Sigma} \pa_\nu (\nabla u \cdot \tilde \nu)\,(\nabla u \cdot \tilde \nu)\,  d \H^2 \\
&+ \e \|u \|_{H^{2}(\Omega)}^{2} + C_\e(\|\pa_\nu u\|_{L^2(\Sigma)}^2 +  \|u \|_{L^{2}(\Omega)}^{2} + \|f\|_{L^2(\Omega)}^2 ).
\end{split}
\end{equation}

Let us denote $\tilde u = (\nabla u \cdot \tilde \nu)$ for short and let $v$ be the harmonic extension of $\tilde u$ to $\Omega$. Let us show that $v$ is close to  $\tilde u$, i.e., we show that
\begin{equation}
    \label{eq:poisson1-6}
\|\nabla(\tilde u - v) \|_{L^2(\Omega)} \leq \e\|\nabla^2 u \|_{L^2(\Omega)} +  C_\e(\|f\|_{L^2(\Omega)} + \|\tilde u - v\|_{L^2(\Omega)}) .
  \end{equation}
To this aim  we calculate (recall that $f= \Delta u$)
\begin{equation}
    \label{eq:poisson1-4}
\Delta \tilde u = \nabla f \cdot \tilde \nu + 2 \nabla^2 u : \nabla \tilde \nu.
\end{equation}
This implies by integration by parts
\[
\begin{split}
\|\nabla &(\tilde u - v) \|_{L^2(\Omega)}^2 =   -\int_{\Omega} \Delta \tilde u (\tilde u - v)\, dx  = \int_{\Omega}( \nabla f \cdot \tilde \nu + 2 \nabla^2 u : \nabla \tilde \nu) (\tilde u  - v )\, dx \\
&= \int_{\Omega} f \Div \, ((\tilde u - v) \, \tilde \nu) - 2( \nabla^2 u : \nabla \tilde \nu) (\tilde u  - v )\, dx \\
&\leq C \|\nabla(\tilde u - v) \|_{L^2(\Omega)}\|f\|_{L^2(\Omega)} + C (\|\nabla^2 u \|_{L^2(\Omega)}+ \|f\|_{L^2(\Omega)}) \|\nabla \tilde \nu  \|_{L^{4}(\Omega)}\|(\tilde u - v) \|_{L^{4}(\Omega)}.
\end{split}
\] 
 By standard estimates from harmonic analysis \cite{Dahl}  and by \eqref{eq:JK-for-nu}  it holds
\begin{equation}
    \label{eq:poisson1-5}
 \|\nabla \tilde \nu \|_{L^{4}(\Omega)} \leq C  \|\nabla \tilde \nu \|_{L^{4}(\Sigma)} \leq C. 
\end{equation}
 On the other hand we have by H\"older's inequality and by Sobolev embedding (recall that $\tilde u - v = 0$ on $\Sigma$)
\[
\|(\tilde u - v) \|_{L^{4}(\Omega)} \leq  \|(\tilde u - v)  \|_{L^6(\Omega)}^{\frac12} \|\tilde u - v\|_{L^2(\Omega)}^{\frac12} \leq C \|\nabla (\tilde u - v)\|_{L^2(\Omega)}^{\frac12} \|\tilde u - v\|_{L^2(\Omega)}^{\frac12}.
 \]
Therefore  by combining the previous inequalities we obtain \eqref{eq:poisson1-6} by Young's inequality. 

We proceed by using  \eqref{eq:poisson1-4}  and by integrating by parts  
\[
\begin{split}
\int_{\Sigma} &\pa_\nu \tilde u\,\tilde u \,  d \H^2 =\int_{\Sigma} \pa_\nu \tilde u\,v \,  d \H^2 = \int_{\Omega} (\nabla \tilde u \cdot \nabla v + \Delta \tilde u \, v) \, dx  \\
&\leq 2\|\nabla v\|_{L^2(\Omega)}^2 + 2\|\nabla (\tilde u- v)\|_{L^2(\Omega)}^2 + \int_\Omega (\nabla f \cdot \tilde \nu + 2 \nabla^2 u : \nabla \tilde \nu) \, v\, dx \\
&= 2\|\nabla v\|_{L^2(\Omega)}^2 + 2\|\nabla (\tilde u - v)\|_{L^2(\Omega)}^2 + \int_{\Sigma} f \, v \, (\tilde \nu \cdot \nu) \, d \H^2\\
&\,\,\,\,\,\,\,\,\,\,\,\,\,\,+ \int_{\Omega} -f \Div \, (v \, \tilde \nu) + 2(\nabla^2 u : \nabla \tilde \nu) v \, dx.
\end{split}
\]
Recall that $\tilde \nu = \nu$ and $v = \tilde u$ on $\Sigma$. Therefore we obtain by the above inequality, by   $\|v\|_{L^4(\Omega)} \leq C \|v\|_{H^1(\Omega)}$,  \eqref{eq:poisson1-6}  and \eqref{eq:poisson1-5}  that 
\[
\begin{split}
\int_{\Sigma} \pa_\nu \tilde u\,\tilde u  - f \, \tilde u \,  d \H^2 &\leq  \e\|\nabla^2 u \|_{L^2(\Omega)} \\
&\,\,\,\,\,\,\,\,\,\,\,+ C_\e(\|\nabla \nu\|_{L^4(\Omega)}\|v\|_{L^4(\Omega)}+  \|v\|_{H^1(\Omega)}^2 +\|f\|_{L^2(\Omega)}^2 + \|\tilde u - v\|_{L^2(\Omega)}^2) \\
&\leq \e\|\nabla^2 u \|_{L^2(\Omega)} + C_\e(\|v\|_{H^1(\Omega)}^2 +\|f\|_{L^2(\Omega)}^2 + \|\tilde u - v\|_{L^2(\Omega)}^2).
\end{split}
\]
The inequality \eqref{eq:poisson1-1} then  follows from the above and \eqref{eq:poisson1-3} together with  
\[
\|v\|_{H^1(\Omega)} = \| \pa_\nu u\|_{H^{\frac12}(\Sigma)}, \qquad  \|\tilde u - v\|_{L^2(\Omega)}^2 \leq \e \|\nabla^2 u \|_{L^2(\Omega)} + C_\e \|u\|_{L^2(\Omega)}^2 + \|v\|_{L^2(\Omega)}^2,
\]
and by  recalling that $\tilde u = (\nabla u \cdot \tilde \nu) = \pa_\nu u$ on $\Sigma$ and $f = \Delta u$. This yields the first claim. The second inequality follows from reversing the previous calculations. 
\end{proof}

We state our lower order version of  Theorem \ref{thm:chen-shkoller}. 

\begin{theorem}
\label{prop:CS-k=1}
Assume that $\Omega$, with $\Sigma= \pa \Omega$, is uniformly  $C^{1,\alpha}(\Gamma)$-regular and  $\|B_\Sigma\|_{L^4(\Sigma)} \leq M$. There exists a constant $C$,  depending on $M$ and  the $C^{1,\alpha}$-norm of the heightfunction, such that  for all vector fields $F \in H^1(\Omega;\R^3)$ it holds
\[
\|F\|_{H^1(\Omega)} \leq  M( \|F_n\|_{H^\frac12(\Sigma)} + \|F\|_{L^2(\Omega)} +\|\Div F\|_{L^2(\Omega)}+\|\curl F\|_{L^2(\Omega)} ).
\]
\end{theorem}

\begin{proof}
By approximation argument we may assume that $F$ and $\Omega$ are smooth. We use the Helmholtz-Hodge decomposition and write $F=\nabla \phi+\psi$ where $\phi$ is the unique solution of the Neumann problem
\[
\begin{cases}
\Delta \phi=\Div F\qquad &x\in \Omega
\\
\pa_\nu \phi= F_n &x\in \Sigma
\end{cases}
\]
with zero average and $\psi$ solves
\[
\begin{cases}
\curl \psi= \curl F \qquad &x\in \Omega
\\
\Div \psi =0 &x\in \Omega
\\
\psi \cdot \nu  = 0  &x\in \Sigma .
\end{cases}
\]
We also note that $\nabla \phi $ and $\psi$ are orthogonal in $L^2(\Omega)$ and thus
\[
\int_{\Omega} |\nabla\phi|^2 +|\psi|^2dx =\int_{\Omega}|F|^2 dx. 
\]
For $\phi$ we have by Lemma \ref{lem:poisson1} that 
\begin{equation}
    \label{eq:CS-k=1.1}
\begin{split}
\|\phi\|_{H^2(\Omega)} &\leq C( \|\pa_\nu \phi\|_{H^\frac12(\Sigma)} + \|\nabla \phi\|_{L^2(\Omega)} +\|\Delta \phi\|_{L^2(\Omega)})\\
&\leq C( \|F_n\|_{H^\frac12(\Sigma)} + \|F\|_{L^2(\Omega)} +\|\Div F\|_{L^2(\Omega)}).
\end{split}
\end{equation}
For $\psi$ we have by \eqref{eq:divcurl1} 
\[
\|\nabla \psi\|_{L^2(\Omega)}^2\, dx= \frac12\|\text{curl}\, \psi\|_{L^2(\Omega)}^2 - \int_{\Sigma} \langle B_\Sigma \, \psi, \psi \rangle \, d\H^2.
\]
We  use the assumption $\|B_\Sigma\|_{L^4} \leq C$, H\"older's inequality and interpolation inequality to deduce   
\[
\begin{split}
-\int_{\Sigma}  \langle B_\Sigma \, \psi, \psi \rangle  \, d\H^2 &\leq \|B_{\Sigma}\|_{L^4(\Sigma)}\|\psi\|_{L^{\frac83}(\Sigma)}^2 \\
&\leq \e\|\psi\|_{H^{\frac12}(\Sigma)}^2 + C_\e \|\psi\|_{L^2(\Omega)}^2 \leq \e \|\nabla \psi\|_{L^{2}(\Omega)}^2 + C_\e \|F\|_{L^2(\Omega)}^2.
\end{split}
\]
Thus we have
\[
\|\nabla \psi\|_{L^2(\Omega)} \leq C(\|\curl F\|_{L^2(\Omega)} + \|F\|_{L^2(\Omega)}) 
\]
This together with \eqref{eq:CS-k=1.1} yields the claim. 
\end{proof}

We proceed by using Theorem \ref{prop:CS-k=1} to control the higher order norms $\|F\|_{H^2(\Omega)}$ and $\|F\|_{H^3(\Omega)}$. We do not need the sharp dependence on the curvature for these estimates and we state the result in a form that is suitable for us. We also treat the case $\|F\|_{H^{\frac32}(\Omega)}$, but only for curl-free vector fields. In this case we need the 'sharp' curvature dependence but this time we have non-optimal dependence on the divergence. 
\begin{lemma}
\label{lem:CS-intermed}
Assume that $\Omega$, with $\Sigma= \pa \Omega$, is uniformly $C^{1,\alpha}(\Gamma)$-regular and  $\|B_\Sigma\|_{H^{\frac12}(\Sigma)} \leq M$. Then for all vector fields $F \in H^3(\Omega;\R^3)$ it holds 
\[
\|F\|_{H^3(\Omega)} \leq  C \big(\|\Delta_\Sigma F_n\|_{H^\frac12(\Sigma)}  + (1+\|H_\Sigma\|_{H^{2}(\Sigma)})\|F\|_{L^{\infty}(\Omega)}  +\|\Div F\|_{H^2(\Omega)}+\|\curl F\|_{H^2(\Omega)} \big)
\]
and
\[
\|F\|_{H^2(\Omega)} \leq  C\big(\|\Delta_\Sigma F_n\|_{H^\frac12(\Sigma)}  + \|F\|_{L^{\infty}(\Omega)}  +\|\Div F\|_{H^1(\Omega)}+\|\curl F\|_{H^1(\Omega)} \big)
\]
for some constant $C$,  depending on $M$ and the $C^{1,\alpha}$-norm of the heightfunction. 
Moreover, if $F = \nabla u$ then it holds 
\[
\| \nabla u\|_{H^{\frac32}(\Omega)} \leq  C(\|\pa_\nu u \|_{H^1(\Sigma)}  + \|u\|_{L^2(\Omega)}  +\|\Delta u\|_{H^1(\Omega)})
\]
and for $k = \frac12, 1$ it holds
\[
\|u\|_{H^{k}(\Sigma)} \leq  C \|u\|_{H^{k+\frac12}(\Omega)}.
\]
\end{lemma}
\begin{proof}
By approximation argument we may assume that $F$ and $\Omega$ are smooth. Note also that by Sobolev embedding $\| B_\Sigma\|_{L^4(\Sigma)} \leq \|B_\Sigma\|_{H^{\frac12}(\Sigma)} \leq M$.

Let $\tilde \nu$ be the harmonic extension of the normal field $\nu$ to $\Omega$. Let us define the vector fields $\tau_i = e_i - (\tilde \nu \cdot e_i) \tilde  \nu$ for $i =1,2,3$, where $\{e_i\}_i$ is a coordinate basis of $\R^3$. For $i,j$ we define a vector field $F_{ij} :\Omega \to \R^3$  as  $(F_{ij})_k =  \nabla^2 F_k \tau_i \cdot  \tau_j  $. We apply Theorem \ref{prop:CS-k=1}  for $F_{ij}$  and obtain 
\[
\|\nabla F_{ij}\|_{L^2(\Omega)} \leq  C( \| F_{ij} \cdot \nu \|_{H^\frac12(\Sigma)} + \| F_{ij}\|_{L^2(\Omega)} +\|\Div  F_{ij}\|_{L^2(\Omega)}+\|\curl F_{ij}\|_{L^2(\Omega)} ).
\]
Recall that  \eqref{eq:poisson1-5}  implies  $\|\nabla \tilde \nu\|_{L^4(\Omega)} \leq C$. Moreover by maximum principle it holds $\|\tilde \nu \|_{L^\infty(\Omega)} \leq C$. Therefore 
\[
\begin{split}
\|\Div F_{ij}\|_{L^2(\Omega)} &\leq C\|\Div F\|_{H^2(\Omega)} + C\|\nabla^2 F \|_{L^4(\Omega)} \|\nabla \nu \|_{L^4(\Omega)} \\
&\leq C\|\Div F\|_{H^2(\Omega)} + C \|\nabla^2 F \|_{L^4(\Omega)}.
\end{split}
\]
By interpolation we have $ \|\nabla^2 F \|_{L^4(\Omega)} \leq \e  \| F\|_{H^3(\Omega)} + C_{\e} \|F\|_{L^2(\Omega)}$ and thus 
\[
\| F_{ij}\|_{L^2(\Omega)} + \|\Div F_{ij}\|_{L^2(\Omega)} \leq  \e  \| F \|_{H^3(\Omega)} +   C_\e(\|\Div F\|_{H^2(\Omega)}+ \|F\|_{L^2(\Omega)}) .
\]
By a similar argument
\[
\|\curl F_{ij}\|_{L^2(\Omega)} \leq  \e  \| F \|_{H^3(\Omega)} + C_\e(\|\curl F\|_{H^2(\Omega)}+ \|F\|_{L^2(\Omega)})
\]
and 
\[
\| \nabla  F_{ij}\|_{L^2(\Omega)}^2 \geq    \sum_{k,l=1}^3 \|  \nabla^2 \nabla_l F_k \, \tau_i \cdot \tau_j \|_{L^2(\Omega)}^2  - \e \| F \|_{H^3(\Omega)}^2 - C_\e \|F\|_{L^2(\Omega)}^2.
\]

Let us fix a point $x \in \Omega$ and estimate  the norm
\[
\sum_{i,j,k,l=1}^3|\nabla^2 \nabla_l F_k(x) \, \tau_i \cdot \tau_j|^2.
\]
First we observe that the above quantity does not depend on the choice of the coordinates in $\R^3$. Let us choose the coordinates  such that  $ \tilde \nu(x) \cdot e_i = 0$ for $i=1,2$. Then we have 
\[
\sum_{i,j,k,l=1}^3 |\nabla^2 \nabla_l F_k(x) \, \tau_i \cdot  \tau_j  |^2 \geq  \sum_{k,l=1}^3 \sum_{i,j=1}^2 | \nabla_i\nabla_j\nabla_l F_k(x)|^2.
\]
By a simple combinatorial argument we deduce 
\[
\begin{split}
&\sum_{i,j,k,l=1}^3 | \nabla_i\nabla_j\nabla_l F_k(x)|^2\\
&\leq C\sum_{k,l=1}^3 \sum_{i,j=1}^2 | \nabla_i\nabla_j\nabla_l F_k(x)|^2  + C |\nabla^2 \Div F(x)|^2 + C |\nabla^2 \curl F(x)|^2. 
\end{split}
\]
By applying the above argument for every $x$ we  have 
\[
\sum_{i,j,k,l=1}^3 \| \nabla  F_{ij}\|_{L^2(\Omega)}^2 \geq c_0 \|\nabla^3 F\|_{L^2(\Omega)}^2 - C(\|\Div F\|_{H^2(\Omega)}^2+ \|\curl F\|_{H^2(\Omega)}^2+  \|F\|_{L^2(\Omega)}^2).
\]
Combing all the previous estimates we obtain 
\[
\|\nabla^3 F\|_{L^2(\Omega)} \leq \sum_{i,j=1}^3 C(\| F_{ij} \cdot \nu \|_{H^\frac12(\Sigma)} + \|F\|_{L^2(\Omega)}  +\|\Div F\|_{H^2(\Omega)}+ \|\curl F\|_{H^2(\Omega)}).
\]
The first inequality  follows once we show
\begin{equation}
    \label{eq:CS-intermed.1}
\sum_{i,j=1}^3 \| F_{ij}\cdot  \nu \|_{H^\frac12(\Sigma)} \leq C\| \Delta_\Sigma F_n  \|_{H^\frac12(\Sigma)}+ \e \|F\|_{H^3(\Omega)} +C_\e(1 + \|H_\Sigma\|_{H^{2}})\|F\|_{L^{\infty}(\Sigma)}.
\end{equation}

To this aim we first note that on $\Sigma$ it holds $\tau_i = e_i - (\nu \cdot e_i)\nu $ and therefore $\tau_i$ is tangential on $\Sigma$. Thus  we have
\begin{equation}
    \label{eq:CS-intermed.2}
   (\nabla^2 F_n) \tau_i \cdot  \tau_j =  F_{ij} \cdot \nu +   \nabla_\tau F \star B_\Sigma + F \star \nabla_\tau B_\Sigma.
    \end{equation}
We use Proposition  \ref{prop:kato-ponce}   and  get 
\[
\|\nabla_\tau F \star B_\Sigma\|_{H^{\frac12}(\Sigma)} \leq C\| \nabla F\|_{H^{\frac12}(\Sigma)} \|B_\Sigma\|_{L^{\infty}(\Sigma)} + C\| \nabla_\tau F\|_{L^{\infty}(\Sigma)} \| B_\Sigma\|_{H^{\frac12}(\Sigma)}.
\]
 Recall that $\| B_\Sigma \|_{H^{\frac12}(\Sigma)} \leq C$. By interpolation we have 
\[
\| \nabla_\tau F\|_{L^{\infty}(\Sigma)} \leq \e \| F\|_{H^3(\Omega)} + C_\e\|  F\|_{L^{\infty}}.
\]
Moreover, by Sobolev embedding, Proposition \ref{prop:interpolation} and by Proposition \ref{prop:meancrv-bound} we have for $\theta < \frac12$
\[
\|B_\Sigma\|_{L^{\infty}(\Sigma)} \leq C\| B_\Sigma\|_{W^{1, \frac73}(\Sigma)} \leq C\| B_\Sigma\|_{H^{2}(\Sigma)}^{\theta}\| B_\Sigma\|_{L^{4}(\Sigma)}^{1 - \theta}\leq C(1+ \| H_\Sigma\|_{H^{2}(\Sigma)}^{\theta}).
\]
On the other hand, Corollary \ref{coro:interpolation} implies  
\[
\| \nabla F\|_{H^{\frac12}(\Sigma)}\leq C \| \nabla F\|_{H^{1}(\Omega)}\leq C\| F\|_{H^3(\Omega)}^{\frac12} \|\nabla F\|_{L^2(\Omega)}^{\frac12}. 
\]
Therefore we have  by the above estimates and by Young's inequality
\[
\begin{split}
\|\nabla_\tau F \star B_\Sigma\|_{H^{\frac12}(\Sigma)} &\leq \e \| F\|_{H^3(\Omega)} + C_{\e}\|F\|_{L^\infty} + C_{\e}(1+ \| H_\Sigma\|_{H^{2}(\Sigma)}^{\theta})\|\nabla F\|_{L^2(\Omega)}\\
&\leq \e \| F\|_{H^3(\Omega)} + C_\e (1+ \| H_\Sigma\|_{H^{2}(\Sigma)}) \|F\|_{L^\infty(\Omega)},
\end{split}
\]
where the last inequality follows from interpolation. 

Let us then bound the last term in \eqref{eq:CS-intermed.2}. We have by Proposition  \ref{prop:kato-ponce} 
\[
\begin{split}
\| F \star  \nabla_\tau B_\Sigma\|_{H^{\frac12}(\Sigma)} &\leq \| F \star  \nabla_\tau B_\Sigma\|_{H^{1}(\Sigma)} \\
&\leq C \|F\|_{L^\infty(\Sigma)}\|  \bar \nabla B_\Sigma\|_{H^{1}(\Sigma)} + C  \| F\|_{W^{1,3}(\Sigma)} \| \bar \nabla B_\Sigma\|_{L^{6}(\Sigma)}.
\end{split}
\]
Proposition \ref{prop:meancrv-bound} yields $\|  B_\Sigma\|_{H^{2}(\Sigma)} \leq C(1+  \| H_\Sigma\|_{H^{2}(\Sigma)})$. Interpolation implies 
\[
 \| \bar \nabla B_\Sigma\|_{L^{6}(\Sigma)}\leq C \|  B_\Sigma\|_{H^{2}(\Sigma)}^{\theta} \|  B_\Sigma\|_{L^{2}(\Sigma)}^{1-\theta}
\]
and 
\[
\| F\|_{W^{1,3}(\Sigma)} \leq \| F\|_{H^{3}(\Omega)}^{1-\theta} \| F\|_{L^{\infty}(\Omega)}^{\theta}
\]
for some $\theta \in (0,1)$.
Therefore we have by \eqref{eq:CS-intermed.2}
\[
\|  F_{ij} \cdot  \nu  \|_{H^\frac12(\Sigma)} \leq \| \bar \nabla^2 F_n \|_{H^\frac12(\Sigma)} + \e \|F\|_{H^3(E)} +C_\e(1 + \|H_\Sigma\|_{H^{2}(\Sigma)})\|F\|_{L^{\infty}(\Omega)}.
\]
The inequality \eqref{eq:CS-intermed.1} then follows from  Proposition \ref{prop:laplace-bound} as
\[
\| \bar \nabla^2 F_n \|_{H^\frac12(\Sigma)} \leq 2\|  \Delta_\Sigma F_n \|_{H^\frac12(\Sigma)} + C_\e \|F_n\|_{L^2(\Sigma)}.
\]

The second inequality follows from a similar argument. 


Let us next prove the last part of the statement, i.e. the inequalities when $F=\nabla u$. Let $u$ be a solution of the Neumann boundary problem
\[
\begin{cases}
\Delta u =f \qquad & \text{in } \Omega
\\
\pa_\nu u = g  &\text{on }  \Sigma,
\end{cases}
\]
where $\int_{\Sigma}g = \int_{\Omega} f$ and $\int_\Omega u \, dx = 0$. First, clearly $\|u\|_{H^{\frac12}(\Sigma)} \leq C \| u\|_{H^1(\Omega)}$. By the equation and by divergence theorem  
\[
\int_\Omega |\nabla u|^2 \, dx + \int_\Omega u \, f  \, dx = \int_{\Sigma} u \, \pa_\nu u \, d \H^2 \leq \|u\|_{H^{\frac12}(\Sigma)} \|g\|_{H^{-\frac12}(\Sigma)} \leq C\|u\|_{H^1(\Omega)} \|g\|_{H^{-\frac12}(\Sigma)}. 
\]
Therefore by  
\[
\left|\int_\Omega u \, f  \, dx \right| \leq \|u\|_{L^{2}(\Omega)} \|f\|_{L^{2}(\Omega)}
\]
and by Poincar\'e inequality $\|u\|_{L^{2}(\Omega)} \leq C \|\nabla u\|_{L^{2}(\Omega)}$ we have 
\[
\|u\|_{H^1(\Omega)} \leq C(\|g\|_{H^{-\frac12}(\Sigma)} + \|f\|_{L^{2}(\Omega)}).
\]
On the other hand Lemma \ref{lem:poisson1} implies 
\[
\begin{split}
 \|u\|_{H^{2}(\Omega)} &\leq C(\|g\|_{H^{\frac12}(\Sigma)} + \|u\|_{L^2} + \|f\|_{L^{2}(\Omega)})\\
 &\leq C(\|g\|_{H^{\frac12}(\Sigma)}  + \|f\|_{L^{2}(\Omega)}).
\end{split}
 \]
We use the two above inequalities and standard interpolation argument to deduce
\begin{equation}
    \label{eq:CS-k=1.3}
\|u\|_{H^{3/2}(\Omega)} \leq C(\|g\|_{L^2(\Sigma)} + \|f\|_{L^2(\Omega)}).
\end{equation}

We proceed by applying \eqref{eq:CS-k=1.3} for $u_{x_i} = \nabla u \cdot e_i - c_i$, for $i = 1,2,3$, $c_i = \fint_{E} u_{x_i}$, and obtain
\begin{equation}
    \label{eq:CS-k=1.4}
\|\nabla u\|_{H^{3/2}(\Omega)} \leq C(\|\pa_\nu (\nabla u)\|_{L^2(\Sigma)}+  \|f\|_{H^1(\Omega)}).
\end{equation}
In order to treat the first term on the RHS we let $\tilde \nu$ be the harmonic extension of  $\nu$ to $\Omega$. We write $\nabla u = \nabla_\tau u + (\pa_\nu u)\, \nu $ and have 
\[
\pa_\nu (\nabla u) = \nabla_\tau (\pa_\nu u) + \pa_\nu (\nabla u \cdot \tilde \nu) \nu + \nabla \tilde \nu \star \nabla u. 
\]
Recall that we have  by maximum principle $\|\tilde \nu\|_{L^\infty} \leq C$ and by \eqref{eq:poisson1-5}  $\|\nabla \tilde \nu\|_{L^4(\Omega)} \leq C$. We argue as in \eqref{eq:poisson1.2}  and obtain  
\[
\|\nabla \tilde \nu \star \nabla u\|_{L^2(\Sigma)} \leq \|\nabla \tilde \nu \|_{L^4(\Sigma)}\,  \|\nabla u\|_{L^{4}(\Sigma)} \leq C\|u\|_{H^2(\Omega)}.
\]
We use Remark \ref{rem:annoying} for $F = \nabla (\nabla u \cdot \tilde \nu)$ and \eqref{eq:poisson1-4} and  have
\[
\begin{split}
\|\pa_\nu (\nabla u \cdot \tilde \nu)\|_{L^2(\Sigma)}^2 &\leq C(\|\nabla_\tau (\nabla u \cdot \tilde \nu)\|_{L^2(\Sigma)}^2 + \|\nabla(\nabla u \cdot \tilde \nu)  \, \Delta (\nabla u \cdot \tilde \nu)\|_{L^1(\Omega)} + \| \nabla u \cdot \tilde \nu\|_{H^1(\Omega)}^2)\\
&\leq C(\|\pa_\nu u \|_{H^1(\Sigma)}^2 + \|u\|_{H^2(\Omega)}^2 + \|f\|_{H^1(\Omega)}^2+ \| \nabla u \cdot \tilde \nu\|_{H^1(\Omega)}^2)\\
&\,\,\,\,\,\,\,+ C(\| \nabla^2 u \star \nabla^2 u \star \nabla \tilde \nu\|_{L^1(\Omega)} + \| \nabla^2 u  \star \nabla u \star  \nabla \tilde \nu \star \nabla \tilde \nu\|_{L^1(\Omega)}).
\end{split}
\]
First, we obtain  by using the previous estimates
\[
\| \nabla u \cdot \tilde \nu\|_{H^1(\Omega)}^2 \leq C \|u
\|_{H^2(\Omega)}^2.  
\]
We bound the second last term by H\"older's inequality and by the Sobolev embedding 
\[
\| \nabla^2 u \star \nabla^2 u \star \nabla \tilde \nu\|_{L^1(\Omega)} \leq C\|\nabla \tilde \nu\|_{L^4(\Omega)}\|\nabla^2 u\|_{L^{\frac83}(\Omega)}^2 \leq \e \|\nabla u\|_{H^{\frac32}(\Omega)}^2 + C_{\e}\|u\|_{H^2(\Omega)}^2.
\]
Similarly we estimate the last term  
\[
\begin{split}
\| \nabla^2 u  \star \nabla u \star  \nabla \tilde \nu \star \nabla \tilde \nu\|_{L^1(\Omega)} &\leq C\|\nabla \tilde \nu\|_{L^4(\Omega)}^2 \|\nabla^2 u\|_{L^{\frac83}(\Omega)} \|\nabla u\|_{L^8(\Omega)}\\
&\leq \e \|\nabla u\|_{H^{\frac32}(\Omega)}^2 + C_{\e}\|u\|_{H^2(\Omega)}^2.
\end{split}
\]
 Therefore we have 
\[
\|\pa_\nu (\nabla u)\|_{L^2(\Sigma)} \leq \e \| \nabla u \|_{H^{\frac32}(\Omega)} + C_\e(\|\pa_\nu u \|_{H^1(\Sigma)}+ \|u\|_{H^2(\Omega)} +\|f\|_{H^1(\Omega)}).
\]
Recall that we have 
\[
\|u\|_{H^2(\Omega)} \leq C(\|g\|_{H^{\frac12}(\Sigma)} + \|f\|_{L^2(\Omega)}).
\]
Therefore the third inequality follows from \eqref{eq:CS-k=1.4}.

For the last inequality we recall that while the Trace operator is not bounded $T: H^{\frac12}(\R^3) \to L^{2}(\R^2)$ it is bounded as $T: H^{\frac32}(\R^3) \to H^{1}(\R^2)$. We prove the statement by localization argument similar to the one in the proof of Proposition \ref{prop:kato-ponce} and we only give the sketch of the proof. 

We cover $\Sigma$ with balls of radius $\delta$, $B_{\delta}(x_i), i = 1, \dots, N$  such that the set $\Sigma \cap B_{2\delta}(x_i)$ is contained in the graph of $\phi_i$ and $\Omega$ is above the graph. We denote the partition of unity by $\eta_i$. Let us fix $i$ and we may assume that $x_i = 0$ and $\phi_i(0) = \nabla \phi_i(0) = 0$. By the regularity assumptions it holds $\|\phi_i\|_{C^{1,\alpha}(\R^2)}, \|\phi_i\|_{W^{2,4}(\R^2)} \leq C$. We define $u_i(x) = \eta_i(x)u_i(x)$ and  
$v_i(x',x_3)=u_i(x', x_3+\phi_i(x'))$ for $x_3 \geq 0$ and  extend $v_i$ to  $\R^3$  by the extension operator.   Then we have by the Trace Theorem 
\[
\|u_i\|_{H^1(\Sigma \cap B_\delta)} \leq C\| v_i\|_{H^1(\R^2)} \leq C \|v_i\|_{H^{\frac32}(\R^3)}.
\]
Recall that the assumption $\|B_\Sigma\|_{L^4} \leq C$ guarantees that  $\Omega$ is an $H^2$-extension domain. Therefore it holds $\|v_i\|_{H^{\frac32}(\R^3)} \leq C\| u\|_{H^{\frac32}(\Omega)}$ and the last inequality follows.  
\end{proof}

\subsection{Regularity estimates for functions}

In this subsection we prove regularity estimates for functions $u : \Omega \to \R$ defined as a solution of  the Dirichlet problem
\begin{equation} \label{eq:dirichlet}
\begin{cases} 
\Delta u= f \qquad& x\in \Omega
\\
u=g \qquad &x\in \Sigma
\end{cases}
\end{equation}

We first  consider the case when $g = 0$ and improve in this case the third inequality in Lemma \ref{lem:CS-intermed}. Here we assume that the boundary has the regularity $\|B_\Sigma\|_{H^{\frac12}(\Sigma)} \leq C$. Note that by the Sobolev embedding this implies $\|B_\Sigma\|_{L^{4}(\Sigma)} \leq C$.
 
\begin{proposition}     \label{prop:dirichlet1}
Assume $\Omega$, with $\Sigma = \pa \Omega$,  is uniformly $C^{1,\alpha}(\Gamma)$-regular and  $\|B_\Sigma\|_{H^{\frac12}(\Sigma)} \leq M$.  There exists a constant $C$, depending on $M$ and the $C^{1,\alpha}$-norm of the heightfunction,  such that the solution of the problem \eqref{eq:dirichlet} with zero Dircihlet boundary datum, i.e.,  $g = 0$
satisfies
\[
\|\pa_\nu u\|_{H^1(\Sigma)} +\| \nabla u\|_{H^{\frac32}(\Omega)} \leq C\|f\|_{H^\frac12(\Omega)}.
\]
\end{proposition}

\begin{proof}
First we note that since $u = 0$ on $\Sigma$ then by \eqref{eq:divcurl2} we have 
\[
\|\nabla^2 u\|_{L^2(\Omega)}^2   =\|f\|_{L^2(\Omega)}^2 -  \int_{\Sigma} H_\Sigma |\pa_\nu u|^2\, d \H^2. 
\]
By \eqref{eq:poisson1.2} it holds 
\[
-\int_{\Sigma} H_\Sigma |\pa_\nu u|^2\, d \H^2 \leq \e \|u\|_{H^2(\Omega)}^2 + C_\e\|\nabla u\|_{L^2(\Sigma)}^2.
\]
We apply Lemma \ref{lem:divcurl} for $F = \nabla u$ and recall that $u=0$ on $\Sigma$ to deduce 
\[
\begin{split}
\|\nabla u\|_{L^2(\Sigma)} &\leq C(\|\nabla_\tau u\|_{L^2(\Sigma)} + \|u\|_{H^1(\Omega)} + \|f\|_{L^2(\Omega)})\\
&\leq \e \|u\|_{H^2(\Omega)}  + C_\e(\|u\|_{L^2(\Omega)} + \|f\|_{L^2(\Omega)}).
\end{split}
\]
Therefore we have 
\[
\|\nabla^2 u\|_{L^2(\Omega)}^2 \leq C(\|u\|_{L^2(\Omega)} + \|f\|_{L^2(\Omega)}).
\] 
We bound $\|u\|_{L^2(\Omega)}$ simply by multiplying the equation \eqref{eq:dirichlet} by $u$ and integrating by parts  $\|\nabla u\|_{L^2(\Omega)}^2 \leq \|f\|_{L^2(\Omega)}\|u\|_{L^2(\Omega)}$. Poincar\'e inequality then implies $\|u\|_{L^2(\Omega)} \leq C \|f\|_{L^2(\Omega)}$
and we have
\beq \label{eq:dirichlet-2}
\|u\|_{H^2(\Omega)} \leq C \|f\|_{L^2(\Omega)}.
\eeq

Let $\tilde \nu$ be the harmonic extension of the normal field and let us define $\tau_i = e_i - \langle e_i,\tilde \nu \rangle \tilde \nu$ as in the proof of Lemma \ref{lem:CS-intermed}. Define  $u_i = \nabla u \cdot \tau_i$. Observe that $u_i = 0$ on $\Sigma$ and apply \eqref{eq:dirichlet-2} to deduce
\begin{equation}
    \label{eq:dirichlet-3}
\|\nabla^2 u_i\|_{L^2(\Omega)} \leq C \|\Delta u_i \|_{L^2(\Omega)}.
\end{equation}
We have (recall $\Delta u = f$)
\[
\Delta u_i = \nabla f \star \tau_i +  \nabla^2 u \star \nabla \tilde \nu  + \nabla u \star \nabla \tilde \nu\star \nabla \tilde \nu.
\]
Arguing similarly as in the proof of  Lemma \ref{lem:poisson1} and using \eqref{eq:dirichlet-2} yields
\begin{equation}
   \label{eq:dirichlet-4}
\|\Delta u_i \|_{L^2(\Omega)} \leq \e  \|u\|_{H^3(\Omega)} +C\|f\|_{H^1(\Omega)} 
\end{equation}

Let us then treat the LHS of \eqref{eq:dirichlet-3}. We have  (recall that $\tau_i = e_i - \langle e_i, \tilde \nu \rangle \tilde \nu$) 
\[
\nabla_{j}\nabla_{k} u_i =  \nabla (\nabla_{j}\nabla_{k} u) \cdot   \tau_i  + \nabla^2 u \star \nabla  \tilde \nu + \nabla u \star \nabla  \tilde \nu  \star \nabla  \tilde \nu + \nabla u \star  \nabla^2  \tilde \nu.
\]
Therefore arguing as in the proof of  Lemma \ref{lem:poisson1}, we obtain 
\begin{equation}
    \label{eq:dirichlet-5}
    \begin{split}
\|\nabla^2 u_i\|_{H^2(\Omega)} \geq \sum_{i,j,k =1}^3 &\|\nabla (\nabla_{j}\nabla_{k} u) \cdot   \tau_i \|_{L^2(\Omega)} \\
&- \e \|u\|_{H^3(\Omega)} - C_\e \|f\|_{H^1(\Omega)} - C\|\nabla u\|_{L^\infty(\Omega)}\|\nabla^2  \tilde \nu\|_{L^2(\Omega)}.
    \end{split}
\end{equation}
Let us fix a point $x \in \Omega$ and as in the proof of Lemma \ref{lem:CS-intermed} we may assume that $\tilde \nu(x) \cdot e_i =0$ for $i =1,2$. Then it is easy to see that 
\[
\begin{split}
\sum_{i,j,k =1}^3 |\nabla (\nabla_{j}\nabla_{k} u(x)) \cdot   \tau_i |^2 &\geq \sum_{i=1}^2\sum_{j,k =1}^3 |\langle \nabla \nabla_j \nabla_k u(x), \tau_i \rangle |^2\\
&\geq c\sum_{i,j,k =1}^3 | \nabla_i \nabla_j \nabla_k u(x)|^2 - C |\nabla \Delta u(x)|^2.
\end{split}
\]
This together with \eqref{eq:dirichlet-5} yields 
\begin{equation}
    \label{eq:dirichlet-6}
    \|u_i\|_{H^2(\Omega)}  \geq c \|\nabla^3 u\|_{L^2(\Omega)} -   C \|f\|_{H^1(\Omega)}- C\|\nabla u\|_{L^\infty}\|\nabla^2  \tilde \nu\|_{L^2(\Omega)}.
\end{equation}    

We proceed by recalling that $\tilde \nu$ is the harmonic extension of $\nu$. We claim that it holds 
\beq \label{eq:dirichlet-66}
\|\nabla^2  \tilde \nu\|_{L^2(\Omega)} \leq  C. 
\eeq
Indeed, this follows from already familiar argument and we only give its outline. Define  $\tau_i = e_i - \langle e_i,\tilde \nu \rangle \tilde \nu$ as in the proof of Lemma \ref{lem:CS-intermed}
 and  let $u_{ij} = \la \nabla \tilde  \nu \, \tau_i \tau_j \ra$. Then it holds $u_{ij} = \la B_\Sigma \tau_i, \tau_j \ra$ on $\Sigma$ and therefore by the assumptions it holds $\|u_{ij}\|_{H^{\frac12}(\Sigma)} \leq C$. Arguing as in the proof of Lemma \ref{lem:poisson1} we deduce
\[
\| \nabla u_{ij}\|_{L^2(\Omega)}^2 \leq \|u_{ij}\|_{H^{\frac12}(\Sigma)}^2 + \e \|\nabla^2  \tilde \nu\|_{L^2(\Omega)}^2  +C_\e. 
\]
By applying this to every $i,j, =1, 2,3$ and arguing as above we obtain \eqref{eq:dirichlet-66}.

We have by interpolation inequality in Corollary \ref{coro:interpolation}
\[
\|\nabla u\|_{L^\infty(\Omega)} \leq C \|\nabla^2 u\|_{L^{4}(\Omega)} \leq C\|\nabla^3 u\|_{L^2(\Omega)}^{\frac34}\|\nabla^2 u\|_{L^2(\Omega)}^{\frac14}. 
\]
Therefore by Young's inequality and by \eqref{eq:dirichlet-2}
\[
\begin{split}
\|\nabla u\|_{L^\infty(\Omega)}\|\nabla^2  \tilde \nu\|_{L^2(\Omega)} &\leq \e \|\nabla^3 u\|_{L^2(\Omega)} +C_\e \|\nabla^2 u\|_{L^2(\Omega)}\\
&\leq \e \|\nabla^3 u\|_{L^2(\Omega)}+ C_\e \|f\|_{L^2(\Omega)}.
\end{split}
\]
Hence, \eqref{eq:dirichlet-3}, \eqref{eq:dirichlet-4} and  \eqref{eq:dirichlet-6} imply
\begin{equation}
    \label{eq:dirichlet-7}
    \|u\|_{H^3(\Omega)} \leq C \|f\|_{H^1(\Omega)}.
\end{equation} 

We set $\mathcal F$ to be the linear operator such that it associates  $f$ with the unique solution $u$ of the problem \eqref{eq:dirichlet}. Then we have by \eqref{eq:dirichlet-2} and \eqref{eq:dirichlet-7}  
\[
\|\mathcal{F}\|_{\mathcal L(L^2,H^2)} \leq C \qquad\text{and} \qquad
\|\mathcal{F}\|_{\mathcal L(H^1,H^3)} \leq C.
\]
Then we have the inequality 
\beq \label{eq:dirichlet-77}
\|\nabla u\|_{H^{\frac32}(\Omega)} \leq C \|f\|_{H^{\frac12}(\Omega)}
\eeq
by standard interpolation theory. 

We need yet to bound $\|\pa_\nu u\|_{H^1(\Sigma)}$. To this aim we extend $\nabla u$ to $\R^3$ by $T$ such that 
\[
\|T(\nabla u)\|_{H^{\frac32}(\R^3)} \leq C  \|\nabla u\|_{H^{\frac32}(\Omega)}. 
\]
Let us denote $U = T(\nabla u)$. Let  $\tilde \nu$ be the Harmonic extension of $\nu$ as before, which we may also extend to $\R^3$. We note that we may assume that the extensions have support in $B_R$. We have by Lemma \ref{lem:CS-intermed}
\[
\|\nabla u \cdot \nu\|_{H^1(\Sigma)} \leq C\| U \cdot \tilde \nu\|_{H^{\frac32}(\Omega)} \leq C \| U \cdot \tilde \nu\|_{H^{\frac32}(\R^3)}. 
\]
The Kato-Ponce inequality \eqref{eq:kato-ponce} with $p_2 =8, q_2=8/3$ yields
\[
\| U \cdot \tilde \nu\|_{H^{\frac32}(\R^3)} \leq C \| U \|_{H^{\frac32}(\R^3)} \|\tilde \nu \|_{L^\infty(\R^3)}+ C \| U \|_{L^{8}(\R^3)}\|\tilde \nu\|_{W^{\frac32,\frac83}(\R^3)}.
\]
We have $\|\tilde \nu \|_{L^\infty(\R^3)}\leq C$ and by the Sobolev embedding $\| U \|_{L^{8}(\R^3)} \leq C\| U \|_{H^{\frac32}(\R^3)}$. We use \eqref{eq:BM} to deduce that 
\[
\|\nabla \tilde \nu\|_{W^{\frac12,\frac83}(\R^3)} \leq C \|
\nabla \tilde\nu\|_{H^{1}(\R^3)}^{\frac12} \|\nabla \tilde \nu\|_{L^{4}(\R^3)}^{\frac12}.
\]
By \eqref{eq:poisson1-5} we have 
\[
\| \nabla \tilde \nu\|_{L^{4}(\R^3)} \leq \|\tilde \nabla \nu\|_{L^{4}(\Omega)} \leq C
\]
 and by \eqref{eq:dirichlet-66}
 \[
 \|\nabla \tilde\nu\|_{H^{1}(\R^3)} \leq C  \|\nabla^2 \tilde\nu\|_{L^{2}(\Omega)} \leq C.
 \]
Therefore by combining the previous inequalities we have  
\[
\|\nabla u \cdot \nu\|_{H^1(\Sigma)} \leq C  \| U \cdot \tilde \nu\|_{H^{\frac32}(\R^3)} \leq C\| U \|_{H^{\frac32}(\R^3)} \leq C\|\nabla u \|_{H^{\frac32}(\Omega)}. 
\]
The result then follows from \eqref{eq:dirichlet-77}.
\end{proof}

We conclude this section by proving the sharp boundary  regularity estimate for  the Dirichlet problem. The proof follows the argument in \cite[Theoreom 4.1]{FJM3D}, with the difference that here we have Dirichlet boundary datum, instead of the zero Neumann case.  
\begin{theorem}\label{teo:reg-capa}
Assume $\Omega$, with $\Sigma = \pa \Omega$,  is  uniformly $C^{1,\alpha}(\Gamma)$-regular and satisfies $\eqref{eq:notationhp}$ for $m \geq 2$. Let  
 $u \in \dot{H}^1(\Omega^c)$ be the solution of 
\begin{equation}
    \begin{cases}
    \Delta u = 0 \qquad &x\in \Omega^c
    \\
    u=g &x\in \Sigma.
    \end{cases}
\end{equation}
Then for all integers $0\leq k \leq m-1$ it holds 
\begin{equation}
    \|\nabla^k u\|_{H^\frac12(\Sigma)} \leq C (1+ \|B_{\Sigma}\|_{H^{k-1}(\Sigma)} +\| g \|_{H^{k+\frac12}(\Sigma)})
\end{equation}
for some constant $C$, depending on $m$ and on the $C^{1,\alpha}$-norm of the heightfunction.
Moreover, if $g$ is constant then the above holds  for  all $k \in \N$.
\end{theorem}
\begin{proof}
{\bf Step 1: Flattening the boundary.}
Since $\Sigma$ is $C^{1,\alpha}(\Gamma)$, for any $x\in \Sigma$ we  find $\delta>0$ such that after rotating and translating the coordinates 
$$
\Omega^c\cap B_\delta= \{(x', x_3):\, x_3>\phi(x') \}
$$
with $\phi \in C^{1,\alpha}(B_\delta) $, $\phi(0)=0$ and $\nabla \phi(0)= 0$. Consider the diffeomorfism $\Psi:\Omega^c\cap B_\delta\to B_\delta^+$
$\Psi (x',x_3)\to (x',x_3-\phi(x'))$ and let $v:= u\circ \Psi^{-1}$ and $w:= g\circ \Psi^{-1}$. Let us extend $g$ by its harmonic extension, denote it by $\tilde g$, and thus $w = \tilde g\circ \Psi^{-1}$ is defined in $B_{\delta}^+$.   By standard calculations we deduce that $v$ is the solution of 
\begin{equation}\label{eq:flatboundary}
    \begin{cases}
    \Div (A_\phi \nabla  v) = 0 \qquad &x\in B^+_\delta
    \\
    v=w &x_3=0,
    \end{cases}
\end{equation}
where $A_\phi$ is symmetric matrix which can be written as  $A_\phi = I + \tilde A(\nabla \phi)$ where $\tilde A(\nabla \phi(x)) = 0$ if $\nabla \phi(x)= 0$. In particular, by choosing $\delta$ small enough $A_\phi$ is postitive definite. 
In  weak form \eqref{eq:flatboundary} reads as
\[
\int_{B_\delta^+}  A_\phi \nabla v \cdot  \nabla \varphi \, dx= 0
\]
for all $\varphi \in C_0^\infty(B_\delta^+)$.

Let $k $ be an integer as in the statement. Let us differentiate the equation \eqref{eq:flatboundary} $k$ times in tangential directions. To this aim let us fix an index vector $\gamma=(\gamma_1,\gamma_2,0)$ with $\gamma_1+\gamma_2=k $, 
and denote $\bar v = \nabla ^\gamma v$  and $\bar w= \nabla^\gamma w$. 
Then $\bar v$  is the solution of 
\begin{equation}\label{eq:flatboundary1}
    \begin{cases}
    \Div (A_\phi \nabla  \bar v) = -\sum_{\tilde \alpha,\beta}  \Div (\nabla^{\tilde \alpha} A_\phi \nabla  \nabla^{\beta} v) \qquad &x\in B^+_\delta
    \\
    \bar v=\bar w &x_3=0 .
    \end{cases}
\end{equation}
with $\tilde \alpha=(\tilde \alpha_1,\tilde \alpha_2,0)$, $\beta=(\beta_1,\beta_2,0)$,   $|\beta|\leq k-1$ and $|\tilde \alpha|+|\beta|\leq k$. In the weak form this reads as
\[
\int_{B_\delta^+}  A_\phi \nabla \bar v \cdot  \nabla \varphi \, dx=
 - \sum_{\tilde \alpha,\beta} \int_{B_\delta^+}  (\nabla^{\tilde \alpha} A_\phi \nabla  \nabla^\beta v ) \cdot \nabla \varphi  \, dx 
\]
for all $\varphi \in C_0^\infty(B_\delta^+)$.

{\bf Step 2: Choice of the test function that has zero boundary value}
Let $\zeta \in C_0^\infty(B_\delta^+)$ be a smooth cut-off function such that $\zeta(x) =1$ for $|x|\leq \frac{\delta}{2}$ and $0\leq \zeta \leq 1$. We choose a test function $\varphi=(\bar v- \bar w )\zeta^2$, which has zero boundary value. With this choice we have
\[
\begin{split}
&\int_{B_\delta^+}  (A_\phi \nabla \bar v \cdot  \nabla \bar v )\,\zeta^2 dx=\int_{B_\delta^+} ( A_\phi \nabla \bar v \cdot  \nabla \bar w ) \zeta^2 \, dx +
\int_{B_\delta^+} ( A_\phi \nabla \bar v \cdot  \nabla \zeta )\, (\bar w  -\bar v) \zeta dx\\
&- \sum_{\tilde \alpha,\beta}\int_{B_\delta^+} (\nabla^{\tilde \alpha} A_\phi \nabla  \nabla^\beta v \cdot \nabla (\bar v- \bar w  )) \zeta^2\, dx 
-2\sum_{\tilde \alpha,\beta}\int_{B_\delta^+} (\nabla^{\tilde \alpha} A_\phi \nabla  \nabla^\beta v \cdot \nabla \zeta ) (\bar v- \bar w  )\zeta\, dx\\
&=I_1+I_2+I_3+I_4.
\end{split}
\]
By the assumption $\phi \in C^{1,\alpha}$ it holds  $\|A_\phi\|_{L^\infty} \leq C$. Thus we may bound  the first two terms as 
\[
I_1+I_2\leq C\|\nabla \bar v\|_{L^2(B_\delta^+)}( \|\nabla \bar w\|_{L^2(B_\delta^+)} +\| \bar v - \bar w\|_{L^2(B^+_\delta)}).
\]

The term $I_3$ is more difficult to treat. Note first since $A_\phi$ is of the form $I + \tilde A(\nabla \phi)$ we have a point-wise bound by the Leibniz rule
\[
\sum_{\tilde \alpha,\beta}|\nabla^{\tilde \alpha} A_\phi|^2||\nabla \nabla^\beta v|^2  \leq
C\sum_{\substack{|\alpha| +|\beta|\leq k\\|\beta|\leq k-1}}(1+ |\nabla^{\alpha_1}  \nabla \phi|^2\dots |\nabla^{\alpha_{k}}\nabla \phi|^2 ) |\nabla  \nabla^{\beta} v|^2.
\]
Hence, we obtain by H\"older's inequality
\[
\begin{split}
I_3 \leq&  C\|\nabla (\bar v-\bar w)\|_{L^2(B^+_\delta)} \\
&\cdot \sum_{\substack{|\alpha| +|\beta|\leq k\\|\beta|\leq k-1}}\Big(1+ \|\nabla^{\alpha_1}  \nabla \phi\|_{L^{\frac{2k}{\alpha_1}}(B_\delta^+)}\dots \|\nabla^{\alpha_{k}}\nabla \phi\|_{L^{\frac{2k}{  \alpha_k}}(B^+_\delta)}  
\Big)\|\nabla  \nabla^{\beta} v\|_{L^{\frac{2k}{|\beta|}}(B^+_\delta)}.
\end{split}
\]
We use interpolation inequality to estimate  
\[
\|\nabla^{\alpha_i} \nabla \phi\|_{L^{\frac{2k}{\alpha_1}}(B_\delta^+)}
\leq  \|\nabla \phi\|_{H^{k}(B_\delta^+)}^{\frac{\alpha_i}{k} } \|\nabla \phi\|^{1-\frac{\alpha_i}{k}}_{L^\infty (B_\delta^+)}.
\]
Also by interpolation we have
\[
\|\nabla\nabla^\beta v \|_{L^{\frac{2k}{|\beta|}} (B_\delta^+)}\leq \|v\|_{H^{k+1}(B_\delta^+)}^{\frac{|\beta|}{k}}\|\nabla v\|_{L^\infty(B_\delta^+)}^{1-\frac{|\beta|}{k}}
\]
and $|\beta|\leq k-1$. Since $\Sigma$ is $C^{1,\alpha}$-regular, we have by  Schauder estimates \cite{GT} that  $\nabla v \in C^{0,\alpha}(B_\delta^+)$. Note that $\sum_i \frac{\alpha_i}{k} \leq \frac{k-|\beta|}{k} $ and $\frac{|\beta|}{k} <1$. Therefore by  Young's inequality  we deduce
\[
\begin{split}
   |I_3|\leq& C \|\nabla (\bar  v-\bar w)\|_{L^2(B^+_\delta)}( 1+ \|\nabla \phi\|_{H^k(B_\delta^+)}^{\frac{k-|\beta|}{k}}) \| v\|_{H^{k+1}(B_\delta^+)}^{\frac{|\beta|}{k}}\\
   \leq&  \e \|\nabla (\bar  v-\bar w)\|_{L^2(B^+_\delta)}^2 + \e \|v\|_{H^{k+1}(B_\delta^+)}^2 + C_\e (1+  \|\nabla \phi\|_{H^k(B_\delta^+)}^2).
\end{split}
\]
We bound the last term $I_4$ similarly. 

Finally we collect the previous estimates, use the ellipticity of the matrix $A_\phi$ and the definition of $\bar w$ and obtain
\[
    \|\nabla \bar v\|_{L^2(B^+_{\delta/2})}^2 
\leq
4\e \|v\|_{H^{k+1}(B_\delta^+)}^2 
   + C(1+ \|\phi\|_{H^{k+1}(B_\delta^+)}^2 + \| w\|_{H^{k+1}(B_\delta^+)}^2).
\]
Summing over all the multi index of the type $(\gamma_1,\gamma_2)$
we have the control over the horizontal derivatives. 
To estimate the vertical derivatives, we use the equation in the strong form as in \cite{FJM3D}, and obtain
\beq \label{eq:reg-boundary-1}
\| v\|_{H^{k+1}(B^+_{\delta/2})}^2 \leq C\e \|v\|_{H^{k+1}(B_\delta^+)}^2 
   + C(1+ \|\phi\|_{H^{k+1}(B_\delta^+)}^2 + \| w\|_{H^{k+1}(B_\delta^+)}^2).
\eeq
\\
{\bf Step 3: Going back to the original function.}
We need to go back to the original function $u$. The argument is similar to \cite{FJM3D} and we merely sketch it.  We note that arguing as in  \cite[Thorem 4.1]{FJM3D}  we may control 
\[
\|\phi\|_{H^{k+1}(B_\delta^+)} \leq C(1+ \|B_{\Sigma}\|_{H^{k-1}(\Sigma)})
\]
for all $k \in \N$. Recall that $\tilde g$ is the harmonic extension of $g$. Using the assumption that  the curvature satisfies the condition \eqref{eq:notationhp} for $m$, we may deduce, arguing as in the proof of Proposition \ref{prop:extension}, that for $k \leq m-1$ it holds 
\[
\| w\|_{H^{k+1}(B_\delta^+)} \leq C\|\tilde g\|_{H^{k+1}(\Omega^c \cap B_{\delta })} \leq C \|g\|_{H^{k+\frac12}(\Sigma)}. 
\]
Obviously if $g$ is constant the above inequality is trivial.

Fix $\sigma$ small such that $\cup_{x\in \Sigma}B_\delta (x)$ covers $\mathcal N_\delta =\{x\in \Omega^c : d(x,\Omega)\leq \delta \}$  and $\sigma_1<\sigma_2<\sigma$. 
By compactness we may choose a finite family of balls covering $\mathcal N_\delta$. Choosing $\e $ small enough we have by \eqref{eq:reg-boundary-1} and by the above inequalities
\[
\|u \|_{H^{k+1}(\mathcal N_{\sigma_1})}^2
\leq C (\|u \|_{H^{k+1}(\mathcal N_{\sigma_2}\setminus \mathcal N_{\sigma_1})}^2 + 1+ \|B_{\Sigma}\|^2_{H^{k-1}(\Sigma)} +\| g \|^2_{H^{k+\frac12}(\Sigma)}).
\]
Since $u$ is harmonic, the interior regularity  yields
\[
\|u \|_{H^{k+1}(\mathcal N_{\sigma_2}\setminus \mathcal N_{\sigma_1})}\leq C \|u\|_{L^2(\mathcal N_{\sigma_2})}.
\]
By standard estimates from harmonic analysis \cite{Dahl} it holds for $R$ large
\[
\|u\|_{L^2(N_{{\sigma}_2})} \leq \|u\|_{L^2(\Omega^c\cap B_R)} \leq C(\|u\|_{L^2(\Sigma)} + \|u\|_{L^2(\pa B_R)} ) \leq C(1 + \|g\|_{L^2(\Sigma)} ).
\]
Therefore  we have
\[
\|u \|_{H^{k+1}(\mathcal N_{\sigma_1})} \leq C (1+  \|B_{\Sigma}\|_{H^{k-1}(\Sigma)} +\| g \|_{H^{k+\frac12}(\Sigma)}).
\]
The claim follows from 
\[
\|\nabla^k u\|_{H^\frac12(\Sigma)} \leq C \| \nabla^{k+1}u \|^2_{L^2(\mathcal N_{\sigma_1})}.
\]
\end{proof}

\section{Useful formulas}

In this section we  focus on the equations \eqref{system} and assume that the family of sets  $(\Omega_t)_{t \in (0,T)}$ and velocities $v(\cdot,t)$ are solution of \eqref{system}. We derive a general formula for the commutators of the material derivative of high order $\D_t^k$ with spatial derivatives. We apply this to calculate  $[\D_t^k,\nabla] v$ and   $[\D_t^k,\nabla] p$, which  will produce two types of error terms, \eqref{eq:def-R-div} and \eqref{def:error-bulk},  defined in the fluid domain $\Omega_t$. We will also calculate the formula for $\D_t^k p$ on the moving boundary $\Sigma_t$ in Lemma \ref{formula:Dtp}, which includes third type of error term defined in \eqref{eq:R-p-in-3}. The precise structures of these error terms are  complicated and we only need to trace the order of the derivatives that appear.  Therefore we effectively use the notation from \cite{Ham}
\[
\nabla^k f \star \nabla^l g 
\]
to denote a contraction of some indexes of tensors  $\nabla^i f$  and $\nabla^j g$ for $i \leq k$ and $j \leq l$. Note that we include the lower order derivatives.

We begin by recalling the following formulas from \cite{CPAM}
\begin{equation}
\label{eq:comm1}
	[\D_t,\nabla] f = \D_t \nabla  f -  \nabla \D_t f =  - \nabla v^T\nabla f,
\end{equation}
\begin{equation}
\label{eq:comm2}
	[\D_t, \nabla_\tau] f =  - (\nabla_\tau v)^T \nabla_\tau f 
\end{equation}
and
\begin{equation}\label{eq:comm3}
	[\D_t, \Delta_\Sigma]  f  =   \nabla_\tau^2 f \star  \nabla v -  \nabla_\tau f \cdot \Delta_\Sigma v + B \star \nabla v \star \nabla_\tau f.
\end{equation}
Let us also recall  the material derivative of the normal field. We use the shorthand notation $\nu = \nu_{\Sigma_t}$, $B= B_{\Sigma_t}$ and  $v_n = v \cdot \nu$. We have by \cite{CPAM}
\begin{equation} \label{eq:normal1}
	\D_t \nu = - (\nabla _\tau v)^T \nu.
\end{equation}
Since $\nabla_\tau v_n = \nabla_\tau v^T \nu + B v_\tau$,  we may write \eqref{eq:normal1} as
\begin{equation} \label{eq:normal2}
	 \D_t \nu = - \nabla_\tau v_n+  B v_\tau.
\end{equation}

We need higher order versions of the commutation formula \eqref{eq:comm1}, i.e., for 
\[
[\D_t^l, \nabla^k] f =  \D_t^l \nabla^k f -\nabla^k  \D_t^l f.
\]
Recall the definition of the norm of an index vector  $\alpha = (\alpha)_{i=1}^k \in \N^k$ 
\[
|\alpha| = \sum_{i =1}^k  \alpha_i
\]
and  note that we include zero in the set of natural numbers $\N$.
\begin{lemma}
\label{lem:comm1}
For $l, k \in \N$ with $l,k \geq 1$ it holds
\[
[\D_t^l, \nabla^k] f =  \sum_{\substack{|\alpha| \leq  k-1 \\|\beta| \leq l-1 }} \nabla^{1+\alpha_1} \D_t^{\beta_1} v \star \cdots \star  \nabla^{1+\alpha_l} \D_t^{\beta_l} v \star \nabla^{1+\alpha_{l+1}} \D_t^{\beta_{l+1}}f.
\]
\end{lemma}

\begin{proof}
Let us first assume $l=1$ and prove 
\beq \label{eq:commgrad}
\D_t\nabla^k f -\nabla^k \D_t f = \sum_{|\alpha| \leq  k-1} \nabla^{1+\alpha_1} v \star  \nabla^{1+\alpha_{2}} f.
\eeq
We argue by induction over $k$ and observe immediately that the case $k=1$ follows from \eqref{eq:comm1}. Assume that \eqref{eq:commgrad} holds  for $k-1$  and note that by \eqref{eq:comm1} we have
\[
\D_t\nabla^k f = \D_t \, \nabla (\nabla^{k-1} f) =  \nabla \, \D_t (\nabla^{k-1} f) + \nabla v \star (\nabla^{k} f).
\]
By  induction assumption we have
\[
\begin{split}
\nabla \, \D_t (\nabla^{k-1} f) &= \nabla \big(\nabla^{k-1} \D_t f +   \sum_{|\alpha| \leq  k-2} \nabla^{1+\alpha_1} v \star \nabla^{1+\alpha_{2}} f\big)
\\&= \nabla^k \D_t f +
\sum_{|\alpha| \leq  k-1} \nabla^{1+\alpha_1} v \star   \nabla^{1+\alpha_{2}} f.
\end{split}
\]
This yields the claim for $l=1$. 

The proof for $l\geq 1$ follows from a similar induction argument. Assume that the claim holds for $l-1$ and 
note that
\[
\D_t^l \nabla^k f= \nabla^k \D_t^l f +\D_t ([ \D_t^{l-1},\nabla^k ] f) + [\D_t, \nabla^k] (\D_t^{l-1} f).
\]
By the first claim we have 
\[
[\D_t, \nabla^k] (\D_t^{l-1} f) = \sum_{|\alpha| \leq  k-1 } \nabla^{1+\alpha_1}  v \star \nabla^{1+\alpha_2} \D_t^{l-1}f.
\]
On the other hand,  by the  induction assumption we have 
\[
\begin{split}
    \D_t[\D_t^{l-1} ,\nabla^k] f =  \D_t  \sum_{\substack{|\alpha| \leq  k-1 \\|\beta| \leq l-2 }} \nabla^{1+\alpha_1} \D_t^{\beta_1} v \star \cdots \star  \nabla^{1+\alpha_{l-1}} \D_t^{\beta_{l-1}} v \star \nabla^{1+\alpha_{l}} \D_t^{\beta_{l}}f .
\end{split}
\]
We use the Leibniz rule and the first claim to deduce that  
\[
\begin{split}
    \D_t\nabla^{1+\alpha_1} \D_t^{\beta_1} v=\nabla^{1+\alpha_1} \D_t^{1+\beta_1} v+
    \sum_{|\tilde \alpha| \leq  \alpha_1}
    \nabla^{1+\tilde \alpha_1} v \star \nabla^{1+\tilde \alpha_{2}} \D_t^{\beta_1}v.
\end{split}
\]
Similar formula holds also for $\D_t\nabla^{1+\alpha_l} \D_t^{\beta_l} f$. Hence, we obtain the claim. 
\end{proof}

Let us next prove higher order commutation formulas for \eqref{eq:comm2} and  a formula for $\D_t^l \nu$ and $\D_t B$. Below $a_{\beta}(\nu)$ and $a_{\alpha,\beta}(\nu, B)$ denote bounded coefficients which depend on $\nu$ and on $\nu $ and $B$ respectively. 

\begin{lemma}
\label{lem:comm2}
 For $l \geq 1$  it holds
\begin{equation} \label{lem:comm:eq0}
[\D_t^l, \nabla_\tau] f =      \sum_{ |\beta|  \leq l-1} a_{\beta}(\nu) \nabla \D_t^{\beta_1} v \star \cdots \star \nabla \D_t^{\beta_l} v \star \nabla_\tau \D_t^{\beta_{l+1}}f
\end{equation}
  and 
\begin{equation} \label{lem:comm:eq1}
[\D_t^l, \nabla^2_\tau] f =      \sum_{\substack{|\alpha|\leq 1 \\|\beta| \leq l-1 } } a_{\alpha, \beta}(\nu, B)   \nabla^{1+\alpha_1}  \D_t^{\beta_1} v \star \cdots \star \nabla^{1+\alpha_{l}} \D_t^{\beta_{l}} v \star \nabla_\tau^{1+ \alpha_{l+1}} \D_t^{\beta_{l+1}}f.
\end{equation}
Moreover we have 
\begin{equation} \label{lem:comm:eq2}
\D_t^l \nu = \sum_{ |\beta|\leq l-1} a_{\beta}(\nu) \,\nabla \D_t^{\beta_1} v \star \cdots \star \nabla \D_t^{\beta_l} v   
\end{equation}
 and
\begin{equation} \label{lem:comm:eq3}
\D_t^l B = \sum_{\substack{|\alpha| \leq 1\\ |\beta| \leq l-1 }} a_{\alpha, \beta}(\nu, B) \nabla^{ 1+\alpha_1} \D_t^{\beta_i} v \star \cdots \star \nabla^{1+\alpha_{l+1}} \D_t^{\beta_{l+1}} v . 
\end{equation}
\end{lemma}

\begin{proof}
Let us first prove \eqref{lem:comm:eq2}. First, the claim holds for $l=1$ by \eqref{eq:normal1}. Let us assume that \eqref{lem:comm:eq2} holds for $l-1$. Then 
\[
\D_t^{l} \nu = \D_t \sum_{ |\beta|\leq l-2} a_{\beta}(\nu) \,\nabla \D_t^{\beta_1} v \star \cdots \star \nabla \D_t^{\beta_{l-1}} v.
\]
By  \eqref{eq:normal1}  it holds $\D_t a_{\beta}(\nu) = \tilde a_{\beta}(\nu) \nabla v$ and by \eqref{eq:comm1} we have
\[
\D_t \nabla \D_t^{\beta_i} v =\nabla \D_t^{\beta_i+1} v + \nabla v \star \nabla \D_t^{\beta_i} v.
\]
Thus we deduce 
\[
\D_t \sum_{|\beta| \leq l-2 } a_{\beta}(\nu) \, \nabla \D_t^{\beta_1} v \star \cdots \star \nabla \D_t^{\beta_{l-1}} v  = \sum_{|\beta| \leq l-1} \tilde a_{\beta}(\nu) \,\nabla\D_t^{\beta_1} v \star \cdots \star \nabla \D_t^{\beta_{l}} v 
\]
which implies \eqref{lem:comm:eq2}.

Let us next prove \eqref{lem:comm:eq0}. By \eqref{eq:comm2}  the claim holds for $l=1$. Let us assume that the claim holds for $l-1$. Then 
\[
\D_t^{l} \nabla_\tau f = \D_t  \nabla_\tau \D_t^{l-1}f+\D_t \sum_{|\beta| \leq l-2} a_{\beta}(\nu) \nabla \D_t^{\beta_1} v \star \cdots \star \nabla\D_t^{\beta_{l-1}} v \star \nabla_\tau \D_t^{\beta_{l}}f .
\]
As before we have by \eqref{eq:normal1} $\D_t a_{\beta}(\nu) = \tilde a_{\beta}(\nu) \nabla v$ and by \eqref{eq:comm2}  
\[
\D_t \nabla_\tau\D_t^{\beta_i} f =  \nabla_\tau \D^{\beta_i+1}_t f + a(\nu)  \nabla v \star \nabla_\tau \D_t^{\beta_i} f.
\]
Therefore we obtain by Leibniz rule
\[
\D_t^{l} \nabla_\tau f = \nabla_\tau \D^{l}_t f +  \sum_{ |\beta| \leq l-1} \tilde a_{\beta}(\nu) \nabla \D_t^{\beta_1} v \star \cdots \star \nabla \D_t^{\beta_{l}} v \star \nabla_\tau \D_t^{\beta_{l}}f
\]
and \eqref{lem:comm:eq0} follows. 

We notice next that  \eqref{lem:comm:eq3} follows from  $B = \nabla_\tau \nu$ and by combining  \eqref{lem:comm:eq0} with  \eqref{lem:comm:eq2}. Finally   we obtain \eqref{lem:comm:eq1} by first applying  \eqref{lem:comm:eq0}  as
\[
\begin{split}
\D_t^l \nabla_\tau^2 f &= \nabla_\tau (\D_t^l \nabla_\tau f) + \sum_{ |\beta| \leq l-1} a_{\beta}(\nu) \nabla \D_t^{\beta_1} v \star \cdots \star \nabla \D_t^{\beta_l} v \star \nabla_\tau \D_t^{\beta_{l+1}} \nabla_\tau f.
\end{split}
\]
 The claim  then follows by differentiating \eqref{lem:comm:eq0}. 
\end{proof}

\begin{remark} \label{rem:comm-vn}
By Lemma \ref{lem:comm2} we have in particular that 
\[
\D_t^l v_n =  \sum_{ |\beta| \leq l} a_{ \beta}(\nu) \,\nabla \D_t^{\beta_1} v \star \cdots \star \nabla \D_t^{\beta_l} v \star D_t^{\beta_{l+1}} v,
  \]
where $\beta_i\leq l-1$ for $i \leq l$. Moreover, since we may write  the Laplace-Beltrami operator as $\Delta_\Sigma f = \text{Tr}(\nabla_\tau^2 f)$ then Lemma \ref{lem:comm2} yields
\[
\D_t^l \Delta_\Sigma f =  \Delta_\Sigma \D_t^l f +  \sum_{\substack{|\alpha| \leq 1 \\ |\beta| \leq l-1  }} a_{\alpha, \beta}(\nu, B)   \nabla^{1+\alpha_1}  \D_t^{\beta_1} v \star \cdots \star \nabla^{1+\alpha_{l}} \D_t^{\beta_{l}} v \star \nabla_\tau^{1+ \alpha_{l+1}} \D_t^{\beta_{l+1}}f .
\]
\end{remark}

Let us next derive  formulas for the divergence and the curl of the vector field $\D_t^l v$. Recall that by \eqref{system} we have $\Div v = 0$ which then implies 
\begin{equation}
    \label{eq:pressure-1}
   - \Delta p=  \Div (\D_t v) = \text{Tr}((\nabla v)^2) = \Div \Div (v \otimes v). 
\end{equation}
For the curl we have   $\curl (D_t v) = 0$ and    $\omega = \curl v = \nabla v - \nabla v^T $ satisfies (see e.g. \cite{CPAM})
\begin{equation}
    \label{eq:omega}
\D_t\omega=  -\nabla v^T \omega - \omega \nabla v.
\end{equation}
We will derive  formulas for $\Div \D_t^l v$ and $\curl \D_t^l v$ below by using \eqref{eq:pressure-1}, \eqref{eq:omega} and  the commutation formula in Lemma \ref{lem:comm1}. 
To this aim we introduce two type of error functions. The first type we denote by $R_{\Div}^l$, which stands for any function which can be written in the form
\begin{equation}
    \label{eq:def-R-div}
 R^{l}_{\Div}=\sum_{|\beta|\leq l} a_{ \beta}(\nabla v) \nabla \D_t^{\beta_1} v \star \cdots \star\nabla \D_t^{\beta_l}v,
\end{equation}
for  $l \geq 0$. We also use the convention that the indexes are ordered as $\beta_1 \geq \beta_2 \geq \dots \geq \beta_l$. The second type of error function is slightly more general and it can be written in the form 
\begin{equation}
\label{def:error-bulk}
R_{bulk}^l = \sum_{|\alpha|\leq 1, |\beta| \leq l} a_{\alpha, \beta}(\nabla v) \nabla \D_t^{\beta_1} v \star  \cdots \star\nabla \D_t^{\beta_{l}} v  \star \nabla^{\alpha_1} \D_t^{\alpha_2 + \beta_{l+1}} v,
\end{equation}
for $l \geq 0$. Note that $R_{bulk}^l$ has one higher order term compared to $ R^{l}_{\Div}$. In particular,  all functions of type $ R^{l}_{\Div}$ are contained in $R_{bulk}^l$. The reason  for introducing  these two notations is that we need to estimate them in different norms. We will do this in the next section. Note that  using Lemma \ref{lem:comm1} and  $-\nabla p = \D_t v$ we deduce that  
\begin{equation}
\label{eq:comm-bulk}
    [\D_t^{l+1},\nabla]p = R_{bulk}^l.
\end{equation}

\begin{lemma}   
\label{lem:curl}
Let  $l\geq 1$ and denote $\omega = \curl v$. Then  it holds
\[
\D_t \nabla ^{l} \omega=   \nabla v\star  \nabla^{l} \omega + \sum_{|\alpha|\leq l}\nabla^{1+\alpha_1} v 
\star \nabla ^{1+\alpha_{2}}v.
\]
Moreover,  $\curl \D_t^l v $ and $\Div \D_t^l v $ can be written in the form 
\[
\curl \D_t^l v  =  R_{\Div}^{l-1}\qquad \text{and} \qquad \Div\D_t^l v  =  R_{\Div}^{l-1}.
\] 
We may also write the divergence  of $\D_t^{l+1} v$ as 
\[
\Div \D_t^{l+1} v= \Div\Div  ( v \otimes \D_t^{l} v) + \Div R^{l-1}_{bulk}.    
\]
\end{lemma}
\begin{proof}
The first claim is an immediate consequence of Lemma \ref{lem:comm1} and  \eqref{eq:omega}. The second claim follows from  Lemma \ref{lem:comm1}  and from $\curl (\D_t v) = 0$. Similarly the third claim follows from Lemma \ref{lem:comm1}  and $\Div v = 0$.

Let us then prove the last claim. We begin by proving two useful identities. First, we claim that for a vector field $F$ it holds 
\begin{equation}
\label{eq:curl-1}
    [\D_t,\Div ]F = -\Div (\nabla v F).
\end{equation}
 Indeed, since $\Div v = 0$ we have 
 \[
 \D_t\Div F- \Div \D_t F 
 =\sum_{i,j=1}^3 v_i\pa_i (\pa_jF_j)- \pa_i(\pa_jF_i v_j)=- \sum_{i,j=1}^3 \pa_i v_j \pa_j F_i= -\Div (\nabla v F).
 \]

The second identity follows also from  $\Div v = 0$  and we may write it 
\begin{equation}
\label{eq:curl-2}
 \Div \Div( v\otimes \D_t^lv )
 =\Div (\nabla \D_t^lv \, v).
\end{equation}
Let us prove the last claim in the case $l = 1$. We use \eqref{eq:comm1}, \eqref{eq:pressure-1}, \eqref{eq:curl-1},   \eqref{eq:curl-2} and the definition of $R_{bulk}$ in \eqref{def:error-bulk}
to deduce
\[
 \begin{split}
\Div \D_t^{2}v &= \D_t \Div \D_t v - [\D_t,\Div] \D_t v \\
&= \D_t \Div (  \nabla v\, v  ) + \Div (\nabla v \, \D_t v)\\
&= \Div ( \D_t ( \nabla v\, v ) ) -\Div (\nabla v \, \nabla v \, v)+ \Div(R_{bulk}^0)\\
&= \Div ( \nabla \D_t v\, v ) ) + \Div(R_{bulk}^0)\\
&= \Div \Div (v \otimes \D_t v) ) + \Div(R_{bulk}^0). 
 \end{split}
\]

Let us assume that the claim holds for $l -1$. We argue as before and obtain by \eqref{eq:curl-1}, \eqref{eq:curl-2} and by the induction assumption 
\[
 \begin{split}
\Div \D_t^{l+1}v &= \D_t \Div \D_t^{l} v - [\D_t, \Div] \D_t^{l} v \\
&= \D_t \Div \left(  \nabla\D_t^{l-1} v\, v  +  R_{\Div}^{l-2} \right) + \Div (\nabla v \, \D_t^{l} v). 
 \end{split}
\] 
We use \eqref{eq:curl-1}, \eqref{eq:comm1} and the definition of $R_{bulk}^{l-1}$ in \eqref{def:error-bulk} and obtain
\[
\begin{split}
\D_t \Div &\left(  \nabla\D_t^{l-1} v\, v\right) =  \Div \left( \D_t( \nabla\D_t^{l-1} v\, v)\right) + [\D_t,\Div](\nabla\D_t^{l-1} v\, v) \\
&=\Div (\nabla \D_t^{l} v \, v) + \Div \left(  \nabla\D_t^{l-1} v \star \nabla v \star v + \nabla\D_t^{l-1} v \star \D_t v   \right)\\
&= \Div (\nabla \D_t^{l} v \, v) + \Div R_{bulk}^{l-1}\\
&= \Div \Div (v \otimes \D_t^{l} v) + \Div R_{bulk}^{l-1}.
\end{split}
\]
 Similarly we have 
\[
\D_t \Div R_{\Div}^{l-2} = \Div R_{\Div}^{l-1} \qquad \text{and} \qquad  \Div (\nabla v \, \D_t^{l} v) = \Div R_{\Div}^{l-1} 
\]
and the claim follows.  
\end{proof}

Let us then turn our focus on the pressure. By \eqref{eq:pressure-1} and \eqref{system} we have that $p$ is a solution of 
\[
\begin{cases}
-\Delta p = \Div \Div (v \otimes v), \quad \text{in }\, \Omega_t, \\
p = H - \frac{Q(t)}{2} |\nabla U|^2, \quad \text{on }\, \Sigma_t,   
\end{cases}
\]
where $Q(t)$ is a real valued function of time  defined in \eqref{def:constant-q} as 
\[
Q(t) = \frac{Q}{(\text{Cap}(\Omega_t))^2},
\]
$U = U_{\Omega_t}$ is the capacitary potential and $H = H_{\Sigma_t}$ is the mean curvature.

We need to derive the equation for $\D_t^l p$. We obtain the equation for $\D_t^l p$ in the bulk from Lemma \ref{lem:curl}. 
\begin{remark} \label{rmk:laplpreassure}
By Lemma \ref{lem:curl} and by \eqref{eq:comm-bulk}  the function $-\Delta \D_t^{l+1}p$ can be written as
\[
\begin{split} 
  -  \Delta  \D_t^{l+1}p &=
-\Div \D_t^{l+1}\nabla p+ \Div [\D_t^{l+1},\nabla]p = \Div \D_t^{l+2} v + \Div R_{bulk}^{l} \\
&=    \Div \Div (v \otimes \D_t^{l+1}v)   + \Div ( R_{bulk}^{l}).
\end{split}
\]
\end{remark}

To find the formula for $\D_t^l p$ on the boundary $\Sigma_t$ is more challenging. To that aim we first need to study the capacitary potential $U$.  We introduce an error term which appears when we deal with the capacity term on the boundary, i.e., for  $l \geq 0 $  we denote by  $R_U^l$  as functions on $\Sigma_t$, which can be written in the form
\begin{equation}
    \label{def:R-U}
    R_U^l = \sum_{\substack{|\alpha| + |\beta| \leq l+1\\|\beta| \leq l}} a_{\alpha,\beta}(v) \D_t^{\beta_1} v \star \cdots \star \D_t^{\beta_{l}} v  \star \nabla^{1+ \alpha_1} \pa_t^{\alpha_2} U.
\end{equation}
We note that $v$ is defined in $\Omega_t$ while $U$ is defined in $\Omega_t^c$, but they are both well-defined on the boundary $\St$.  We have the following formulas for $U$ on $\St$.
\begin{lemma}
    \label{lem:mat-capa}
    Let $l \geq 1$. Then on $\Sigma_t$  it holds 
    \[
    \D_t^l \nabla U = R_U^{l-1}
    \]
  and  
    \[
    \begin{split}
    \D_t^{l+1} \nabla U = &\nabla \pa_t^{l+1} U + \nabla^2 U \, \D_t^l v\\
    &+ \sum_{\substack{\alpha + |\beta| + \gamma \leq  l+1\\ |\beta|\leq l-1, \gamma \leq l}} a_{\alpha, \beta,\gamma}(v) \D_t^{\beta_1} v \star \cdots \star \D_t^{\beta_{l-1}} v \star \nabla^{1+ \alpha} \pa_t^{\gamma} U.
       \end{split}
    \]
 Moreover we have the following formula for $\pa_t^{l+1} U$     
\[
\pa_t^{l+1} U = - \pa_{\nu} U \, (\D_t^{l} v \cdot \nu) + R_U^{l-1} \qquad \text{on } \, \Sigma_t  .
\]
\end{lemma}
\begin{proof}
The proof of the first statement is straightforward. Note that
\[
  \D_t \nabla U =  \nabla \pa_t U +    \nabla^2 U v
\]
and
\[
\D_t^2 \nabla U =  \nabla \pa_t^2 U +    \nabla^2 U \D_t v + \nabla^3 U\star v \star v  + \nabla^2 \pa_t U \star v. 
\]
Thus the first claim holds for $l =1,2$ and the second for $l=1$. The general case $l \geq 1$ follows by an induction argument. 

For the third claim we recall that the potential satisfies $U = 1$ on $\Sigma_t$. Therefore  it holds $\D_t U = 0$ on $\St$ which we write as
\[
\pa_t U = - ( \nabla U \cdot v ).
\]
Differentiating this  yields
\[
\D_t^{l} \pa_t U =  - (\nabla U \cdot \D_t^{l}  v ) + \sum_{\substack{i+j = l\\i \leq l-1}} \D_t^{i}  v \star \D_t^{j} \nabla U. 
\]
By the first claim we have $\D_t^{j} \nabla U = R_U^{j-1}$ and thus by the definition of $R_U^{j-1}$ in \eqref{def:R-U} we may write 
\[
\D_t^{l} \pa_t U = - (\nabla U \cdot \D_t^{l}  v )  + \sum_{\substack{|\alpha| + |\beta| \leq l\\|\beta| \leq l-1}} a_{\alpha,\beta}(v) \D_t^{\beta_1} v \star \cdots \star \D_t^{\beta_{l-1}} v  \star \nabla^{1+ \alpha_1} \pa_t^{\alpha_2} U.
\]
It also  holds 
\[
\D_t^{l} \pa_t U  = \pa_t^{l+1} U + \sum_{\substack{|\alpha| + |\beta| \leq l\\|\beta| \leq l-1}} a_{\alpha,\beta}(v) \D_t^{\beta_1} v \star \cdots \star \D_t^{\beta_{l-1}} v  \star \nabla^{1+ \alpha_1} \pa_t^{\alpha_2} U.
\]
Since  $U$ is constant on $\Sigma_t$ it holds  $\nabla U = \pa_{\nu} U \,  \nu $. This implies the third claim.
\end{proof}

We conclude this section by deriving a formula for $\D_t^{l+1} p$. Recall that 
\begin{equation}
    \label{eq:D_tp-0}
p = H - \frac{\Qt}{2} |\nabla U|^2 \qquad \text{on } \, \Sigma,
\end{equation}
where $\Qt$ is defined in \eqref{def:constant-q}. It is well known that (e.g. \cite{DDM})
\begin{equation}
    \label{eq:D_tp-1}
   \D_t H = - \Delta_\Sigma v_n - |B|^2 v_n + \nabla_\tau H\cdot v
\end{equation}
where $v_n = v \cdot \nu$. Using the geometric fact 
\begin{equation} \label{eq:geom}
\Delta_\Sigma \nu = -|B|^2 \nu + \nabla_\tau H  
\end{equation}
and \eqref{eq:D_tp-0} we obtain the formula 
\begin{equation} \label{eq:form-Dtp}
\D_t p = - \Delta_\Sigma v \cdot \nu - 2 B : \nabla_\tau v  -  \Qt( \D_t \nabla U \cdot \nabla U) -\frac{Q'(t)}{2} |\nabla U|^2 .
\end{equation}

We may write \eqref{eq:form-Dtp} in a different form. Indeed we use  $\nabla U =  \pa_\nu U \, \nu = - |\nabla U|\nu$ and obtain 
\[
\begin{split}
- \D_t \nabla U \cdot \nabla U &= -(\nabla \pa_t U \cdot \nabla U) - (\nabla^2 U  \nu \cdot  v) \, \pa_\nu U\\
&= -(\nabla \pa_t U \cdot \nabla U)  + ( \nabla^2 U  \nu \cdot    \nu ) \, v_n\,  |\nabla U|  - ( \nabla^2 U  \nabla U  \cdot  v_\tau ).
\end{split}
\]
We notice that  
\begin{equation}
    \label{eq:D_tp-2}
    (\nabla^2 U  \nabla U \cdot v_\tau) = \frac12 ( \nabla_\tau  |\nabla U|^2  \cdot  v).
\end{equation}
Moreover, we recall that $U$ is harmonic in $\Omega_t^c$ and constant on $\St$. Therefore it holds by \eqref{eq:tang-laplace}
\begin{equation}
    \label{eq:D_tp-3}
0 = \Delta U = \overbrace{\Delta_{\tau} U}^{=0} + (\nabla^2 U\nu \cdot \nu) + H \pa_{\nu} U = (\nabla^2 U\nu \cdot \nu) - H|\nabla U|.
\end{equation}
Thus we have by \eqref{eq:D_tp-0},  \eqref{eq:D_tp-1},  \eqref{eq:D_tp-2},  \eqref{eq:D_tp-3}  that
\begin{equation}
    \label{eq:D_tp-4}
 \D_t p = - \Delta_\Sigma v_n - |B|^2 v_n - \Qt\,\big( \pa_\nu U \, (\pa_\nu \pa_t U)  -  H |\nabla U|^2 v_n\big)  + \langle \nabla_\tau p, v\rangle  - Q'(t) \frac{|\nabla U|^2}{2}.
\end{equation}

In the next lemma we find a suitable expression  for $\D_t^{l} p$ for $l \geq 1$. Again we will have an error term which in this case is defined on the boundary $\Sigma_t$ and is more complicated than the previous ones. We define the error term $R_p^l$ for $l \geq 1$ as  
\begin{equation}
    \label{eq:R-p-in-3}
    R_{p}^l = R_{I}^l + R_{II}^l + R_{III}^l.
\end{equation}
where
\begin{equation}
    \label{def:R_I}
\begin{split}
    &R_{I}^l = - (|B|^2 - Q(t) H\, |\nabla U|^2 ) (\D_t^{l}v \cdot \nu) + (\nabla_\tau p \cdot \D_t^{l}v ), \\
    &R_{II}^l = \sum_{|\alpha|\leq 1, \, |\beta|\leq l-1}a_{\alpha, \beta,\gamma}(B) \nabla^{1+\alpha_1} \D_t^{\beta_1} v \star \cdots \star \nabla^{1+\alpha_{l+1}} \D_t^{\beta_{l+1}} v \qquad \text{and} \\
    &R_{III}^l =\sum_{\substack{|\alpha| +  |\beta| + |\gamma |\leq l+1\\ |\beta|\leq l-1, \gamma_i \leq l}} a_{\alpha, \beta,\gamma, Q}(v)  \D_t^{\beta_1} v \star \cdots \star \D_t^{\beta_{l-1}} v  \star \nabla^{1+\alpha_1} \pa_t^{\gamma_1} U \star \nabla^{1+\alpha_1} \pa_t^{\gamma_2} U ,
    \end{split}
\end{equation}
where the coefficients $a_{\alpha, \beta,\gamma,Q}(v)$
depend on $v$ and on the derivatives $Q^{(k)}(t)$ for $k \leq l+1$.
Above $a_{\alpha,\beta, \gamma}(B)$ means that the coefficient depends on the second fundamental form.   For $l=1$ we need to quantify this dependence in which case $R_{II}^1$ reads as  
\begin{equation}
    \label{def:R_II-2}
R_{II}^1 =  a_1(\nu, \nabla v) \star B + a_2(\nu, \nabla v) \star \nabla^2 v.
\end{equation}
 The reason why $R_p^l$ has three terms is that $R_{II}^l$ contains  the error terms arising from the surface tension and $R_{III}^l$ from the capacity. The first term $R_{I}^l$ is separate merely from notational reasons as it contains the  highest order material derivatives. 

\begin{lemma}
    \label{formula:Dtp}
    For $l \geq 2$ It holds
    \[
    \D_t^{l} p = -\Delta_\Sigma (\D_t^{l-1} v \cdot \nu) - Q(t) \, \pa_{\nu} U \,  (\pa_{\nu} \pa_t^{l} U)   + R_p^{l-1}
    \]
    on $\Sigma_t$, where $\Qt$ is defined in \eqref{def:constant-q}.
\end{lemma}
\begin{proof}

We first claim that it holds   
\begin{equation} \label{eq:form-Dtpl}
\begin{split}
\D_t^{l} p &= - (\Delta_\Sigma \D_t^{l-1} v) \cdot \nu - 2 B : \nabla_\tau (\D_t^{l-1} v) \\
&- Q(t) (  \nabla \pa_t^{l} U \cdot \nabla U ) -  Q(t) ( \nabla^2 U  \nabla U  \cdot  \D_t^{l-1} v) + R_{II}^{l-1} +  R_{III}^{l-1}.
\end{split}
\end{equation}
To obtain the claim \eqref{eq:form-Dtpl} for $l=2$ we first recall that by \eqref{eq:comm3} we have 
\[
[\D_t,\Delta_{\Sigma}] v =  a_1(\nu, \nabla v) \star  B  + a_2(\nu, \nabla v) \star \nabla^2 v, 
\]
and that \eqref{eq:normal1} implies $\D_t \nu = -(\nabla_\tau v)^T \nu $ and \eqref{eq:comm2} implies $[\D_t, \nabla_\tau] v= a(\nu) \nabla v \star \nabla v$. We  use   \eqref{eq:comm2} and \eqref{eq:normal1} to obtain
\[
\begin{split}
\D_t B =\D_t (\nabla_\tau \nu) &=  \nabla_\Sigma (\D_t \nu ) + [\D_t, \nabla_\tau] \nu \\
&= -\nabla_\tau ((\nabla_\tau v)^T \nu )  + a_1(\nu, \nabla v) \star  B  \\
&= a_1(\nu, \nabla v) \star  B + a_2(\nu, \nabla v )  \star \nabla^2 v.
\end{split}
\]
We differentiate  \eqref{eq:form-Dtp} and use the above identities and have 
\[
\begin{split}
\D_t^{2} p = - (\Delta_\Sigma \D_t v) \cdot \nu - 2 B : \nabla_\tau (\D_t v) -\D_t \Big(Q(t) (  \D_t \nabla  U \cdot \nabla U ) +\frac{Q'(t)}{2} |\nabla U|^2 \Big)    + R_{II}^1.
\end{split}
\]
 Lemma \ref{lem:mat-capa} yields 
\[
\begin{split}
\D_t (\D_t \nabla U \cdot \nabla U) &= (\D_t^2 \nabla U \cdot \nabla U)+ (\D_t \nabla U \cdot \D_t \nabla U)\\
&= (  \nabla \pa_t^{2} U \cdot \nabla U )  + (\nabla^2 U  \nabla U \cdot \,  \D_t v ) \\
&\,\,\,\,\,\,\,\,\,\,+ \sum_{\substack{|\alpha| + |\beta| \leq 2,\\ \beta_i \leq 1}} a_{\alpha,\beta}(v)  \nabla^{1+\alpha_1} \pa_t^{\beta_1} U \star \nabla^{1+\alpha_2} \pa_t^{\beta_2} U.
\end{split}
\]
We may embed the rest of the terms to $R_{III}^1$.  This implies  \eqref{eq:form-Dtpl} for $l=2$.

To obtain the claim  \eqref{eq:form-Dtpl} for $l\geq 2$   we  differentiate \eqref{eq:form-Dtp} $(l-1)$-times. Since the argument is similar to the case $l=2$, we only highlight the most subtle steps. To identify the error terms  we recall that by Lemma \ref{lem:comm2} we have for $i \leq l-1$   
\[
\D_t^i \nu = \sum_{|\beta| \leq i-1} a_{\beta}(\nu)\nabla \D_t^{\beta_1} v \star  \cdots \star \nabla \D_t^{\beta_i} v ,
\]
 \[
[\D_t^i, \nabla_\Sigma ] v = \sum_{|\beta| \leq i-1} a_{\beta}(\nu)\nabla \D_t^{\beta_1} v \star  \cdots \star \nabla \D_t^{\beta_{i+1}} v,
\]
and  by Remark \ref{rem:comm-vn}
\[
[\D_t^i,\Delta_\Sigma] v =  \sum_{\substack{|\alpha| \leq 1 \\ |\beta| \leq i-1  }} a_{\alpha, \beta}( B)   \nabla^{1+\alpha_1}  \D_t^{\beta_1} v \star \cdots \star \nabla^{1+\alpha_{i}} \D_t^{\beta_{i}} v \star \nabla_\tau^{1+ \alpha_{i+1}} \D_t^{\beta_{i+1}}v .
\]
In order to treat the capacitary terms we  first observe that 
\[
\D_t^{l-1} \langle \D_t \nabla U, \nabla U\rangle = \langle \D_t^{l} \nabla U, \nabla U \rangle + \sum_{\substack{i+j\leq l\\i,j \leq l-1}} \D_t^{i} \nabla U \star  \D_t^{i}\nabla U 
\]
and then use   Lemma \ref{lem:mat-capa} to deduce  
\[
\begin{split}
\D_t^{l-1} &(\D_t \nabla U \cdot \nabla U) = (  \nabla \pa_t^{l} U \cdot \nabla U )  + ( \nabla^2 U  \nabla U \cdot  \D_t^{l-1} v) \\
&+ \sum_{\substack{|\alpha| + |\beta| +|\gamma| \leq l,\\ |\beta| \leq l-2, \, \gamma_i \leq l-1}} a_{\alpha,\beta,\gamma}(v) \D_t^{\beta_1} v \star \cdots \star \D_t^{\beta_{l-1}} \star \nabla^{1+\alpha_1} \pa_t^{\gamma_1} U \star \nabla^{1+\alpha_2} \pa_t^{\gamma_2} U.
\end{split}
\]
This implies \eqref{eq:form-Dtpl}.

We proceed by  calculating  and  by  using  \eqref{eq:geom} 
\[
\begin{split}
\Delta_\Sigma (\D_t^{l-1} v \cdot \nu) &= (\Delta_\Sigma \D_t^{l-1} v) \cdot \nu + 2 B : \nabla_\tau (\D_t^{l-1} v) + ( \Delta_\Sigma  \nu) \cdot (\D_t^{l-1} v)\\
&= (\Delta_\Sigma \D_t^{l-1} v) \cdot \nu + 2 B : \nabla_\tau (\D_t^{l-1} v)  -|B|^2  (\D_t^{l-1} v  \cdot \nu) + ( \nabla_\tau H \cdot \D_t^{l-1} v). 
\end{split}
\]
Moreover, we recall that  $\nabla U = \pa_{\nu} U  \nu$ and that  \eqref{eq:D_tp-3} implies $( \nabla^2 U \nu \cdot \nu ) =  H \pa_\nu U $. Therefore we have 
\[
\begin{split}
\langle  \nabla^2 U  \nabla U , \D_t^{l-1} v \rangle &= \langle  \nabla^2 U  \nu , \nu  \rangle  \, \pa_{\nu} U \, (\D_t^{l-1} v \cdot \nu) + \langle  \nabla^2 U  \nabla U , (\D_t^{l-1} v)_\tau \rangle\\
&= H |\nabla U|^2  (\D_t^{l-1} v \cdot \nu) + ( \nabla_\tau  \frac{|\nabla U|^2}{2} \cdot \D_t^{l-1} v).
\end{split} 
\]
By combining the previous identities with \eqref{eq:form-Dtpl}  implies
\[
\begin{split}
    \D_t^l p &= - \Delta_\Sigma (\D_t^{l-1} v \cdot \nu)  - Q(t) \, \pa_{\nu} U \,  (\pa_{\nu} \pa_t^{l} U)   \\
    &-(|B|^2 - Q(t)\, H|\nabla U|^2)  (\D_t^{l-1} v  \cdot \nu) + \big(\nabla_\tau (H - Q(t)\frac{|\nabla U|^2}{2} )\cdot \D_t^{l-1} v\big)+ R_p^{l-1}.
\end{split}
\]
Hence, the claim follows from \eqref{eq:D_tp-0}.  
\end{proof}

\section{Estimation of the error terms}

In the previous section we  introduced four error  terms $R_{\Div}^l, R_{bulk}^l, R_U^l$ and $R_p^l$, which will appear later in the proof of the Main Theorem. These are nonlinear and   characterized  by their order $l \geq 0$ and their precise forms can be found in \eqref{eq:def-R-div}, \eqref{def:error-bulk} \eqref{def:R-U} and \eqref{eq:R-p-in-3} respectively. The first two terms  $R_{\Div}^l$  and $R_{bulk}^l$ are defined in the fluid domain and appear already in the case when the shape of the drop does not change. The term $R_U^l$  is due to the nonlinearity of the capacitary term. The term $R_p^l$ is due to the nonlinear behavior of the pressure on the moving boundary and it is by far the most difficult to treat. 

In this section our goal is to estimate these error terms by the energy quantity of order $l \in\ \N$ which we define as     
\begin{equation}
\label{def:E_l}
E_l(t) :=  \sum_{k=0}^{l}  \|\D_t^{l+1-k}v\|_{H^{\frac32 k}(\Omega_t)}^2 + \|v\|_{H^{\lfloor \frac32(1+1)\rfloor}(\Omega_t)}^2+ 
 \|\D_t^l v \cdot \nu \|_{H^1(\Sigma_t)}^2+1.
\end{equation}
The most difficult is to estimate the lowest order terms $R_{div}^1, \dots$, i.e., the case $l=1$, and we treat it separately. The difficulty of the case $l=1$ makes the arguments in this section long and cumbersome.

As we explained in the introduction, the proof of the Main Theorem is by induction argument, where we assume that we have the bound $E_{l-1}(t) \leq C$ and then use this to bound  $E_l(t)$. We begin this section by proving that the bound $E_{l-1}(t) \leq C$ for $l \geq 2$ implies
 \[
 \|B\|_{H^{\frac32l -1}(\St)} \leq M(C).
 \]
This will guarantee that every step improves the regularity of the flow. Perhaps the most challenging part is to start the argument and we show in Section 6 that the a priori bounds \eqref{eq:apriori_est}  imply the following estimate on the pressure  
\[
\|p\|_{H^1(\Omega_t)} \leq C 
\]
for all $t \leq T$. We will show that this implies the following curvature bound 
\[
\|B\|_{H^{\frac12}(\St)} \leq C,
\]
which in particular implies the bound $\|B\|_{L^{4}(\St)} \leq C$.

We notice that the above curvature bounds ensure that $\St$ satisfies the condition \eqref{eq:notationhp} for  $m = \lfloor \frac32l\rfloor +1$ when $l \geq 2$ and $m=2$ for $l =1$. This means that  the results from Section 2 such as Proposition \ref{prop:extension},  Proposition \ref{prop:sobolev},  Corollary \ref{coro:interpolation}, Proposition \ref{prop:kato-ponce}, Proposition \ref{prop:laplace-bound} and Proposition \ref{prop:meancrv-bound}  hold for all $k \leq m$. We take this for granted in the calculations throughout this section without further mention in order to make the presentation less heavy.  

We begin by estimating the capacitary potential $U$ by the pressure via the identity \eqref{eq:D_tp-0}. We note that in the next lemma the a priori  $C^{1,\alpha}$-bound for the boundary  $\St = \pa \Omega_t$ is crucial. 

\begin{lemma}
\label{lem:reg-capa}
Let $l\geq 1$ and assume that $\St$ is uniformly $C^{1,\alpha}(\Gamma)$-regular and satisfies the condition $\|B\|_{L^4(\St)} \leq M$ when $l = 1$ and $\|B\|_{H^{\frac32 l- 1}(\Sigma_t)} \leq M$ when $l \geq 2$. Let  $U$ be the  capacitary potential defined in \eqref{eq:def-capapot}. There exists a constant $C$, depending on $M, l$ and  on the $C^{1,\alpha}$-norm of the heightfunction, such that  
\[
\|\nabla^{1+k} U \|_{H^{\frac12}(\St)} \leq C (1+ \|p \|_{H^{k}(\St)} ) \qquad \text{on } \, \St
\]
for all integers $k \leq \frac32l +\frac12$.
\end{lemma}

\begin{proof}
Let us note that the assumptions on the curvature imply that $\St$ satisfies the condition \eqref{eq:notationhp} for $m =  \lfloor \frac32l\rfloor+1$.  In particular, the condition $k \leq \frac32l +\frac12$ implies $k \leq m$. 

Let us prove the claim by induction over $k$ and consider first the case $k = 0$. This follows immediately from  Theorem \ref{teo:reg-capa} and from $\|B\|_{L^4(\St)} \leq C$ as 
\[
\|\nabla U\|_{H^{\frac12}(\St)} \leq C(1+ \|B\|_{L^2(\St)}) \leq C. 
\]
Let us then fix $k$ and assume that the claim holds for $k-1$. Since $U$ is constant on $\Sigma_t$, Theorem \ref{teo:reg-capa}  implies
\[
\|\nabla^{1+k} U\|_{H^{\frac12}(\St)} \leq C(1 + \|B\|_{H^k(\St)}).
\]
  By Proposition \ref{prop:meancrv-bound} we have
\[
\|B\|_{H^k(\St)} \leq C(1+ \|H\|_{H^k(\St)})
\leq C(1+\|p\|_{H^k(\St)}+\||\nabla U|^2\|_{H^k(\St)}).
\]
Proposition \ref{prop:kato-ponce} yields
\[
\||\nabla U|^2\|_{H^k(\St)} \leq C\|\nabla U\|_{L^\infty(\St)} \|\nabla U\|_{H^k(\St)}.
\]
Since $\Omega_t$ is uniformly $C^{1,\alpha}$-regular we have  by Schauder estimates $\| U\|_{C^{1,\alpha}(\St)} \leq C$. By combining the previous inequalities we obtain
\begin{equation}
    \label{eq:reg-capa-0}
\|\nabla^{1+k} U\|_{H^{\frac12}(\St)} \leq C(1 + \|p\|_{H^k(\St)} + \|\nabla U\|_{H^k(\St)})
\end{equation}

We claim next that under the assumptions of the lemma,  it holds for every smooth function $u: \Omega_t \to \R$ and for all $k \leq m$  
\begin{equation}
    \label{eq:reg-capa-1}
\|\nabla u\|_{H^k(\Sigma_t)} \leq C(\|u\|_{L^2(\St)} + \|\nabla^{1+k} u\|_{L^{2}(\St)}). 
\end{equation}
Indeed, for $k =1$ we have $\|\nabla u\|_{H^1(\Sigma_t)} \leq \|u\|_{L^2(\St)} +  \|\nabla^2 u\|_{L^2(\Sigma_t)}$, while  for $k=2$ the assumption $\|B\|_{L^4} \leq C$ and the Sobolev embedding  imply
\[
\begin{split}
\|\nabla u\|_{H^2(\Sigma_t)} &\leq  \|u\|_{L^2(\St)}+  \|\nabla^3 u\|_{L^2(\Sigma_t)} + \|B\star \nabla^2 u \|_{L^2(\Sigma_t)}\\
&\leq   \|u\|_{L^2(\St)} + \|\nabla^3 u\|_{L^2(\Sigma_t)} + \|B\|_{L^4(\Sigma_t)}  \|\nabla^2 u \|_{L^4(\Sigma_t)}\\
&\leq C(\|u\|_{L^2(\St)}+\|\nabla^3 u\|_{L^2(\Sigma_t)}).
\end{split}
\]
The case $k \geq 3$ follows from the same argument. We will take \eqref{eq:reg-capa-1} for granted from now on. 

We obtain by  \eqref{eq:reg-capa-0} and \eqref{eq:reg-capa-1} that 
\[
\|\nabla^{1+k} U\|_{H^{\frac12}(\St)} \leq C(1 + \|p\|_{H^k(\St)} + \|\nabla^{1+k} U\|_{L^{2}(\St)}).
\]
We deduce by Lemma \ref{lem:divcurl}, by interpolation and by the induction assumption (that the claim holds for $k-1$)
\[
\begin{split}
\|\nabla^{1+k} U\|_{L^{2}(\St)} &\leq C(1+\|\nabla^{k} U\|_{H^{1}(\St)}) \leq \e  \|\nabla^{k} U\|_{H^{\frac32}(\St)} +C_\e (1+\|\nabla^{k} U\|_{L^{2}(\St)})\\
&\leq \e  \|\nabla^{1+k} U\|_{H^{\frac12}(\St)} +  C_\e(1+\|p\|_{H^{k-1}(\St)}).
\end{split}
\]
 Thus by choosing $\e$ small enough we obtain  the claim.
\end{proof}

From  Lemma \ref{lem:reg-capa} we deduce that an estimate on the pressure implies bound on the curvature. The statement follows from the proof of Lemma \ref{lem:reg-capa}  and we leave the proof for the reader. 

 \begin{lemma}
\label{lem:press-curv}
Let $l \geq 1$ and assume that $\St$ is uniformly $C^{1,\alpha}(\Gamma)$-regular and for $l=1$ it holds  $\|p\|_{H^1(\Omega_t)} \leq M$ and for $l \geq 2 $ it holds  $E_{l-1}(t) \leq M$. In the case $l=1$ we have  
\[
\|B\|_{H^{\frac12}(\St)} \leq C
\]
and $\|B\|_{L^{4}(\St)} \leq C$. In the case  $l \geq 2$   we have
\[
\|B\|_{H^{\frac32 l- 1}(\Sigma_t)} \leq C.
\]
Moreover for $l \geq 1$ we have 
\[
\|B \|_{H^{k}(\Sigma_t)} \leq M(1+ \|p \|_{H^{k}(\Sigma_t)} )
\]
for integers  $k \leq \frac32l + \frac12$. The constants depend on $M, l$ and on the $C^{1,\alpha}$-norm of the heightfunction.
\end{lemma}

From now on we will assume that, in addition to \eqref{eq:apriori_est}, we have  for $l=1$  the estimate $\|p\|_{H^1(\Omega_t)} \leq C$  and for $l \geq 2$  $E_{l-1}(t)\leq C$. By  Lemma \ref{lem:press-curv} these imply curvature bounds that we mentioned at the beginning of the section.  

We begin to estimate the error terms and we begin with $R_{\Div}^l$ defined in \eqref{eq:def-R-div}.
\begin{lemma}
\label{lem:R-div-est}
Consider $R_{\Div}^l$ defined in \eqref{eq:def-R-div}. Assume that  \eqref{eq:apriori_est} holds and  $\|p\|_{H^1(\Omega_t)} \leq M$. Then we have 
\[
\|R_{\Div}^1\|_{H^\frac12(\Omega_t)}^2 \leq C(1+ \|p\|_{H^2(\Omega_t)}^2)E_1(t)
\]
for $C= C(M)$. \\
Let  $l \geq 2$ and assume  also that $E_{l-1}(t)\leq M$. Then there exists $C = C(M,l)$, such that 
\beq \label{eq:divestimate}
\| R_{\Div}^l\|_{H^{\frac12}(\Omega_t)}^2  \leq C E_l(t)
\eeq
and for integers $1 \leq k \leq l$ and every $\e>0$ it holds
\beq \label{eq:divestimatehigh}
\|R_{\Div}^{l-k}\|_{H^{\frac32 k -1}(\Omega_t)}^2  \leq \e E_{l}(t) + C_\e
\eeq
for some $C_\e = C_\e(M,l,\e)$.
\end{lemma}

\begin{proof}
For $l=1$ we have by  the definition of $R_{\Div}^1 $ \eqref{eq:def-R-div} that 
\[
R_{\Div}^1 = a(\nabla v) \star \nabla \D_t v,
\]
where $a$ is smooth. Note that in this case $E_1(t)$, defined in \eqref{def:E_l}, reads as 
\[
E_1(t) =  \|\D_t^{2} v\|_{L^{2}(\Omega_t)}^2 + \|\D_t v\|_{H^{\frac32}(\Omega_t)}^2 + \|v\|_{H^3(\Omega_t)}^2+ 
 \|\D_t v \cdot \nu \|_{H^1(\Sigma_t)}^2+1.
\]
Since $\|B\|_{L^4} \leq C$, we may extend $\nabla v$ and $\nabla \D_t v$ to  $\R^3$, denote the extensions $F_v$ and $G_{v_t} $ respectively, such that  the extensions satisfy  
\[
\|F_v\|_{L^\infty(\R^3)} \leq C \|\nabla v\|_{L^\infty(\Omega_t)} ,\qquad  \|F_v\|_{H^2(\R^3)} \leq C \|\nabla v\|_{H^2(\Omega_t)} 
\]
and
\[
\|G_{v_t}\|_{L^2(\R^3)} \leq C \|\nabla \D_t v\|_{L^2(\Omega_t)} , \qquad \|G_{v_t}\|_{H^{\frac12}(\R^3)} \leq C \|\D_t v\|_{H^{\frac32}(\Omega_t)}. 
\]
Moreover, since $\Omega_t$ is bounded we may assume that $F_{v}, G_{v_t} \in C_0^\infty(B_R)$. 

We use the Kato-Ponce  inequality  \eqref{eq:kato-ponce} in $\R^3$ with $p_1 = 2, q_1 = \infty$,  $p_2 = \frac{12}{5}$ and $q_2= 12$ to deduce
\[
\begin{split}
\|R_{\Div}^1\|_{H^{\frac12}(\Omega_t)} &\leq C \|\nabla \D_t v\star a(\nabla v)\|_{H^{\frac12}(\Omega_t)} \leq C\| G_{v_t} \star a(F_v)\|_{H^{\frac12}(\R^3)} \\
&\leq C \|G_{v_t}\|_{H^{\frac12}(\R^3)} \|F_v\|_{L^\infty(\R^3)} + C \|G_{v_t}\|_{L^{\frac{12}{5}}(\R^3)}\|F_v\|_{W^{\frac12, 12}(\R^3)}.
\end{split}
\]
Since $\|F_v\|_{L^\infty} \leq C$ we have 
\[
\|G_{v_t}\|_{H^{\frac12}(\R^3)} \|F_v\|_{L^\infty(\R^3)} \leq C \|\D_t v\|_{H^{\frac32}(\Omega_t)} \leq C E_1(t)^{\frac12}.
\]
We have by using the Sobolev embedding $\|u\|_{L^p(B_R)} \leq C \|u\|_{H^s(B_R)} = C \|u\|_{W^{s,2}(B_R)}$, for $p = \frac{6}{3-2s}$ and $s=\frac14$, and by the general Gagliardo-Nirenberg inequality \eqref{eq:BM} that  
\[
\|G_{v_t}\|_{L^{\frac{12}{5}}(\R^3)} \leq C\|G_{v_t}\|_{H^{\frac{1}{4}}(\R^3)}  \leq C \|G_{v_t}\|_{H^{\frac12}(\R^3)}^{\frac12}\|G_{v_t}\|_{L^{2}(\R^3)}^{\frac12}.
\]
By   \eqref{eq:BM} we also have 
\[
\|F_{v}\|_{W^{\frac12, 12}(\R^3)} \leq C \|F_{v}\|_{W^{1,6}(\R^3)}^{\frac12}\|F_{v}\|_{L^{\infty}(\R^3)}^{\frac12} \leq C \|F_{v}\|_{H^{2}(\R^3)}^{\frac12}\|F_{v}\|_{L^{\infty}(\R^3)}^{\frac12}.
\]
Therefore by $\|F_{v}\|_{L^{\infty}(\R^3)} \leq C$, $\|F_{v}\|_{H^{2}(\R^3)} \leq C\|\nabla v\|_{H^2(\Omega_t)}\leq C\|v\|_{H^3(\Omega_t)}$ and 
\[
\|G_{v_t}\|_{L^{2}(\R^3)} \leq C \|\nabla \D_t v \|_{L^2(\Omega_t)} = C \|p\|_{H^2(\Omega_t)}
\]
we have 
\[
\|G_{v_t}\|_{L^{\frac{12}{5}}(\R^3)}\|F_{v}\|_{W^{\frac12, 12}(\R^3)} \leq C \|p\|_{H^2(\Omega_t)}^{\frac12} \|\D_t v\|_{H^{\frac32}(\Omega_t)}^{\frac12}  \|F_{v}\|_{H^{2}(\R^3)}^{\frac12} \leq C \|p\|_{H^2(\Omega_t)}^{\frac12} E_1(t)^{\frac12}.
\]
This implies the first inequality.

Let  $l\geq2$.  In order to estimate the product \eqref{eq:def-R-div} we use Proposition  \ref{prop:kato-ponce} to deduce
 \[
 \|R^l_{\Div}\|_{H^\frac12(\Omega_t)}
 \leq \sum_{|\beta|\leq l}\|\nabla \D^{\beta_1}_t v\|_{H^\frac12(\Omega_t)} \prod_{i=2}^l\|\nabla\D_t^{\beta_i} v\|_{L^\infty(\Omega_t)},
\]
where we use the convention that $\beta_1\geq \beta_2 \geq \dots \geq \beta_l $. By Recalling the definition of $E_{l-1}(t)$ in \eqref{def:E_l}, by the assumption $E_{l-1}(t)\leq C$ and by Sobolev embedding it holds for $\beta_i \leq l-2$
\[
\|\nabla\D^{\beta_i}_t v\|_{L^{\infty}(\Omega_t)}^2 \leq C \|\D_t^{\beta_i}v\|_{H^3(\Omega_t)}^2 \leq \sum_{k=0}^{l-1} \|\D_t^{l-k} v\|_{H^{\frac32 k}(\Omega_t)}^2 \leq C E_{l-1}(t) \leq C.
\]
For future purpose we also note that by the same argument we  have
\begin{equation}
    \label{eq:induction}
\|\nabla^{1+ \alpha}\D^{\beta}_t v\|_{L^{\infty}(\Omega_t)}^2  \leq C E_{l-1}(t) \quad\text{for } \, \alpha + \beta \leq l-2.
\end{equation}
Moreover, by the same argument it holds
\[
\|\nabla \D_t^{l-1}v\|^2_{L^\infty(\Omega_t)}\leq C E_l(t) .
\]
We also have 
\[
\|\nabla \D_t^{l-1}v\|_{H^\frac12(\Omega_t)}^2\leq C\| \D_t^{l-1}v\|_{H^\frac32(\Omega_t)}^2 \leq CE_{l-1}(t)\leq C.
\]
Recall that by the definition of $R_{\Div}^l$  above, the norm of the index is $|\beta| \leq l$. Therefore since $l \geq 2$ it holds $\beta_i \leq l-1$ for $i \geq 2$ and $\beta_i \leq l-2$ for $i \geq 3$. 
Thus we conclude by the above estimates that  
\[
 \|R^l_{\Div}\|_{H^\frac12(\Omega_t)}
 \leq C(1+ \|\nabla \D^l_t v\|_{H^\frac12(\Omega_t)} +\|\nabla \D_t^{l-1}v\|_{H^\frac12(\Omega_t)} \|\D_t^{l-1} v\|_{L^\infty(\Omega_t)}) \leq C E_{l}(t)^{\frac12},
\]
which implies  \eqref{eq:divestimate}.
 
The proof of \eqref{eq:divestimatehigh} follows from similar argument  and we merely sketch it. For $k =l$ the statement is trivial. 
For $1 \leq k\leq l-1$ 
we recall that 
\[
 R^{l-k}_{\Div}=\sum_{|\beta|\leq l-k} a_\beta(\nabla v) \nabla \D_t^{\beta_1} v \star \cdots \star\nabla \D_t^{\beta_{l-k}}v.
\]
First, if $k = 1$ then  by applying the previous estimate for $l-1$ we obtain
\[
\|R^{l-1}_{\Div}\|_{H^{\frac12}(\Sigma_t)}^2 \leq C(1 + \|p\|_{H^2(\Omega_t)}^2) E_{l-1}(t).
\]
But now the condition $E_{l-1}(t)\leq C$ yields
\[
\|p\|_{H^2(\Omega_t)}^2 \leq \|\D_t v\|_{H^1(\Omega_t)}^2 \leq C E_{l-1}(t) \leq C.  
\]
This implies the inequality for $k = 1$. 

Assume $2 \leq k \leq l-1$. We apply  Proposition \ref{prop:kato-ponce} to bound 
\[
\|R^{l-k}_{\Div}\|_{H^{\frac32k -1}(\St)} \leq \sum_{|\beta|\leq l-k} \|\nabla \D_t^{\beta_1} v \|_{L^\infty(\St)} \cdots \|\nabla \D_t^{\beta_{l-k-1}} v \|_{L^\infty(\St)} \|\nabla \D_t^{\beta_{l-k}}v\|_{H^{\frac32k -1}(\St)}.
\]
Since $k \geq 2$ then  $\beta_i \leq  l-2$ for all $i$. Therefore by \eqref{eq:induction}  we have
$\|\nabla \D_t^{\beta_i}v\|_{L^\infty(\Omega_t)}\leq C$ for all $i$. Moreover since $\beta_i \leq l-k$ it holds by the  Trace Theorem and by  interpolation
\[
\begin{split}
\| \nabla \D_t^{\beta_i} v\|_{H^{\frac32k -1}(\St)}^2 &\leq C\|\D_t^{\beta_i} v\|_{H^{\frac32k+1}(\Omega_t)}^2 \\
&\leq \e \|\D_t^{\beta_i} v\|_{H^{\frac32(k+1)}(\Omega_t)}^2 + C_\e\|\D_t^{\beta_i} v\|_{H^{\frac32k}(\Omega_t)}^2\\
&\leq \e E_l(t) + C_\e E_{l-1}(t) \leq \e E_l(t) + C_\e. 
\end{split}
\]
Hence, we have \eqref{eq:divestimatehigh}.
\end{proof}

We proceed to bound  the   $L^2$-norm of $R_{bulk}^l$, which is defined in \eqref{def:error-bulk},  in the fluid domain $\Omega_t$. Formally  $R_{bulk}^l$ is of order $1/2$ higher than $R^{l}_{\Div}$, and therefore this bound is of the same order than the previous lemma. 

\begin{lemma}
    \label{lem:bound-R_B}
   Consider $R_{bulk}^l$ defined in   \eqref{def:error-bulk}. Assume  that  \eqref{eq:apriori_est} holds and  $\|p\|_{H^1(\Omega_t)} \leq M$. There exists $C = C(M)$ such that 
    \[ 
    \|R_{bulk}^1 \|_{L^2(\Omega_t)}^2 \leq C (1+ \|p\|_{H^2(\Omega_t)}^2)E_1(t).
    \]
    Let $l \geq 2$ and assume  also that $E_{l-1}(t)\leq M$.
    There exists $C=C (M,l)$ such that
    \[
    \|R_{bulk}^l \|_{L^2(\Omega_t)}^2 \leq C  E_l(t)
    \]
    and  for integers $1\leq  k \leq l-1$ and for $\e >0 $ it holds
    \[
    \| R_{bulk}^{l-k} \|_{H^{\frac32(k-1)}(\Omega_t)}^2 \leq \e E_l(t) + C_\e
    \]
    for some constant $C_\e = C_\e(M,l,\e)$. 
\end{lemma}

\begin{proof}
By \eqref{def:error-bulk} and the uniform bound on $\nabla v$  given by  \eqref{eq:apriori_est} we have a  pointwise  estimate
    \[
    |R_{bulk}^l| \leq C \sum_{|\alpha|\leq 1, |\beta| \leq l} |\nabla \D_t^{\beta_1} v |   \cdots |\nabla \D_t^{\beta_{l}} v | \, | \nabla^{\alpha_1} \D_t^{\alpha_2 + \beta_{l+1}} v|.
    \]
Let us first consider the case $l=1$. Then we have by the above inequality, by $\D_t v = -\nabla p$ and by ignoring the terms of the form $|\nabla v |$, as they are uniformly bounded, and obtain a pointwise bound
 \[
|R_{bulk}^1|  \leq  C \left(1+ |\D_t^{2} v| +   | \nabla \D_t v||\nabla p|\right). 
 \]
 Therefore we have by H\"older's inequality and  by the  Sobolev embedding 
 \[
 \begin{split}
 \|R_{bulk}^1 \|_{L^2(\Omega_t)}^2 &\leq C(\|\D_t^{2} v\|_{L^2(\Omega_t)}^2 + \|\nabla p  \|_{L^6(\Omega_t)}^2  \|\nabla \D_t v  \|_{L^3(\Omega_t)}^2  )\\
 &\leq C(1+ \| p\|_{H^2(\Omega_t)}^2)(1+ \|\D_t^2 v\|_{L^2(\Omega_t)}^2 + \|\D_t v\|_{H^{3/2}(\Omega_t)}^2)\\
 &\leq C(1+ \| p\|_{H^2(\Omega_t)}^2)(1+ E_1(t)).
 \end{split}
 \]
 This implies the claim for $l =1$.

Let us then treat the case $l \geq 2$. Let us assume that the first $l$ indexes are ordered as $\beta_1 \geq \beta_2 \geq \dots \geq \beta_{l}$. As before we ignore all the terms in above which are uniformly bounded by the a priori assumption and by the assumption  $E_{l-1}(t) \leq C$. Recall first that by \eqref{eq:induction} it holds  
\[
\|\nabla \D_t^{\beta_i} v\|_{L^\infty} \leq C \quad \text{when } \quad \beta_i \leq l-2.
\]
Recall that it holds $|\beta| \leq l$. Therefore, if  $\beta_{l+1} \geq l-1$  then $\beta_1 \leq 1$ and $\beta_i =0$ for $i \geq 2$. When $\beta_{l+1} = l-2$ then the only possible other none-zero indexes  are when $\beta_1 = 2$  or when  $\beta_1 = \beta_2 = 1$. Finally when $l \geq 3$ and  $\beta_{l+1} \leq l-3$ then $\nabla^{\alpha_1}\D_t^{\alpha_2+\beta_{l+1}} v$  is itself uniformly bounded and the only  nontrivial terms  are given by the indexes  $\beta_1 = l-1$ and $\beta_2 = 1$. Hence, we have a pointwise bound which we write by relabeling the indexes as
 \[
 \begin{split}
    &|R_{bulk}^l| \leq C \sum_{|\alpha|\leq 1} |\nabla^{\alpha_1} \D_t^{\alpha_2 + l} v | +C\sum_{|\alpha|\leq 1, \beta_i \leq l-1}|\nabla \D_t^{\beta_1}  v| |\nabla^{\alpha_1} \D_t^{\alpha_2 + \beta_2} v | \\
    &+  C \sum_{|\alpha| \leq 1}(|\nabla \D_t^2 v | +|\nabla \D_t v |^2 ) | \nabla^{\alpha_1} \D_t^{\alpha_2 + l-2} v| +C(|\nabla \D_t^{l}  v| + |\nabla \D_t^{l-1}  v||\nabla \D_t  v|) |\nabla\D_t v| 
    \end{split}
    \]
Then by  H\"older's inequality we deduce
\begin{equation}
    \label{eq:Rbulk11}
\begin{split}
\| R_{bulk}^l\|_{L^2(\Omega_t)}^2 &\leq C \sum_{|\alpha|\leq 1} \|\nabla^{\alpha_1} \D_t^{\alpha_2 +l}  v\|_{L^2}^2 +C \sum_{|\alpha| \leq 1, \beta_i \leq l-1}   \|\nabla \D_t^{\beta_1}  v\|_{L^3}^2  \| \nabla^{\alpha_1} \D_t^{\alpha_2 +\beta_2} v \|_{L^6}^2\\
&+ C \sum_{|\alpha| \leq 1}   (  \| \nabla \D_t^2 v\|_{L^3}^2 +  \| \nabla\D_t  v\|_{L^6}^4 ) \|\nabla^{\alpha_1}  \D_t^{\alpha_2 + l-2}   v\|_{L^6}^2 \\
&+ C(\|\nabla \D_t^l v\|_{L^3}^2  + \|\nabla \D_t^{l-1} v\|_{L^6}^2\|\nabla \D_t v\|_{L^6}^2) \| \D_t v\|_{L^6}^2. 
\end{split}
\end{equation}
We bound the  first term on RHS of \eqref{eq:Rbulk11} for $\alpha_1+ \alpha_2 \leq  1$
\[
\|\nabla^{\alpha_1} \D_t^{\alpha_2 +l}  v\|_{L^2(\Omega_t)}^2 \leq \| \D_t^{l+1} v \|_{L^2(\Omega_t)}^2 + \| \D_t^{l} v \|_{H^{3/2}(\Omega_t)}^2  \leq  E_l(t).
\]
We claim  that the  next term  with the $L^3$-norm,  $\| \nabla \D_t^{\beta_1}  v \|_{L^3}$,  is  bounded. Indeed, we use the Sobolev embedding, the induction assumption and the fact that  $\beta_1 \leq l-1$ and have
\[
 \|\nabla \D_t^{\beta_1} v \|_{L^3(\Omega_t)}^2 \leq  \|\D_t^{\beta_1} v \|_{H^{\frac32}(\Omega_t)}^2 \leq  \sum_{k=0}^l \|\D_t^{l -k} v\|_{H^{\frac32  k}(\Omega_t)}^2  \leq  E_{l-1} (t)  \leq C.
\]
We bound the coupling term  with  $\alpha_1 + \alpha_2 \leq 1$ and $\beta_2 \leq l-1$ by
\[
\| \nabla^{\alpha_1} \D_t^{\alpha_2 +\beta_2}   v\|_{L^6(\Omega_t)}^2 \leq C\| \D_t^{\alpha_2+\beta_2}   v\|_{H^{1+\alpha_1}(\Omega_t)}^2 \leq C \sum_{k=0}^l \|\D_t^{l+1 -k} v\|_{H^{\frac32 k}(\Omega_t)}^2 \leq C E_l(t).
\]

We proceed to the next row in \eqref{eq:Rbulk11} and  for $\alpha_1 + \alpha_2 \leq 1$ we have
\[
\| \nabla^{\alpha_1} \D_t^{\alpha_2 +l-2} v \|_{L^6(\Omega_t)}^2 \leq \| \D_t^{\alpha_2 +l-2} v \|_{H^{1+\alpha_1}(\Omega_t)}^2 \leq  C \sum_{k=0}^l \|\D_t^{l -k} v\|_{H^{\frac32 k}(\Omega_t)}^2 \leq C E_{l-1}(t) \leq C.
\]
We also have 
\[
\|\nabla \D_t^2  v\|_{L^3(\Omega_t)}^2 \leq   \|\D_t^{2} v \|_{H^{3/2}(\Omega_t)}^2 \leq C E_2(t) \leq C E_l(t)
\] 
since $l \geq 2$. Moreover, by  interpolation it holds 
\[
\| \nabla \D_t  v\|_{L^6}  \leq C\|\nabla \D_t v \|_{H^2}^{\frac12}  \| \nabla \D_t  v\|_{L^2}^{\frac12} \leq C E_2(t)^{\frac14} E_1(t)^{\frac14}  \leq C E_l(t)^{\frac14},
\]
when $l \geq 2$. Hence, the second row in \eqref{eq:Rbulk11} is bounded by $E_l(t)$. 

We are left with the last row in \eqref{eq:Rbulk11}. We bound the first term by 
\[
\|\nabla \D_t^l v\|_{L^3(\Omega_t)}^2 \leq \|\nabla \D_t^l v\|_{H^{\frac12}(\Omega_t)}^2\leq \|\D_t^l v\|_{H^{\frac32}(\Omega_t)}^2 \leq E_l(t)
\]
and the last by $\| \D_t v\|_{L^6(\Omega_t)}^2 \leq \| \D_t v\|_{H^1(\Omega_t)}^2 \leq C E_1(t) \leq C $.   Finally we treat the two remaining terms by the same argument. Indeed for $\beta \leq l-1$ we have by interpolation as before
\[
\| \nabla \D_t^{\beta}  v\|_{L^6(\Omega_t)}  \leq C\|\nabla \D_t^{\beta} v \|_{H^2(\Omega_t)}^{\frac12}  \| \nabla \D_t^{\beta}  v\|_{L^2(\Omega_t)}^{\frac12} \leq C E_l(t)^{\frac14} E_{l-1}(t)^{\frac14}  \leq C E_l(t)^{\frac14}.
\]
Hence, we have
\[
\| R_{bulk}^l\|_{L^2(\Omega_t)}^2 \leq CE_l(t).
\]

We are left with the last inequality.  For $k=1$ the claim follows  by applying the previous inequality with $l-1$. Let us then assume $l\geq 3$ and $2 \leq k\leq l-1$. By definition of $R_{bulk}^{l-k}$ in \eqref{def:error-bulk} it holds $|\beta| \leq l-k \leq l-2$. Therefore \eqref{eq:induction} implies
\[
\|\nabla \D_t^{\beta_i} v\|_{L^\infty(\Omega_t)}
\leq \|\D_t^{\beta_i} v\|_{H^3(\Omega_t)}\leq C
\]
for all $i$. Therefore by Proposition \ref{prop:kato-ponce} and by relabeling the indexes  we have  
\[
\begin{split}
\|R_{bulk}^{l-k}\|_{H^{\frac32(k-1)}(\Omega_t)}
\leq 
C\sum_{\substack{{|\beta|\leq l-k}\\{|\alpha|\leq 1}}}
&\|\nabla^{\alpha_1}\D_t^{\alpha_2+\beta_{2}}v\|_{H^{\frac32(k-1)}(\Omega_t)} \\
&+
\|\nabla\D_t^{\beta_{1}}v\|_{H^{\frac32(k-1)}(\Omega_t)}\|\nabla^{\alpha_1}\D_t^{\alpha_2+\beta_{2}}v\|_{L^\infty(\Omega_t)}.
\end{split}
\]
Since $\beta_2 \leq l-k$, we may estimate the first term on RHS as
\[
\|\nabla^{\alpha_1}\D_t^{\alpha_2+\beta_{2}}v\|_{H^{\frac32(k-1)}(\Omega_t)}^2  \leq C \sum_{i=0}^{l-1} \|\D_t^{l-i}v\|_{H^{\frac32i}(\Omega_t)}^2\leq CE_{l-1}(t) \leq C.
\]
We estimate the second similarly by using $\beta_1 \leq l-k$
\[
\|\nabla\D_t^{\beta_{1}}v\|_{H^{\frac32(k-1)}(\Omega_t)}^2 \leq \|\D_t^{\beta_{1}}v\|_{H^{\frac32k}(\Omega_t)}^2 \leq CE_{l-1}(t) \leq C. 
\]
Finally we bound the last term by the Sobolev embedding, by  $\beta_{2}\leq l-k$, $\alpha_1 + \alpha_2 \leq 1$ and by interpolation
\[
\begin{split}
\|\nabla^{\alpha_1}\D_t^{\alpha_2+\beta_{2}}v\|_{L^\infty(\Omega_t)}^2 &\leq C\|\D_t^{\alpha_2+\beta_{2}}v\|_{H^{2+\alpha_1}(\Omega_t)}^2 \\
&\leq \e \|\D_t^{\alpha_2+\beta_{2}}v\|_{H^{\frac32(2+\alpha_1)}(\Omega_t)}^2 + C_\e \|\D_t^{\alpha_2+\beta_{2}}v\|_{L^{2}(\Omega_t)}^2\\
&\leq \e E_l(t) + C_\e E_{l-1}(t) \leq \e E_l(t) + C_\e.
\end{split}
\]
Thus we have 
\[
\|R_{bulk}^{l-k}\|_{H^{\frac32(k-1)}(\Omega_t)}^2 \leq  \e E_l(t) + C_\e.
\]
\end{proof}

The two previous error bounds in Lemma \ref{lem:R-div-est} and Lemma \ref{lem:bound-R_B} are similar in the sense that they only involve the material derivatives of the velocity field. We proceed to the error terms which involve  the time derivatives of the capacity potential. Note that  $\pa_t^k U$ for all $k$ is a harmonic function  in $\Omega_t^c$ but not constant on $\St$. We will  use again  Theorem \ref{teo:reg-capa} together with  Lemma \ref{lem:mat-capa} which gives the formula for $\pa_t^k U$ on the boundary $\St$. 


We first prove a generic bound which will be useful when we bound the pressure. 
 \begin{lemma}
\label{lem:reg-capa1}
Let $l\geq 2$ and assume that \eqref{eq:apriori_est}   and  the condition   $E_{l-1}(t)  \leq M$ hold. Let $\alpha, \beta \geq 0$ be integers.  When  $\alpha + \beta \leq l$, it holds  
\beq \label{eq:regcapa1}
\|\nabla^{1+\alpha}\pa_t^{\beta} U \|_{H^{\frac{\alpha}{2}+ \frac12}(\Sigma_t)}^2 \leq \e E_l(t) +C_\e,
\eeq
for $C_\e = C_\e(M, l, \e)$. On the other hand, when $\alpha + \beta \leq l+1$ and $\beta \leq l$ then 
\beq \label{eq:regcapa0}
\| \nabla^{1+\alpha}\pa_t^\beta U \|_{H^{\frac12}(\Sigma_t)}^2 \leq C E_l(t)
\eeq
for  $C = C_\e(M, l)$.
\end{lemma}

\begin{proof}
Instead of  \eqref{eq:regcapa1} we prove in fact a slightly stronger result, namely 
\beq \label{eq:regcapa-2}
\begin{split}
 &\|\nabla^{1+\alpha}\pa_t^{\beta} U \|_{H^{\frac{\alpha}{2} + \frac12 }(\Sigma_t)}^2 \leq \e E_l(t) +C_\e \quad \text{when } \, \alpha \,
 \text{is  even and } \\
 &\|\nabla^{1+\alpha}\pa_t^{\beta} U \|_{H^{\frac{\alpha}{2} + 1}(\Sigma_t)}^2 \leq C E_l(t) \quad \text{when } \, \alpha \,
 \text{ is  odd}.
 \end{split}
\eeq
The inequality \eqref{eq:regcapa1} then follows from \eqref{eq:regcapa-2}  by interpolation.

We prove \eqref{eq:regcapa-2} by induction over $\beta$ and consider  first the case  $\beta = 0$. Note that then $\alpha \leq l$. Let us first consider the case when $\alpha$ is even. Then by Lemma \ref{lem:reg-capa} and  by interpolation we have
\[
\begin{split}
\| \nabla^{1+\alpha} U\|_{H^{\frac{\alpha}{2}+\frac12}(\Sigma_t)} &\leq C(1+ \|  p\|_{H^{\frac32\alpha}(\St)}) \leq \e  \|  p\|_{H^{\frac32\alpha +\frac12}(\St)} + C_\e \|p\|_{L^2(\St)}\\
&\leq \e  \|  p\|_{H^{\frac32\alpha +\frac12}(\St)} + C_\e.
\end{split}
\]
We use Lemma \ref{lem:CS-intermed},  $-\nabla p = \D_t v$, $\alpha \leq l$ and the  definition of $E_l(t)$  to estimate
\[
 \|  p\|_{H^{\frac32\alpha +\frac12}(\St)}^2  \leq C \|   p\|_{H^{\frac32\alpha +1}(\Omega_t)}^2 \leq C(1+ \| \D_t v \|_{H^{\frac{3l}{2}}(\Omega_t)}^2) \leq C E_l(t).
 \]
This implies \eqref{eq:regcapa-2} when $\alpha$ is even.  When $\alpha$ is odd we have again by  Lemma \ref{lem:reg-capa},  Lemma \ref{lem:CS-intermed}  and by the definition of $E_l$ that
\[
\| \nabla^{1+\alpha} U\|_{H^{\frac{\alpha}{2}+1}(\Sigma_t)}^2 \leq C (1+ \|p\|_{H^{\frac{3\alpha}{2} +\frac12}(\Sigma_t)}^2) \leq C (1+ \|p\|_{H^{\frac{3l}{2} +1}(\Omega_t)}^2) \leq CE_{l}(t). 
\]

Let us assume that \eqref{eq:regcapa-2} holds for $\beta \leq k-1$ for $2 \leq k \leq l$. In particular, this implies that \eqref{eq:regcapa1} holds for $\beta \leq k-1$. Let us  consider only the case when $\alpha$ is even, since the argument for $\alpha$ odd is similar.  We first observe that since $\alpha \leq l-k \leq l-1$ then by the Trace Theorem  
\beq \label{eq:regcapa-4}
\|p\|_{H^{\frac{3}{2}\alpha}(\St)}^2 \leq C(1+\|\nabla p \|_{H^{\frac{3}{2}\alpha}(\Omega_t)}^2) \leq C(1+\|\D_t v \|_{H^{\frac{3}{2}(l-1)}(\Omega_t)}^2) \leq C E_{l-1}(t) \leq C. 
\eeq
We have then  by Theorem \ref{teo:reg-capa}, Lemma \ref{lem:press-curv} and by \eqref{eq:regcapa-4} that 
\[
\begin{split}
\|\nabla^{1+\alpha } \pa_t^k U \|_{H^{ \frac{\alpha}{2}+\frac12}(\St)} &\leq C(1+ \|p\|_{H^{\frac{3}{2}\alpha}(\St)} + \|\pa_t^k U\|_{H^{ \frac{3}{2}\alpha+\frac32}(\St)})\\
&\leq C(1+ \|\pa_t^k U\|_{H^{ \frac{3}{2}(1+\alpha)}(\St)}).
\end{split}
\]
We use the expression of $\pa_t^k U$ from  Lemma \ref{lem:mat-capa} and Proposition \ref{prop:kato-ponce} to estimate the last term in above  as follows
\[
\begin{split}
\|\pa_t ^{k} U\|_{H^{\frac{3}{2}(1+\alpha)}(\St)}
\leq &\|\nabla U\|_{H^{\frac{3}{2}(1+\alpha)}(\St)} \|\D_t^{k-1}v \|_{L^\infty(\St)} \\
&+ \|\nabla U\|_{L^{\infty}(\St)}\|\D_t^{k-1}v\|_{H^{\frac{3}{2}(1+\alpha)}(\St)}+\|R^{k-2}_U\|_{H^{\frac{3}{2}(1+\alpha)}(\St)}.
\end{split}
 \]
To estimate the first term on RHS let us first assume that $k \leq l-1$. Then by \eqref{eq:induction} $\|\D_t^{k-1}v \|_{L^\infty(\St)} \leq C$ and since $\alpha \leq l-1$ we have by \eqref{eq:regcapa1} for $\beta = 0$ that 
\[
\|\nabla U\|_{H^{\frac{3}{2}(1+\alpha)}(\St)}^2  \leq \e E_l(t) + C_\e.
\]
On the other hand, when $k = l$ then $\alpha = 0$. Thus  again by  \eqref{eq:regcapa1} it holds 
\[
\|\nabla U\|_{H^{\frac{3}{2}}(\St)}^2 \leq CE_{1}(t) \leq CE_{l-1}(t) \leq C. 
\]
By the Sobolev embedding, by interpolation, by the Trace Theorem and by $\alpha + k \leq l$ we have
\[
\begin{split}
\|\D_t^{k-1}v \|_{L^\infty(\St)}^2  &\leq C \|\D_t^{k-1}v\|_{H^{\frac{3}{2}(1+\alpha)}(\St)}^2 \\
&\leq \e \|\D_t^{k-1}v\|_{H^{ \frac{3}{2}\alpha+2}(\St)}^2+ C_\e \|\D_t^{k-1}v\|_{L^2(\St)}^2\\
&\leq  C\e  \|\D_t^{k-1}v\|_{H^{ \frac{3}{2}(\alpha+2)}(\Omega_t)}^2 + C_\e E_{l-1}(t) \leq C\e E_{l}(t) + C_\e.
\end{split}
\]
Hence, we need yet to prove 
\beq 
\label{eq:regcapa-6}
\|R^{k-2}_U\|_{H^{\frac{3}{2}(1+\alpha)}(\St)}^2 \leq \e E_l(t) +C_\e ,
\eeq
where $k \leq l$.

By the definition in \eqref{def:R-U} and by the Kato-Ponce inequality (Proposition \ref{prop:kato-ponce}) we  may estimate  
\beq \label{eq:regcapa-7}
\begin{split}
    \|R^{k-2}_U\|_{H^{\frac{3}{2}(1+\alpha)}(\St)} \leq &C\sum_{\substack{|\gamma|+|\beta| \leq k-1\\|\beta|\leq k-2}}\Big( \|\D_t^{\beta_1}v \|_{L^\infty} \cdots \|\D_t^{\beta_{k-2}}v \|_{L^\infty}\|\nabla^{1+ \gamma_1} \pa_t^{\gamma_2} U\|_{H^{\frac{3}{2}(1+\alpha)}(\St)}\\
    &\,\,\,\,\,\,\,+ \|\D_t^{\beta_1}v \|_{H^{\frac{3}{2}(1+\alpha)}(\St)} \cdots \|\D_t^{\beta_{k-2}}v \|_{L^\infty}\|\nabla^{1+ \gamma_1} \pa_t^{\gamma_2} U\|_{L^{\infty}(\St)} \Big).
\end{split}
\eeq
Since $|\beta| \leq k-2 \leq l-2$, \eqref{eq:induction} implies $\|\D_t^{\beta_1}v \|_{L^\infty} \leq C$ for all $i$. Note that   $\alpha+k \leq l$ and  $\beta_1 \leq k-2$ implies $\alpha + \beta_1 \leq l-2$. Therefore we have  by the Trace Theorem and by the definition of $E_{l-1}(t)$  
\[
\|\D_t^{\beta_1}v \|_{H^{\frac{3}{2}(1+\alpha)}(\St)}^2 \leq C\|\D_t^{\beta_1}v \|_{H^{
\frac{3}{2}(\alpha +2)}(\Omega_t)}^2 \leq CE_{l-1}(t) \leq C.
\]
We bound the both capacitary terms by the Sobolev embedding and by the induction assumption, which states that \eqref{eq:regcapa1} holds for $\beta \leq k-1$. Indeed we  have by $\gamma_1 + \alpha \leq |\gamma| + \alpha \leq k-1 + \alpha \leq l-1$ and $\gamma_2 \leq k-1$ that 
\beq \label{eq:regcapa-8}
\begin{split}
    \|\nabla^{1+ \gamma_1} \pa_t^{\gamma_2} U\|_{L^{\infty}(\St)}^2 &\leq C \|\nabla^{1+ \gamma_1} \pa_t^{\gamma_2} U\|_{H^{\frac{3}{2}(1+\alpha)}(\St)}^2\\
    &\leq C(1+\|\nabla^{1+ (1+ \alpha+ \gamma_1)} \pa_t^{\gamma_2} U\|_{H^{\frac{1}{2}(1+\alpha+ \gamma_1)}(\St)}^2)\\
    &\leq \e E_l(t) + C_\e.
\end{split}
\eeq
Hence, we have \eqref{eq:regcapa-6} when $\alpha$ is even.

 Let us then prove \eqref{eq:regcapa0}. We notice that  Theorem \ref{teo:reg-capa} and \eqref{eq:regcapa1}  imply \eqref{eq:regcapa0} when $2 \leq \alpha \leq l$. On the other  Lemma \ref{lem:reg-capa} implies \eqref{eq:regcapa0}  when $\alpha = l+1$. (Note that the assumption $\alpha \leq \frac32 l$ is satisfied for all $\alpha \leq l+1$ since $l \geq 2$.) We need thus to consider the case $\alpha =1$ and $\beta = l$. For this the argument is similar than before and we only sketch it. 
 
By Theorem \ref{teo:reg-capa} and by \eqref{eq:regcapa-4} we have 
\[
\begin{split}
\|\nabla^{2} \pa_t^l U \|_{H^{\frac12}(\St)} &\leq C(1+ \|p\|_{H^{1}(\St)} + \|\pa_t^l U\|_{H^{\frac52}(\St)})\\
&\leq C(1+ \|\pa_t^l U\|_{H^{\frac52}(\St)}).
\end{split}
\]
We use the expression of $\pa_t^l U$ from  Lemma \ref{lem:mat-capa} and Proposition \ref{prop:kato-ponce} to estimate the last term in above  as follows
\[
\begin{split}
\|\pa_t ^{l} U\|_{H^{\frac52}(\St)}
&\leq \|\nabla U\|_{H^{\frac52}(\St)} \|\D_t^{l-1}v \|_{L^\infty(\St)} \\
&+ \|\nabla U\|_{L^{\infty}(\St)}\|\D_t^{l-1}v\|_{H^{\frac52}(\St)}+\|R^{l-2}_U\|_{H^{\frac52}(\St)}.
\end{split}
 \]
 By   Lemma \ref{lem:reg-capa}  and Lemma \ref{lem:CS-intermed} we have
 \[
 \begin{split}
 \|\nabla U\|_{H^{\frac52}(\St)}^2 \leq  C(1+ \|p\|_{H^2(\St)}^2) \leq C(1+ \|\nabla p\|_{H^{\frac32}(\Omega_t)}^2) \leq CE_{1}(t) \leq C.
 \end{split}
 \]
 Recall also that $\|\nabla U\|_{L^\infty} \leq C$. Sobolev embedding and Trace Theorem  yield
 \[
  \|\D_t^{l-1}v\|_{L^{\infty}(\St)}^2\leq  C\|\D_t^{l-1}v\|_{H^{2}(\St)}^2 \leq C\|\D_t^{l-1}v\|_{H^{3}(\Omega_t)}^2 \leq C E_l(t).
 \]
Therefore we need yet to estimate $ \|R^{l-2}_U\|_{H^{\frac52}(\St)} $. 

Arguing as in \eqref{eq:regcapa-7} we obtain
 \[
 \begin{split}
    \|R^{l-2}_U\|_{H^{\frac52}(\St)} \leq &\sum_{\substack{|\gamma|+|\beta| \leq l-1\\|\beta|\leq l-2}}\Big( \|\D_t^{\beta_1}v \|_{L^\infty} \cdots \|\D_t^{\beta_{l-2}}v \|_{L^\infty}\|\nabla^{1+ \gamma_1} \pa_t^{\gamma_2} U\|_{H^{\frac52}(\St)}\\
    &\,\,\,\,\,\,\,+ \|\D_t^{\beta_1}v \|_{H^{\frac52}(\St)} \cdots \|\D_t^{\beta_{l-2}}v \|_{L^\infty}\|\nabla^{1+ \gamma_1} \pa_t^{\gamma_2} U\|_{L^{\infty}(\St)} \Big).
\end{split}
\]
Arguing as before we deduce $\|\D_t^{\beta_i}v \|_{L^\infty} \leq C$  for all $i$ and \[
\|\D_t^{\beta_1}v \|_{H^{\frac52}(\St)}^2 \leq \|\D_t^{\beta_1}v \|_{H^{3}(\Omega_t)}^2 \leq E_{l-1}(t) \leq C.
\]
Finally we use the fact that \eqref{eq:regcapa0} holds for $\beta \leq l-1$ and $\gamma_1 + \gamma_2 \leq l-1$  to conclude 
\[
\begin{split}
\|\nabla^{1+ \gamma_1} \pa_t^{\gamma_2} U\|_{L^{\infty}(\St)}^2 &\leq C \|\nabla^{1+ \gamma_1} \pa_t^{\gamma_2} U\|_{H^{\frac52}(\St)}^2\\
&\leq C(1+ \|\nabla^{1+ (2+\gamma_1)} \pa_t^{\gamma_2} U\|_{H^{\frac12}(\St)}^2) \leq CE_l(t).
\end{split}
\]
This concludes the proof. 
\end{proof}

We need also the following bound on the capacitary potential and for the error term $R_U^l$, defined in \eqref{def:R-U}, associated with it. In the first statement of the following lemma we need to relax the usual assumption on the quantity  \eqref{def:apriori_est}  being bounded to assume that the set $\Omega_t$ is uniformly $C^{1,\alpha}(\Gamma)$-regular and that the velocity satisfies $\|v\|_{W^{1,4}(\St)} \leq C$. The point is that we need the following estimate when we do not have the Lipschitz bound on $v$. This does not complicate the proof and  will be useful later.   
\begin{lemma}
    \label{lem:estimate-R-U}
    Consider $R_U^l$ defined in \eqref{def:R-U}. Assume that  $\St$ is uniformly $C^{1,\alpha}(\Gamma)$-regular and satisfies   
    \[
    \|p\|_{H^1(\Omega_t)} + \|v\|_{W^{1,4}(\St)} \leq M.
    \]
    There exists $C$  such that
    \begin{equation}
        \label{eq:est-R-U-1}
        \|\nabla \pa_t^2 U\|_{L^2(\St)}^2 + \|R_U^1\|_{L^2(\St)}^2 \leq C(1+ \|p\|_{H^2(\Omega_t)}^2)E_1(t). 
    \end{equation}
    Let $l \geq 2$ and assume that $E_{l-1}(t) \leq M$. Then it holds 
  \begin{equation}
        \label{eq:est-R-U-2}
         \|\nabla \pa_t^{l+1} U\|_{L^2(\St)}^2 + \|R_U^l\|_{L^2(\St)}^2 \leq C E_l(t).
    \end{equation}
    The constants  depend on $M, l$ and on the $C^{1,\alpha}$-norm of the heightfunction.
    \end{lemma}
\begin{proof}
This time we only  prove \eqref{eq:est-R-U-1} since \eqref{eq:est-R-U-2} follows from similar argument. Note also that the $C^{1,\alpha}$-bound on $\St$ implies $C^{1,\alpha}$-bound on $U$. Recall also that  $\|p\|_{H^1(\Omega_t)} \leq C$  implies $\|B\|_{L^{4}(\St)} \leq \|B\|_{H^{\frac12}(\St)} \leq C$ by  Lemma \ref{lem:press-curv}.  We begin by noticing that, since $\pa_t^{2} U$ is harmonic, it holds by Lemma \ref{lem:divcurl}  
\[
\| \nabla \pa_t^{2} U \|_{L^2(\Sigma_t)}\leq C(1+ \|  \pa_t^{2} U \|_{H^1(\Sigma_t)}).
\]
Then by Lemma \ref{lem:mat-capa}  we have  
\[
\| \pa_t^{2} U \|_{H^1(\Sigma_t)}\leq  C(1+ \| \pa_\nu U  (\D_t v \cdot \nu ) \|_{H^1(\Sigma_t)} + \|R_U^0\|_{H^1(\Sigma_t)}),
\]
where 
\[
R_U^0 = \sum_{|\alpha|\leq 1} a_{\alpha}(v) \nabla^{1+\alpha_1} \pa_t^{\alpha_2} U.
\]
Since $\|\nabla U\|_{L^{\infty}(\Sigma_t)} \leq C$,  we may  estimate by  Proposition \ref{prop:kato-ponce}, $\|B\|_{L^4} \leq C$ and by the Sobolev embedding 
\begin{equation}
    \label{eq:est-R-U-4}
\begin{split}
&\|  \pa_t^{2} U \|_{H^1(\Sigma_t)}\\
&\leq C+ C \| \D_t v \cdot \nu  \|_{H^1(\Sigma_t)} + C(\| \nabla^2 U\|_{L^4(\Sigma_t)}+ \|B\|_{L^4(\Sigma_t)}) \|  \D_t v \cdot \nu  \|_{L^4(\Sigma_t)} + \|R_U^0\|_{H^1(\Sigma_t)}\\
&\leq C+ C(1+ \| \nabla^2 U\|_{L^4(\Sigma_t)}) \| \D_t v \cdot \nu  \|_{H^1(\Sigma_t)} + \|R_U^0\|_{H^1(\Sigma_t)}.
\end{split}
\end{equation}
We have by the Sobolev embedding and by  Lemma \ref{lem:reg-capa} 
\beq
\label{eq:est-R-U-5}
\| \nabla^2 U\|_{L^4(\Sigma_t)} \leq \| \nabla^2 U\|_{H^{1/2}(\Sigma_t)} \leq C(1+ \|p\|_{H^1(\Sigma_t)}).
\eeq
Since $\| \D_t v \cdot \nu  \|_{H^1(\Sigma_t)}^2 \leq E_1(t)$, we need yet to show that $\|R_U^0\|_{H^1(\Sigma_t)}^2 \leq CE_1(t)$. 

Since $\|v\|_{L^\infty(\St)} \leq C \|v\|_{W^{1,4}(\St)} \leq C$  we have by the Sobolev embedding 
\[
\begin{split}
\|R_U^0\|_{H^1(\Sigma_t)} &\leq C \sum_{|\alpha|\leq 1} \|v\|_{L^\infty} \|\nabla^{1+\alpha_1} \pa_t^{\alpha_2} U\|_{H^1(\Sigma_t)} + C  \sum_{|\alpha|\leq 1} \|v\|_{W^{1,4}} \|\nabla^{1+\alpha_1} \pa_t^{\alpha_2} U\|_{L^4(\Sigma_t)}\\
&\leq C(1 +  \|\nabla^{2}  U\|_{H^1(\Sigma_t)}  + \|\nabla \pa_t U\|_{H^1(\Sigma_t)}).
\end{split}
\]
Lemma \ref{lem:reg-capa} and Lemma \ref{lem:CS-intermed} yield 
\beq
\label{eq:est-R-U-6}
\begin{split}
\|\nabla^3  U\|_{H^{\frac12}(\Sigma_t)}^2 \leq C(1+ \|p\|_{H^2(\St)}^2) \leq C(1+ \|\nabla p\|_{H^{\frac32}(\Omega_t)}^2) \leq CE_{1}(t).
\end{split}
\eeq

We bound $\|\nabla^2 \pa_t U\|_{H^{\frac12}(\Sigma_t)}$ with a similar  argument  and thus we only sketch it. First, we recall that  it holds $\pa_t U = - \nabla U \cdot v$ on $\St$. We use Theorem \ref{teo:reg-capa}, Lemma \ref{lem:press-curv}, Proposition \ref{prop:kato-ponce} and \eqref{eq:est-R-U-6} to deduce
\beq \label{eq:est-R-U-7}
\begin{split}
  \|\nabla^2 \pa_t U\|_{H^{\frac12}(\Sigma_t)}^2 &\leq C(1 + \|p\|_{H^1(\St)}^2 + \| \nabla U \cdot v\|_{H^{\frac52}(\Sigma_t)}^2)\\
  &\leq C( 1+ \|p\|_{H^1(\St)}^2 +  \|\nabla^3  U\|_{H^{\frac12}(\Sigma_t)}^2 + \|v\|_{H^3(\Omega_t)}^2)\\
  &\leq CE_1(t).
\end{split}
\eeq
We thus deduce by \eqref{eq:est-R-U-4}, \eqref{eq:est-R-U-5}, \eqref{eq:est-R-U-6} and \eqref{eq:est-R-U-7} that 
 \beq
\label{eq:est-R-U-8}
\| \nabla \pa_t^{2} U \|_{L^2(\Sigma_t)}^2 \leq C(1+ \|p\|_{H^2(\Omega_t)}^2)E_1(t).
\eeq

We are left with 
\beq
\label{eq:est-R-U-9}
\|R_U^1\|_{L^2(\St)}^2 \leq  C(1+ \|p\|_{H^2(\Omega_t)}^2)E_1(t).
\eeq
To this aim  we recall the definition of $R_U^1$ in \eqref{def:R-U}
\[
R_U^1 = \sum_{|\alpha|+\beta \leq 2, \,\beta \leq 1} a_{\alpha, \beta} \D_t^\beta v \star \nabla^{1+ \alpha_1} \pa_t^{\alpha_2} U.
\]
We use H\"older's inequality as 
\[
\|R_U^1\|_{L^2(\St)} \leq C \big(\sum_{|\alpha|\leq 2} \|\nabla^{1+ \alpha_1} \pa_t^{\alpha_2} U\|_{L^2(\St)} + \|\D_t v\|_{L^4(\St)}  \sum_{|\alpha|\leq 1} \|\nabla^{1+ \alpha_1} \pa_t^{\alpha_2} U\|_{L^4(\St)}\big).
\]
We have by \eqref{eq:est-R-U-6}, \eqref{eq:est-R-U-7} and  \eqref{eq:est-R-U-8}
\[
\sum_{|\alpha|\leq 2} \|\nabla^{1+ \alpha_1} \pa_t^{\alpha_2} U\|_{L^2(\St)}^2 \leq  C(1+ \|p\|_{H^2(\Omega_t)}^2)E_1(t).
\]
By the Sobolev embedding and $-\nabla p = \D_t v$ it holds
\[
\|\D_t v\|_{L^4(\St)} \leq \|\D_t v\|_{H^1(\Omega_t)} \leq C \|p\|_{H^2(\Omega_t)}.
\]
Moreover \eqref{eq:est-R-U-6} and \eqref{eq:est-R-U-7} imply for $\alpha_1 + \alpha_2 \leq 1$
\[
\|\nabla^{1+ \alpha_1} \pa_t^{\alpha_2} U\|_{L^4(\St)}^2 \leq \|\nabla^{1+ \alpha_1} \pa_t^{\alpha_2} U\|_{H^{\frac12}(\St)}^2 \leq CE_1(t).
\]
Hence we have \eqref{eq:est-R-U-9}.
\end{proof}

Finally we need to bound the error term $R_p^l$, defined in \eqref{eq:R-p-in-3}, which is associated with the pressure. This term is  the most difficult to treat and it turns out that the lower order   case $l=1$ is  the most challenging to deal with. Therefore we state it as an own lemma. 
\begin{lemma}
    \label{lem:estimate-R-p-1}
    Let $R_p^l$ be as defined in \eqref{eq:R-p-in-3}. Assume that \eqref{eq:apriori_est} holds and  $\|p\|_{H^1(\Omega_t)}\leq M$. Then  it holds
    \[
    \|R_p^1\|_{H^{\frac12}(\Sigma_t)}^2 \leq C(1+ \|p\|_{H^2(\Omega_t)}^2) E_1(t).
    \]
    for some constant $C=C(M)$.
    \end{lemma}
    
    \begin{proof}
    Let us begin by recalling that by the definition of $R_p^1$ in \eqref{eq:R-p-in-3}, \eqref{def:R_I} \eqref{def:R_II-2} we may write 
 \begin{equation}
     \label{eq:Rp1-again}
     \begin{split}
    R_p^1 = &- |\nabla_\tau p|^2 - (|B|^2 - Q(t) H |\nabla U|^2)  \,  \pa_\nu p  +  a_1(\nu, \nabla v) \star B + a_2(\nu, \nabla v) \star \nabla^2 v \\
    &+ \sum_{|\alpha|+ |\gamma| \leq 2,\, \gamma_i \leq 1} a_{\alpha,\gamma, Q}(v)  \nabla^{1+\alpha_1} \pa_t^{\gamma_1} U \star \nabla^{1+\alpha_2} \pa_t^{\gamma_2} U.
    \end{split}
    \end{equation}
We first recall it holds $\|U\|_{C^{1,\alpha}(\Sigma_t)} \leq C$.  We may bound the curvature by Sobolev embedding and interpolation as
\[
\|B\|_{C^\alpha(\Sigma_t)}  \leq C (1+\|p\|_{C^\alpha(\Sigma_t)}) \leq C(1+\|\bar \nabla p\|_{L^{\frac{11}{5}}(\Sigma_t)}) \leq C(1+ \| p\|_{H^{2}(\Sigma_t)}^{\theta} \| p\|_{L^{4}(\Sigma_t)}^{1-\theta}),
\]
for $\theta < \frac49$. Recall that  $ \|p\|_{H^1(\Omega_t)} \leq C$ implies  $\| p\|_{L^{4}(\Sigma_t)}, \|B\|_{L^4} \leq C$. By the Sobolev embedding and by Lemma \ref{lem:CS-intermed} we have
\begin{equation}
     \label{eq:pH2-bound}
     \begin{split}
\| p\|_{H^{2}(\Sigma_t)}^2 \leq C\| \nabla p\|_{H^{1}(\Sigma_t)}^2 \leq  C \|\nabla p\|_{H^{\frac32}(\Omega_t)}^2 \leq C E_{1}(t).
     \end{split}
\end{equation}
Therefore we obtain 
\begin{equation}
     \label{eq:B-bound1} 
\|B\|_{C^\alpha(\Sigma_t)}^2  \leq CE_1(t)^{\theta} \qquad \text{for } \, \theta < \frac49.
\end{equation}
We may also bound the curvature simply as
\begin{equation}
     \label{eq:B-bound2}
\|B\|_{L^\infty(\Sigma_t)}  \leq \|B\|_{C^\alpha(\Sigma_t)} \leq C (1+\|p\|_{C^\alpha(\Sigma_t)}) \leq C(1+ \|p\|_{H^2(\Omega_t)}).
\end{equation}

Let us bound the first term in \eqref{eq:Rp1-again} which is of the highest order.  We observe that by interpolation it holds 
\[
 \|\nabla p\|_{L^{8}(\St)} \leq C\|\nabla p\|_{H^{1}(\St)}^{\frac12} \|\nabla p\|_{L^{4}(\St)}^{\frac12}.
\]
By using this,  the Sobolev embedding and \eqref{eq:kato-weak} we have 
\[
\begin{split}
\||\nabla_\tau p|^2 \|_{H^{1/2}(\Sigma_t)} &\leq C\| |\nabla p|^2 \|_{H^{\frac12}(\St)} +  C \| \nu\|_{W^{1,4}(\St)}  \| |\nabla p|^2\|_{L^{4}(\St)} \\
&\leq C\| |\nabla p|^2 \|_{H^{1}(\Omega_t)} + C\|B\|_{L^4(\St)} \|\nabla p\|_{L^{8}(\St)}^2  \\
&\leq C\|\nabla p \|_{L^{6}(\Omega_t)} \|\nabla^2 p  \|_{L^{3}(\Omega_t)}  + C \|\nabla p\|_{L^{4}(\St)}\|\nabla p\|_{H^{1}(\St)} \\
&\leq C\|p \|_{H^{2}(\Omega_t)} \|\nabla p  \|_{H^{\frac32}(\Omega_t)}  \\
&\leq C (1+ \| p \|_{H^{2}(\Omega_t)}) E_1(t)^{\frac12}. 
\end{split}
\]
This gives bound for the first term.

In order to bound the next term in \eqref{eq:Rp1-again} we let   $\tilde \nu$ and  $\tilde B$ be the harmonic extensions  of the normal $\nu $ and of  the second fundamental form $B$ to $\Omega_t$. By maximum principle $\|\tilde B\|_{L^\infty(\Omega_t)} \leq \| B\|_{L^\infty(\Sigma_t)}$, while  by standard results from harmonic analysis \cite{Dahl} it holds 
\[
\| \tilde B\|_{W^{1,3}(\Omega_t)} \leq C\|   B\|_{W^{1,3}(\Sigma_t)} 
\]
and by \eqref{eq:poisson1-5} $\| \tilde \nu\|_{W^{1,4}(\Omega_t)} \leq C$. Then   we have 
\begin{equation}
     \label{eq:p-H2-bound2}
\|\nabla p \cdot \tilde \nu \|_{H^{1}(\Omega_t)} \leq C\|p \|_{H^{2}(\Omega_t)}  + C \| \tilde \nu\|_{W^{1,4}(\Omega_t)}  \| p\|_{W^{1,4}(\Omega_t)}\leq C\|p \|_{H^{2}(\Omega_t)}.
\end{equation}
We have  by \eqref{eq:B-bound1}, \eqref{eq:p-H2-bound2} and by the Sobolev embedding 
\[
\begin{split}
\||B|^2   \,  &\pa_\nu p \|_{H^{1/2}(\Sigma_t)}^2 \leq C(1+\|\nabla (|\tilde B|^2   \, (\nabla p \cdot \tilde \nu))\|_{L^{2}(\Omega_t)}^2\\
&\leq C(1+ \|B\|_{L^\infty(\Sigma_t)}^4\|\nabla p \cdot \tilde \nu \|_{H^1(\Omega_t)}^2 + \|B\|_{L^\infty(\Sigma_t)}^2\|\nabla  \tilde B\|_{L^3(\Omega_t)}^2 \|\nabla p\|_{L^6(\Omega_t)}^2)\\
&\leq C(1+ \|p\|_{H^2(\Omega_t)}^2) E_1(t)^{2\theta} + E_1(t)^{\theta} \|   B\|_{W^{1,3}(\Sigma_t)}^2\|p\|_{H^2(\Omega_t)}^2),
\end{split}
\]
for $\theta < \frac49$. By interpolation in Proposition \ref{prop:interpolation},  by Lemma  \ref{lem:press-curv} and \eqref{eq:pH2-bound} we have
\begin{equation}
    \label{eq:R-p-1-00}
\|   B\|_{W^{1,3}(\Sigma_t)} \leq C\| B\|_{H^2(\Sigma_t)}^{\frac59}\| B\|_{L^4(\Sigma_t)}^{\frac49} \leq C(1+ \|p\|_{H^2(\Sigma_t)}^{\frac59})\leq C E_1(t)^{\frac12 \cdot \frac59}.
\end{equation}
Therefore, since $\theta < \frac49$,  we may bound the second term as  
\[
\||B|^2   \,\pa_\nu p \|_{H^{1/2}(\Sigma_t)}^2 \leq C(1+ \|p\|_{H^2(\Omega_t)}^2)E_1(t).
\]

By the same argument we also have  
\begin{equation}
    \label{eq:R-p-1-0}
\| H |\nabla U|^2 \pa_\nu p \|_{H^{1/2}(\Sigma_t)}^2 \leq C(1+ \|p\|_{H^2(\Omega_t)}^2)E_1(t).
\end{equation}
Indeed, the same calculations as above lead to 
\[
\| H |\nabla U|^2 \pa_\nu p \|_{H^{1/2}(\Sigma_t)}^2 \leq  C(1+ \|p\|_{H^2(\Omega_t)}^2) E_1(t)^{\frac49}(\| B\|_{W^{1,3}(\Sigma_t)}^2+ \||\nabla U|^2\|_{W^{1,3}(\Sigma_t)}^2).
\]  
Recall that \eqref{eq:R-p-1-00} yields $\| B\|_{W^{1,3}(\Sigma_t)}^2\leq CE_1(t)^{\frac59}$. By   Lemma \ref{lem:reg-capa} we deduce
\[
\begin{split}
 \||\nabla U|^2\|_{H^2(\St)}  &\leq C\|\nabla U\|_{L^\infty}\|\nabla U\|_{H^2(\St)} \\
&\leq C(1+ \| \nabla^2 U\|_{H^{\frac12}(\St)})\leq C(1+ \|p\|_{H^2(\St)}) \leq CE_1(t)^{\frac12}.
\end{split}
\]
Hence, by interpolation  
\[
\||\nabla U|^2\|_{W^{1,3}(\Sigma_t)} \leq C\| |\nabla U|^2 \|_{H^2(\St)}^{\frac13} \||\nabla U|^2\|_{L^\infty(\St)}^{\frac23} \leq C E_1(t)^{\frac12 \cdot \frac13}
\]
and \eqref{eq:R-p-1-0} follows.

 The term $a_1(\nu, \nabla v) \star B $ is easy to bound and leave the details for the reader. Also the term $a_2(\nu, \nabla v) \, \nabla^2 v $ is not difficult and we merely point out that by interpolation 
 \[
 \|\nabla^2 v\|_{L^4(\Omega_t)} \leq C \|v\|_{H^3(\Omega_t)}^{1/2} \|\nabla v\|_{L^\infty(\Omega_t)}^{1/2} \leq C \|v\|_{H^3(\Omega_t)}^{1/2}. 
 \]
 Thus we have by \eqref{eq:B-bound2}
 \[
 \begin{split}
 \|a_1(\nu, \nabla v) \, \nabla^2 v\|_{H^{\frac12}(\Sigma_t)} &\leq \|a_1(\nu, \nabla v) \, \nabla^2 v\|_{H^{1}(\Omega_t)}\\ 
 &\leq C\|\nabla ^3 v\|_{H^3(\Omega_t)} + C \|\nabla^2 v\|_{L^4(\Omega_t)}^2 +C \|B\|_{C^\alpha(\Sigma_t)}  \|\nabla^2 v\|_{L^2(\Omega_t)}\\
 &\leq  C(1 + \|p \|_{H^2(\Omega_t)}) \|v \|_{H^3(\Omega_t)}.
  \end{split}
 \]

Before we  treat the last term in \eqref{eq:Rp1-again} we need to show that the coefficients $a_{\alpha,\gamma, Q}$ are bounded. To  this aim we need to show that $Q^{(1)}= Q'(t), Q^{(2)} = Q''(t)$, where $\Qt$ is defined in \eqref{def:constant-q},  are bounded since $a_{\alpha, \gamma ,Q}$ depend smoothly on them. It is clear that it is enough to show that the first and second derivative of  $\Cap (\Omega_t)$ are bounded. It is easy to see, and in fact we already calculated, that  
\[
 \frac{d}{dt} \Cap (\Omega_t) = -\frac12 \int_{\Sigma_t} |\nabla U|^2 v_n d\H^2.
\]
This is clearly bounded. We calculate further and obtain  
\beq \label{eq:Q-2-derivat}
 \frac{d^2}{dt^2} \Cap (\Omega_t) = -\frac12 \int_{\St} H_{\St}|\nabla U|^2 v_n d\H^2 -  \int_{\St} (\nabla \pa_t U \cdot \nabla U) v_n + |\nabla U|^2 \pa_t (v_n)\, d\H^2.
\eeq
The first term on RHS is clearly bounded. 
For the second term on RHS in \eqref{eq:Q-2-derivat} we note  that since $U$ is constant on $\Sigma_t$ we have $\nabla U= -|\nabla U| \nu$. Therefore 
\[
|\int_{\Sigma_t} \nabla \pa_t U \cdot \nabla U v_n d\H^2|
=|\int_{\Sigma_t}|\nabla U|  v_n \nabla \pa_t U \cdot   \nu  d\H^2| \leq \|\nabla U v_n \|_{H^\frac12(\Sigma_t)}\|\nabla \pa_t U \cdot \nu \|_{H^{-\frac12}(\Sigma)}.
\]
We note that we may use the Kato Ponce inequality (Proposition \ref{prop:kato-ponce}) and Lemma \ref{lem:reg-capa}  to deduce
\[
\|\nabla U v_n \|_{H^\frac12(\Sigma_t)} \leq \|\nabla U  \|_{H^\frac12(\Sigma_t)} \|v_n\|_{L^\infty(\Sigma_t)}+\|v_n  \|_{H^\frac12(\Sigma_t)}\|\nabla U \|_{L^\infty(\Sigma_t)} \leq C.
\]
Next we let $\tilde U_t$ the harmonic extension of $\pa_t U$  in $\Omega_t$ and note that for any $ \phi \in H^\frac12(\Sigma_t)$ it holds
\[
\begin{split}
\int_{\Sigma_t} \phi \nabla \pa_t U \cdot \nu\, d\H^2
= \int_{\Omega_t} \Div (\phi \nabla   \tilde U_t  ) \,dx
&\leq \|\nabla \phi \|_{L^2(\Omega_t)} \|\nabla \pa_t \tilde U\|_{L^2(\Omega_t)}\\
&\leq  \|\phi\|_{H^\frac12(\Sigma_t)} \|\pa_t U\|_{H^\frac12(\Sigma_t)}.
\end{split}
\]
This and  $\pa_t U = - \nabla U \cdot v$ imply
\[
\|\nabla \pa_t U \cdot \nu \|_{H^{-\frac12}(\Sigma_t)}
\leq
\|\pa_t U\|_{H^\frac12(\Sigma_t)}= \|\pa_\nu U  v_n\|_{H^\frac12(\Sigma_t)} \leq C.
\]
Let us  estimate the last term in \eqref{eq:Q-2-derivat}.  We note that 
\[
\pa_t(v_n) = \D_t v \cdot \nu + a(\nabla v)\star B_{\St}.
\]
By  \eqref{eq:pressure-1} it holds  $\Div \D_t v=- \text{Tr}((\nabla v)^2)$. Therefore we estimate 
\[
\begin{split}
 |\int_{\Sigma_t} |\nabla U|^2 \pa_t v_n\,  d\H^2|&\leq C +   \int_{\Sigma_t} |\nabla U|^2\D_t v\cdot \nu d\H^2\\
&\leq C(1 + \| \Div \D_t v\|_{L^2(\Omega_t)} + \|\nabla U\|_{H^{\frac12}(\St)} \|\D_t v\|_{L^2(\Omega_t)})
\leq C.
\end{split}
\]
Thus we have $|Q^{(2)}(t)| \leq C$  and the coefficients $a_{\alpha,\gamma, Q}$ are bounded. 

Let us treat the last term in \eqref{eq:Rp1-again}. We may assume that $\alpha_1+\gamma_1 \geq \alpha_2+\gamma_2$ and assume first that $\alpha_1 + \gamma_1=2$ (in which case $\alpha_2 = \gamma_2 = 0$). This means that either $\alpha_1 =2,  \gamma_1 = 0$ or $\alpha_1 =1,  \gamma_1 = 1$. Therefore we have by \eqref{eq:est-R-U-6} and \eqref{eq:est-R-U-7}
\begin{equation}
    \label{eq:R-p-before}
\|\nabla^{1+ \alpha_1}\pa_t^{\gamma_1} U \|_{H^{\frac12}(\St)}^2 \leq CE_1(t).
\end{equation}
We extend $v$ to the complement $\Omega_t^c$ such that it remains Lipschitz. Recall that $U$ (and $\pa_t U$) is harmonic in $\Omega_t^c$. Since $\Omega_t$ is bounded we may choose a large ball such that $\Omega_t \subset B_{R/2}$ and $\|U\|_{H^{\frac12}(\St)}\simeq \|\nabla U\|_{L^{2}(B_R \setminus \Omega_t)}$.    Then by \eqref{eq:est-R-U-5}, \eqref{eq:R-p-before} and by the Sobolev embedding it holds 
 \[
  \begin{split}
 \|a_{\alpha,\gamma}(v) &\nabla^{1+ \alpha_1}\pa_t^{\gamma_1} U \star \nabla U  \|_{H^{\frac12}(\St)} \leq C \|a_{\alpha,\gamma}(v) \nabla^{1+ \alpha_1}\pa_t^{\gamma_1} U \star \nabla U  \|_{H^{1}(B_R \setminus \Omega_t)} \\
 &\leq C(1+\|\nabla^{2+ \alpha_1}\pa_t^{\gamma_1} U \|_{L^{2}(B_R \setminus \Omega_t)} +\|\nabla^{1+ \alpha_1}\pa_t^{\gamma_1} U\|_{L^4(B_R \setminus \Omega_t)} \|\nabla^{2} U\|_{L^4(B_R \setminus \Omega_t)}) \\
 &\leq C(1+ \|\nabla^2 U\|_{H^{\frac12}(\St)})(1+ \|\nabla^{1+ \alpha_1}\pa_t^{\gamma_1} U \|_{H^{\frac12}(\St)}) \\
 &\leq (1+ \|p\|_{H^{1}(\St)})E_1(t)^{\frac12}.
 \end{split}
 \]
 We are left with the last term \eqref{eq:Rp1-again} in the case when $\alpha_i + \gamma_i \leq 1$ for $i = 1,2$.  We estimate this  by the Kato-Ponce inequality (Proposition \ref{prop:kato-ponce}) and the Sobolev embedding  as
 \[
  \begin{split}
\sum_{\alpha_i + \gamma_i\leq 1, \, i =1,2} &\|a_{\alpha,\gamma}(v)  \nabla^{1+\alpha_1} \pa_t^{\gamma_1} U \star \nabla^{1+\alpha_2} \pa_t^{\gamma_2} U\|_{H^{\frac12}(\St)}
\\
&\leq C \sum_{\alpha + \gamma \leq 1 } \|\nabla^{1+\alpha} \pa_t^{\gamma} U\|_{L^{\infty}(\St)}  \sum_{\alpha + \gamma \leq 1} \|\nabla^{1+\alpha} \pa_t^{\gamma} U\|_{H^{\frac12}(\St)} \\
&\leq C \sum_{\alpha + \gamma \leq 1 } \|\nabla^{1+\alpha} \pa_t^{\gamma} U\|_{H^{\frac32}(\St)}  \sum_{\alpha + \gamma \leq 1} \|\nabla^{1+\alpha} \pa_t^{\gamma} U\|_{H^{\frac12}(\St)}.
  \end{split}
 \]
 We bound the first term in the last row by \eqref{eq:R-p-before} 
 \[
 \|\nabla^{1+\alpha} \pa_t^{\gamma} U\|_{H^{\frac32}(\St)}^2  \leq CE_1(t).
 \]
We bound the last term in the last row when $\alpha=1$ and $\gamma= 0$ by \eqref{eq:est-R-U-5} as before $\|\nabla^{2} U\|_{H^{\frac12}(\St)} \leq C(1 + \|p\|_{H^1(\St)})$. We need yet to prove 
 \begin{equation}
     \label{eq:R-1-1-4}
   \|\nabla \pa_t U\|_{H^{\frac12}(\St)}  \leq C(1 + \|p\|_{H^1(\St)})
 \end{equation}
 to conclude the proof.   We use the fact that $\pa_t U$ is harmonic and $\pa_t U =  \pa_\nu U \, v_n$ and therefore by  Theorem \ref{teo:reg-capa} we deduce 
  \[
  \|\nabla \pa_t U\|_{H^{\frac12}(\St)} \leq C(1  + \|\pa_\nu U \, v_n\|_{H^{3/2}(\Sigma_t)}). 
  \]
We recall that   $\|B\|_{L^4(\Sigma_t)} \leq  C $  and $\|B\|_{H^1(\Sigma_t)} \leq C(1+ \|p\|_{H^1(\Sigma_t)})$. Therefore we deduce by by Proposition \ref{prop:kato-ponce}, the a priori bound $\|v_n\|_{H^2(\St)} \leq C$ in  \eqref{eq:apriori_est}  and by \eqref{eq:est-R-U-5} that 
\[
\begin{split}
\|\pa_\nu U \, v_n\|_{H^{3/2}(\Sigma_t)} &\leq  C(\|v_n\|_{H^{\frac32}(\Sigma_t)} + \| \nabla U\|_{H^{\frac32}(\Sigma_t)} + \|\nu\|_{H^{\frac32}(\Sigma_t)}) \\
&\leq C(1+ \|v_n\|_{H^{2}(\Sigma_t)} + \| \nabla^2 U\|_{H^{\frac12}(\Sigma_t)} + \|B\|_{H^{1}(\Sigma_t)})\\
&\leq C(1+ \|p\|_{H^1(\St)}).
\end{split}
\]
Hence, we have \eqref{eq:R-1-1-4} and the claim follows. 
    \end{proof}
    
    We conclude this section with the higher order version of Lemma \ref{lem:estimate-R-p-1}. 
    \begin{lemma}
    \label{lem:estimate-R-p-2}
    Let $l \geq 2$  and consider  $R_p^l$ defined in \eqref{eq:R-p-in-3}. Assume that \eqref{eq:apriori_est} holds and  $E_{l-1}(t) \leq M$. There exists $C = C(M,l)$ such that
    \begin{equation}
    \label{eq:R-p-l-1}
    \|R_p^l\|_{H^{\frac12}(\Sigma_t)}^2 \leq C E_l(t)
    \end{equation}
    and for integers $1 \leq k  \leq l-1$ and $\e>0$ it holds 
    \begin{equation}
    \label{eq:R-p-l-2}
    \|R_p^{l-k}\|_{H^{\frac32k-1}(\Sigma_t)}^2 \leq \e E_l(t) +C_\e
    \end{equation}
    for some constant $C_\e=C(M,l\e)$.
\end{lemma}

\begin{proof}
Let us first prove \eqref{eq:R-p-l-1}. We begin by showing 
\begin{equation}
    \label{eq:R-p-l-6}
    \|R_I^l\|_{H^{\frac12}(\Sigma_t)}^2 \leq CE_l(t),
\end{equation}
where 
\[
R_{I}^l = - (|B|^2 + Q(t)\, H\, |\nabla U|^2 ) (\D_t^{l}v \cdot \nu) + \langle \nabla_\tau p, \D_t^{l}v \rangle,
\]
here  $\Qt$ is defined in \eqref{def:constant-q}.  Let us first recall that $E_{1}(t)\leq C$ implies $\|B\|_{H^{2}(\Sigma_t)} \leq C$ and $\|p\|_{H^2(\Sigma_t)} \leq C$. By Lemma \ref{lem:reg-capa}  this implies $\|\nabla^3 U\|_{H^{\frac12}(\Sigma_t)} \leq C$. In particular, $\|B\|_{L^\infty} \leq C$ and  $\| \nabla^2 U \|_{L^\infty} \leq C$. Therefore we may bound  by Sobolev embedding and by Proposition \ref{prop:kato-ponce} 
\[
\begin{split}
\|(|B|^2 + Q(t)\,H\, |\nabla U|^2 ) &(\D_t^{l}v \cdot \nu)\|_{H^{\frac12}(\Sigma_t)}^2 \leq  \|(|B|^2 + Q(t)\,H\, |\nabla U|^2 ) (\D_t^{l}v \cdot \nu)\|_{H^{1}(\Sigma_t)}^2 \\
&\leq C(1+ \| B\|_{W^{1,4}(\Sigma_t)}^2 \|\D_t^{l}v \cdot \nu\|_{L^4(\Sigma_t)}^2 + \| \D_t^{l}v\cdot \nu\|_{H^1(\Sigma_t)}^2)\\
&\leq C(1+ \| B\|_{H^2(\Sigma_t)}^2) \| \D_t^{l}v \cdot \nu \|_{H^1(\Sigma_t)}^2 \\
&\leq C E_l(t).
\end{split}
\]
In order bound $\|\nabla_\tau p \cdot  \D_t^{l}v\|_{H^{\frac12}(\Sigma_t)}$ we observe that by the curvature bound $\|B\|_{L^\infty}\leq C$, by $-\nabla p = \D_t v$ and by the Sobolev embeddings $\|u\|_{L^3(\Omega_t)} \leq C\|u\|_{H^{\frac12}(\Omega_t)}$ and $\|u\|_{L^6(\Omega_t)} \leq C\|u\|_{H^{1}(\Omega_t)}$ we have 
\[
\begin{split}
\|\nabla_\tau p \cdot \D_t^{l}v\|_{H^{\frac12}(\Sigma_t)} &\leq C\|\D_t v \cdot  \D_t^{l}v\|_{H^{1}(\Omega_t)}\\
&\leq C(1+ \|\nabla \D_t v \cdot \D_t^{l}v\|_{L^{2}(\Omega_t)} + \| \D_t v \cdot \nabla \D_t^{l}v\|_{L^{2}(\Omega_t)})\\
&\leq C(1 +  \|\nabla \D_t v\|_{L^{3}(\Omega_t)}  \| \D_t^{l}v\|_{L^{6}(\Omega_t)} + \| \D_t v\|_{L^{6}(\Omega_t)}  \| \nabla \D_t^{l}v\|_{L^{3}(\Omega_t)})\\
&\leq C(1 +  \| \D_t v\|_{H^{\frac32}(\Omega_t)}  \| \D_t^{l}v\|_{H^{1}(\Omega_t)} + \| \D_t v\|_{H^{1}(\Omega_t)}  \|  \D_t^{l}v\|_{H^{\frac32}(\Omega_t)}).
\end{split}
\]
By definition of $E_l(t)$ it holds 
\[
\| \D_t v\|_{H^{\frac32}(\Omega_t)}^2 \leq E_1(t) \leq E_{l-1}(t) \leq C \quad \text{and} \quad  \|  \D_t^{l}v\|_{H^{\frac32}(\Omega_t)}^2\leq E_{l}(t).
\]
Therefore we have $\|\nabla_\tau p \cdot \D_t^{l}v\|_{H^{\frac12}(\Sigma_t)}^2 \leq C E_{l}(t) $ and \eqref{eq:R-p-l-6} follows. 

Let us next show 
\begin{equation}
    \label{eq:R-p-l-7}
    \|R_{II}^l\|_{H^{\frac12}(\Sigma_t)}^2 \leq CE_l(t),
\end{equation}
where 
\[
R_{II}^l =  \sum_{|\alpha|\leq 1, \, |\beta|\leq l-1}a_{\alpha, \beta}(B) \overbrace{\nabla^{1+\alpha_1} \D_t^{\beta_1} v \star \cdots \star \nabla^{1+\alpha_{l+1}} \D_t^{\beta_{l+1}} v}^{=: R_{\alpha, \beta}(v)}.
\]
We first observe that we may ignore the coefficients $a_{\alpha, \beta}(B) $. Indeed, we may extend $B$ to $\Omega_t$, call the extension $\tilde B$,  such that $\|\tilde B\|_{H^2(\Omega_t)} \leq C$.  Then by the above notation 
\begin{equation}
    \label{eq:Rplgeq3}
\begin{split}
 \|R_{II}^l\|_{H^{\frac12}(\Sigma_t)} &\leq C  \|R_{II}^l\|_{H^{1}(\Omega_t)} \leq  C(1+ \sum_{\alpha, \beta} \|\nabla (\nabla a_{\alpha, \beta}(\tilde B) \star R_{\alpha, \beta}(v))\|_{L^2(\Omega_t)}\\
 &\leq C(1+ \sum_{\alpha, \beta} \|\nabla \tilde B \star R_{\alpha, \beta}(v)\|_{L^2(\Omega_t)} + \| \nabla R_{\alpha, \beta}(v)\|_{L^2(\Omega_t)}) \\
 &\leq C(1+ \sum_{\alpha, \beta} \|\nabla \tilde B\|_{L^4(\Omega_t)} \|R_{\alpha, \beta}(v)\|_{L^4(\Omega_t)} + \| \nabla R_{\alpha, \beta}(v)\|_{L^2(\Omega_t)}) \\
 &\leq  C(1+ \sum_{\alpha, \beta} \| \nabla R_{\alpha, \beta}(v)\|_{L^2(\Omega_t)}) \\
 &\leq C\big(1 + \sum_{|\alpha|\leq 2, \, |\beta|\leq l-1} \|\nabla^{1+\alpha_1} \D_t^{\beta_1} v \star \cdots \star \nabla^{1+\alpha_{l+1}} \D_t^{\beta_{l+1}} v\|_{L^2(\Omega_t)}\big).
 \end{split}
\end{equation}

By the  assumption $E_{l-1}(t) \leq C$ and by \eqref{eq:induction} we deduce $\| \nabla^{1+ \alpha_i} \D_t^{\beta_i} v\|_{L^\infty} \leq C$ for $\alpha_i + \beta_i \leq l-2$. We note also that by $|\alpha| \leq 2$ and $|\beta|\leq l-1$ it follows that $|\alpha| + |\beta| \leq l+1$.  We ignore all the terms in the last row of \eqref{eq:Rplgeq3}  which indexes satisfy $\alpha_i + \beta_i \leq l-2$ as these are uniformly bounded. For the rest of the terms we  use  H\"older's inequality   and relabel the indexes (note that below we assume $\alpha \leq 2$ and $\beta\leq l-1$)
\begin{equation} \label{eq:Rp-holder}
\begin{split}
\sum_{|\alpha|\leq 2, \, |\beta|\leq l-1} &\|\nabla^{1+\alpha_1} \D_t^{\beta_1} v \star \cdots \star \nabla^{1+\alpha_{l+1}} \D_t^{\beta_{l+1}} v\|_{L^2(\Omega_t)} \\
&\leq C\sum_{\alpha\leq 2, \beta \leq l-1} \|\nabla^{1 + \alpha} \D_t^{\beta} v\|_{L^2}^2  +  \sum_{\alpha+ \beta = l} \|\nabla^{1 + \alpha} \D_t^{\beta} v\|_{L^6}^2 \cdot \sum_{\alpha+ \beta = 1} \|\nabla^{1 + \alpha} \D_t^{\beta} v\|_{L^3}^2 \\
&\,\,\,\,\,\,\,\,+ \sum_{\alpha+ \beta = l-1} \|\nabla^{1 + \alpha} \D_t^{\beta} v\|_{L^3}^2 \cdot \sum_{\alpha+ \beta = 2} \|\nabla^{1 + \alpha} \D_t^{\beta} v\|_{L^6}^2  + \sum_{\alpha+ \beta \leq l-1} \|\nabla^{1 + \alpha} \D_t^{\beta} v\|_{L^6}^6 .
\end{split}
\end{equation} 
To bound the first term on the RHS of \eqref{eq:Rp-holder} we simply note that for $\beta \leq l-1$ and $\alpha \leq 2$ it holds
\[
\|\nabla^{1+\alpha} \D_t^{\beta}  v\|_{L^2(\Omega_t)}^2 \leq C\| \D_t^{\beta}  v\|_{H^3(\Omega_t)}^2  \leq C E_{l}(t).
\]
For the second and the third terms we have first for $\alpha + \beta \leq l$ and $\beta \leq l-1$ (which include the case $\alpha + \beta = 2$ as $l \geq 2$)  that 
\[
\|\nabla^{1 + \alpha} \D_t^{\beta} v\|_{L^6(\Omega_t)}^2 \leq C\| \D_t^{\beta} v\|_{H^{l -\beta+2}(\Omega_t)}^2 \leq C \| \D_t^{l+1 -(l+1+\beta)} v\|_{H^{\frac32(l+1+\beta)}(\Omega_t)}^2  \leq C E_{l}(t) .
\]
For $\alpha + \beta \leq l-1$ (which includes the case  $\alpha + \beta = 1$) we deduce
\begin{equation}
\label{eq:Rp:intermediate}
\|\nabla^{1 + \alpha} \D_t^{\beta} v\|_{L^3(\Omega_t)}^2 \leq C \| \D_t^{\beta}v\|_{H^{\frac12 +(l- \beta)}(\Omega_t)}^2 \leq C E_{l-1}(t) \leq C.
\end{equation}
For the last term we interpolate in the fluid domain $\Omega_t \subset \R^3$ for $\alpha + \beta \leq l-1$ as
\[
\begin{split}
\|\nabla^{1 + \alpha} \D_t^{\beta} v\|_{L^6(\Omega_t)} &\leq \|\nabla^{1+\alpha} \D_t^{\beta} v\|_{H^2(\Omega_t)}^{1/3} \|\nabla^{1 + \alpha} \D_t^{\beta} v\|_{L^3(\Omega_t)}^{2/3}\leq C E_l(t)^{1/6}  \|\nabla^{1 + \alpha} \D_t^{\beta} v\|_{L^3}^{2/3}.
\end{split}
\]
By \eqref{eq:Rp:intermediate} we have $\|\nabla^{1 + \alpha} \D_t^{\beta} v\|_{L^3}\leq C$ and thus
\[
\|\nabla^{1 + \alpha} \D_t^{\beta} v\|_{L^6}^6 \leq C E_l(t).
\]
By combing the previous estimates with \eqref{eq:Rplgeq3} and \eqref{eq:Rp-holder} we obtain 
\[
\|R_{II}^l\|_{H^{1/2}(\Sigma_t)}^2 \leq C(1+\| \nabla  R_{II}^l\|_{L^{2}(\Omega_t)}^2) \leq C  E_l(t),
\]
and \eqref{eq:R-p-l-7} follows. 

We are left to prove 
\begin{equation}
    \label{eq:R-p-l-8}
    \|R_{III}^l\|_{H^{\frac12}(\Sigma_t)}^2 \leq CE_l(t),
\end{equation}
where 
\[
R_{III}^l = \sum_{\substack{|\alpha| +  |\beta| + |\gamma |\leq l+1\\ |\beta|\leq l-1, \gamma_i \leq l}} a_{\alpha, \beta,\gamma,Q}(v)  \D_t^{\beta_1} v \star \cdots \star \D_t^{\beta_{l-1}} v  \star \nabla^{1+\alpha_1} \pa_t^{\gamma_1} U \star \nabla^{1+\alpha_1} \pa_t^{\gamma_2} U 
\]
and the coefficients $a_{\alpha,\beta,\gamma,Q}$ depend on the time derivatives of $\Qt$ up to order $l+1$. Recall that  $\Qt$ is defined in \eqref{def:constant-q}

This time we will not give the argument which proves the  boundedness of  $Q(t)^{(l+1)} =  \frac{d^{l+1}}{dt^{l+1}} Q(t)$  as it simpler than the rest of the proof and is similar to the argument in \eqref{eq:Q-2-derivat}.

Recall that by \eqref{eq:induction} it holds $\|\D_t^{\beta_i} v\|_{L^\infty} \leq C$ for $\beta_i \leq l-2$. On the other hand for $\alpha + \gamma \leq l-1$ we have by Lemma \ref{lem:reg-capa1}
\begin{equation}
\label{eq:induction3}
\begin{split}
    \|\nabla^{1+\alpha}\pa_t^\gamma U\|_{L^{\infty}(\Sigma_t)}^2 \leq C\|\nabla^{1+\alpha}\pa_t^\gamma U\|_{H^{\frac32}(\Sigma_t)}^2\leq C(1+\|\nabla^{2+\alpha}\pa_t^\gamma U\|_{H^{\frac12}(\Sigma_t)}^2) \leq CE_{l-1}(t) \leq C.  
\end{split}
\end{equation}
Similarly we have 
\begin{equation}
\label{eq:R-p-l-a}
\|\D_t^{\beta} v\|_{H^{\frac12}(\Sigma_t)} \leq C \quad \text{for }\, \beta \leq l-1\quad \text{and} \quad \|\nabla^{1+\alpha}\pa_t^\gamma U\|_{H^{\frac12}(\Sigma_t)} \leq C \quad  \text{for} \,  \alpha+\gamma \leq l.
\end{equation}
 Therefore we may ignore all the terms with indexes which satisfy  $\beta_i \leq l-2$ and $\alpha_i + \gamma_i \leq l-1$. Recall that we assume $\beta_1 \geq \dots \geq \beta_{l-1}$. 
Therefore we may estimate by Proposition \ref{prop:kato-ponce}, \eqref{eq:induction3} and by  \eqref{eq:R-p-l-a} 
 \[
 \begin{split}
  &\|R_{III}^l\|_{H^{\frac12}(\Sigma_t)} \\
  &\leq \sum_{\alpha + \gamma \leq l}C\big(\|\D_t^{l-1} v\|_{L^{\infty}(\Sigma_t)}\|\nabla^{1+\alpha}\pa_t^\gamma U\|_{H^{\frac12}(\Sigma_t)} + \|\D_t^{l-1} v\|_{H^{\frac12}(\Sigma_t)}\|\nabla^{1+\alpha}\pa_t^\gamma U\|_{L^{\infty}(\Sigma_t)}\big)\\
  &\,\,\,\,\,\,\,\,\,\,\,+C \sum_{\alpha + \gamma \leq l+1, \gamma \leq l}\|\nabla^{1+\alpha}\pa_t^\gamma U\|_{H^{\frac12}(\Sigma_t)}\\
  &\leq C\big(\|\D_t^{l-1} v\|_{L^{\infty}(\Sigma_t)} + \sum_{\alpha + \gamma \leq l}\|\nabla^{1+\alpha}\pa_t^\gamma U\|_{L^{\infty}(\Sigma_t)} + \sum_{\substack{\alpha + \gamma \leq l+1 \\ \gamma \leq l}}\|\nabla^{1+\alpha}\pa_t^\gamma U\|_{H^{\frac12}(\Sigma_t)} \big).
   \end{split}
 \]
We estimate the first term in the last row by Sobolev embedding
\[
\|\D_t^{l-1} v\|_{L^{\infty}(\Sigma_t)}^2  \leq C \|\D_t^{l-1} v\|_{H^{3}(\Omega_t)}^2 \leq CE_{l}(t). 
\]
For the second we use Sobolev embedding and Lemma \ref{lem:reg-capa1} and obtain for $\alpha +\gamma \leq l$
\[
\|\nabla^{1+\alpha}\pa_t^\gamma U\|_{L^{\infty}(\Sigma_t)}^2 \leq C(1+\|\nabla^{2+\alpha}\pa_t^\gamma U\|_{H^{\frac12}(\Sigma_t)}^2  )\leq CE_l(t).
\]
The same argument also implies 
\[
\|\nabla^{1+\alpha}\pa_t^\gamma U\|_{H^{\frac12}(\Sigma_t)}^2 \leq CE_{l}(t)
\]
for $\alpha + \gamma \leq l+1$ and $\gamma \leq l$. Hence we obtain \eqref{eq:R-p-l-8} and this concludes the proof of \eqref{eq:R-p-l-1}.

Let us then prove \eqref{eq:R-p-l-2}. Let us first treat  the first term in the definition of $R_{p}^{l-k}$ and  bound $\|R_I^{l-k}\|_{H^{\frac32k-1}(\Sigma_t)}$. We first observe that the case $k=1$ follows from \eqref{eq:R-p-l-1}. Let us then assume $k \geq 2$. By the Sobolev embedding it holds $\|u\|_{L^\infty(\Sigma)} \leq C \|u\|_{L^{\frac32k-1}(\Sigma)}$.   We use this and  the Kato-Ponce inequality in  Proposition \ref{prop:kato-ponce} to deduce that 
\[
\begin{split}
&\|R_I^{l-k}\|_{H^{\frac32k-1}(\Sigma_t)} \\
&\leq C(1+  \| B\|_{H^{\frac32k-1}(\Sigma_t)}^{2} + \| \nabla U\|_{H^{\frac32k-1}(\Sigma_t)}^{2}+ \| \D_t^{l-k} v\|_{H^{\frac32k-1}(\Sigma_t)}^{2} + \|\nabla p \|_{H^{\frac32k-1}(\Sigma_t)}^{2}).
\end{split}
\]
Let us show that all the terms on RHS are bounded. 
 
 To this aim we first recall that the bound  $E_{l-1}(t) \leq C$ implies $\|B\|_{H^{\frac32 l -1}(\Sigma_t)} \leq C$. Since $k \leq l-1$ we deduce $\| B\|_{H^{\frac32k-1}(\Sigma_t)} \leq C$. Lemma \ref{lem:reg-capa} implies $\| \nabla U\|_{H^{\frac32k-1}(\Sigma_t)}^2 \leq C E_{l-1}(t) \leq C$. The condition  $k \leq l-1$ and the Trace Theorem also yields
 \[
 \| \D_t^{l-k} v\|_{H^{\frac32k-1}(\Sigma_t)}^2 \leq C\| \D_t^{l-k} v\|_{H^{\frac32k}(\Omega_t)}^2 \leq C E_{l-1}(t) \leq C. 
 \]
Similarly we deduce by $-\nabla p = \D_t v$ that  $ \|\nabla p \|_{H^{\frac32k-1}(\Sigma_t)}\leq C$. Hence, we have  
\[
\|R_I^{l-k}\|_{H^{\frac32k-1}(\Sigma_t)}\leq C.
\]

Let us then bound $\|R_{II}^{l-k}\|_{H^{\frac32k-1}(\Sigma_t)}$. As before, Proposition \ref{prop:kato-ponce} and the Sobolev embedding yield
\[
\|R_{II}^{l-k}\|_{H^{\frac32k-1}(\Sigma_t)} \leq \sum_{\alpha \leq 1, \beta \leq l-k-1}C( 1+ \|a_{\alpha,\beta}(B)\|_{H^{\frac32k-1}(\Sigma_t)}^2 + \|\nabla^{1+\alpha} \D_t^{\beta} v\|_{H^{\frac32k-1}(\Sigma_t)}^q)
\]
for $q \geq 1$. Recall that $\|B\|_{H^{\frac32k-1}(\Sigma_t)} \leq C$ and $k \geq 2$. Therefore $\|B\|_{L^\infty} \leq C $ and thus $\|a_{\alpha,\beta}(B)\|_{H^{\frac32k-1}(\Sigma_t)} \leq C$. On the other hand by $\alpha\leq 1$, $\beta \leq l - (k+1)$  and by Lemma \ref{lem:CS-intermed} we deduce
\begin{equation}
\label{eq:R-p-ineq2-1}
\|\nabla^{1+\alpha} \D_t^{\beta} v\|_{H^{\frac32k-1}(\Sigma_t)}^2 \leq C \|\D_t^{\beta} v\|_{H^{\frac32(k+1)}(\Omega_t)}^2 \leq C \sum_{i =0}^{l-1} \|\D_t^{l-i} v\|_{H^{\frac32i}(\Omega_t)}^2  \leq C E_{l-1}(t) \leq C. 
\end{equation}
Hence, we have $\|R_{II}^{l-k}\|_{H^{\frac32k-1}(\Sigma_t)} \leq C$.

Let us finally treat  $R_{III}^{l-k}$. We note that it holds $\|v\|_{H^{\frac32k-1}(\St)}^2 \leq CE_{l-1}(t) \leq C$ and therefore we may ignore the coefficients $a_{\alpha,\beta,\gamma, Q}(v)$.  Then  by Proposition \ref{prop:kato-ponce} and by the Sobolev embedding we have
\begin{equation}
    \label{eq:R-p-l-3}
    \begin{split}
\|R_{III}^{l-k}\|_{H^{\frac32k-1}(\Sigma_t)} &\leq   C(1+ \sum_{\beta \leq l-(k+1)} \| \D_t^{\beta} v\|_{H^{\frac32k-1}(\Sigma_t)}^{q} \cdot \\ &\cdot \sum_{\substack{|\alpha| +  |\gamma| \leq l-k+1, \\ |\gamma| \leq l-k}} \| \nabla^{1+\alpha_1}\pa_t^{\gamma_1} U\|_{H^{\frac32k-1}(\Sigma_t)}\| \nabla^{1+\alpha_2}\pa_t^{\gamma_2} U\|_{H^{\frac32k-1}(\Sigma_t)})
    \end{split}
\end{equation} 
for $q \geq 1$.  By \eqref{eq:R-p-ineq2-1} we have $\| \D_t^{\beta} v\|_{H^{\frac32k-1}(\Sigma_t)} \leq C$ for all $\beta \leq l-(k+1)$. To bound the last term we may assume that $\alpha_1 + \gamma_1 \geq \alpha_1 + \gamma_1 $. If $\alpha_1 + \gamma_1 = l-k +1$ then necessarily $\alpha_2 + \gamma_2 \leq l-k$. Therefore by Lemma \ref{lem:reg-capa1}
\[
\| \nabla^{1+\alpha_1}\pa_t^{\gamma_1} U\|_{H^{\frac32k-1}(\Sigma_t)}^2 \leq C(1+ \| \nabla^{1+(\alpha_1 +k-1)}\pa_t^{\gamma_1} U\|_{H^{\frac{k-1}{2} + \frac12}(\Sigma_t)}^2) \leq \e E_l(t) + C_\e
\]
and 
\[
\| \nabla^{1+\alpha_2}\pa_t^{\gamma_2} U\|_{H^{\frac32k-1}(\Sigma_t)}^2 \leq C(1+ \| \nabla^{1+(\alpha_2+k-1) }\pa_t^{\gamma_2} U\|_{H^{\frac{k-1}{2} + \frac12}(\Sigma_t)}^2) \leq  CE_{l-1}(t) \leq C.
\]
Therefore we deduce  by \eqref{eq:R-p-l-3} that 
\[
\|R_{III}^{l-k}\|_{H^{\frac32k-1}(\Sigma_t)}^2 \leq \e E_l(t) + C_\e.
\]
This concludes  the proof of \eqref{eq:R-p-l-2}.
\end{proof}

\section{First regularity estimates}

In this section we prove our first regularity estimates for the solution of \eqref{system}. We  assume that the solution satisfies the a priori estimates \eqref{eq:apriori_est}, i.e., $\Lambda_T <\infty$ and $\sigma_T>0$, where $\Lambda_T$ and $\sigma_T$ are defined in \eqref{def:apriori_est} and \eqref{eq:height_well}. We recall that 
\[
\Lambda_T := \sup_{t \in(0,T]} \left( \|h(\cdot ,t)\|_{C^{1,\alpha}(\Sigma_t)} + \|\nabla v(\cdot,t) \|_{L^\infty(\Omega_t)} + \|v_n(\cdot,t) \|_{H^2(\Sigma_t) }  \right).
\]
    In particular,  bound on $\Lambda_T$ does not imply curvature bounds, and thus we need to be careful as we may not use e.g. the interpolation results from Proposition \ref{prop:interpolation}. Our goal in this section is to show that the a priori estimates \eqref{eq:apriori_est} imply the following bounds for the pressure 
    \[
    \sup_{t \leq T}\|p\|_{H^1(\Omega_t)} \leq C \qquad \text{and} \qquad \int_0^T \|p\|_{H^2(\Omega_t)}^2 \, dt \leq C. 
    \]
    The first bound above is important as it implies $\|B_{\Sigma}\|_{L^4(\Sigma_t)} \leq C$, which is crucial  e.g. for the interpolation inequality in Proposition \ref{prop:interpolation} to hold. The second estimate is important for the first order energy estimate which we prove in the next section in Proposition \ref{prop:case-1}.  
    
    Let us begin by stating regularity estimates that we have by the a priori estimate. First, recall once again that by the uniform $C^{1,\alpha}(\Gamma)$-bound we have for the capacitary potential $U$ that $\|U\|_{C^{1,\alpha}(\bar \Omega_t^c)} \leq C$.   Let us prove the following estimates for the second fundamental form and for the capcacitary potential. 
\begin{lemma}
\label{lem:B^4}
Assume that  \eqref{eq:apriori_est} holds for $T>0$. Then for all $t < T$ we have  
\[
\|B\|_{L^4(\Sigma_t)}^4 \leq \e \|p\|_{H^1(\Sigma_t)}^2 + C_{\e}
\]
for $C= C(M, \e)$ and 
\[ 
\|B\|_{H^1(\Sigma_t)} + \|\nabla^2 U \|_{H^{\frac12}(\St)}\leq C(1+ \|p\|_{H^1(\Sigma_t)})
\]
for $C= C(M)$.
\end{lemma}
  \begin{proof}
    We denote the height-function by  $h = h(\cdot,t)$. Then by standard calculations (see e.g. \cite{FJM3D, MantegazzaBook}) we  may write the second fundamental form on $\Sigma$ as $B = a(h,\bar \nabla h)  \bar \nabla^2 h$. Therefore we may bound 
  \[
\|B\|_{L^4(\Sigma_t)}^4 \leq C(1+\|\bar \nabla^2 h\|_{L^4(\Gamma)}^4) \quad \text{and} \quad  \|\bar \nabla^3 h\|_{L^2(\Gamma)}^2 \leq C(1+ \|\bar  \nabla B \|_{L^2(\Sigma_t)}^2 +  \|\bar \nabla^2 h\|_{L^4(\Gamma)}^4). 
  \]
  We use interpolation on $\Gamma$ as
   \[
  \|\bar \nabla^2 h\|_{L^4(\Gamma)} \leq C \|\bar \nabla^3 h\|_{L^2(\Gamma)}^{\theta} \| h\|_{C^{1,\alpha}(\Gamma)}^{1-\theta} \leq C  \|\bar \nabla^3 h\|_{L^2(\Gamma)}^{\theta}
  \]
  for $\theta< 1/2$. This implies by Young's inequality $\|\bar \nabla^2 h\|_{L^4(\Gamma)}^4 \leq  \e  \|\bar \nabla^3 h\|_{L^2(\Gamma)}^2 + C_\e$.
  Thus  by choosing $\e$ small we obtain
  \[
  \|\bar \nabla^3 h\|_{L^2(\Gamma)}^2 \leq C(1+ \|\bar  \nabla B \|_{L^2(\Sigma_t)}^2) 
   \]
  and
  \[
  \|B\|_{L^4(\Sigma_t)}^4 \leq  \e  \|\bar \nabla^3 h\|_{L^2(\Gamma)}^2 + C_\e.
   \]
  By the Simon's identity \eqref{eq:Simon}  we deduce
  \[
  \|\bar  \nabla B \|_{L^2(\Sigma_t)}^2 \leq  \|\bar \nabla H \|_{L^2(\Sigma_t)}^2 + C \|B\|_{L^4(\Sigma_t)}^4. 
  \]
  Therefore we have  
  \begin{equation} \label{eq:B^4-3}
  \|B\|_{L^4(\Sigma_t)}^4 \leq  \e  \| H \|_{H^1(\Sigma_t)}^2 + C_\e
   \end{equation}
  and
  \begin{equation} \label{eq:B^4-2}
  \| B \|_{H^1(\Sigma_t)} \leq  C(1+\|H \|_{H^1(\Sigma_t)}). 
  \end{equation}
  
 Let us consider the capacitary potential $U$. Let us show that even though we may not use Proposition \ref{prop:interpolation}, the $C^{1,\alpha}(\Gamma)$-regularity still implies the following weak interpolation inequality  
   \beq \label{eq:B^4-4}
   \|\nabla U \|_{H^1(\Sigma_t)} \leq \e  \| \nabla^2 U \|_{H^{\frac12}(\Sigma_t)} + C_\e \| \nabla U \|_{L^{2}(\Sigma_t)}\leq \e  \| \nabla^2 U \|_{H^{\frac12}(\Sigma_t)} + C_\e.
   \eeq
In order to  prove \eqref{eq:B^4-4} we first observe that the $C^{1,\alpha}(\Gamma)$-regularity of $\St$ implies the following inequalities  for $p \in (1,2)$ and $u,v \in C^\infty(\St)$
 \[
 \|u\|_{L^2(\St)} \leq \e  \|u\|_{H^{\frac12}(\St)} + C_\e \|u\|_{L^p(\St)}  \quad \text{and} \quad \|\nabla_\tau v\|_{L^p(\St)} \leq  \delta   \|v\|_{H^1(\St)}  + C_\delta  \|v\|_{L^2(\St)}.
 \]
We apply these for $u = \nabla^2 U$ and $v = \nabla U$ and have
\[
\begin{split}
&\|\nabla ^2 U\|_{L^2(\St)} \leq \e  \|\nabla^2 U\|_{H^{\frac12}(\St)} + C_\e \|\nabla^2 U\|_{L^p(\St)}  \quad \text{and} \\
&\|\nabla_\tau \nabla U\|_{L^p(\St)} \leq  \delta   \|\nabla U\|_{H^1(\St)}  + C_\delta  \|\nabla U\|_{L^2(\St)}.
\end{split}
\]
Since $\pa_{x_i} U$ is harmonic and $\St$ is $C^{1,\alpha}(\Gamma)$-regular  we have by  \cite{FJR} that \[
 \|\nabla^2 U\|_{L^p(\St)}  \leq C( \|\nabla_\tau \nabla U\|_{L^p(\St)} + \|\nabla U\|_{L^2(\Omega_t)}). 
\]
Therefore by first choosing $\e$ small and then $\delta$ even smaller we obtain \eqref{eq:B^4-4}.

  By $p = H - \frac{Q(t)}{2} |\nabla U|^2$, where $\Qt$ is defined in \eqref{def:constant-q},  we may estimate 
   \[
   \|H\|_{H^{1}(\Sigma_t)} \leq   C( \|p\|_{H^1(\Sigma_t)} + \|\nabla U \|_{H^1(\Sigma_t)}). 
   \] 
 Then by \eqref{eq:B^4-4} we have 
\[
 \|H\|_{H^{1}(\Sigma_t)} \leq   C_\e(1+ \|p\|_{H^1(\Sigma_t)}) + \e \| \nabla^2 U \|_{H^{\frac12}(\Sigma_t)}.
\]
We use Theorem \ref{teo:reg-capa} and \eqref{eq:B^4-2} and have 
\[
\| \nabla^2 U \|_{H^{\frac12}(\Sigma_t)} \leq C( 1 + \|B\|_{H^1(\Sigma_t)}) \leq C(1+\|H \|_{H^1(\Sigma_t)}).
\]
Therefore we deduce from the two above inequalities that
\[
 \|H\|_{H^{1}(\Sigma_t)} \leq   C(1+ \|p\|_{H^1(\Sigma_t)}).
\]
The claim follows from this together with \eqref{eq:B^4-3} and \eqref{eq:B^4-2}.
  \end{proof}

 Let us proceed to the following regularity estimate. 
 \begin{lemma}
     \label{lem:claim1}
     Assume that the a priori estimates \eqref{eq:apriori_est} hold for $T>0$.  Then 
     \[
     \sup_{t \in [0,T]} \| p\|_{L^2(\Sigma_t)}^2 +  \int_0^T \| p\|_{H^1(\Sigma_t)}^2 \, dt \leq C(1+ T),
     \]
      for $C = C(M)$. 
 \end{lemma}
 
 \begin{proof}
The idea is to consider the following function
\[
\Psi(t) :=  \int_{\Sigma_t} p\,  ( \nabla v \, \nu )\cdot \nu + \e p^2 \, d \H^2 ,
\] 
where the choice of $\e$ will be clear later. First, we observe that under the a priori estimates \eqref{eq:apriori_est} $v$ is uniformly Lipschitz  and therefore $\Psi$ is bounded from below by
\begin{equation} \label{eq:claim1-11}
\Psi(t) \geq - C\|p(\cdot, t)\|_{L^1(\St)} + \e \|p(\cdot, t)\|_{L^2(\St)}^2 \geq -C_\e + \frac{\e}{2}\|p(\cdot, t)\|_{L^2(\St)}^2,
\end{equation}
where the last inequality follows from the Young's inequality, i.e.,  $\|p(\cdot, t)\|_{L^1(\St)} \leq  \frac{\e}{2} \|p(\cdot, t)\|_{L^2(\St)}^2 + C_\e$ and $C_\e$ is a large constant that depends on $\e$.  By differentiating and using the a priori estimates \eqref{eq:apriori_est} we obtain 
\begin{equation} \label{eq:claim1-1}
\begin{split}
\frac{d}{dt}  &\int_{\Sigma_t} p\,  ( \nabla v \, \nu) \cdot \nu  \, d \H^2 \\
&=  \int_{\Sigma_t} p \,  (\nabla v \, \nu  )\cdot \nu  \Div_\tau v \, d \H^2  + \int_{\Sigma_t} \D_t p  \, (\nabla v \, \nu ) \cdot \nu  \, d \H^2 +  \int_{\Sigma_t} p\, \D_t ((\nabla v \, \nu)  \cdot \nu ) \, d \H^2\\
&\leq C_\e  + \e\int_{\Sigma_t} |\D_t p |^2 \, d \H^2 +  \int_{\Sigma_t} p\, \D_t  ((\nabla v \, \nu)  \cdot \nu ) \, d \H^2.
\end{split}
\end{equation}
We estimate the second last term in \eqref{eq:claim1-1} by \eqref{eq:D_tp-4} and  \eqref{eq:apriori_est} and have
\[
\int_{\Sigma_t} |\D_t p |^2 \, d \H^2 \leq C (1+ \| p\|_{H^1(\Sigma_t)}^2  + \|B\|_{L^4(\Sigma_t)}^4 + \| \nabla  \pa_t U \|_{L^2(\Sigma_t)}^2).
\]
Lemma \ref{lem:B^4} yields $\|B\|_{L^4(\Sigma_t)}^4 \leq  C(1+ \|p\|_{H^1(\Sigma_t)}^2)$. On the other hand, by   Lemma \ref{lem:divcurl}  it holds  $\| \nabla  \pa_t U \|_{\Sigma_t}^2 \leq C\| \pa_t U \|_{H^1(\Sigma_t)}^2$. Since $\pa_t U = - \nabla U \cdot v$, we  may estimate  by Theorem \ref{teo:reg-capa} and by Lemma \ref{lem:B^4}
\[
\begin{split}
\| \nabla  \pa_t U \|_{L^2(\Sigma_t)}^2 &\leq C(1+\| \bar \nabla  \pa_t U \|_{L^2(\Sigma_t)}^2) \leq C(1+ \| \bar \nabla  (\nabla U \cdot v) \|_{L^2(\Sigma_t)}^2) \leq C(1+ \|\nabla^2 U\|_{L^2(\Sigma_t)}^2)\\ &\leq C(1+ \|\nabla^2 U\|_{H^{1/2}(\Sigma_t)}^2) \leq C(1+ \|p\|_{H^1(\Sigma_t)}^2).
\end{split}
\]
Therefore we may bound
\begin{equation} \label{eq:claim1-2}
\int_{\Sigma_t} |\D_t p |^2 \, d \H^2 \leq C (1+ \| p\|_{H^1(\Sigma_t)}^2 ).
\end{equation}

Let us treat the last in term in \eqref{eq:claim1-1}. First, we have by \eqref{eq:comm1}, \eqref{eq:normal1} and \eqref{eq:apriori_est} that 
\[
\begin{split}
\int_{\Sigma_t} p\, \D_t ( (\nabla v \, \nu) \cdot \nu ) \, d \H^2 &\leq  \int_{\Sigma_t} p\,  ((\nabla  \D_t v \, \nu) \cdot \nu) \, d \H^2 + \e \|p\|_{L^2(\St)}^2 + C_\e \\
&= - \int_{\Sigma_t} p\,  ( \nabla ^2 p \, \nu )\cdot \nu ) \, d \H^2 + \e \|p\|_{L^2(\St)}^2 + C_\e. 
\end{split}
\]
We use \eqref{eq:pressure-1} to estimate $|\Delta p| \leq C\|\nabla v\|_{L^\infty}^2 \leq C $  and recall that  by \eqref{eq:tang-laplace}  it holds $(\nabla^2 p \, \nu )\cdot \nu  =   \Delta p  - \Delta_\Sigma p - H \pa_\nu p$ to deduce
\[
\begin{split}
-\int_{\Sigma_t} p\,  ( \nabla^2 p \, \nu )\cdot \nu \, d \H^2 &\leq  C + \int_{\Sigma_t} p\,  \Delta_\Sigma p \, d \H^2 + \int_{\Sigma_t} H p\,  \pa_\nu p \, d \H^2\\
&\leq C - \int_{\Sigma_t} |\bar \nabla p |^2 \, d \H^2 + \int_{\Sigma_t} (\e |\pa_\nu p|^2 + C_\e(1+ |H|^4)) \, d \H^2,
\end{split}
\]
where in the last inequality we have used $p = H - \frac{Q(t)}{2}|\nabla U|^2$, where $\Qt$ is defined in \eqref{def:constant-q}. Lemma \ref{lem:divcurl} yields $\|\pa_\nu p\|_{L^2(\Sigma_t)} \leq C(1 + \| p\|_{H^1(\Sigma_t)})$ while Lemma \ref{lem:B^4} implies $\|H\|_{L^4}^4 \leq \delta \| p\|_{H^1(\Sigma_t)}^2 + C_{\delta}$. Therefore by first choosing $\e$ small and then $\delta$ even smaller,  we obtain
\[
-\int_{\Sigma_t} p\,  (\nabla^2 p \, \nu )\cdot \nu \, d \H^2 \leq - \|\bar \nabla p \|_{L^2(\Sigma_t)}^2 + \e \| p\|_{H^1(\Sigma_t)}^2 + C_\e. 
\]
By direct calculation and by using \eqref{eq:claim1-2} we have 
\[
\frac{d}{dt} \int_{\St}  p^2 \, d \H^2 \leq  C (1+ \|p\|_{H^1(\St)}^2 ).
\]
Combining the two above inequalities with \eqref{eq:claim1-1} and \eqref{eq:claim1-2} we conclude
\begin{equation} \label{eq:claim1-3}
\frac{d}{dt} \Psi(t) \leq - \frac12\|\bar \nabla p \|_{L^2(\Sigma_t)}^2 + \e \| p\|_{L^2(\Sigma_t)}^2  + C_\e,
\end{equation} 
when $\e$ is small. Finally we use Lemma \ref{lem:B^4} to estimate 
\[
\| p\|_{L^2(\Sigma_t)}^2  \leq C(1+\|H \|_{L^2(\Sigma_t)}^2) \leq  \|B\|_{L^4(\St)}^4 + C \leq \e \| p\|_{H^1(\Sigma_t)}^2 + C_\e.
\]
This yields $\| p\|_{L^2(\Sigma_t)}^2 \leq2 \|\bar \nabla p \|_{L^2(\Sigma_t)}^2 + C_\e$ when $\e$ is small. Thus we may estimate \eqref{eq:claim1-3} 
\[
\frac{d}{dt} \Psi(t) \leq - \frac14\|\bar \nabla p \|_{L^2(\Sigma_t)}^2  + C_\e,
\]
when $\e$ is small. The conclusion  follows by integrating the above over  $[0,T]$ and using \eqref{eq:claim1-11}. 
 \end{proof}

We proceed to higher order regularity estimate which is uniform in time. 
\begin{proposition}
\label{prop:claim-reg1}
Assume that the a priori estimates \eqref{eq:apriori_est} hold for $T>0$.  Then 
     \[
     \sup_{t \in (0,T]} \|\nabla p\|_{L^2(\Omega_t)}^2  \leq e^{C(1+T)}(1+ \|\nabla p\|_{L^2(\Omega_0)}^2) 
     \]
     for $C = C(M)$.
\end{proposition}

\begin{proof}
We differentiate
\[
\begin{split}
\frac{d}{dt} \frac12 \int_{\Omega_t} |\nabla p|^2 \, dx &=\frac12 \int_{\Omega_t} |\nabla p|^2  \overbrace{\Div(v)}^{=0}\, dx + \int_{\Omega_t} ( \D_t \nabla p \cdot  \nabla p)  \, dx \\
&= \int_{\Omega_t} (\nabla \D_t p \cdot \nabla p)  \, dx + \int_{\Omega_t}  ([\D_t,\nabla] p \cdot  \nabla p)  \, dx.
\end{split}
\]
By \eqref{eq:comm1} and  $\|\nabla v\|_{L^\infty}\leq C$ we have a pointwise estimate $|[\D_t,\nabla] p| \leq C |\nabla v| |\nabla p| \leq C|\nabla p|$. Therefore we deduce 
\begin{equation} \label{eq:claim_reg1-1}
\begin{split}
\frac{d}{dt} \frac12 \int_{\Omega_t} |\nabla p|^2 \, dx &\leq \int_{\Omega_t} ( \nabla \D_t p \cdot  \nabla p )  \, dx + C \|\nabla p\|_{L^2(\Omega_t)}^2 \\
&= \int_{\Omega_t} \Div ( \D_t p \, \nabla p)  \, dx - \int_{\Omega_t}  \D_t p \, \Delta p   \, dx  +  C \|\nabla p\|_{L^2(\Omega_t)}^2 \\
&= \int_{\Sigma_t}  \D_t p \, \pa_\nu p  \, dx - \int_{\Omega_t}  \D_t p \, \Delta p   \, dx  +  C \|\nabla p\|_{L^2(\Omega_t)}^2\\
&\leq \| \D_t p\|_{L^2(\Sigma_t)}^2 + \| \pa_\nu p\|_{L^2(\Sigma_t)}^2 - \int_{\Omega_t}  \D_t p \, \Delta p   \, dx  +  C \|\nabla p\|_{L^2(\Omega_t)}^2.
\end{split}
\end{equation}

We have by \eqref{eq:claim1-2} that $\| \D_t p\|_{L^2(\Sigma_t)}^2 \leq C(1 + \|p\|_{H^1(\Sigma_t)}^2)$ and by Lemma \ref{lem:divcurl} \[
\| \pa_\nu  p\|_{L^2(\Sigma_t)}^2 \leq  C(\|p\|_{H^1(\Sigma_t)}^2 + \|\Delta p\|_{L^2(\Omega_t)}^2) \leq  C(1 + \|p\|_{H^1(\Sigma_t)}^2).
\]
We are left with the second last term in \eqref{eq:claim_reg1-1}. 

To that aim let   $u : \Omega_t \to \R$ be the solution of  
\[
\begin{cases}
-\Delta u = \Delta  p \quad \text{in }\, \Omega_t\\
u= 0  \quad \text{on }\, \Sigma_t .
\end{cases}
\]
Then it holds 
\[
-\int_{\Omega_t}  \D_t p \, \Delta p   \, dx = \int_{\Omega_t}  \D_t p \, \Delta u   \, dx = \int_{\Omega_t} \Delta \D_t p \,  u   \, dx  + \int_{\Sigma_t} \D_t p \, \pa_\nu u   \, d \H^2 .
\]
Since  $|\Delta u| = |\Delta p| \leq C$ and $u =0 $ on $\St$ it holds $\|u\|_{H^1(\Omega_t)} \leq C$ and by Lemma \ref{lem:divcurl} we deduce $\|\nabla u\|_{L^2(\St)} \leq C$.  We  may bound the last term on the RHS by \eqref{eq:claim1-2} 
\[
\int_{\Sigma_t} \D_t p \, \pa_\nu u   \, d \H^2 \leq \|  \D_t p\|_{L^2(\Sigma_t)}^2 + \| \pa_\nu u\|_{L^2(\Sigma_t)}^2 \leq C(1 + \|p\|_{H^1(\Sigma_t)}^2).
\]
We are thus left with the second last term.

We have   by   \eqref{eq:divcurl2}, by   Lemma \ref{lem:B^4}, by Sobolev embedding  and by interpolation  that 
\[
\begin{split}
    \|\nabla^2 u\|_{L^2(\Omega_t)}^2 &\leq C +  \int_{\St}|H_{\St}| |\nabla u|^2 \, d \H^2 \leq C + C_\e\|H_{\St}\|_{L^3(\St)}^3  + \e \|\nabla u \|_{L^{3}(\St)}^{3}\\
    &\leq C_\e(1+  \|p\|_{L^2(\St)}^2) +  \e\|\nabla u \|_{L^{4}(\St)}^{2} \|\nabla u\|_{L^2(\St)}\\
    &\leq C_\e(1+  \|p\|_{L^2(\St)}^2) +  C \e\|\nabla^2 u \|_{L^{2}(\Omega_t)}^{2}.
\end{split}
\]
Therefore it holds 
\[
\| u\|_{H^2(\Omega_t)}^2 \leq C(1+\|p\|_{H^1(\St)}^2).
\]
By Remark \ref{rmk:laplpreassure} we have
\[
\Delta \D_t p = \Div \Div (v \otimes \nabla p) + \Div( R_{bulk}^0).
\]
Therefore  by integrating   by parts
\[
\begin{split}
\int_{\Omega_t} \Delta \D_t p \,  u   \, dx &= \int_{\Omega_t} ( v \otimes \nabla p) : \nabla^2 u   \, dx  + \int_{\Omega_t} R_{bulk}^0 \star \nabla u \,  dx - \int_{\St} (\nabla p \cdot \nu) (\nabla u \cdot v) \, d \H^2 \\
&\leq C(1+\|p\|_{H^1(\St)}^2 + \|\nabla p\|_{L^2(\Omega_t)}^2).
\end{split}
\]

We deduce by \eqref{eq:claim_reg1-1} and by the above estimates that 
\[
\frac{d}{dt} \frac12 \| \nabla p \|_{L^2(\Omega_t)}^2  \leq C(1+ \|p\|_{H^1(\Sigma_t)}^2 + \| \nabla p \|_{L^2(\Omega_t)}^2 ). 
\]
This implies
\[
\frac{d}{dt} \log(1 + \| \nabla p \|_{L^2(\Omega_t)}^2)\leq C(1+ \|p\|_{H^1(\Sigma_t)}^2) 
\]
and the claim follows from Lemma \ref{lem:claim1}.
\end{proof}

An important consequence  of Proposition \ref{prop:claim-reg1} is that by Lemma \ref{lem:press-curv} we have the following  bound for the curvature  
\begin{equation}
    \label{eq:B-4-bound}
\| B\|_{L^4(\Sigma_t)} +  \| B\|_{H^{\frac12}(\Sigma_t)}\leq C.
\end{equation}
This means that  from now on we may use the general interpolation inequality from Proposition \ref{prop:interpolation}. 

At the end of this section  we improve  Lemma \ref{lem:claim1}. We recall the definition of the energy quantity $E_1(t)$ from \eqref{def:E_l}
\[
E_1(t) = \|\D_t^2 v\|_{L^{2}(\Omega_t)}^2 + \|\nabla  p\|_{H^{\frac32}(\Omega_t)}^2 + \|v\|_{H^{3}(\Omega_t)}^2 + \|\pa_\nu p\|_{H^1(\St)}^2 +1.
\]
In particular, $E_1(0)$ denotes the above quantity at time $t = 0$. It is clear that 
\[
\|p\|_{H^1(\Omega_t)}^2 \leq  E_1(t).
\]

\begin{lemma}
    \label{lem:claim2}
    Assume that the a priori estimates \eqref{eq:apriori_est} hold for $T>0$.  Then 
     \[
     \int_0^T \| p\|_{H^2(\Omega_t)}^2 \, dt \leq C,
     \]
     where the constant $C$ depends on $M, T$ and on $E_1(0)$. 
\end{lemma}

\begin{proof}
The proof is similar to Lemma \ref{lem:claim1}. 
This time we differentiate
\[
\Phi(t) := -\int_{\Sigma_t} p \Delta_{\St} v_n\, d \H^2.
\]
Note that  by the a priori estimates \eqref{eq:apriori_est} and by Proposition \ref{prop:claim-reg1}, $\Phi$  is uniformly bounded on $[0,T]$. Note also that by Proposition \ref{prop:claim-reg1} it holds 
\[
\sup_{t < T} \|p\|_{H^1(\Omega_t)}^2 \leq C, 
\]
where the  constant depends on $T, \Lambda_T$ and on $E_1(0)$. 

We calculate as in \eqref{eq:claim1-1} by using \eqref{eq:claim1-2} and $\|v_n\|_{H^2(\Sigma_t)} \leq C$ that  
\[
\begin{split}
\frac{d}{dt} \Phi(t) &\leq C + \|\D_t p\|_{L^2(\Sigma_t)}^2   -\int_{\Sigma_t} p\, (\D_t \Delta_{\St} v_n)\, d \H^2\\
&\leq C(1+ \|p\|_{H^1(\Sigma_t)}^2) -\int_{\Sigma_t} p\, (\D_t \Delta_{\St} v_n)\, d \H^2.
\end{split}
\]
To bound the last term we recall that by \eqref{eq:comm3} it holds 
\[
(\D_t \Delta_{\St} v_n) =  \Delta_{\St} (\D_t v_n) +  \nabla_\tau^2 v_n \star  \nabla v -  \nabla_\tau v_n \cdot \Delta_{\St} v + B \star \nabla v\star \nabla_\tau v_n. 
\]
Therefore by $\|\nabla v\|_{L^\infty} +\|v_n\|_{H^2(\Sigma_t)} \leq C$, by $\|\nabla_\tau v_n\|_{L^4(\Sigma_t)} \leq \|v_n\|_{H^2(\Sigma_t)} $  and by \eqref{eq:B-4-bound} it holds 
\begin{equation}
    \label{eq:claim2-1}
\begin{split}
&-\int_{\Sigma_t} p\, (\D_t \Delta_{\St} v_n)\, d \H^2 \leq - \int_{\Sigma_t} p\, \Delta_{\St} (\D_t v_n)\, d \H^2 +\int_{\Sigma_t} p\, \nabla_\tau v_n \cdot \Delta_{\St} v\, d \H^2 \\
&\,\,\,\,\,\,\,\, + C(1+ \|p\|_{L^2(\Sigma_t)}\|B\|_{L^4(\Sigma_t)}\|\bar \nabla v_n\|_{L^4(\Sigma_t)}) \\
&\leq-\int_{\Sigma_t} p\, \Delta_{\St} (\D_t v \cdot \nu)\, d \H^2 - \int_{\Sigma_t} p\, \Delta_{\St} (\D_t \nu \cdot v)\, d \H^2  - \int_{\Sigma_t}  (\nabla_\tau ( p\,  \nabla_\tau v_n) \cdot \nabla_\tau v) \, d \H^2 + C .
\end{split}
\end{equation}

We may write the first term on RHS in \eqref{eq:claim2-1}  by the formula \eqref{eq:divcurl2},  $\D_t v = - \nabla p$   and by the estimate \eqref{eq:B-4-bound} 
\[
\begin{split}
-\int_{\Sigma_t} p\, \Delta_{\St} (\D_t v \cdot \nu)\, d \H^2 &= \int_{\Sigma_t} \Delta_{\St} p\, \pa_\nu p\, d \H^2\\
&\leq -\frac12 \int_{\Omega_t} (|\nabla^2 p|^2-|\Delta p|^2) \, dx + C\int_{\Sigma_t}|B||\nabla p|^2\, d \H^2\\
&\leq -\frac12 \int_{\Omega_t} |\nabla^2 p|^2 \, dx + C+ C_\e \|B\|_{L^4}+ \e \|\nabla p\|_{L^{\frac83}(\Sigma_t)}^2\\
&\leq  -\frac12 \int_{\Omega_t} |\nabla^2 p|^2 \, dx + C_\e  + \e \|\nabla p\|_{L^{\frac83}(\Sigma_t)}^2. 
\end{split}
\]
By the Sobolev embedding it holds
$\|\nabla p\|_{L^{\frac83}(\Sigma_t)}^2 \leq C\|\nabla p\|_{H^{1/2}(\Sigma_t)}^2 \leq C(1+ \|\nabla^2 p\|^2_{L^2(\Omega_t)}).$
Therefore by choosing $\e>0$ small we deduce
\[
-\int_{\Sigma_t} p\, \Delta_{\St} (\D_t v \cdot \nu)\, d \H^2 \leq -\frac13 \int_{\Omega_t} |\nabla^2 p|^2 \, dx  + C_\e.
\]
We bound the third term on RHS in \eqref{eq:claim2-1} simply by   $\|\nabla v\|_{L^\infty} +\|v_n\|_{H^2(\Sigma_t)} \leq C$ as 
\[
-\int_{\Sigma_t}  \nabla_\tau ( p\,  \nabla_\tau v_n)  \cdot \nabla_\tau v\, d \H^2 \leq C(1+ \| p\|_{H^1(\Sigma_t)}^2 ).
\]
For the remaining  term in \eqref{eq:claim2-1}  we recall \eqref{eq:normal2} which states $\D_t \nu = -  \nabla_\tau v_n + B v_\tau$. Therefore we obtain by Lemma \ref{lem:B^4} and by $\|\nabla v\|_{L^\infty} +\|v_n\|_{H^2(\Sigma_t)} \leq C$ that 
\[
\begin{split}
-\int_{\Sigma_t} p\, \Delta_{\St} (\D_t \nu \cdot v)\, d \H^2  &= \int_{\Sigma_t} \langle  \bar \nabla p ,  \bar \nabla(\D_t \nu \cdot v)\rangle \, d \H^2 \\
&\leq C(1+ \| p\|_{H^1(\Sigma_t)}^2 + \| B \|_{H^1(\Sigma_t)}^2 + \|B\|_{L^4(\Sigma_t)}^4)\\
&\leq C(1+ \| p\|_{H^1(\Sigma_t)}^2).
\end{split}
\]
Hence, we have 
\[
\frac{d}{dt} \Phi(t) \leq - \frac13 \int_{\Omega_t} |\nabla^2 p|^2 \, dx + C(1+ \| p\|_{H^1(\Sigma_t)}^2)
\]
and the claim follows from Lemma \ref{lem:claim1}.
\end{proof}

\section{Energy estimates}

As we mentioned before, the fundamental property of the solution of \eqref{system} is the conservation of the  energy \eqref{energy}. In this section we define high order energy functions and show that their derivatives are controlled by the  quantity   \eqref{def:E_l} of the same order. This will be the first step in proving that the high order energy quantities remain bounded along the flow. 

We define the energy of order $l \geq 1$ as
\begin{equation}
\label{eq:high-energy}
\begin{split}
\mathcal{E}_l(t) =\frac12 \int_{\Omega_t} &|\D_t^{l+1} v|^2 \, dx +  \frac12\int_{\Sigma_t} |\nabla_\tau (\D_t^l v \cdot \nu)|^2 \, d \H^2 \\
&-  \frac{Q(t)}{2} \int_{\Omega_t^c} |\nabla (\pa_t^{l+1} U) |^2 \, dx  +  \int_{\Omega_t} |\nabla^{\lfloor \frac12(3l+1) \rfloor}\omega|^2 \, dx,
\end{split}
\end{equation}
where $\lfloor \frac12(3l+1) \rfloor$ is the integer part of  $\frac12(3l+1)$,  $\omega$ is the curl of $v$ defined as 
\[
\omega = \curl v= \nabla v - \nabla v^T
\]
and  $\Qt$ is defined in \eqref{def:constant-q}.

In this section we  calculate $\frac{d}{dt}\mathcal{E}_l$ and estimate it in terms of the $E_l(t)$, which we recall is  defined in \eqref{def:E_l} as
\[
E_l(t) =  \sum_{k=0}^{l}  \|\D_t^{l+1-k}v\|_{H^{\frac32 k}(\Omega_t)}^2 + \|v\|_{H^{\lfloor \frac32(1+1)\rfloor}(\Omega_t)}^2+ 
 \|\D_t^l v \cdot \nu \|_{H^1(\Sigma_t)}^2+1.
\]
In particular, it  holds $\En_l(t) \leq C E_l(t)$. We state the main results of this section below  and prove them later. 
\begin{proposition}
    \label{prop:case-1}
    Assume that  the a priori estimates \eqref{eq:apriori_est} hold for $T>0$. Then for all $t <T$ it holds
    \[
\frac{d}{dt} \mathcal{E}_1(t) \leq C(1 + \|p\|_{H^2(\Omega_t)}^2)E_1(t),
\]
where the constant depends on $M, T$ and  $E_1(0)$, i.e.,  $E_1(t)$ at time $t=0$.
\end{proposition}

\begin{proposition}
    \label{prop:case-l}
    Let $l \geq 2$ and assume that  \eqref{eq:apriori_est} and $E_{l-1}(t) \leq M$ hold for all $t \in [0,T)$. Then  all $t <T$ it holds  
    \[
\frac{d}{dt} \mathcal{E}_l(t) \leq CE_l(t)
\]
where the constant depends on $M, l, T$ and on $\sup_{t < T} E_{l-1}(t)$.
\end{proposition}

The proof of the both above energy estimates is based on the calculations of the differential of $\mathcal{E}_l(t)$, which we state first for all $l \geq 1$. In the proof of Proposition \ref{prop:case-1} and Proposition \ref{prop:case-l} we then need to estimate the remainder terms by the quantity $E_l(t)$.  

We begin by differentiating the first term of $\En_l(t)$ in \eqref{eq:high-energy} and obtain by $\Div v = 0$, by \eqref{eq:comm-bulk} and by the definition of $E_l(t)$ that 
\begin{equation} \label{eq:diff1}
\begin{split}
&\frac{d}{dt} \frac12 \int_{\Omega_t}|\D_t^{l+1} v|^2 \, dx = \int_{\Omega_t} ( \D_t^{l+2} v \cdot  \D_t^{l+1} v)  \, dx\\
&= -  \int_{\Omega_t} (\nabla \D_t^{l+1}  p \cdot \D_t^{l+1} v )  \, dx - \int_{\Omega_t} ([\D_t^{l+1},\nabla]  p \cdot  \D_t^{l+1} v )  \, dx \\
&\leq   -  \int_{\Omega_t} (\nabla \D_t^{l+1}  p \cdot \D_t^{l+1} v)  \, dx + \|\D_t^{l+1} v \|_{L^2(\Omega_t)}^2 + \|[\D_t^{l+1},\nabla]  p \|_{L^2(\Omega_t)}^2 \\
&= -  \int_{\Omega_t} \Div (\D_t^{l+1}  p \,  \D_t^{l+1} v)  \, dx + \int_{\Omega_t} \D_t^{l+1}  p \, \Div (\D_t^{l+1} v)  \, dx  + E_l(t) + \|R_{bulk}^l \|_{L^{2}(\Omega_t)}^2\\
&= -  \int_{\Sigma_t} \D_t^{l+1}  p \, (\D_t^{l+1} v \cdot \nu)  \, d\H^2   + E_l(t) + \|R_{bulk}^l \|_{L^{2}(\Omega_t)}^2 + \int_{\Omega_t} \D_t^{l+1}  p \, \Div (\D_t^{l+1} v)  \, dx.
\end{split}
\end{equation}

Next we differentiate the second term in the energy and obtain by $\|\nabla v\|_{L^\infty} \leq C$ and \eqref{eq:comm2}
\begin{equation} \label{eq:diff2}
\begin{split}
&\frac{d}{dt} \frac12 \int_{\Sigma_t} |\nabla_\tau (\D_t^{l} v \cdot \nu)|^2\, d\H^2 \\
&=\int_{\Sigma_t} ( \D_t \nabla_\tau (\D_t^{l} v \cdot \nu) \cdot \nabla_\tau (\D_t^{l} v \cdot \nu))  \, d\H^2 +   \frac12 \int_{\Sigma_t} |\nabla_\tau (\D_t^{l} v \cdot \nu)|^2 (\Div_\tau v) \, d\H^2 \\
&\leq   \int_{\Sigma_t} (\nabla_\tau \D_t (\D_t^{l} v \cdot \nu) \cdot \nabla_\tau (\D_t^{l} v \cdot \nu)) \, d\H^2 + C  \|\D_t^{l} v \cdot \nu\|_{H^1(\Sigma_t)}^2\\
&=  -\int_{\Sigma_t} (\Delta_{\Sigma_t} (\D_t^{l} v \cdot \nu)) ( \D_t (\D_t^{l} v \cdot \nu))\, d\H^2  +C E_l(t) \\
&= -\int_{\Sigma_t} (\Delta_{\Sigma_t} (\D_t^{l} v \cdot \nu)) ( \D_t^{l+1} v \cdot \nu)\, d\H^2 \\
&\,\,\,\,\,\,\,\,\,\,\,\,\,\,\,+ \int_{\Sigma_t} ( \nabla_{\tau} (\D_t^{l} v \cdot \nu) \cdot \nabla_\tau ( \D_t^{l} v \cdot \D_t \nu)) \rangle\, d\H^2 +C E_l(t)\\
&\leq -\int_{\Sigma_t} (\Delta_{\Sigma_t} (\D_t^{l} v \cdot \nu)) ( \D_t^{l+1} v \cdot \nu)\, d\H^2 +C E_l(t) + \|\D_t^{l} v \cdot \D_t \nu\|_{H^1(\Sigma_t)}^2. 
\end{split}
\end{equation}

We differentiate the third  term, use the fact that $\pa_t^{l+1} U$ is harmonic and Lemma \ref{lem:mat-capa} and have (recall that it holds  $\|\nabla \pa_t^{l+1} U\|_{L^2(\Omega_t^c)} \leq C\| \nabla \pa_t^{l+1} U\|_{L^2(\Sigma_t)}$)
\begin{equation} \label{eq:diff3}
\begin{split}
\frac{d}{dt} &-\frac{Q(t)}{2} \int_{\Omega_t^c} |\nabla \pa_t^{l+1} U|^2\, dx = - Q(t)\int_{\Omega_t^c} \langle \nabla \pa_t^{l+2} U,\nabla \pa_t^{l+1} U \rangle\, dx \\
&\,\,\,\,\,\,\,\,+ \frac{Q(t)}{2} \int_{\Sigma_t} |\nabla \pa_t^{l+1} U|^2 v_n \, d\H^2  -\frac{Q'(t)}{2} \int_{\Omega_t^c} |\nabla \pa_t^{l+1} U|^2\, dx \\
&\leq   Q(t)\int_{\Sigma_t}  \pa_t^{l+2} U \, (\pa_{\nu}  \pa_t^{l+1} U)\, d\H^2 + C\|\nabla \pa_t^{l+1} U\|_{L^2(\Sigma_t)}^2 \\
&=  - Q(t)\int_{\Sigma_t} (\pa_{\nu} U (\D_t^{l+1} v \cdot \nu) +R_U^{l})\, (\pa_{\nu}  \pa_t^{l+1} U)\, d\H^2 + C\|\nabla \pa_t^{l+1} U\|_{L^2(\Sigma_t)}^2\\
&\leq  -Q(t)\int_{\Sigma_t} (\pa_{\nu} U \,  \pa_{\nu}  \pa_t^{l+1} U) (\D_t^{l+1} v \cdot \nu) \, d\H^2 + \|R_U^{l}\|_{L^2(\Sigma_t)}^2 +C\| \nabla \pa_t^{l+1} U\|_{L^2(\Sigma_t)}^2,  
\end{split}
\end{equation}
where $R_U^l$ in the remainder term defined in \eqref{def:R-U} and $Q'(t) = \frac{d}{dt }\Qt$, where $\Qt$ is defined in \eqref{def:constant-q}. Recall that in the proof of Lemma \ref{lem:estimate-R-p-1} we proved that $Q'(t)$ and $Q''(t)$ are bounded, see \eqref{eq:Q-2-derivat}.

Finally, we differentiate the fourth term involving the curl. To that aim we denote $\lambda_l := \lfloor \frac12(3l+1)\rfloor$.
We have by Lemma \ref{lem:curl} and by \eqref{eq:est-product}
\beq\label{eq:diff4}
\begin{split}
    \frac{d}{dt}\int_{\Omega_t} |\nabla^{\lambda_l}\omega|^2\, dx
    &= \int_{\Omega_t} \nabla v \star \nabla^{\lambda_l} \omega \star \nabla^{\lambda_l}\omega\, dx    +  C\|\nabla^{\lambda_l}\omega\|_{L^2(\Omega_t)}^2 \\
    &\,\,\,\,\,\,\,\,\,\,\,\,\,\,\,+ \sum_{|\alpha|\leq \lambda_l}\|\nabla^{1+\alpha_1} v \star \nabla^{1+\alpha_2} v \|_{L^2(\Omega_t)}^2
    \\
    &\leq C\|\nabla^{\lambda_l}\omega\|_{L^2(\Omega_t)}^2  + C\|\nabla v\|_{L^\infty(\Omega_t)}^2 \|\nabla v\|_{H^{\lambda_l}(\Omega_t)}^2 \\
    &\leq C\|\nabla^{\lambda_l}\omega\|_{L^2(\Omega_t)}^2 + C\| v\|_{H^{\lfloor \frac32(l+1)\rfloor}(\Omega_t)}^2 \leq C E_{l}(t).
\end{split}
\eeq

Let us next show that the highest order terms in \eqref{eq:diff1},  \eqref{eq:diff2}  and \eqref{eq:diff3} cancel out. Indeed, this follows from  Lemma \ref{formula:Dtp} which states that 
\[
\D_t^{l+1}p = -\Delta_{\St} (\D_t^l v \cdot \nu) - Q(t) \, \pa_{\nu} U \,  (\pa_{\nu} \pa_t^{l+1} U) + R_p^{l},
\]
where $R_p^l$ denotes the error term defined in \eqref{eq:R-p-in-3} and $\Qt$ is defined in \eqref{def:constant-q}. Note that we may estimate  
\[
\big|\int_{\St} R_p^{l} (\D_t^{l+1} v \cdot \nu ) \, d \H^2 \big| \leq  \|R_p^l\|_{H^{\frac12}(\St)}^2 + \|\D_t^{l+1} v \cdot \nu\|_{H^{-\frac12}(\St)}^2.
\]
By divergence theorem and by Lemma \ref{lem:curl} we have 
\[
\|\D_t^{l+1} v \cdot \nu\|_{H^{-\frac12}(\St)}^2 \leq C\|\D_t^{l+1} v\|_{L^2(\Omega_t)}^2 + \|R_{\Div}^l\|_{L^2(\Omega_t)}^2 \leq CE_l(t) + \|R_{\Div}^l\|_{L^2(\Omega_t)}^2. 
\]
Therefore we have by \eqref{eq:diff1}, \eqref{eq:diff2}, \eqref{eq:diff3} and \eqref{eq:diff4}  that
\begin{equation} \label{eq:all-errors}
\begin{split}
\frac{d}{dt}  \En_l(t) &\leq C E_l(t) + \|R_{bulk}^l\|_{L^2(\Omega_t)}^2 + \|R_{U}^l\|_{L^2(\St)}^2 +\|R_{\Div}^l\|_{L^2(\Omega_t)}^2+ \|R_p^l\|_{H^{\frac12}(\St)}^2\\
&\,\,\,\,+ C \|\nabla \pa_t^{l+1} U\|_{L^2(\St)}^2 +\| \D_t^l v \cdot \D_t \nu\|_{H^1(\St)}^2 + \int_{\Omega_t} \D_t^{l+1}  p \, \Div (\D_t^{l+1} v)  \, dx.
\end{split}
\end{equation}

We need thus to bound the remainder terms in \eqref{eq:all-errors}. As we mentioned before, we prove the energy bounds by induction such that we bound $E_l(t)$ when  we know that $E_{l-1}(t)$ is already bounded. The difficulty is to bound the first order quantity $E_1(t)$ and  we do this separately in Proposition \ref{prop:case-1}.

\begin{proof}[\textbf{Proof of Proposition \ref{prop:case-1}}]
First we recall that Proposition \ref{prop:claim-reg1} implies
\[
 \sup_{t \in (0,T)} \|\nabla p\|_{L^2(\Omega_t)}^2  \leq e^{CT}(1+\|\nabla p\|_{L^2(\Omega_0)}^2) \leq C,
\]
where the constant $C$ on the RHS depends on $T, \Lambda_T$  defined in \eqref{eq:apriori_est} and on $E_1(0)$, which is the energy quantity $E_1(t)$ at time $t=0$. Then we have the bound \eqref{eq:B-4-bound} which we recall is 
\[
\|B\|_{L^4(\St)} + \|B\|_{H^{\frac12}(\St)} \leq C
\]
for all $t <T$.

We recall the estimate \eqref{eq:all-errors} for $l=1$ and use Lemma \ref{lem:R-div-est}, Lemma \ref{lem:bound-R_B}, estimate \eqref{eq:est-R-U-1} from Lemma \ref{lem:estimate-R-U} and Lemma \ref{lem:estimate-R-p-1} to deduce
\[
\frac{d}{dt}   \En_1(t) \leq C(1+ \|p\|_{H^2(\Omega_t)}^2) E_1(t) + \| \D_t v \cdot \D_t \nu\|_{H^1(\St)}^2 + \int_{\Omega_t} \D_t^{2}  p \, \Div (\D_t^{2} v)  \, dx.
\]
Hence, we need to bound the two last terms. The second last term is easy to treat and we merely claim that it holds 
\[
\| \D_t v \cdot \D_t \nu\|_{H^1(\St)}^2 \leq C(1+ \|p\|_{H^2(\Omega_t)}^2) E_1(t).
\]
Indeed, this follows from $\D_t v = -\nabla p$, $\D_t \nu = -(\nabla_\tau v)^T \nu$ from \eqref{eq:normal1} and from Proposition \ref{prop:kato-ponce}. We leave the details for the reader. The last term is challenging and we will prove that  
\beq
\label{eq:case-1-1}
\int_{\Omega_t} \D_t^{2}  p \, \Div (\D_t^{2} v)  \, dx \leq C(1+ \|p\|_{H^2(\Omega_t)}^2) E_1(t).
\eeq

The argument is similar than in the proof of Proposition \ref{prop:claim-reg1}. The  idea is to use the fact that the term $\Div (\D_t^{2} v)$ is  lower order due to the fact that $\Div v = 0$. Indeed, we have by Lemma \ref{lem:curl} that $\Div (\D_t^{2} v) =  R^{1}_{\Div}$, where $R^{1}_{\Div}$ is defined in \eqref{eq:def-R-div} and is lower order than $\D_t^{2} v$. To this aim let $u$ be a solution of
\[
\begin{cases}
&-\Delta u = \Div (\D_t^{2} v), \quad  \text{in } \, \Omega_t\\
&\,\,\,\,\,\,\,u = 0 \qquad \text{on }\, \St.
\end{cases}
\]
We have by integration by parts 
\[
\int_{\Omega_t} \D_t^{2}  p \, \Div (\D_t^{2} v)  \, dx = -\int_{\Omega_t} \D_t^{2}  p \, \Delta u  \, dx =   -\int_{\Omega_t} \Delta \D_t^{2}  p \,  u  \, dx - \int_{\St} \D_t^2 p\, \pa_\nu u\, d \H^2.
\]
We use Remark \ref{rmk:laplpreassure} and write 
\[
-\Delta \D_t^{2}  p = \Div \Div (v \otimes \D_t^2 v) + \Div (R_{bulk}^1).
\]
By integration by parts we deduce
\beq \label{eq:div-term-1}
\begin{split}
-\int_{\Omega_t} \Delta \D_t^{2}  p \,  u  \, dx = &\int_{\Omega_t} (v \otimes \D_t^2 v): \nabla^2 u  \, dx + \int_{\Omega_t} R_{bulk}^1 \star \nabla u  \, dx \\
&- \int_{\St} (\D_t^2 v \cdot \nu)  (\nabla u \cdot v) \, d \H^2.
\end{split}
\eeq
We use divergence theorem, the definition of $E_1(t)$, $\Div (\D_t^{2} v) =  R^{1}_{\Div}$  and Lemma \ref{lem:R-div-est} for  the last term 
\[
\begin{split}
- \int_{\St} (\D_t^2 v \cdot \nu)  (\nabla u \cdot v) \, d \H^2 &= - \int_{\Omega_t} \Div\big( (\nabla u \cdot v)\D_t^2 v\big)\\
&\leq CE_1(t) + \|u\|_{H^2(\Omega_t)}^2 +C\|R^{1}_{\Div}\|_{L^2(\Omega_t)}^2\\
&\leq  C(1+\|p\|_{H^2(\Omega_t)}^2)E_1(t)   + \|u\|_{H^2(\Omega_t)}^2.
\end{split}
\]
Since $-\Delta u = \Div (D_t^2 v)$,  we may  use the inequality  \eqref{eq:dirichlet-2} and  Lemma \ref{lem:R-div-est}  to obtain
\[
 \|u\|_{H^2(\Omega_t)}^2 \leq C\|R^{1}_{\Div}\|_{L^2(\Omega_t)}^2 \leq C(1+\|p\|_{H^2(\Omega_t)}^2)E_1(t).
\]
Therefore we have by \eqref{eq:div-term-1}, by the definition of $E_1(t)$ and by Lemma \ref{lem:bound-R_B} that 
\beq \label{eq:div-term-2}
\begin{split}
\int_{\Omega_t} \D_t^{2}  p \, \Div (\D_t^{2} v)  \, dx \leq   C(1+\|p\|_{H^2(\Omega_t)}^2)E_1(t) - \int_{\St} \D_t^2 p\, \pa_\nu u\, d \H^2. 
\end{split}
\eeq

We proceed by using Lemma \ref{formula:Dtp} to write
\[
\D_t^2 p = - \Delta_{\St} (\D_t v \cdot \nu) - Q(t) \nabla U \cdot \nabla \pa_t^2 U + R_p^1,
\]
where $\Qt$ is defined in \eqref{def:constant-q}. We integrate by parts on $\St$  the term 
\[
- \int_{\St}\Delta_{\St} (\D_t v \cdot \nu) \pa_\nu u\, d \H^2 = \int_{\St}\langle \bar \nabla (\D_t v \cdot \nu), \bar \nabla  \pa_\nu u \rangle \, d \H^2
\]
and deduce
\[
\begin{split}
-\int_{\St} \D_t^2 p\, \pa_\nu u\, d \H^2 \leq &\|\D_t v \cdot \nu\|_{H^1(\St)}^2 + \|\pa_\nu u\|_{H^1(\St)}^2 \\
&+ C\|\nabla \pa_t^2 U\|_{L^{2}(\St)}^2 + C\|R_p^1\|_{L^{2}(\St)}^2.
\end{split}
\]
Lemma \ref{lem:estimate-R-p-1}  and  \eqref{eq:est-R-U-8} imply 
\[
\|\nabla \pa_t^2 U\|_{L^{2}(\St)}^2 + \|R_p^1\|_{H^{\frac12}(\St)}^2 \leq C(1+\|p\|_{H^2(\Omega_t)}^2)E_1(t).
\]
We use   Proposition \ref{prop:dirichlet1}, Lemma \ref{lem:curl} and   Lemma \ref{lem:R-div-est}   to deduce
\[
\begin{split}
\|\pa_\nu u\|_{H^1(\St)}^2 &\leq C\|\Div(\D_t^2 v)\|_{H^{\frac12}(\Omega_t)}^2 \leq C\|R_{\Div}\|_{H^{\frac12}(\Omega_t)}^2 \\
&\leq  C(1+\|p\|_{H^2(\Omega_t)}^2)E_1(t).
\end{split}
\]
Therefore since $\|\D_t v \cdot \nu\|_{H^1(\St)}^2  \leq E_1(t)$ we obtain by combining the previous estimates 
\[
\int_{\St} \D_t^2 p\, \pa_\nu u\, d \H^2 \leq C(1+\|p\|_{H^2(\Omega_t)}^2)E_1(t). 
\]
Hence, \eqref{eq:div-term-1} implies  \eqref{eq:case-1-1}  which concludes the proof. 
\end{proof}

By a similar argument we prove the higher order case. We begin again with   \eqref{eq:all-errors} and we need to bound the remainder terms. 

\begin{proof}[\textbf{Proof of Proposition \ref{prop:case-l}}]
The assumption $E_{l-1}(t) \leq C$ and  Lemma \ref{lem:press-curv} imply the curvature bound $\|B\|_{H^{\frac32l-1}(\St)} \leq C$. Thus we may use   the estimate \eqref{eq:all-errors}, Lemma   \ref{lem:R-div-est}, Lemma \ref{lem:bound-R_B}, estimate \eqref{eq:est-R-U-2} from Lemma \ref{lem:estimate-R-U} and Lemma \ref{lem:estimate-R-p-1} to deduce
\[
\frac{d}{dt} \En_l(t) \leq C E_l(t) + \| \D_t^l v \cdot \D_t \nu\|_{H^1(\St)}^2 + \int_{\Omega_t} \D_t^{l+1}  p \, \Div (\D_t^{l+1} v)  \, dx.
\]
Hence, we need to bound the two last terms. Again the second last term is easy to treat and we merely sketch it. We have by \eqref{eq:normal1} $\D_t \nu = - (\nabla_\tau v)^T \nu$ and have by Proposition \ref{prop:kato-ponce} 
\[
\| \D_t^l v \cdot \D_t \nu\|_{H^1(\St)}^2 \leq C \| \D_t \nu\|_{L^\infty(\St)}^2\| \D_t^l v\|_{H^1(\St)}^2 + C \| \D_t \nu\|_{W^{1,4}(\St)}^2\| \D_t^l v\|_{L^4(\St)}^2 \leq C E_l(t).
\]

We treat the last term similarly as in the previous proof and claim that it holds 
\beq
\label{eqcase-l-1}
\int_{\Omega_t} \D_t^{l+1}  p \, \Div (\D_t^{l+1} v)  \, dx \leq C E_l(t).
\eeq
Since the argument is almost the same as with  \eqref{eq:case-1-1} we only sketch it. We let $u$ be the solution of 
\[
\begin{cases}
&-\Delta u = \Div (\D_t^{l+1} v) \quad  \text{in } \, \Omega_t,\\
&\,\,\,\,\,\,\,u = 0 \qquad \text{on }\, \St.
\end{cases}
\]
Again by integration by parts and by using Remark \ref{rmk:laplpreassure}, Lemma \ref{lem:R-div-est} and Lemma \ref{lem:bound-R_B} we obtain the higher order version of \eqref{eq:div-term-2} which reads as 
\[
\int_{\Omega_t} \D_t^{l+1}  p \, \Div (\D_t^{l+1} v)  \, dx \leq C E_l(t) - \int_{\St} \D_t^{l+1}p \, \pa_\nu u \, d \H^2.
\]
Lemma \ref{formula:Dtp} yields
\[
\D_t^{l+1}p = -\Delta_{\St}(\D_t^l v \cdot \nu) - Q(t)\nabla U \cdot \nabla \pa_t^{l+1} U + R_p^l,
\]
where $\Qt$ is defined in \eqref{def:constant-q}. By integration by parts on $\St$ we have 
\[
\begin{split}
- \int_{\St} \D_t^{l+1}p \, \pa_\nu u \, d \H^2 \leq &\|\D_t^l v \cdot \nu\|_{H^1(\St)}^2 +  \|\pa_\nu u\|_{H^1(\St)}^2 \\
&+ C\|\nabla \pa_t^{l+1} U\|_{L^{2}(\St)}^2 + C\|R_p^l\|_{L^{2}(\St)}^2.
\end{split}
\]
Note that  $\|\D_t^l v \cdot \nu\|_{H^1(\St)}^2 \leq E_l(t)$.
By Lemma \ref{lem:estimate-R-U} and Lemma \ref{lem:estimate-R-p-2} we have 
\[
\|\nabla \pa_t^{l+1} U\|_{H^{\frac12}(\St)}^2 + \|R_p^l\|_{H^{\frac12}(\St)}^2 \leq CE_l(t).
\] 
Finally by Proposition \ref{prop:dirichlet1}, by the curvature bound $\|B\|_{H^2(\St)} \leq \|B\|_{H^{\frac32l-1}(\St)} \leq  C$, by Lemma \ref{lem:curl} and by Lemma \ref{lem:R-div-est} we have
\[
\|\pa_\nu u\|_{H^{1}(\St)}^2 \leq  C\|\Div (\D_t^{l+1} v)\|_{H^{\frac12}(\Omega_t)}^2 \leq C \|R_{\Div}^l\|_{H^{\frac12}(\Omega_t)}^2 \leq C E_l(t).
\]
This proves \eqref{eqcase-l-1} and concludes the proof.
\end{proof}

\section{Higher regularity estimates}

Let us recall the definition of the energies  $E_l(t)$ and $\En_l(t)$ for $l \geq 1$ in \eqref{def:E_l} and \eqref{eq:high-energy} respectively. In the previous section we proved energy estimates where we control the derivative of $\En_l(t)$ by $E_l(t)$.  In this section we complete the estimate and prove that the energy  $\En_l(t)$ in fact controls $E_l(t)$. This,  together with the results in the previous section, will give us control for $\En_l(t)$ and implies   the regularity of the flow.

Note that the energy $\En_l(t)$ is  defined in \eqref{eq:high-energy} as
\[
\begin{split}
\mathcal{E}_l(t) =\frac12 \int_{\Omega_t} &|\D_t^{l+1} v|^2 \, dx +  \frac12\int_{\Sigma_t} |\nabla_\tau (\D_t^l v \cdot \nu)|^2 \, d \H^2 \\
&-  \frac{Q(t)}{2} \int_{\Omega_t^c} |\nabla (\pa_t^{l+1} U) |^2 \, dx  +  \int_{\Omega_t} |\nabla^{\lfloor \frac12(3l+1) \rfloor}\curl v |^2 \, dx,
\end{split}
\]
where $\Qt$ is defined in \eqref{def:constant-q}. Since $Qt \geq 0$,  the energy  has one negative term and we define its positive part as
\begin{equation}
\label{def:Eplus}
\mathcal{E}_l^+(t) := 1+ \frac12\|\D_t^{l+1}v\|_{L^{2}(\Omega_t)}^2 + \|\curl v\|_{H^{\lfloor \frac32l+\frac12\rfloor}(\Omega_t)}^2+
\frac12 \|\D_t^l v \cdot \nu \|_{H^1(\Sigma_t)}^2.
\end{equation}
Then it holds 
\[
c \,  \En_l^+(t) \leq \En_l(t) + \frac{Q(t)}{2} \int_{\Omega_t^c} |\nabla (\pa_t^{l+1} U)|^2 \, dx+ 1,  
\]
for $c>0$. 

The first main result of this section states that the energy $ \mathcal{E}_1(t)$ controls $E_1(t)$. For later purpose we need this bound when the boundary $\St$ is $C^{1,\alpha}$-regular but the velocity field is only bounded in $W^{1,4}$.   This makes the statement slightly heavy. 
\begin{proposition}
    \label{prop:reg-est-1}
    Assume that $\St$ is uniformly $C^{1,\alpha}(\Gamma)$-regular and the pressure and the velocity satisfies
    \[
    \|p\|_{H^1(\Omega_t)} + \|v\|_{W^{1,4}(\St)} + \|v\|_{W^{1,4}(\Omega_t)} \leq M. 
    \]
    Then there are constants $C$ and $C_0$   such that 
    \[
    E_1(t) \leq C(C_0 + \mathcal{E}_1(t)). 
\]
The constants depend on  $\sigma_t$ defined in \eqref{eq:height_well}, the  $C^{1,\alpha}$-norm of the heightfunction and on  $M$.
\end{proposition}

Before the proof we remark that  if the a priori estimates \eqref{eq:apriori_est} hold for $T>0$, then the above assumptions hold for constants $C, C_0$, which depend on $T, \Lambda_T, \sigma_T$ and  $E_1(0)$. Indeed, then by Proposition \ref{prop:claim-reg1} it holds 
\[
\|p\|_{H^1(\Omega_t)}  \leq C.
\]

\begin{proof}
We first recall that by Lemma \ref{lem:press-curv}  we have 
\[
 \|B\|_{L^4(\St)} + \|B\|_{H^{\frac12}(\St)} \leq C\|p\|_{H^1(\Omega_t)}\leq  C.
\]
The claim follows once we prove that for any $\e>0$ it holds 
\beq \label{eq:reg-est-1-1}
 \En_1^+(t) \leq  \En_1(t) + \e E_1(t) +C_\e
\eeq
and
\beq \label{eq:reg-est-1-2}
 E_1(t) \leq C \En_1^+(t).
\eeq

Let us first prove \eqref{eq:reg-est-1-1}. Since $\pa_t^2 U$ is harmonic in $\Omega_t^c$ it holds 
\[
\int_{\Omega_t^c} |\nabla \pa_t^2 U|^2 \, dx \leq C \| \pa_t^2 U\|_{H^{\frac12}(\St)}^2. 
\]
We use  interpolation  (Corollary \ref{coro:interpolation}) to deduce
\beq \label{eq:reg-est-1-3}
\| \pa_t^2 U\|_{H^{\frac12}(\St)} \leq C\| \pa_t^2 U\|_{H^1(\St)}^{\frac12}\| \pa_t^2 U\|_{L^{2}(\St)}^{\frac12}.
\eeq
By \eqref{eq:est-R-U-8} it holds
\beq \label{eq:reg-est-1-4}
\|\pa_t^2 U\|_{H^1(\St)} \leq C(1 + \|p\|_{H^1(\St)})E_1(t)^{\frac12}. 
\eeq
In order to estimate $\| \pa_t^2 U\|_{L^{2}(\St)}$ we use Lemma \ref{lem:mat-capa} and $\|\nabla U\|_{L^\infty}, \|v\|_{L^\infty} \leq C$ and have 
\[
\begin{split}
\| \pa_t^2 U\|_{L^{2}(\St)} &\leq C \|\D_t v\|_{L^{2}(\St)} + C\sum_{|\alpha| \leq 1}\| \nabla^{1+\alpha_1}\pa_t^{\alpha_2} U\|_{L^{2}(\St)}\\
&\leq C\|p\|_{H^1(\St)} + C(1+ \| \nabla^{2}U\|_{L^{2}(\St)} + \| \nabla\pa_t U\|_{L^{2}(\St)}) .
\end{split}
\]
We have by \eqref{eq:est-R-U-5} $ \| \nabla^{2}U\|_{L^{2}(\St)} \leq C(1+ \|p\|_{H^1(\St)})$. 
We use Lemma \ref{lem:divcurl} and $\pa_t U = - \nabla U \cdot v$  to deduce
\[
\begin{split}
\| \nabla\pa_t U\|_{L^{2}(\St)}\leq C(1+\| \nabla_\tau \pa_t U\|_{L^{2}(\St)}) &\leq C\|v\|_{L^\infty} \|\nabla^2  U\|_{L^{2}(\St)} + C \|\nabla U\|_{L^\infty} \|v\|_{H^1(\St)} \\
&\leq C(1+ \|\nabla^2  U\|_{L^{2}(\St)}) \leq C(1+  \|p\|_{H^1(\St)}).
\end{split}
\]
Therefore by  \eqref{eq:reg-est-1-3},  \eqref{eq:reg-est-1-4}, $\|p\|_{L^4(\St)}\leq \|p\|_{H^1(\Omega_t)} \leq C$ and by interpolation we obtain
\[
\begin{split}
\| \pa_t^2 U\|_{H^{\frac12}(\St)}^2 &\leq C +  C\|p\|_{H^1(\St)}^2 E_1(t)^{\frac12} \\
&\leq C + C\|p\|_{H^2(\St)}^{\frac23}\|p\|_{L^4(\St)}^{\frac43} E_1(t)^{\frac12}\\
&\leq \e(\|p\|_{H^2(\St)}^2 + E_1(t)) + C_\e.
\end{split}
\]
Lemma \ref{lem:CS-intermed} yields $\|p\|_{H^2(\St)}^2 \leq C\|\nabla p\|_{H^{\frac32}(\Omega_t)}^2 \leq C E_1(t)$ and \eqref{eq:reg-est-1-1} follows. 

Let us then prove \eqref{eq:reg-est-1-2}. Recall that it holds 
\beq \label{eq:reg-est-1-55}
E_1(t) \leq 2 \En_{1}^+(t) + \|v\|_{H^3(\Omega_t)}^2 + \|\D_t v\|_{H^{\frac32}(\Omega_t)}^2. 
\eeq
By \eqref{eq:pressure-1} we have   $ -\Delta p  = \text{Tr}((\nabla v)^2)$. We use the third inequality in Lemma \ref{lem:CS-intermed} and $\|\bar \nabla \pa_\nu p\|_{L^{2}(\St)}^2 \leq 2 \En_{1}^+(t)$ and have by interpolation 
\beq \label{eq:reg-est-1-5}
\begin{split}
\|\D_t v\|_{H^{\frac32}(\Omega_t)}^2 &= \|\nabla p\|_{H^{\frac32}(\Omega_t)}^2 \leq C (\|\pa_\nu p\|_{H^{1}(\St)}^2 + \|p\|_{L^2(\Omega_t)}^2 + \|\Delta p\|_{H^1(\Omega_t)}^2) \\
&\leq  C (\|\pa_\nu p\|_{H^{1}(\St)}^2 + \|p\|_{L^2(\Omega_t)}^2 + \|\nabla^2 v \|_{L^4(\Omega_t)}^2 \|\nabla v\|_{L^4(\Omega_t)}^2 +1 ) \\
&\leq C_\e \En_1^+(t) +  \e \|v\|_{H^3(\Omega_t)}^2.
\end{split}
\eeq

We proceed  estimating  $\|v\|_{H^3(\Omega_t)}$. By the first inequality in  Lemma \ref{lem:CS-intermed}, by $\| v\|_{L^\infty(\Omega_t)} \leq \| v\|_{W^{1,4}(\Omega_t)}\leq C$ and  by Lemma \ref{lem:press-curv} we have 
\[
\begin{split}
\|v\|_{H^3(\Omega_t)}^2 &\leq C\big(\|\Delta_{\St} v_n\|_{H^{\frac12}(\St)}^2 +  (1+ \|H_{\St}\|_{H^2(\St)}^2)\| v\|_{L^\infty}^2 +  \|\curl v \|_{H^1(\Omega_t)}^2\big) \\
&\leq C\big(\|\Delta_{\St} v_n\|_{H^{\frac12}(\St)}^2  + \|p\|_{H^2(\St)}^2 + \En_1^+(t)\big).
\end{split}
\]
By the fourth inequality in Lemma  \ref{lem:CS-intermed} and by \eqref{eq:reg-est-1-5} we have 
\[
\|p\|_{H^2(\St)}^2 \leq C(\|p\|^2_{L^2(\St)} +\|\nabla p\|_{H^{\frac32}(\Omega_t)}^2) \leq    C_\e \En_1^+(t) +\e\|v\|_{H^3(\Omega_t)}^2.
\]
Therefore by choosing $\e$ small enough we deduce
\beq \label{eq:reg-est-1-7}
\|v\|_{H^3(\Omega_t)}^2 \leq C\|\Delta_{\St} v_n\|_{H^{\frac12}(\St)}^2 + C\En_1^+(t)
\eeq
In order to control $\|\Delta_{\St} v_n\|_{H^{\frac12}(\St)}$ we use  \eqref{eq:D_tp-4} which states
\beq \label{eq:reg-est-1-77}
\D_t p = -\Delta_{\Sigma_t} v_n - Q(t) \nabla U \cdot \nabla \pa_t U + R_p^0,
\eeq
where  $\Qt$ is defined in \eqref{def:constant-q} and
\[
R_p^0 = -(|B|^2  -Q(t)\,  H|\nabla U|^2)v_n + (\nabla p \cdot v_\tau ) -\frac{Q'(t)}{2} |\nabla U|^2.
\]
Therefore we have 
\[
\|\Delta_{\St} v_n\|_{H^{\frac12}(\St)} \leq \|\D_t p\|_{H^{\frac12}(\St)} + \|\nabla U \cdot \nabla \pa_t U\|_{H^{\frac12}(\St)} + \|R_p^0\|_{H^{\frac12}(\St)}.
\]
We estimate the first term on RHS as
\[
\begin{split}
 \|\D_t p\|_{H^{\frac12}(\St)}^2 &\leq C(1+  \|\nabla \D_t p\|_{L^{2}(\Omega_t)}^2) \\
 &\leq  C(1+  \| \D_t \nabla p\|_{L^{2}(\Omega_t)}^2 + \|  [\D_t, \nabla] p\|_{L^{2}(\Omega_t)}^2)\\
 &\leq C(1+  \| \D_t^2 v\|_{L^{2}(\Omega_t)}^2 + \|  \nabla v\|_{L^{4}(\Omega_t)}^2\|  \nabla p\|_{L^{4}(\Omega_t)}^2) \leq C \En_1^+(t).
 \end{split}
\]
By  an already familiar argument we get
\[
\|\nabla U \cdot \nabla \pa_t U\|_{H^{\frac12}(\St)}^2 + \|R_p^0\|_{H^{\frac12}(\St)}^2 \leq \e E_1(t) + C_{\e}.
\]
We leave the details for the reader. Combing the previous three inequalities yield 
\[
\|\Delta_{\St} v_n\|_{H^{\frac12}(\St)}^2 \leq C_{\e} \En_1^+(t) + \e E_1(t).
\]
By combining  \eqref{eq:reg-est-1-55}, \eqref{eq:reg-est-1-5}, \eqref{eq:reg-est-1-7}  with the above inequality and by choosing $\e$ small enough imply    
\[
E_1(t) \leq C\En_1^+(t)
\]
and the claim \eqref{eq:reg-est-1-2} follows. 
\end{proof}

Proposition \ref{prop:reg-est-1} implies that the bound on $\curl v$ and $ \D_t^2 v$ in the fluid domain and on $\D_t v$ on the boundary imply the bound on $v $ and $\D_t v$ in the domain. In the next lemma we show the converse for the initial set $t=0$, i.e., the bound on $v $ in the domain and on the mean curvature $H_{\Sigma_0}$  imply that $E_1(0)$ is bounded. 

\begin{lemma}
    \label{lem:initial-reg}
    Assume that $\Omega_0$ is a smooth set such that $\|h_0\|_{L^\infty(\Sigma)} < \eta$. Then it holds 
    \[
    E_1(0) \leq C_0,
    \]
    for a constant $C_0$ which depends  on $\sigma_0 = \eta - \|h_0\|_{L^\infty(\Sigma)}$,  $\|v\|_{H^3(\Omega_0)}$ on $\|H_{\Sigma_0}\|_{H^2(\Sigma_0)}$ and on $\|h_0\|_{C^{1,\alpha}}$.
\end{lemma}

\begin{proof}
The bound $\|H_{\Sigma_0}\|_{H^2(\Sigma_0)} \leq C$ and  Proposition \ref{prop:meancrv-bound}  imply  $\|B_{\Sigma_0}\|_{H^2(\Sigma_0)} \leq C$. Then we obtain by Theorem \ref{teo:reg-capa} that  $\|\nabla^3 U\|_{H^{\frac12}(\Sigma_0)} \leq C$. Hence, we have 
\beq \label{eq:initial-reg1}
\|p\|_{H^2(\Sigma_0)} \leq C. 
\eeq

Let us show that 
\beq \label{eq:initial-reg2}
\|\pa_\nu p\|_{H^1(\Sigma_0)} \leq C. 
\eeq
Let $\tilde \nu$ be the harmonic extension of the normal field. Note that since 
\[
\|B\|_{C^{\alpha}(\Sigma_0)} \leq C\|B_{\Sigma_0}\|_{H^2(\Sigma_0)} \leq C,
\]
then by standard elliptic regularity theory \cite{GT} we deduce that $\|\nabla \tilde \nu \|_{C^{\alpha}(\Sigma_0)} \leq C$. Then    \eqref{eq:pressure-1}, $\nabla \Delta p = \nabla^2 v \star \nabla v$ and  $\|v\|_{H^3(\Omega_0)}\leq C$ imply that 
\[
\begin{split}
\|\Delta (\nabla p \cdot \tilde \nu)\|_{L^2(\Omega_0)} &\leq C\|\nabla^2 v \star \nabla v\|_{L^2(\Omega_0)} + C\|\nabla^2 p \star \nabla \tilde \nu\|_{L^2(\Omega_0)} \\
&\leq C(1+ \| p\|_{H^2(\Omega_0)}). 
\end{split}
\]
Lemma \ref{lem:poisson1} together with interpolation yields 
\[
\| p\|_{H^2(\Omega_0)} \leq C(1+ \|\pa_\nu p\|_{H^{\frac12}(\Sigma_0)}) \leq \e \|\pa_\nu p\|_{H^1(\Sigma_0)} + C_\e.
\]
Therefore by combing the two estimates with Lemma \ref{lem:divcurl} we obtain
\[
\|\pa_\nu p\|_{H^1(\Sigma_0)} = \| \nabla  p \cdot \tilde \nu\|_{H^1(\Sigma_0)} \leq C_\e(1+ \| \pa_\nu (\nabla  p \cdot \tilde \nu)\|_{L^2(\Sigma_0)})  + \e  \|\pa_\nu p\|_{H^1(\Sigma_0)}. 
\]
Choosing $\e$ small yields
\[
\|\pa_\nu p\|_{H^1(\Sigma_0)}  \leq C(1+ \| \pa_\nu (\nabla  p \cdot \tilde \nu)\|_{L^2(\Sigma_0)}).
\]
Note that by  $\| \nabla \tilde \nu \|_{C^{\alpha}(\Sigma_0)} \leq C$,  Lemma \ref{lem:divcurl} and  by \eqref{eq:initial-reg1}  we have 
\[
\begin{split}
\| \pa_\nu (\nabla  p \cdot \tilde \nu)\|_{L^2(\Sigma_0)} &\leq C( \| (\nabla^2  p \,  \nu \cdot \nu) \|_{L^2(\Sigma_0)} + C\|\nabla  p\|_{L^2(\Sigma_0)})\\ 
&\leq C \| (\nabla^2  p \,  \nu \cdot \nu)  \|_{L^2(\Sigma_0)} + C(1 + \|p\|_{H^1(\Sigma_0)})\\
&\leq  C(1+ \| (\nabla^2  p \,  \nu \cdot \nu) \|_{L^2(\Sigma_0)}).
\end{split}
\]
Therefore since 
\[
\Delta_{\Sigma_0} p = \Delta p -  (\nabla^2  p \,  \nu \cdot \nu)  - H_{\Sigma_0} \pa_\nu p
\]
we obtain by  \eqref{eq:initial-reg1} and by $\|\pa_\nu p\|_{L^2(\St)} \leq C(1+  \| p\|_{H^1(\St)}) \leq C$ that 
\[
\|  (\nabla^2  p \,  \nu \cdot \nu)\|_{L^2(\Sigma_0)} \leq C( 1+  \|p\|_{H^2(\Sigma_0)})\leq C. 
\]
Thus we have \eqref{eq:initial-reg2} by the three  inequalities above. 

We estimate $\|\nabla p\|_{H^{\frac32}(\Omega_0)}$ similarly.  We use  \eqref{eq:initial-reg2} and Lemma \ref{lem:CS-intermed}  to estimate 
\beq \label{eq:initial-reg3} 
\|\nabla p\|_{H^{\frac32}(\Omega_0)} \leq C(\|\pa_\nu p\|_{H^1(\Sigma_0)} + \|p \|_{L^2(\Omega_0)} + \|\Delta p \|_{H^1(\Omega_0)}) \leq C.  
\eeq

In order to show that $\|\D_t^2 v\|_{L^2(\Omega_0)}$ is bounded we first observe that by \eqref{eq:comm1} and by \eqref{eq:initial-reg3}  we have 
\[
\|\D_t^2 v\|_{L^2(\Omega_0)} \leq \|\nabla \D_t p\|_{L^2(\Omega_0)} + \|\nabla v \star \nabla p\|_{L^2(\Omega_0)} \leq \|\nabla \D_t p\|_{L^2(\Omega_0)} + C. 
\]
Recall that we define the $H^{\frac12}(\Sigma_0)$-norm using harmonic extension.  Then it holds 
\[
\|\nabla \D_t p\|_{L^2(\Omega_0)} \leq C(\|\D_t p\|_{H^{\frac12}(\Sigma_0)} +\| \D_t p\|_{L^2(\Omega_0)} +  \|\Delta \D_t p\|_{L^2(\Omega_0)}). 
\]
Note that it holds $\| \D_t p\|_{L^2(\Omega_0)}  \leq C \|\D_t p\|_{H^{\frac12}(\Sigma_0)}$. By Remark \ref{rmk:laplpreassure} and  Lemma \ref{lem:curl}  we have   
\[
\begin{split}
 \|\Delta \D_t p\|_{L^2(\Omega_0)} &\leq C\|R_{\Div}^1\|_{L^2(\Omega_0)} + C\|R_{bulk}^0\|_{H^1(\Omega_0)}  \\
 &\leq C(1+ \|p\|_{H^2(\Omega_0)}+ \|v\|_{H^2(\Omega_0)} )\leq C. 
 \end{split}
\]
We proceed by using \eqref{eq:form-Dtp} to write 
\[
\D_t p = - \Delta_{\Sigma_0} v \cdot \nu -2 B :\nabla_\tau v - Q(0)(D_t \nabla U \cdot \nabla U) - \frac{Q'(0)}{2} |\nabla U|^2.
\]
We only bound the first term on RHS as the others are lower order. By the $C^{1,\alpha}$-regularity of $\nu$  we immediately  estimate 
\[
\|\Delta_{\Sigma_0} v \cdot \nu \|_{H^{\frac12}(\Sigma_0)} \leq C \|\Delta_{\Sigma_0} v \|_{H^{\frac12}(\Sigma_0)} \leq C\|v\|_{H^3(\Omega_0)}.
\]
This concludes the proof. 
\end{proof}

Let us next prove the higher order version of Proposition \ref{prop:reg-est-1}. 

\begin{proposition}
    \label{prop:reg-est-l}
    Let $l \geq 2$ and assume that  \eqref{eq:apriori_est} and $E_{l-1}(t) \leq M$ hold for all $t \in [0,T)$.
     Then there are constants $C$ and $C_0$ such that
    \[
    E_l(t) \leq C( C_0+ \En_l(t)), 
\]
where the constants $C$ and $C_0$ depend on $M, l$ and  $T$.  
\end{proposition}

\begin{proof}
We recall that by the definition of $E_l(t)$ in \eqref{def:E_l}, $\En_l(t)$ in \eqref{eq:high-energy} and of $\En_l^+(t)$ it holds 
\[
 c\, \En_l^+(t)\leq \En_l(t) +  \frac{Q(t)}{2} \int_{\Omega_t^c}|\nabla (\pa_t^{l+1} U)|^2 \, dx + 1
\]
for $c>0$,  $\Qt$ defined in \eqref{def:constant-q}, and
\begin{equation}
    \label{eq:ger-est-22}     
 E_{l}(t) \leq 2 \mathcal{E}_{l}^+(t) + \sum_{k=1}^{l} \|\D_t^{l+1-k} v\|_{H^{\frac32 k}(\Omega_t)}^2 + \| v \|_{H^{\lfloor \frac32(l +1)\rfloor}(\Omega_t)}^2.
 \end{equation}
The claim follows once we prove that for any $\e>0$ it holds 
\beq \label{eq:reg-est-l-1}
 \En_l^+(t) \leq  \En_l(t) + \e E_l(t) +C_\e
\eeq
and
\beq \label{eq:reg-est-l-2}
 E_l(t) \leq C \En_l^+(t).
\eeq

In order to prove \eqref{eq:reg-est-l-1} we use the fact that $\pa^{l+1} U$ is  harmonic in $\Omega_t^c$, interpolation (Corollary \ref{coro:interpolation}) and Lemma \ref{lem:estimate-R-U} and have 
\[
\begin{split}
\int_{\Omega_t^c} |\nabla \pa_t^{l+1} U|^2 \, dx &\leq C \| \pa_t^{l+1} U\|_{H^{\frac12}(\St)}^2 \leq C \| \pa_t^{l+1} U\|_{H^{1}(\St)}\| \pa_t^{l+1} U\|_{L^{2}(\St)}\\
&\leq C E_l(t)^{\frac12} \| \pa_t^{l+1} U\|_{L^{2}(\St)} \leq \e_1 E_l(t) + C_{\e_1} \| \pa_t^{l+1} U\|_{L^{2}(\St)}^2.
\end{split}
\]
We use Lemma \ref{lem:mat-capa},  $\|\nabla U\|_{L^\infty} \leq C$,  Lemma \ref{lem:estimate-R-U} and  the assumption $E_{l-1}(t) \leq \tilde C$ to deduce
\[
\| \pa_t^{l+1} U\|_{L^{2}(\St)} \leq \|\nabla U \cdot \D_t^l v\|_{L^{2}(\St)} + \|R_U^{l-1}\|_{L^{2}(\St)} \leq C\|\D_t^l v\|_{L^{2}(\St)} + C.
\]
By the Trace Theorem, by interpolation  (Corollary \ref{coro:interpolation}) and by the definition of $E_l(t)$ it holds
\[
\begin{split}
\|\D_t^l v\|_{L^{2}(\St)}^2 &\leq C \|\D_t^l v\|_{H^{1}(\Omega_t)}^2\leq C\|\D_t^l v\|_{H^{\frac32}(\Omega_t)}^{\frac43} \|\D_t^l v\|_{L^{2}(\Omega_t)}^{\frac23} \\
&\leq  C E_l(t)^{\frac23} E_{l-1}(t)^{\frac13} \leq  \e_2 E_l(t) + C_{\e_2}.
\end{split}
\]
By choosing first $\e_1$ and then $\e_2$ small implies  \eqref{eq:reg-est-l-1}. 

Let us then prove \eqref{eq:reg-est-l-2}. By \eqref{eq:ger-est-22} we have to bound $\|\D_t^{l+1-k} v\|_{H^{\frac32 k} (\Omega_t)}$ for all $k = 1, \dots, l$ and $\|v\|_{H^{\lfloor\frac32(l+1)\rfloor}(\Omega_t)}$. We  claim first that it holds
 \begin{equation}
     \label{eq:reg-est-l-3}
     \|\D_t^{l} v\|_{H^{\frac32}(\Omega_t)}^2 \leq C \En_{l}^+(t).
 \end{equation}
 Indeed, by Theorem \ref{thm:chen-shkoller}, Lemma \ref{lem:curl} and Lemma \ref{lem:R-div-est} it holds 
 \[
 \begin{split}
 \|&\D_t^{l} v\|_{H^{\frac32}(\Omega_t)}^2 \\
 &\leq C(\|(\D_t^{l} v\cdot \nu)\|_{H^{1}(\Sigma_t)}^2 + \|\D_t^l v\|_{L^{2}(\Omega_t)}^2 + \|\Div \D_t^l v\|_{H^{\frac12}(\Omega_t)}^2 +   \|\curl \D_t^l v\|_{H^{\frac12}(\Omega_t)}^2)\\
 &\leq C(\En_l^+(t) + E_{l-1}(t) +  \|R_{\Div}^{l-1} \|_{H^{\frac12}(\Omega_t)}^2 )\\
 &\leq C(\En_l^+(t) + E_{l-1}(t)) \leq C\En_l^+(t) 
  \end{split}
 \]
 and \eqref{eq:reg-est-l-3}  follows. 
 
Next we claim that for $2 \leq k \leq l$ it holds
 \begin{equation}
     \label{eq:reg-est-l-4}
    \|\D_t^{l+1-k} v\|_{H^{\frac32k}(\Omega_t)}^2 \leq C  \| \D_t^{l+3 -k} v \|_{H^{\frac32k-3}(\Omega_t)}^2  + \e E_l(t) + C_{\e}.
 \end{equation}
This inequality means that two derivatives in time implies regularity for three derivatives in space. We first  use Proposition \ref{prop:vector-high}, Lemma \ref{lem:curl} and  Lemma \ref{lem:R-div-est} to deduce 
 \begin{equation}
    \label{eq:reg-est-l-5}
\begin{split}
   \|\D_t^{l+1-k} &v\|_{H^{\frac32k}(\Omega_t)}^2 \\
   &\leq C \big(\|\Delta_\Sigma (\D_t^{l+1-k} v \cdot \nu)\|_{H^{\frac32k-\frac52}(\Sigma_t)}^2 + \|\D_t^{l+1-k} v\|_{L^2(\Omega_t)}^2 \\
   &\,\,\,\,\,\,\,+ \|\Div (\D_t^{l+1-k} v)\|_{H^{\frac32k -1}(\Omega_t)}^2 + \|\curl (\D_t^{l+1-k} v)\|_{H^{\frac32k -1}(\Omega_t)}^2\big)\\
   &\leq C\big( \|\Delta_\Sigma (\D_t^{l+1-k} v \cdot \nu)\|_{H^{\frac32k-\frac52}(\Sigma_t)}^2 + E_{l-1}(t) +  \|R_{\Div}^{l-k}\|_{H^{\frac32k -1}(\Omega_t)}^2\big)\\
   &\leq  C \|\Delta_\Sigma (\D_t^{l+1-k} v \cdot \nu)\|_{H^{\frac32k-\frac52}(\Sigma_t)}^2 + \e E_{l} + C_\e.
 \end{split}
 \end{equation}
 We proceed by using  Lemma \ref{formula:Dtp} to write
 \begin{equation}
    \label{eq:reg-est-l-6}
 \D_t^{l+2 -k} p = - \Delta_\Sigma  (\D_t^{l+1-k} v \cdot \nu) - Q(t)(\nabla U \cdot \nabla \pa_t^{l+2-k} U) + R_p^{l+1-k}.
 \end{equation}
Lemma \ref{lem:estimate-R-p-2} yields
\begin{equation}
    \label{eq:reg-est-l-7}
  \|R_p^{l+1-k}\|_{H^{\frac32 k -\frac52}(\Sigma_t)}^2 = \|R_p^{l-(k-1)}\|_{H^{\frac32 (k-1) -1}(\Sigma_t)}^2 \leq  \e E_l(t) +  C_{\e}.
 \end{equation}

Next we claim that 
\begin{equation}
    \label{eq:reg-est-l-8}
    \|(\nabla U \cdot \nabla \pa_t^{l+2-k} U)\|_{H^{\frac32k-\frac52}(\Sigma_t)}^2 \leq \e E_l(t) + C_\e.
\end{equation}
If $k =2$ then we use the fact that by the assumption $\|B\|_{H^2(\St)} \leq \|B\|_{H^{\frac32l-1}(\St)} \leq C$ and by Theorem \ref{teo:reg-capa}  the function $U$ is uniformly $C^{2,\alpha}$-regular. Therefore we have by Lemma \ref{lem:reg-capa1} 
\[
\|(\nabla U \cdot \nabla \pa_t^{l} U)\|_{H^{\frac12}(\Sigma_t)}^2 \leq C\| \nabla \pa_t^{l} U\|_{H^{\frac12}(\Sigma_t)}^2\leq  \e E_{l}(t) + C_{\e}.
\]
If $k \geq 3$ then $2 \leq \frac32k-\frac52 \leq \lfloor \frac32 l \rfloor -2$. We have by Proposition \ref{prop:kato-ponce}, by the Sobolev embedding, by Lemma \ref{lem:reg-capa} and  by  Lemma \ref{lem:reg-capa1} that 
\[
\begin{split}
\|(\nabla U \cdot \nabla \pa_t^{l+2-k} &U)\|_{H^{\frac32k-\frac52}(\Sigma_t)}^2 \leq  C\|\nabla U\|_{L^\infty(\St)}^2 \|\nabla \pa_t^{l+2-k} U\|_{H^{\frac32k-\frac52}(\Sigma_t)}^2 \\
&\,\,\,\,\,\,\,\,\,\,\,\,+\|\nabla U\|_{H^{\frac32k-\frac52}(\Sigma_t)}^2 \|\nabla \pa_t^{l+2-k} U\|_{L^\infty(\St)}^2\\
&\leq C(1+ \|\nabla U\|_{H^{\frac32k-\frac52}(\Sigma_t)}^2 )\|\nabla \pa_t^{l+2-k} U\|_{H^{\frac32k-\frac52}(\Sigma_t)}^2 \\
&\leq C(1+ \|p\|_{H^{\lfloor \frac32 l \rfloor -2}(\St)}^2) (\e E_{l}(t) + C_\e).
\end{split}
\]
Hence, \eqref{eq:reg-est-l-8} follows from the Trace Theorem as 
\[
\|p\|_{H^{\lfloor \frac32 l \rfloor -2}(\St)}^2 \leq C(1+\|\nabla p\|_{H^{\frac32 l -2}(\Omega_t)}^2) \leq CE_{l-1}(t) \leq C.  
\]

We deduce by  \eqref{eq:reg-est-l-5},  \eqref{eq:reg-est-l-6}, \eqref{eq:reg-est-l-7}, \eqref{eq:reg-est-l-8}, Lemma \ref{lem:CS-intermed}   and \eqref{eq:comm-bulk}  that 
\[
\begin{split}
&\|\D_t^{l+1-k} v\|_{H^{\frac32k}(\Omega_t)}^2 \\
&\leq C\|\D_t^{l+2 -k} p\|_{H^{\frac32k-\frac52}(\Sigma_t)}^2 + \e E_l(t) +  C_{\e}\\
&\leq C\|\nabla \D_t^{l+2 -k} p\|_{H^{\frac32k-3}(\Omega_t)}^2 + \e E_l(t) +  C_{\e}\\
&\leq C\| \D_t^{l+2 -k} \nabla p\|_{H^{\frac32k-3}(\Omega_t)}^2 +C\| [\D_t^{l+2 -k},\nabla] p\|_{H^{\frac32k-3}(\Omega_t)}^2 +  \e E_l(t) +  C_{\e}\\
&\leq C\| \D_t^{l+3 -k} v \|_{H^{\frac32k-3}(\Omega_t)}^2 +C\|R_{bulk}^{l+1-k}
\|_{H^{\frac32k-3}(\Omega_t)}^2 +  \e E_l(t) +  C_{\e}.
\end{split}
\]
Lemma \ref{lem:R-div-est} implies 
 \[
 \| R_{bulk}^{l+1-k}
\|_{H^{\frac32k-3}(\Omega_t)}^2 \leq \| R_{bulk}^{l-(k-1)}
\|_{H^{\frac32(k-1)-1}(\Omega_t)}^2 \leq \e E_l(t) +  C_{\e}
 \]
 and the estimate \eqref{eq:reg-est-l-4} follows.

Let us then prove 
\begin{equation}
 \label{eq:vctrh-final}
     \| v \|_{H^{\lfloor \frac32(l +1)\rfloor}(\Omega_t)}^2 \leq  C\|\D_t^2 v\|_{H^{\frac32(l-1)}(\Omega_t)}^2+ C \|\D_t v\|_{H^{\frac32l}(\Omega_t)}^2 + \e E_{l}(t) + C_\e \En_l^+(t).
 \end{equation}
 We denote $\lambda_l = \lfloor \frac32(l +1)\rfloor -1$ and  use the second inequality in Proposition \ref{prop:vector-high} and have 
 \[
 \| v \|_{H^{\lfloor \frac32(l +1)\rfloor}(\Omega_t)}^2 \leq C(1+ \|\Delta_{\St} v_n\|_{H^{\lambda_l -\frac32}(\St)}^2 +\|B\|_{H^{\frac32l}(\St)}^2+ \|\curl v \|_{H^{\lambda_l}(\Omega_t)}^2).
 \]
 By the definition of $\En_l^+(t)$ in \eqref{def:Eplus} it holds $\|\curl v \|_{H^{\lambda_l}(\Omega_t)}^2 \leq \En_l^+(t)$. Lemma \ref{lem:press-curv}  and Trace Theorem yield
 \[
 \begin{split}
 \|B\|_{H^{\frac32l}(\St)}^2 &\leq  C(1+ \| p\|_{H^{\lfloor \frac32l +\frac12\rfloor}(\St)}^2) \\
 &\leq C(1+ \|\nabla p\|_{H^{\frac32l}(\Omega_t)}^2) = C (1+\|\D_t v\|_{H^{\frac32l}(\Omega_t)}^2) .
  \end{split}
 \]
 We treat the term $\|\Delta_{\St} v_n\|_{H^{\lambda_l -\frac32}(\St)}$  by using \eqref{eq:reg-est-1-77} and have 
 \[
 \|\Delta_{\St} v_n\|_{H^{\lambda_l -\frac32}(\St)} \leq C\|\D_t p\|_{H^{\lambda_l -\frac32}(\St)} + C\|\nabla U \cdot \nabla \pa_t U\|_{H^{\frac32l -1}(\St)}  + \|R_p^0\|_{H^{\frac32l -1}(\St)}. 
 \]
 By Lemma \ref{lem:reg-capa1} we have 
 \[
  \|\nabla \pa_t U\|_{H^{\frac32l -1}(\St)}^2 \leq \e E_l + C_\e
 \]
 and 
 \[
   \|\nabla U\|_{H^{\frac32l -1}(\St)}^2 \leq CE_{l-1}(t) \leq C.
 \]
 Therefore we have by Proposition \ref{prop:kato-ponce} and by the Sobolev embedding 
 \[
 \|\nabla U \cdot \nabla \pa_t U\|_{H^{\frac32l -1}(\St)}^2 \leq C\|\nabla U\|_{H^{\frac32l -1}(\St)}^2   \|\nabla \pa_t U\|_{H^{\frac32l -1}(\St)}^2  \leq  \e E_l + C_\e.
 \]
 Similarly we obtain 
 \[
 \|R_p^0\|_{H^{\frac32l -1}(\St)}^2 \leq \e E_l + C_\e.
 \]
 We leave the details for the reader. Therefore we have by arguing as before 
 \[
 \begin{split}
   \| v \|_{H^{\lfloor \frac32(l +1)\rfloor}(\Omega_t)}^2 &\leq C\|\D_t p\|_{H^{\lambda_l -\frac32}(\St)} +    \e E_l + C_\e\En_l^+(t)\\
   &\leq C\|\nabla \D_t p\|_{H^{\lambda_l -2}(\Omega_t)}^2 +    \e E_l + C_\e\En_l^+(t)\\
   &\leq C\|\D_t \nabla  p\|_{H^{\lambda_l -2}(\Omega_t)}^2 + \|[\D_t, \nabla]  p \|_{H^{\lambda_l -2}(\Omega_t)}^2 +    \e E_l + C_\e\En_l^+(t)\\
   &\leq C\|\D_t^2 v \|_{H^{\lambda_l -2}(\Omega_t)}^2 + \|\nabla v \star \nabla p\|_{H^{\lambda_l -2}(\Omega_t)}^2 +    \e E_l + C_\e\En_l^+(t).
 \end{split}
 \]
Note that  $\lambda_l -2  \leq  \frac32(l-1)$ and $\lambda_l -1 \leq \lfloor \frac32 l\rfloor$. Thus  by the definition of $E_{l-1}(t)$  it holds  
\[
\| \nabla p\|_{H^{\lambda_l -2}(\Omega_t)}^2 +  \|\nabla v\|_{H^{\lambda_l -2}(\Omega_t)}^2 \leq  \| \D_t v \|_{H^{\frac32(l-1)}(\Omega_t)}^2 + \|v\|_{H^{\lfloor \frac32 l\rfloor}(\Omega_t)}^2 \leq C E_{l-1}(t)\leq C.
\]
Proposition \ref{prop:kato-ponce}, the assumption  $ \| \nabla v \|_{L^\infty(\Omega_t)} \leq C$, and  the Sobolev embedding  then  imply
 \[
 \|\nabla v \star \nabla p\|_{H^{\lambda_l -2}(\Omega_t)}^2 \leq C E_{l-1}^2(t)\leq C   
 \]
and the inequality \eqref{eq:vctrh-final} follows.

We deduce by \eqref{eq:reg-est-l-3}, \eqref{eq:vctrh-final}    and by using  \eqref{eq:reg-est-l-4} an iterative way that  
\[     
 \sum_{k=1}^{l} \|\D_t^{l+1-k} v\|_{H^{\frac32 k}(\Omega_t)}^2 + \| v \|_{H^{\lfloor \frac32(l +1)\rfloor}(\Omega_t)}^2 \leq C_\e \En_{l}^+(t) + \e E_{l}(t). 
 \]
Thus we obtain   \eqref{eq:reg-est-l-2} by using  the above inequality and  \eqref{eq:ger-est-22}. 
\end{proof}

\section{Proof of the Main Theorem}

In this short section we collect the results from Section 6, Section 7 and Section 8 and prove the Main Theorem. The proof is fairly straightforward, and the only delicate part is to show that the a priori estimates \eqref{eq:apriori_est} hold for a short time. 
\begin{proof}[\textbf{Proof of the Main Theorem}]
 Let us assume that the quantities $\Lambda_T$ and $\sigma_T$, which are defined in \eqref{def:apriori_est} and \eqref{eq:height_well} respectively,  satisfy   $\Lambda_T \leq M $ and $\sigma_T\geq \frac{1}{M}$ for $T>0$. We show that this implies the bound
 \beq \label{eq:mainthm-1}
E_l(t) \leq C_l \qquad \text{for all }\, t \leq T 
 \eeq
for every positive integer $l$, where the constant $C_l$ depends on $l, T, M$ and on $E_l(0)$. Here the dependence on $T$ means that if $T<1$, then the constant $C_l$ may be chosen to be independent of $T$.   The estimate \eqref{eq:mainthm-1} is crucial as it quantifies the smoothness of the flow under the assumption that  the a priori estimates are bounded.

 We obtain first by Lemma \ref{lem:claim2} that
 \beq \label{eq:mainthm-11}
 \int_0^T \|p\|_{H^2(\Omega_t)}^2 \, dt \leq \tilde C,
 \eeq
 where $\tilde C$ depends on $T,M$ and on $E_1(0)$.  Proposition \ref{prop:case-1} and Proposition \ref{prop:reg-est-1} in turn  imply 
  \beq \label{eq:mainthm-22}
 \begin{split}
 \frac{d}{dt} \En_1(t) &\leq C(1+ \|p\|_{H^2(\Omega_t)}^2)E_1(t)\\
 &\leq C(1+ \|p\|_{H^2(\Omega_t)}^2)(C_0 + \En_1(t))
  \end{split}
 \eeq
 for all $t \leq T$. In particular, the quantity $C_0 + \En_1(t)$ is positive. Therefore we obtain by integrating over $(0,T)$ and using \eqref{eq:mainthm-11} that
\[
C_0 + \En_1(t) \leq \hat C(C_0 + \En_1(0))
\]
 for all $t \leq T$. By using  Proposition \ref{prop:reg-est-1} again we have 
\[
 E_1(t) \leq C(C_0 + \En_1(t)) \leq C \hat C(C_0 + \En_1(0)) \leq C_1,
 \]
 where the constant $C_1 $ depends on $M, T$ and  $E_1(0)$. 
 
 We may then use Proposition \ref{prop:case-l} and Proposition \ref{prop:reg-est-l} in an  inductive way and deduce that if $E_{l-1}(t) \leq C_{l-1}$ for $t \leq T$ then  it holds 
 \[
 \frac{d}{dt} \En_l(t) \leq CE_l(t) \leq C(C_0 + \En_l(t)).
 \]
  By integrating we deduce
 \[
 C_0 + \En_l(t) \leq (C_0 + \En_l(0))e^{C T}
 \]
 and using Proposition \ref{prop:reg-est-l} again we have
  \beq \label{eq:mainthm-3}
 E_l(t) \leq C(C_0 + \En_l(0))e^{C T} \leq C_l,
 \eeq
 where the constant $C_l$ depends on $l, T, M$ and on $E_{l}(0)$. Note that we obtain \eqref{eq:mainthm-3} under the assumption $E_{l-1}(t) \leq C_{l-1}$ for $t \leq T$ and thus an induction argument implies that \eqref{eq:mainthm-3} holds for all $l$ for a constant which depends on $T, l, M$ and on $E_l(0)$. Therefore we have  \eqref{eq:mainthm-1}.

 Let us then prove the last  claim, i.e., that the  a priori estimates \eqref{eq:apriori_est} hold for $M$  for a short time 
 \beq \label{eq:mainthm-4}
 T_0 \geq c_0
 \eeq
 for a positive constant $c_0$ which depends on $\|H_{\Sigma_0}\|_{H^2(\Sigma)}$, $\|v\|_{H^3(\Omega_0)}$ and on $\sigma_0$.  
 
 To this aim we define the quantity 
 \[
 \lambda_t := \|\nabla p\|_{L^2(\Omega_t)}^2 + \|B\|_{L^4(\St)}^4 + \|\nabla v\|_{L^4(\St)}^4 + \|\nabla v\|_{L^4(\Omega_t)}^4 +1, 
 \]
 where $p$ is the pressure and $v$ the velocity field. Let us also denote by 
\[
\delta(t) : = d_{\mathcal{H}}(\Omega_t, \Omega_0)
\]
the Hausdorff distance between the sets  $\Omega_t$ and $\Omega_0$. The point is that if we would know that it holds $\lambda_t \leq 2\lambda_0$ and $\delta(t) \leq \e_0$, where   $\lambda_0$ is the value at time $t=0$ and $\e_0$ is a small number, then we have by the curvature bound and by standard argument from regularity theory (e.g. by Allard regularity theory)  that $\St$ is uniformly $C^{1,\alpha}(\Gamma)$-regular. We choose the number $\e_0$ such that it depends also on $\sigma_0$ so that  $\delta(t) \leq \e_0$  implies $\sigma_t \geq \frac{\sigma_0}{2}$. Moreover,  by Proposition \ref{prop:reg-est-1} we deduce that there are constants $C$ and $C_0$ such that 
 \beq\label{eq:mainthm-7}
 E_1(t) \leq C(C_0 + \En_1(t)).
 \eeq
 
 Let us then define  $T_0 \in (0, T]$ to be the largest number such that  
 \[
 \begin{split}
 \sup_{t \leq T_0} \lambda_t \leq 2 \lambda_0 , \quad 
\sup_{t \leq T_0 } \delta(t) \leq \e_0  \quad \text{and} \quad  \sup_{t \leq T_0 }  \En_1(t) \leq C_0 + \En_1(0), 
 \end{split}
 \]
 where $C_0$ is the constant in \eqref{eq:mainthm-7}. We note that the last condition together with  \eqref{eq:mainthm-7} implies that 
 \beq\label{eq:mainthm-8}
 E_1(t) \leq C(C_0 + \En_1(t))\leq C(2C_0 + \En_1(0)) \leq \tilde C E_1(0), 
 \eeq
for $t \leq T_0$. It is also easy to see that for $\Lambda_{T}$ defined in \eqref{def:apriori_est} it holds $\Lambda_{T}^2 \leq C \sup_{t \leq T} E_1(t)$. This means that  \eqref{eq:mainthm-8} ensures that the a priori estimates \eqref{eq:apriori_est} hold for the time interval $[0,T_0]$.  Therefore it is enough to show that $T_0 \geq c_0$. We may assume that $T_0 < \min \{T,1\}$ since otherwise the claim is trivially true. 

If $T_0 < \min \{T,1\}$ then at least in one of the three conditions in the definition of $T_0$ we have an equality. Assume that $\lambda_{T_0} = 2 \lambda_0$. Note that by \eqref{eq:mainthm-8} it holds $E_1(t) \leq  \tilde C E_1(0)$ for all $t \leq T_0$. We remark that it holds 
\[
\|B\|_{L^\infty(\St)}^2 + \|\nabla v\|_{L^\infty(\Omega_t)}^2 \leq CE_1(t).
\]
Moreover, by using the formula \eqref{lem:comm:eq3} in Lemma \ref{lem:comm2} we obtain   
\[
\|\D_t \nabla p\|_{L^2(\Omega_t)} +  \|\D_t B\|_{L^2(\St)}  + \|\D_t \nabla v\|_{L^2(\St)} \leq CE_1(t).
\]
We leave the details for the reader. Therefore by a straightforward calculation we deduce that  for some $q \geq 1$ it holds 
\[
\frac{d}{dt} \lambda_t \leq C E_1(t)^q \leq C E_1(0)^q
\]
where the last inequality follows from \eqref{eq:mainthm-8}. By integrating the above over $(0,T_0)$ and using $\lambda_{T_0} = 2 \lambda_0$ we obtain 
\[
\lambda_0 \leq C E_1(0)^q T_0. 
\]
Since $\lambda_0 \geq 1$ we have $T_0 \geq c_0$, for a constant that depends on $E_1(0)$ and $\sigma_0$. 

We argue similarly if we have an equality in the third condition in the definition of $T_0$, i.e.,   $ \En_1(T_0) = C_0 + \En_1(0)$. Indeed, then by the definition of $E_1(t)$ we have that 
\[
 \| p\|_{H^2(\Omega_t)}^2 \leq   E_1(t).
\]
Therefore we obtain by \eqref{eq:mainthm-22} and \eqref{eq:mainthm-8}  that 
\[
\frac{d}{dt} \En_1(t) \leq C(1+ \|p\|_{H^2(\Omega_t)}^2)E_1(t) \leq C,
\]
 where the constant $C$ depends on $E_1(0)$ and on $\sigma_0$. We integrate the above over $(0,T_0)$ and
 obtain 
 \[
C_0 =  \En_1(T_0) -  \En_1(0) \leq CT_0. 
 \]
Thus we have again $T_0 \geq c_0$.

Finally assume that it holds  $ \delta(T_0) = \e_0$. By definition the flow gives a  diffeomorphism  $\Phi_{T_0} : \Sigma_0 \to \Sigma_{T_0}$. We note that the velocity is uniformly bounded by the Sobolev embedding and by \eqref{eq:mainthm-8}
  \[
  \|v\|_{L^\infty(\Omega_t)}^2 \leq CE_1(t) \leq C E_1(0).
  \]
 Therefore we have by the fundamental Theorem of Calculus  that for every $x \in  \Sigma_0$ it holds
  \[
 |\Phi_{T_0}(x) - x| \leq   \int_0^{T_0} \|v\|_{L^\infty}\, dt \leq   C  T_0. 
  \]
Since $\sup_{x \in \Sigma_0}|\Phi_{T_0}(x) - x|\geq \delta(T_0) \geq \e_0$,  we again have $T_0 \geq c_0$. 
  
  We have thus obtained \eqref{eq:mainthm-4} for a constant $c_0$ which depends on $\sigma_0$ and $E_1(0)$. By Lemma \ref{lem:initial-reg} we deduce that $c_0$ in fact depends on $\sigma_0, \|v\|_{H^3(\Omega_0)}$,  $\|H_{\Sigma_0}\|_{H^2(\Sigma_0)}$ and on $\|h_0\|_{C^{1,\alpha}(\Gamma)}$. This concludes the proof of the second claim. 
  
\end{proof}

\section{Statements and Declarations}

\subsection*{Funding and competing interests}
The research was supported by the Academy of Finland grant 314227. The authors have no non-financial competing interests to declare that are relevant to the consent of this article.

\subsection*{Data availability statement}
Data sharing not applicable to this article as no datasets were generated or analysed during the current study.

\end{document}